\documentclass[11pt]{amsart}

\usepackage[utf8]{inputenc}
\usepackage[T1]{fontenc}
\usepackage[french, english]{babel}

\usepackage{amsthm}
\usepackage{amsmath}
\usepackage{amssymb}
\usepackage{mathrsfs}
\usepackage{bbm}
\usepackage{graphicx}
\usepackage{listings}
\usepackage{mathrsfs}
\usepackage{enumerate}
\usepackage{pict2e}
\usepackage{stmaryrd}
\usepackage{mathtools}

\usepackage[all,cmtip]{xy}

\usepackage{multirow}

\usepackage{color}
\definecolor{marin}{rgb}   {0.,   0.1,   0.5} 
\definecolor{rouge}{rgb}   {0.8,   0.,   0.} 
\definecolor{sepia}{rgb}   {0.4,   0.25,   0.} 
\definecolor{mag}{rgb}   {0.3,   0,   0.3} 

\usepackage[colorlinks,citecolor=sepia,linkcolor=marin,urlcolor=mag ,bookmarksopen,bookmarksnumbered]{hyperref}

\newtheorem{theorem}{Theorem}[section]
\newtheorem{corollary}[theorem]{Corollary}
\newtheorem{lemma}[theorem]{Lemma}
\newtheorem{proposition}[theorem]{Proposition}
\newtheorem{definition}[theorem]{Definition}
\newtheorem{remark}[theorem]{Remark}

\newcommand{\ic}{\mathrm{i}}

\newcommand{\dd}{\mathrm{d}}

\newcommand{\Z}{\mathbb{Z}}
\newcommand{\cS}{\mathcal{S}}
\newcommand{\cSc}{\mathcal{S}^c}
\newcommand{\N}{\mathbb{N}}

\newcommand{\R}{\mathbb{R}}

\newcommand{\T}{\mathbb{T}}

\newcommand{\eps}{\varepsilon}

       \renewcommand{\S}{{\mathcal S}}

\begin{document}
\title[Infinite dimensional invariant tori for NLS]{Infinite dimensional invariant tori for nonlinear Schr\"odinger equations}
%\title[Almost periodic solutions to NLS]{Existence of infinite dimensional invariant tori for analytic nonlinear Schr\"odinger equations without external parameters}

\author{Joackim Bernier}

\address{\small{Nantes Universit\'e, CNRS, Laboratoire de Math\'ematiques Jean Leray, LMJL,
F-44000 Nantes, France
}}

\email{joackim.bernier@univ-nantes.fr}

\author{Beno\^it Gr\'ebert}

\address{\small{Nantes Universit\'e, CNRS, Laboratoire de Math\'ematiques Jean Leray, LMJL,
F-44000 Nantes, France
}}

\email{benoit.grebert@univ-nantes.fr}

\author{Tristan Robert}

\address{\small{Institut  \'Elie Cartan, Universit\'e de Lorraine, B.P. 70239, F-54506 Vand\oe uvre-l\`es-Nancy Cedex, France
}}

\email{tristan.robert@univ-lorraine.fr}



\begin{abstract} We prove that nonlinear Schr\"odinger equations on the circle, without external parameters, admit plenty of almost periodic solutions. 
Indeed, we prove that arbitrarily close to most of the  finite dimensional KAM tori constructed by Kuksin--Poschel in \cite{KP96}, there exist infinite dimensional  non resonant Kronecker  tori, i.e. rotational invariant tori. 
This result answers a natural and longstanding question, well identified by the Hamiltonian PDE community since the first KAM-type result for PDEs in \cite{Kuk87}.
\end{abstract} 
\maketitle

\setcounter{tocdepth}{1} 
\tableofcontents
\setcounter{tocdepth}{2}

\section{Introduction}

 We consider the  nonlinear Schr\"odinger equation on the circle $\mathbb T:=\R/2\pi\Z$
\begin{equation}
\tag{NLS}
\label{eq:NLS}
\ic \partial_t u + \partial_x^2 u = f(|u|^2)u, \quad x\in \mathbb{T},\ t\in\R
\end{equation}
where $f$ is a real entire function such that $f(0)=0$, $f'(0)\neq0$  (i.e. the cubic term is present) and $z\mapsto f(z^2)$ grows at most exponentially fast\footnote{i.e. there exists $C>0$ such that for all $z\in \mathbb{C}$, $|f(z)|\leq C \exp(C\sqrt{|z|})$.}. For instance, $f(|u|^2)u = (\sinh \, |u|)^2 u$ or $f(|u|^2)u = |u|^2 u + |u|^4 u$
are possible choices of nonlinearity. We point out that \eqref{eq:NLS} is a classical and emblematic Hamiltonian system whose structure is recalled in Subsection \ref{sub:ham_form} below.

\medskip

We are looking for  solutions with Sobolev or analytic regularity in the space variable, i.e. we consider solutions living in the phase space
$$
\ell^2_{\varpi} := \big\{ u \in L^2(\mathbb{T}) \ | \ \| u\|_{\ell_\varpi^2}^2 := \sum_{k\in \mathbb{Z}} |u_k|^2 \varpi_k^2  < \infty\big \}
$$
where the sequence of weights $(\varpi_k)_{k\in\Z}$ is given by 
$$
\varpi_k = \langle k \rangle^{s} e^{\mathfrak a|k|} \quad \mathrm{with} \quad \mathfrak a\geq 0 \quad \mathrm{and} \quad s\geq 1 
$$ 
and we always identify functions $ u \in L^2(\mathbb{T})$ with their Fourier coefficients 
\begin{equation}
\label{eq:def_Fourier}
 u_k:=(2\pi)^{-1}\int_{\mathbb{T}}u(x)e^{ikx}\dd x, \quad k\in \mathbb{Z}.
 \end{equation}
The phase space $\ell^2_{\varpi}$ is foliated by invariant tori for the linear part of the equation \eqref{eq:NLS}. Namely
for all $\xi\in (\mathbb{R}_+)^\mathbb{Z}$ the torus
\begin{equation}
\label{eq:def_T_droit}
\mathrm T_\xi :=\{ u\in \ell^2_{\varpi}\mid \forall k\in\mathbb{Z}, \quad  |u_k|^2=\xi_k\, \},
\end{equation}
is an invariant set for the system
\begin{equation}
\label{eq:LS}
\ic \partial_t u = \partial_x^2 u .
\end{equation}
The dynamics on these tori is trivial, it consists of parallel translations. Such tori play an important role and are often called Kronecker tori or rotational tori  (see for instance \cite{KP03}).  In this linear case the velocities of translation, or frequencies, are $\omega_j=j^2$, $j\in\Z$. This set of frequencies is totally resonant : the solutions are periodic.

\medskip

When we consider small solutions of \eqref{eq:NLS}, the nonlinear part is perturbative.
Then a natural question is:
\begin{center}
\emph {Which linear tori persist, slightly deformed, as invariant sets of \eqref{eq:NLS}?}
 \end{center}
 Kuksin--P\"oschel  proved in \cite{KP96} that most of the finite dimensional tori, close to the origin, persist as KAM tori (i.e. Kronecker tori that are linearly stable). Unfortunately, their result suffers a limitation : the larger the dimension of the torus is, the closer to the origin it has to be. This limitation is not specific to the work of Kuksin--P\"oschel, it appears in all the papers about existence of finite dimensional invariant tori. Thus, up to now, there have been no results concerning the existence of infinite dimensional invariant tori to  \eqref{eq:NLS} and, more generally, for (non-integrable) Hamiltonian PDEs without external parameter. Let us define the kind of tori we are going to prove the persistency of:

\begin{definition}[Non resonant infinite dimensional Kronecker tori] A subset $\mathcal{T}$ of $\ell^2_\varpi$ is a \emph{non resonant infinite dimensional Kronecker torus} (for \eqref{eq:NLS}) if there exists a homeomorphism\footnote{Naturally, we always equip $\mathbb{T}^\mathbb{N}$ with the product topology.} $\Psi : \mathbb{T}^\mathbb{N} \to \mathcal{T}$ and a sequence of rationally independent numbers $\omega \in \mathbb{R}^{\mathbb{N}}$ such that
 \begin{center}
 $\forall \theta\in \mathbb{T}^{\mathbb{N}}$, $t\mapsto \Psi( \omega t +\theta)$ is a global  solution to  \eqref{eq:NLS}.
\end{center}
\end{definition}
Since the frequencies are rationally independent (or non resonant), i.e. $k\cdot\omega\neq 0$ for all $0\neq k\in\Z^\N$ of finite support, on these tori each orbit is dense and the flow is ergodic (for the product measure). Our main result can be summarized in a very short theorem:
\begin{theorem} \label{thm:0} There exists non resonant infinite dimensional Kronecker tori for \eqref{eq:NLS}.
\end{theorem}
Actually, in Theorem \ref{thm:main} below, we construct a whole family of non resonant infinite dimensional Kronecker tori which accumulates on the finite dimensional Kronecker tori constructed by Kuksin--P\"oschel in \cite{KP96}.

\medskip

Another classical way of considering this kind of problems, consists in looking for almost periodic solution (which are not quasi periodic). Up to now, there have been no results concerning the existence of such solutions to  \eqref{eq:NLS} and more generally for not completely integrable Hamiltonian PDEs without external parameter.
\begin{definition}[Quasi-periodic  function] A continuous function $q:\R\to\R$ is called quasi-periodic with frequencies $\omega=(\omega_1,\cdots,\omega_n)$, if there exists a continuous function, $Q:\mathbb T^n\to \R$, such that $q(t)=Q(\omega t)$ for all $t\in\R$.
\end{definition}
Then following H. Bohr (see \cite{Boh47}) we  recall the definition of almost-periodic functions
\begin{definition}[Almost-periodic function] A continuous function  is called almost-periodic if it is a uniform limit of quasi-periodic functions.
\end{definition}
Of course a function leaving on an infinite dimensional torus could be quasi-periodic or even periodic, depending on resonances. In Corollary \ref{cor:ap_functions} of the Appendix, we prove that each orbit on a non resonant infinite dimensional Kronecker torus corresponds to an almost periodic solution which is not quasi-periodic. Therefore, as a corollary of Theorem \ref{thm:0}, we get:
\begin{corollary}
There exist almost periodic solutions to \eqref{eq:NLS} which are not quasi periodic.
\end{corollary}

\subsection{Context}\label{sec:context} For several decades now, many mathematicians have been interested in KAM theory for nonlinear Hamiltonian PDEs. The aim is to show the existence of invariant sets for such PDEs, which is quite natural for conservative dynamical systems.  Typically, the flow of the linear part of the PDE foliates the phase space (of infinite dimension) into invariant tori of different dimensions - finite or infinite. As for \eqref{eq:NLS}, the general question is whether these tori persist, slightly deformed, in the presence of nonlinear perturbations. From a more PDE point of view, finite-dimensional invariant tori correspond to quasi-periodic in time solutions, while infinite-dimensional tori correspond to almost periodic in time solutions.

\medskip 

For finite dimensional Hamiltonian system, standard KAM theory ensures that, under some non-degeneracy assumptions, most tori survive and thus most of the solutions are quasi-periodic (see e.g. \cite{KP03}). 
 Naturally, in an infinite-dimensional phase space, infinite-dimensional invariant sets are expected. In other words, for PDEs, \emph{we expect almost periodic solutions to be typical}.
 % as in the linear case or in the case of integrable PDEs such as Korteweg de-Vries,  cubic nonlinear Schr\"odinger  or  Benjamin--Ono equations (see \cite{KP03, GK14, GK21}).  

\medskip 
 
The construction of invariant tori of infinite dimension  was first achieved for certain models with short range interaction (see \cite{FSW86, Pos90}). But clearly  the interactions in nonlinear PDEs (in Fourier variables) are long range. Then Cherchia--Perfetti  \cite{CP95}  constructed invariant tori of infinite dimension for Hamiltonian systems with long range interactions but they require a very rapidly increasing frequencies. Again this does not correspond to standard PDEs.
On the other hand, there exists some completely integrable Hamiltonian PDEs (such as Korteweg de-Vries,  cubic nonlinear Schr\"odinger  or  Benjamin--Ono equations) for which it has been proven that all the linear tori survive (see \cite{KP03, GK14, GK21}). Nevertheless, although very interesting, these systems are very specific and, as for finite dimensional systems, this property is not expected to be true in presence of perturbations.

\medskip

Concerning non integrable Hamiltonian PDEs, many results have been obtained, but mainly concerning the existence of finite-dimensional invariant tori (see e.g. \cite{Kuk87, Way90, CW93, Bou94, Pos96, KP96} for the earliest results for 1d PDEs, \cite{Bou98, EK10, BB13, PP15,  EGK16, GP16} for some results in higher dimension, \cite{BBHM18, FG24, BHM23} for some results about quasi-linear equations and \cite{BKM18} for large amplitude solutions) and often at the cost of adding external parameters to the equation to simplify the treatment of the small denominators inherent in KAM theory. Most of these results allow to get invariant tori of arbitrarily large dimension. Nevertheless, in the proofs, the larger the dimension is, the smaller the perturbation parameter has to be. Therefore, they say nothing about the existence of infinite dimensional invariant tori.

\medskip

 The desire to show the existence of almost periodic solutions of \eqref{eq:NLS} or other Hamiltonian PDEs is obviously not new. The grail would be to show that these almost periodic solutions represent the typical behavior of small solutions. We are still a long way from that at best, with the notable exception of the recent  paper \cite{BCGP24}, it has been proved the existence of some almost periodic solutions and until now only for PDEs with external parameters. The first result in a PDE context is due to Bourgain in \cite{Bou96} who proved the existence of almost periodic solutions to the wave equation in presence of a random potential. Then, few years later, in \cite{Bou05}, the same author considered \eqref{eq:NLS} but with an extra random convolution potential $V$ as external parameter:
\begin{equation}
\label{eq:NLSconv}
\ic \partial_t u + \partial_x^2 u+ V*u = f(|u|^2)u, \quad x\in \mathbb{T},\ t\in\R.
\end{equation}
Again he proved the existence of almost periodic solutions but, this time, with more reasonable decrease  of the Fourier coefficients of the solution. This result has been revisited and improved by a series of paper (in particular \cite{BMP20,Con23} and  
\cite{CY21} for the wave equation). Further, using sparsity arguments, Biasco--Masetti--Procesi proved in \cite{BMP24} the existence of almost periodic solutions of finite regularity. In all these results the Fourier coefficients of the potential $V$ are bounded and does not converge to $0$. This is a critical ingredient of these proofs which is used to deal with the small divisors (since the linear frequencies for \eqref{eq:NLSconv}  are $\omega_j=j^2-V_j,\ j\in\Z$). 
This assumption on $V$ seems to be the major obstruction to get similar results for other models or just to remove $V$ in \eqref{eq:NLSconv} (and so to get \eqref{eq:NLS}).  Indeed, following an observation due to  Kuksin--P\"oschel in \cite{KP96}, it can be proved that \eqref{eq:NLS} is somehow equivalent to \eqref{eq:NLSconv} where $V$ would depend nonlinearly on $u$ through the relation $V_j = |u_j|^2$, $j\in \mathbb{Z}$. It therefore makes sense to try to match the regularity of $V$ and $u$ in \eqref{eq:NLSconv} in order to get rid of the external parameters. In this direction, introducing a new approach based on tree expansions and renormalization groups, Corsi--Gentile--Procesi (see \cite{CGP24})  succeeded in greatly improving the regularity of $V$, indeed they can consider $V$ of any finite regularity $C^n$. Unfortunately, their almost periodic solutions are still Gevrey in the space variable. Very recently, Biasco--Corsi--Gentile--Procesi extended this result (in \cite{BCGP24}) by proving that their set of almost periodic solutions is asymptotically of full measure.

\medskip

Without external parameters, the only known results prove that most of the linear infinite dimensional tori are almost preserved for very long times (see e.g. \cite{Bou00,BFG20,LX24,BC24} for \eqref{eq:NLS} and \cite{BG21} for perturbations of KdV or Benjamin--Ono equations). The proofs of these results, somewhere between Birkhoff normal forms and KAM, deal with many similar difficulties (in particular the degeneracy of small divisors) but make crucial use of the fact that time, although potentially very long, is finite.

\subsection{About the approach} 
  First, we present our approach without going into technical details.
\begin{itemize}
\item Our approach is different from that adopted in the above-mentioned articles which concern infinite-dimensional tori (in particular in \cite{Bou05, BMP20}): we don't try to construct an infinite-dimensional torus directly, but, iteratively,
a convergent sequence of invariant tori whose dimension goes to $+\infty$.
 This was the scheme of proof proposed by P\"oschel in \cite{Pos02} but then rarely used because it requires the nonlinearity to be regularizing which is not the case for \eqref{eq:NLS} neither for most of the Hamiltonian PDEs (see also \cite{GX13}).

\item However, revisiting techniques from the world of dispersive PDEs, and in particular techniques used to get nonlinear smoothing \cite{ErTz1,ErTz2,ErTz3,ErTz4,McConnell}, unconditional uniqueness \cite{GuoKwonOh,Kishimoto}, or invariant renormalized Gibbs measures \cite{BO96,DNY}, we succeed in Theorem \ref{thm:reg} below,  by a canonical change of variable,  in making the nonlinearity of \eqref{eq:NLS} regularizing up to a gauge transform\footnote{Roughly speaking, it means that the only obstruction to transform \eqref{eq:NLS} into a Hamiltonian PDE with a regularizing nonlinearity is the existence of terms depending on the mass $\| u\|_{L^2}^2$}.  This step is fundamental in our approach and deeply uses dispersive properties of \eqref{eq:NLS}. We think that Theorem \ref{thm:reg}, although probably not surprising in the dispersive community (where this kind of result are obtained without preserving symplecticity), could be  useful for other purpose  in the Hamiltonian PDEs world.  Note, however, that this regularization modulo a gauge transform has already been observed in the integrable case (see Theorem 2.1 in \cite{KST17}). 
In the preamble to section \ref{sec:reg}, we give more details about this dispersive effect. 

\item With this "regularizing normal form",  the P\"oschel's scheme makes sense but it is not the end of the story since, first, the regularization is just partial, but also P\"oschel was considering external parameters (namely the Fourier coefficients of a potential) and we don't. Nevertheless,
as in \cite{Pos02}, we iterate a loop that consists in adding one mode to the finite-dimensional tori already constructed. The central tool for achieving this loop is a KAM theorem. So we iterate an infinite number of KAM theorems. The main difficulty lies in preserving the so called \emph{twist condition} which provides small divisor estimates at the cost of excluding few internal parameters (which are essentially the squared moduli of the initial datum's Fourier coefficients in the final coordinates).

%frequencies that can be modulated using internals   parameters which are essentially the squared moduli of the initial data's Fourier coefficients. {\color{red} bof, non ?}

\item To do this, we revisit the classical KAM theorem: we prove a normal-form theorem  (see Theorem \ref{thm:KAM}) where we eliminate much more than, as usual in KAM scheme, the 2-jet of the perturbation. Roughly speaking we eliminate something between  the 3-jet and the 4-jet, see Definition \ref{def:jet}. This forces us to have a stronger smallness condition on this perturbation and leads us, between each application of KAM (each loop), to apply a Birkhoff procedure (see Theorem \ref{thm:Birk}).
\end{itemize}

  We expect that our approach is quite robust and not too specific to \eqref{eq:NLS}. Indeed, on the one hand, there exist many Hamiltonian PDEs without external parameters for which it is possible to modulate similarly the frequencies using initial data (perturbations of KdV and Benjamin--Ono equations \cite{BG21}, \eqref{eq:NLS} on non rectangular flat tori \cite{BC24}, the pure gravity water waves \cite{FG24}...). 
Moreover, in general, the way in which frequencies depend on external parameters is rarely as favorable as for NLS with a convolutional potential \eqref{eq:NLSconv} while the kind of dependency with respect to the internal parameters is more universal.
   On the other hand, the fact that dispersion provides regularizing effects is a quite general idea which has been applied in many contexts. So we do not expect our regularizing normal form to be specific to \eqref{eq:NLS}, but to be equally valid for other one-dimensional periodic dispersive PDEs (see e.g. \cite{ErTz2,ErTz3}).
   %({\color{red}@Tristan, penses tu qu'il faut en dire plus ici ?})

  \subsection{Extended version of the main result}

Now, we give a more precise version of our main result which includes some information about the shape of the infinite dimensional tori.
Actually, we prove that they accumulate on the finite dimensional Kronecker tori constructed by Kuksin--P\"oschel in \cite{KP96} for the Hausdorff distance\footnote{it is indeed a distance on the set of the compact sets of $\ell^2_\varpi$.} 
$$
\mathrm{dist}(T_1,T_2) := \max\big( \sup_{u\in T_1} \inf_{v\in T_2} \| u-v\|_{\ell^2_\varpi} ,  \sup_{v\in T_2} \inf_{u\in T_1} \| u-v\|_{\ell^2_\varpi}  \big), \quad T_1,T_2 \subset \ell^2_\varpi.
$$
\begin{theorem}\label{thm:main}
Let $\mathcal{S}_1\subset \mathbb{Z}$ be a non empty finite set  and let $\Xi_0$ be a compact set included in $(0,1)^{\mathcal{S}_1}$. For all $\varepsilon \leq 1$, there exist a Borel set   $\Xi_\varepsilon \subset \Xi_0$, a family $(\mathcal{T}_{\xi}^\varepsilon)_{\xi \in\Xi_\varepsilon }$ of subsets of $\ell^2_\varpi$, a family of homeomorphisms $\Psi_{\xi}^\varepsilon : \mathbb{T}^{\mathcal{S}_1} \to \mathcal{T}_{\xi}^\varepsilon$ indexed by $\xi \in \Xi_\varepsilon$ and an injective map $\omega^{(\varepsilon)} : \Xi_\varepsilon \to \mathbb{R}^{{\mathcal{S}_1}}$
 such that 
\begin{itemize}
\item[(i)] for all $\xi \in \Xi_\varepsilon$, $\mathcal{T}_{\xi}^\varepsilon$ is a Kronecker torus of frequencies vector $\omega^{(\varepsilon)} $, i.e. 
\begin{center}
 $\forall \theta\in \mathbb{T}^{\mathcal{S}_1}$, $t\mapsto \Psi_{\xi}^\varepsilon(\omega^{(\varepsilon)}  t +\theta)$ is a global  solution to  \eqref{eq:NLS}.
\end{center}
\item[(ii)]  for all $\xi \in \Xi_\varepsilon$, $\mathcal{T}^\eps_{\xi}$ is close to $\varepsilon \mathrm T_{\xi}$ for the Hausdorff distance induced by the $\ell^2_\varpi$ norm
\begin{equation}
\label{eq:Hauss}
\mathrm{dist}(\mathcal{T}_{\xi}^\varepsilon,\varepsilon \mathrm T_{\xi})\leq \varepsilon^{1+\frac1{10}},
\end{equation}
\item[(iii)]  the set $\Xi_\varepsilon$ is asymptotically of full measure, i.e.
$$
\lim_{\epsilon\to 0}\mathrm{Leb}(\Xi_\epsilon)=\mathrm{Leb}(\Xi_0),
$$
\item[(iv)]  for all $j\in \mathcal{S}_1$,  there exists a $1$-Lipschitz function\footnote{for the $\ell^1$ norm.} $\lambda_j^{(\varepsilon)}: \Xi_\varepsilon \to \mathbb{R}$, bounded by $1$, such that for all $\xi \in \Xi_\varepsilon$,
\begin{equation}
\label{eq:freq_thm}
\omega_{j}^{(\varepsilon)}(\xi) = j^2 - \varepsilon^2 f'(0) \big( \xi_j -2 \sum_{i \in \mathcal{S}_1} \xi_i \big) + \varepsilon^{5/2} \lambda_j^{(\varepsilon)}(\xi).
\end{equation}
\item[(v)] for all $\xi \in \Xi_\varepsilon$, $\mathcal{T}_{\xi}^\varepsilon$ is accumulated by infinite dimensional tori: for all $ \rho \leq1$, there exists a non resonant infinite dimensional Kronecker torus $\mathcal{T}_{\xi}^{\varepsilon,\rho}  \subset \ell^2_\varpi$ such that
$$
\mathrm{dist}(\mathcal{T}_{\xi}^{\varepsilon,\rho} , \mathcal{T}_{\xi}^{\varepsilon})\leq \rho.
$$

\end{itemize}
\end{theorem}

Let us comment our result, in particular comparing it with Kuksin--P\"oschel result \cite{KP96}.
\begin{itemize}
\item items (i)-(iv) are consequences of the result by Kuksin--P\"oschel \cite{KP96}. In particular items (ii)-(iv) ensures that, up to an asymptotically zero measure set, we consider the same invariant tori of finite dimension that Kuksin--P\"oschel. More precisely, items (iii) and (iv) imply that we could parametrize finite dimensional tori by frequencies living in the same set as Kuksin--P\"oschel tori (up to an asymptotic zero measure set) while item {\rm(ii)} ensures that the tori we consider indeed bifurcate from the linear ones.
\item item {\rm(v)} is the new part: it says that we are able to construct {\it infinite dimensional} invariant tori (as stated in Theorem \ref{thm:0}) but furthermore, that Kuksin--P\"oschel tori, which are in a sense exceptional since of finite dimension, are not isolated in the phase space: they are {\it accumulated} by  infinite dimensional invariant tori. 
\item All the finite dimensional invariant tori we construct are KAM tori, i.e. with linear flow and linearly stable, but we cannot conserve the linear stability when we pass to the limit (essentially because the normal form that we construct is only defined on the torus). Thus our infinite dimensional invariant tori are Kronecker tori.
\item 
The infinite dimensional tori that we construct are not maximal in the sense that they are built on a set $\mathcal{S}_\infty$ of Fourier modes which is sparse in $\Z$ and contains $\mathcal{S}_1$. To be more specific, in assertion (v), the frequencies vector $\omega^{\eps,\rho}(\xi)\in \R^\N$ of the tori $\mathcal{T}_{\xi}^{\varepsilon,\rho}$  satisfies $\omega_p^{\eps,\rho}(\xi)=i_p^2+o(1)$ where the integers $i_p$ are distinct and belong to a sparse set. Ideally we would like a full dimensional torus, i.e. $\mathcal{S}_\infty:=\{ i_p \ | \ p\in \mathbb{N} \} =\mathbb{Z}$.
This restriction, although technical, does not seem easy to remove. Nevertheless, in some sense we are not that far; we have only used it to ensure the invertibility of the twist matrix. In particular our sparsity condition is simple and explicit:
\begin{equation}
\label{eq:def_sparcity_condition_intro}
\sup_{k\in \mathcal{S}_\infty} \sum_{j\in \mathcal{S}_\infty} \Big( 1\wedge \frac{\langle  k \rangle^{4}}{\langle j \rangle}  \wedge  \frac{\langle  j \rangle^{4}}{\langle k \rangle}   \Big) \vee  \frac1{\langle j\rangle} < \infty.
\end{equation}
 Note that in particular, it does not require the sequence $(i_p)_{p\in \mathbb{N}}$ to be fast increasing.

%, satisfying a condition of the kind $j_i=i^5$ (see section \ref{sec:sparse}). Ideally we would like a full dimensional torus, i.e. $j_i=i$. This restriction, although technical, does not seem easy to remove.
\end{itemize}

\subsection{About the proof}\label{sec:scheme}

 A detailed scheme of proof would be very long and maybe too technical, here we just want to give the principal keys in an informal way. Note that each section has a short introduction that attempts to clarify the section's place in the overall proof. 
 Now, we assume that the reader is familiar with the techniques of KAM theory and Birkhoff normal form theory.
 
 \medskip
 
 \noindent \underline{$\triangleright$ Modulation of frequencies and twist condition.} The crucial hypothesis in any KAM theorem is the twist condition: assume that the frequencies vector $\omega=(\omega_j)_{j\in\mathcal S}$, $\mathcal S$ being finite or infinite subset of $\Z$, depends smoothly on parameters $\xi\in\R^{\mathcal S}$, we want to be able to modulate the frequencies when moving the parameters. So we want the derivative of $\omega$, denoted $\mathrm{d}\omega$, to be  invertible in some sense depending on the topology we choose on the set of frequencies and on the set of parameters. 
%Here we consider $\omega:\ell^1(\mathcal S)\to\ell^1(\mathcal S)$.\\   
In the case of \eqref{eq:NLS} with $f'(0)=-1$ the first correction to $\omega_j$ is $\xi_j$ (it comes from the term $\frac12|u_j|^4$ in the Hamiltonian which, once you recenter around $\xi_j$ writing $|u_j|^2=\xi_j+y_j$, shows up a quadratic term $\xi_j y_j$). Let us write $\omega_j=\xi_j+\eps\tilde\omega_j$ where $\eps$ is a small parameter linked to the size of the perturbation (i.e. of the solution here). The twist condition results in invertibility of $\text{Id}+\eps \mathrm{d}\tilde{\omega}$ and, if we want a diagonally dominant matrix\footnote{which is equivalent to consider  $\omega:\ell^1(\mathcal S)\to\ell^1(\mathcal S)$ and to apply a Neumann series argument}, 
can be rewritten as
\begin{equation}
\label{eq:bound_l1}
\sup_{k\in\mathcal S}\sum_{j\in\mathcal S} |\partial_{\xi_k}\tilde\omega_j |\lesssim 1.
\end{equation}
The non-linearity regularization improves the estimate on $\tilde\omega$.  By $\delta$-regularizing Hamiltonian we mean an analytic Hamiltonian $P$ that satisfies (locally) $\nabla P: \ \ell^1_{s_0}\to \ell^1_{s_0+\delta}$ where 
$$
\ell^1_{s_0}:=\{ u \in \mathbb{C}^{\mathbb{Z}} \mid \sum_{k\in \mathbb{Z}} \langle k\rangle^{s_0}|u_k|<\infty  \}, \quad \mathrm{for\ some\ } s_0\geq0.
$$ 
Without regularization ($\delta=0$), we  expect estimates on $\tilde\omega$ of the type $|\partial_{\xi_k}\tilde\omega_j |\lesssim 1$ and we cannot expect to prove the twist condition except if $\mathcal S$ is finite (this is actually the case in \cite{KP96}).
When $\mathcal S$ is infinite, we need regularization ($\delta>0$) and we can still distinguish two cases:
\begin{itemize}
\item with a regularization without loss ($\delta>0$ and  $s_0=0$) we  expect estimates on $\tilde\omega$ of the type $|\partial_{\xi_k}\tilde\omega_j |\lesssim \frac1{\langle j\rangle^\delta}$ uniformly in $k$ and we can  prove the twist condition for $\delta>1$ or $S$ sparse.
\item with a regularization with loss ($\delta>0$ and  $s_0>0$), which is actually our case with $s_0=2\delta=2$, we prove estimates on $\tilde\omega$ of the type\footnote{in fact $\xi\mapsto\tilde\omega$ will be only Lipschitz and we have to adapt the twist condition but we omit this problem here}
\begin{equation}
\label{estim:omeg}
\left|\partial_{\xi_k}\tilde\omega_j\right|\lesssim  (1\wedge \langle j \rangle^{-\delta} \langle k\rangle^{2s_0} \wedge  \langle k \rangle^{-\delta} \langle  j \rangle^{2s_0}  ) \vee  \langle j\rangle^{-\delta}
\end{equation}
 and we prove the twist condition for $\delta>0$ and $S$ sparse (see section \ref{sec:sparse} for a proof). Note that putting together \eqref{eq:bound_l1} and \eqref{estim:omeg}, we get the sparsity condition \eqref{eq:def_sparcity_condition_intro} with $s_0=2\delta=2$.
\end{itemize}

%\medskip
%
%
% \noindent \underline{$\triangleright$ About the sparsity condition on $\mathcal S$.}  We can try to explain why estimate \eqref{estim:omeg} is natural and thus  why, for the meantime, sparsity seems unavoidable without a significant improvement in our regularizing normal form theorem (section \ref{sec:reg}). In fact the problem comes from integrable terms in the Hamiltonian.  Theorem \ref{thm:reg} allows terms of the form $c_{k,\ell,j}|u_j|^2|u_\ell|^2 |u_k|^2$ in the nonlinear term with, if we assume $\langle k\rangle \gg \langle\ell\rangle\geq \langle j\rangle$,  the estimate\footnote{Indeed we can compute explicitly all the sixtic terms and we obtain (see for instance \cite{BFG20} where these terms are very useful) $Z_6=\sum_{k\neq\ell}\frac{|u_k|^4 |u_\ell|^2}{(k-\ell)^2}$ which leads to a quite better estimates of $c_{k,\ell,j}$, but it is difficult to expect such "convolution" structure to be stable at higher degree. So in general we cannot expect better estimate than the one given by Theorem \ref{thm:reg}. } $|c_{k,\ell,j}|\leq \langle k\rangle^{-\delta}\langle \ell\rangle^{4\delta}\langle j\rangle^{4\delta}$. After opening modes $j$ and $\ell$ this term generates a contribution to the frequency of index $k$: $\tilde\omega_k(\xi)=c_{k,\ell,j}\xi_\ell\xi_j$ and we have $\partial_{\xi_\ell} \tilde\omega_k =c_{k,\ell,j}\xi_j$. If $j$ is very small, this term could be of order $\langle k\rangle^{-\delta}\langle \ell\rangle^{4\delta}$.

\medskip

 \noindent \underline{$\triangleright$ About the sparsity condition \eqref{eq:def_sparcity_condition_intro} on $\mathcal S_\infty$.} As explained previously, the explicit condition \eqref{eq:def_sparcity_condition_intro} is the only one we impose on the sites.  We can try to explain why estimate \eqref{estim:omeg} (implying \eqref{eq:def_sparcity_condition_intro}) is natural and thus  why, for the meantime, sparsity seems unavoidable without a significant improvement in our regularizing normal form theorem (section \ref{sec:reg}). In fact the problem comes from integrable terms in the Hamiltonian.  Theorem \ref{thm:reg} allows terms of the form $c_{k,\ell,j}|u_j|^2|u_\ell|^2 |u_k|^2$ in the nonlinear term with, if we assume $\langle k\rangle \gg \langle\ell\rangle\geq \langle j\rangle$,  the estimate\footnote{Indeed we can compute explicitly all the sixtic terms and we obtain (see for instance \cite{BFG20} where these terms are very useful) $Z_6=\sum_{k\neq\ell}\frac{|u_k|^4 |u_\ell|^2}{(k-\ell)^2}$ which leads to a quite better estimates of $c_{k,\ell,j}$, but it is difficult to expect such "convolution" structure to be stable at higher degree. So in general we cannot expect better estimate than the one given by Theorem \ref{thm:reg}. } $|c_{k,\ell,j}|\leq \langle k\rangle^{-\delta}\langle \ell\rangle^{4\delta}\langle j\rangle^{4\delta}$. After opening modes $j$ and $\ell$ this term generates a contribution to the frequency of index $k$: $\tilde\omega_k(\xi)=c_{k,\ell,j}\xi_\ell\xi_j$ and we have $\partial_{\xi_\ell} \tilde\omega_k =c_{k,\ell,j}\xi_j$. If $j$ is very small, this term could be of order $\langle k\rangle^{-\delta}\langle \ell\rangle^{4\delta}$.

\medskip
 
 \noindent \underline{$\triangleright$ KAM with internal parameters.} To perform a KAM theorem using internal parameters we know, after Kuksin-P\"oschel \cite{KP96}, that we first need to perform a Birkhoff normal form up to order four which brings out the 4th-order integrable terms, $\sum_k |u_k|^4$, from the cubic term in the original equation. As we have seen above these terms are used to satisfy the twist condition.  Actually, after our regularization step, we obtain a normal form ready for the first KAM step which essentially constructs the Kuksin--P\"oschel invariant tori. Then our strategy consists in adding a new internal site and to apply again a KAM theorem, and then to iterate this loop. So we need, at the loop output, a Hamiltonian in normal form up to order four. We have two possibilities:
\begin{itemize}
\item Following the standard KAM theorem we remove the 2-jet of the perturbation during the KAM step and then we apply a Birkhoff normal form procedure to eliminate the order 3 terms and order 4 terms that are not integrable. This strategy seems to be the most reasonable, just mimicking \cite{KP96}, but it faces an obstruction. In fact when eliminating order 3 terms by Birkhoff, we will create terms of order 4 by taking the Poisson bracket between terms of order 3 (typically $\{P_3,\chi_3\}$ where $P_3$ is an order 3 term of the perturbation and $\chi_3$ is an order 3 term coming from the resolution of the cohomological equation). Among them we can create $a_{k,j}|u_k|^2 |u_j|^2 $ where, a priori\footnote{Here $P_3$ and $\chi_3$, as part of the residual terms of KAM,  do not have  small coefficients.  }, $a_{k,j}$ is of size $1$. Then, when we open the site $k$, $|u_k|^2$ becomes $\xi_k+y_k$ and  this  generates the term $a_{k,j}\xi_k |u_j|^2$. In other words, we have created a contribution to $\omega_j$ equals to $a_{k,j}\xi_k$ whose $C^1$ (or Lipschitz) norm is, a priori, of size $1$. Thus, as seen above, we cannot satisfy the twist condition at the next step. We note that this obstruction does not appear when eliminating order 4 terms.
\item We remove all the 3-jet (and actually a part of the 4-jet) during the KAM step. This is very expensive in terms of small divisors. In fact to perform such KAM theorem, we have to solve successively 6 cohomological equations and at the end we have to assume that this part of the jet of our perturbation is very small ($r^{4000}$ where $r$ is the size of the neighborhood of the current torus). 
%This condition will be satisfied initially by  taking $r$ small with respect to 1 but very large with respect to the size $\varepsilon$ of the solution we are going to design{\color{red} pas compréhensible je crois}. 
Then after each KAM step and before each new opening step we will apply a Birkhoff procedure to make much smaller the terms of the perturbation that will appear in this jet when we open the new site. Then for the next KAM step this jet will be small enough to eliminate it by applying our KAM theorem.
\end{itemize}
Clearly, although less natural, we adopt the second solution in this paper.

\medskip

 \noindent \underline{$\triangleright$ Internal versus external parameters.} To try to be more explicit about how we handle small divisors with internal parameters, let us return to a simpler model that already contains the essential difficulties.  As explain in section \ref{sec:context}, following an observation due to  Kuksin--P\"oschel in \cite{KP96}, very roughly  \eqref{eq:NLS} is  equivalent to \eqref{eq:NLSconv} where $V$ would depend nonlinearly on $u$ through the relation $V_j = |u_j|^2$, $j\in \mathbb{Z}$. This leads, at step $p$ of our construction, to frequencies of the form $\omega_j=j^2+\xi_j,\ j\in \mathcal{S}_p = \{ i_1,\cdots,i_p\}$ (assuming here that $\sharp \mathcal{S}_1 = 1$ to lighten the notations), where the parameters $\xi_j$ satisfy $\xi_{i_q} \sim r_q^{2\nu}$, $\nu$ is an exponent very close but less than $1$ and $r_p$ is the size of the neighborhood on which we have built our normal form around the finite dimensional invariant torus we have constructed at this step. 
 Note that of course we have a constraint of the kind $r_p^{2\nu} \lesssim \varpi_{i_p}^2$ to ensure that the solution remains in the phase space $\ell^2_\varpi$ (and so that all the changes of variables are well defined).
 Then, at step $p+1$, we open a new site and get $\omega_j=j^2+\xi_j,\ j\in S_{p+1}$ with  $\xi_{i_{p+1}}\sim r_{p+1}^{2\nu}$ and $r_{p+1}$ to be determined. Then we want to construct an invariant torus based on $\mathcal{S}_{p+1}$. Forgetting the linear part (coming from the Laplacian) which at this stage no longer plays a role, we are therefore led to control small divisors whose internal part has the form $\sum_{j\in \mathcal{S}_{p+1}}\boldsymbol{k}_j \xi_j$. By homogeneity arguments we can expect that the term we want to remove by a KAM procedure have  size $O(r_{p+1}^5)$.  There are two possible cases
\begin{itemize}
\item Either there exists $j\in \mathcal{S}_p$ such that $|\boldsymbol{k}_j| \gtrsim |\boldsymbol{k}_{i_{p+1}}|$ and thus we expect to control the small divisors by $r_p^2$. It is then sufficient to take $r_{p+1}$ to be sufficiently small compared to $r_p$. By the way, that is why we have opted for a  iterative process \`a la P\"oschel, giving us complete freedom in our choice of $r_{p+1}$.
\item Either $|\boldsymbol{k}_{i_{p+1}}|\gg |\boldsymbol{k}_j|$ for all $j\in \mathcal{S}_p$  and thus we expect to control the small divisors by $r_{p+1}^2$. This case is much less favorable because, as explained above, we want to eliminate many more terms than usual from the Hamiltonian and therefore solve a cascade of cohomological equations, which means paying the price of the small divisor several times over. The size $O(r_{p+1}^5)$ of the terms to be removed is not sufficient. A Birkhoff normal form step is then required to reduce the size of these terms before they can be eliminated by KAM. Here we take advantage of opening the sites one by one, so that we only have to consider very specific terms in this Birkhoff step.
\end{itemize}

\medskip
 
 \noindent \underline{$\triangleright$ About our choice of  Hamiltonian class.} As explain previously, in order to be able to get the twist condition, we need to propagate (in our normal form construction) sharp Lipschitz estimates on the modulated frequencies (like \eqref{estim:omeg}). In particular, it requires to have sufficiently good estimates on the Hamiltonians appearing in our construction (implying in particular vector field estimates from $\ell^1$ to $\ell^1$).
Since these estimates are quite involved, we have chosen the most direct option: we simply propagate estimates on the coefficients of the Hamiltonians and we do not modify the phase space and its topology. It generates spaces of analytic functions (see Definition \ref{def:class}) but the drawback of such approach is that the stability of such spaces with respect to the composition by Lie transforms is not automatic. Nevertheless, we prove it (in Lemma \ref{lem:flow}) by using an ad-hoc Cauchy Kowaleskaya theorem (see Lemma \ref{lem:CK}). 
\medskip
 
 \noindent \underline{$\triangleright$ About our small divisor estimates.} To prove our KAM theorem (see Theorem \ref{thm:KAM}), but also our Birkhoff theorem (see Theorem \ref{thm:Birk}), we need to control more small divisor than usual (e.g. to remove the $4-$jet in the KAM theorem). Therefore, at each step, we impose a fourth Melnikov condition on the frequencies (see \eqref{eq:cequonveut_sd}). We prove that they are typically satisfied using arguments similar to those developed by two of the authors in \cite{BG22} to perform Birkhoff normal forms for PDEs in low regularity\footnote{Note that the situation is actually quite similar here in the sense that the estimates we propagate implies that the Hamiltonians enjoys vector fields estimate from $\ell^1$ into $\ell^1$.}.
%To prove our KAM theorem (see Theorem \ref{thm:KAM}), but also our Birkhoff theorem (see Theorem \ref{thm:Birk}), we need to control more small denominators than usual. Therefore, at each step, the linear frequencies have to satisfy a particularly strong non resonance condition (see \eqref{eq:cequonveut_sd}). To prove it, as in \cite{BG21}, we use an argument based on the fact that the frequencies remains asymptotically close to integers. 
%
%
%
\medskip
 
 \noindent \underline{$\triangleright$ About the amplitudes of the tori.} Very schematically, if we denote by $(r_p)_{p\in\N}$ the sequence of sizes of the circles that make up the infinite dimensional tori we are constructing, we strongly use, as in \cite{Pos02}, that $r_{p+1}$ can be chosen much smaller than $r_p$. In a way we have an infinite reservoir of smallness and this enables us to make assumptions such as "the adapted jet is of size $r^{4000}$" in our KAM theorem. The downside is that we relinquish all reasonable control over the decay of $r_p$ (contrary to \cite{BMP20} or \cite{Con23}). More generally, we are not trying to optimize sizes.

\medskip
 
 \noindent \underline{$\triangleright$ About the geometry of our invariant sets.} For simplicity, we only focused on topological considerations about the geometry of the Kronecker tori we constructed. In particular, we do not think there is any serious obstruction to prove that the associated embeddings $\Psi$ are in fact analytic (as it is usually done for finite dimensional tori, see e.g. \cite{KP96}). However, it would generate technicalities that we preferred to avoid to lighten the proofs.

\medskip

 \noindent \underline{$\triangleright$ On the smoothness assumption on $f$.} We require $f$ to be an entire function such that $z\mapsto f(z^2)$ grows at most exponentially fast. It is clearly more than just analytic in a neighborhood of the origin (i.e. as \cite{KP96}). Actually, by a Cauchy estimate, it can be seen as a smoothness assumption : we require bounds of the form $|f^{(p)}(0)| \leq (p!)^{-1} C^p$ instead of the classical bounds $|f^{(p)}(0)| \leq p ! C^p$ (for a constant $C>0$ independent of $p$). We use this extra $(p!)^{-2}$ factor to perform the \emph{Wick resummation} of the nonlinearity in order to put the equation in regularizing normal form (see Section \ref{sec:reg}). 

\subsection{Notations, constants and Hamiltonian formalism} 
\subsubsection{Global constants} Three global constants are omnipresent in the paper :  
$$
\delta := 1, \quad \eta_0= 10 \quad \mathrm{and} \quad \eta_{\max} = 100.
$$
It is convenient to fix them once for all to  avoid to generate irrelevant dependencies.
The first one, $\delta$, correspond to the smoothing effect we get from the regularizing normal form. It could be any number between $0$ and $1$ (but of course, it would influence a lot the numerology of most of our exponents).
The two others, $\eta_0$ and $\eta_{\max}$ encode respectively the minimal and maximal radii of analyticity of the Hamiltonians we consider from Section \ref{sec:Ham_form}. These are just normalisation constants.

\subsubsection{Notations} We use some quite standard notations. For all $a,b\in \mathbb{R}$, $a\wedge b:=\min(a,b)$ and $a\vee b:=\max(a,b)$. The Japanese bracket is defined by $\langle a \rangle := \sqrt{1+a^2}$ for $a\in \mathbb{R}$.  We shall use the notation $A\lesssim B$ to denote $A\le C B$ where $C$ is a positive constant
depending on  parameters fixed once for all. 
We will emphasize by writing $\lesssim_{\alpha}$ 
when the constant $C$ 
depends on some other parameter $\alpha$. The phase space $\ell^2_\varpi$ and the nonlinearity $f$ are also considered as fixed once for all. 
Finally, all along the paper, we will use the following notation to rearrange the indices. If $E$ is a set $\mathbbm{1}_E$ denotes the indicator function of this set. If $P$ denotes a property $\mathbbm{1}_P$ is equal to $1$ is $P$ is true and it is equal to $0$ is $P$ is false.
\begin{definition}[Rearrangement $\boldsymbol{\ell}_j^*$] Given  $q\geq 0$ and a vector $\boldsymbol{\ell}\in \mathbb{Z}^q$,  $\boldsymbol{\ell}^*_1,\cdots,\boldsymbol{\ell}^*_q$ denotes a rearrangement of $\boldsymbol{\ell}_1,\cdots,\boldsymbol{\ell}_q$ satisfying $|\boldsymbol{\ell}^*_1|\geq\cdots\geq|\boldsymbol{\ell}^*_q|$.
\end{definition}

\subsubsection{Hamiltonian formalism} \label{sub:ham_form} As usual, discrete Lebesgue spaces are defined, for $1\leq p<\infty$, by
 $$
 \ell^{p} := \{ u\in \mathbb{C}^{\mathbb{Z}} \ | \ \|u\|_{\ell^p}^p := \sum_{k\in \mathbb{Z}}  |u_k|^p <\infty \}, \quad \ell^{\infty} := \{ u\in \mathbb{C}^{\mathbb{Z}} \ | \ \|u\|_{\ell^\infty} := \sup_{k\in \mathbb{Z}}  |u_k| <\infty \}.
 $$
 We extend these definitions for sequences indexed by subsets of $\mathbb{Z}$. Recalling that we identify functions with their sequence of Fourier coefficients (defined by \eqref{eq:def_Fourier}), the Fourier--Plancherel isometry writes
$$
\forall u \in \ell^2, \quad \|u\|_{\ell^2}^2 = \int_{\mathbb{T}} |u(x)|^2 \frac{\mathrm{d}x}{2\pi} =: \| u\|_{L^2}^2.
$$ 
We equip $\ell^2$ of the standard real scalar product $(u,v)_{\ell^2}$  defined by
$$
\forall u,v\in \ell^2(\mathbb{T}), \quad  (u,v)_{\ell^2} :=  \Re \sum_{k\in \mathbb{Z}} u_k \overline{v_k} = \Re \int_{\mathbb{T}} u(x) \overline{v(x)} \frac{\mathrm{d}x}{2\pi}.
$$
Given an open subset $\mathcal{U} \subset \ell^2_\varpi$, a smooth function $P :\mathcal{U} \to \mathbb{R}$ and $u\in \mathcal{U}$, its gradient $\nabla P(u)$ is the unique element of $\ell^{2}_{\varpi^{-1}}(\mathbb{Z})$\footnote{where $(\varpi^{-1})_k:= \varpi^{-1}_k $ for all $k\in \mathbb{Z}$}  satisfying
$$
\forall v\in \ell^2_\varpi, \ (\nabla P(u),v)_{\ell^2} = \mathrm{d}P(u)(v).
$$
It can be easily checked that
\begin{equation}
\label{eq:formula_grad}
\forall k \in \mathbb{Z}, \quad (\nabla P(u))_k = 2\partial_{\overline{u_k}} P(u):= (\partial_{\Re u_k} +\ic \partial_{\Im u_k} )P(u).
\end{equation}
We always equip $\ell^2_\varpi$ with the usual symplectic form $(\ic\cdot,\cdot)_{\ell^2}$.  Therefore \eqref{eq:NLS} is a Hamiltonian system as it writes
\begin{equation}
\label{eq:sympl_struct}
\ic\partial_t u = \nabla H^{\eqref{eq:NLS}}(u), \quad \mathrm{with} \quad H^{\eqref{eq:NLS}}(u) := \frac12 \int_{\mathbb{T}} |\partial_x u(x)|^2 + F(|u(x)|^2) \frac{\mathrm{d}x}{2\pi} 
\end{equation}
where $F$ is the primitive of $f$ vanishing at the origin.
Finally, we recall that a $C^1$ map $\tau : \mathcal{U} \to \ell^2_\varpi$ is \emph{symplectic} if
$$
\forall u \in \mathcal{U}, \forall v,w \in \ell^2_\varpi, \quad (\ic v,w)_{\ell^2} = (i\mathrm{d}\tau(u)(v),\mathrm{d}\tau(u)(w))_{\ell^2}.
$$
Moreover, if  $H,K :  \mathcal{U}  \to \mathbb{R}$ are two smooth functions such that $\nabla H$ (or $\nabla K$) is  $ \ell^2_\varpi$ valued then the \emph{Poisson bracket} of $H$ and $K$ is defined by
$$
\{  H,K\}(u):= (\ic \nabla H(u),\nabla K(u))_{\ell^2}.
$$
Note that, as usual, it also satisfies
\begin{equation}
\label{eq:Poisson_formula}
\{  H,K\}
  =2\ic \sum_{k\in \mathbb{Z}} \partial_{\overline{u_k}}H \partial_{u_k} K - \partial_{u_k}H \partial_{\overline{u_k}} K.
\end{equation}

\subsection{Organization of the paper} 
Section \ref{sec:reg} is independent of the rest of the paper and provides a regularizing normal form result, Theorem \ref{thm:reg}, that is crucial in the proof of Theorem \ref{thm:main} and that is interesting in itself. The other sections are all related and focus on proving our main result, Theorem \ref{thm:main}. The heart of the proof is in section \ref{sec:loop}, where we show how to go from a Hamiltonian in normal form with $p$ internal modes to a Hamiltonian in normal form with $p+1$ internal modes, and therefore how to go from $p$-dimensional invariant torus to a $p+1$-dimensional invariant torus. Then in section \ref{sec:proof} we implement the iterative scheme that allows us to take the limit $p\to+\infty$. 
Sections \ref{sec:Ham_form} \ref{sec:PD} \ref{sec:KAM} \ref{sec:Birk} \ref{sec:open} are preparatory, providing all the necessary tools, in particular a KAM theorem (Theorem \ref{thm:KAM}) and a Birkhoff theorem (Theorem \ref{thm:Birk}) adapted to our purpose. Each section has a short introduction that attempts to clarify the section's place in the overall proof.

\subsection{Acknowledgments} During the preparation of this work the authors benefited from the support of the Centre Henri Lebesgue ANR-11-LABX-0020-0 and J.B. was also supported by the region "Pays de la Loire" through the project "MasCan".
T.R. was partially supported by the ANR project Smooth ANR-22-CE40-0017. J.B. and B.G. were partially supported by the ANR project KEN ANR-22-CE40-0016.

\section{Regularizing normal form}\label{sec:reg}
As explained in the introduction, the first step towards the proof of Theorem~\ref{thm:main} is to put the NLS equation \eqref{eq:NLS} into regularizing normal form. Indeed, it is known in the dispersive PDE community \cite{ErTz1,ErTz3,McConnell} that by using a \emph{Poincar\'e-Dulac} normal form reduction after a suitable gauge transformation, one can show a nonlinear smoothing effect for the periodic NLS equation with polynomial nonlinearity. Namely, after transformation of the equation, the nonlinear part of the Duhamel formula can be controlled at a higher regularity than the initial data. In this section, our goal is to establish a similar result for \eqref{eq:NLS}, but through a kind of convergent \emph{Birkhoff} normal form transformation. While the Poincar\'e-Dulac transformation works well to show nonlinear smoothing at the level of the equation, it does not preserve symplecticity. Here we instead work directly at the level of the Hamiltonian.

The first step is to remove trivial resonances. At the level of the equation, this is usually done by a suitable gauge transformation. Here we instead expand our Hamiltonian in terms of Wick products, which are known in the quantum field theory literature to be the suitable renormalization of usual power nonlinearities obtained by removing single pairings (leading to the so-called UV divergencies in EQFT). Then we remove non-resonant interactions through a convergent Birkhoff procedure.

\subsection{Statement}

In the following theorem, which is the main result of this section, we conjugate the Hamiltonian of \eqref{eq:NLS}, $H^{\eqref{eq:NLS}}$, defined by \eqref{eq:sympl_struct}, to a $1$-smoothing Hamiltonian.

\begin{theorem} \label{thm:reg} There exist some constants $\rho,C >0$  and a $C^1$ symplectic map $\tau : B_{\ell^2_\varpi}(0,\rho) \to \ell^2_\varpi$ such that for all $u\in B_{\ell^2_\varpi}(0,\rho)$
$$
H^{\eqref{eq:NLS}} \circ \tau =  H^{(0)}_{1}(u)   + \sum_{2n+q \geq 6} \sum_{\substack{\boldsymbol{\sigma} \in \{-1,1\}^q \\ \boldsymbol{\sigma}_1+ \cdots + \boldsymbol{\sigma}_q = 0}} \sum_{ \boldsymbol{\sigma}_1 \boldsymbol{\ell}_1+\cdots+  \boldsymbol{\sigma}_q \boldsymbol{\ell}_q =0} H_{n}^{\boldsymbol{\ell},\boldsymbol{\sigma}} \|u\|_{L^2}^{2n} \prod_{1\leq j \leq q} u_{\boldsymbol{\ell}_j}^{\boldsymbol{\sigma}_j} 
$$
with the usual convention $ u_{\ell}^{-1} :=  \overline{u_{\ell}}$ and where
$$
 H^{(0)}_{1}(u)  := \frac12 \sum_{k\in \mathbb{Z}} \big(  |k|^2 - f'(0) \frac{|u_k|^2}2 \big) |u_k|^2 + \frac{f'(0)}2\| u\|_{L^2}^4
$$
and the coefficients $ H_{n}^{\boldsymbol{\ell},\boldsymbol{\sigma}} \in \mathbb{C}$ satisfy the bound
$$
| H_{n}^{\boldsymbol{\ell},\boldsymbol{\sigma}}| \lesssim C^{q+2n} \big( 1 \wedge \frac{  \langle \boldsymbol{\ell}_3^* \rangle^2  }{  \langle \boldsymbol{\ell}_1^* \rangle } \big).
$$
Moreover, the map $\tau$ satisfies the following properties for all $u\in B_{\ell^2_\varpi}(0,\rho)$:
\begin{itemize}
\item it preserves the $L^2$ norm and commutes with the gauge transform : $\| \tau(u) \|_{L^2} = \| u\|_{L^2}$ and $\tau(e^{i \theta} u) = \tau( u)  $ for all $\theta \in \mathbb{T}$,
\item it is close to the identity : $\| \tau(u) - u\|_{\ell_\varpi^2} \lesssim \|u\|_{\ell_\varpi^2}^3$,
\item its derivative is close enough to the identity : $\| \mathrm{d}\tau(u) - \mathrm{Id}\|_{\ell^2_\varpi \to \ell^2_\varpi} \leq \frac12$.
\end{itemize}
\end{theorem}

\begin{remark} A quite direct extension of our construction would allow to prove that $\tau- \mathrm{Id}$ is actually smoothing up to a gauge transform, i.e. that for all $u\in B_{\ell^2_\varpi}(0,\rho)$, there exists $\theta(u)\in \mathbb{T}$ such that $\| \tau(u) - e^{i\theta(u)} u\|_{\ell_{\varpi_\sharp}^2} \lesssim \|u\|_{\ell_\varpi^2}^3$ where $(\varpi_\sharp)_k: =\varpi_k \langle k\rangle $. It would make Theorem \ref{thm:reg} more similar to \cite{KST17} (dealing with the cubic \eqref{eq:NLS} case) but since this is not useful for our application, we did not include it.
\end{remark}

This theorem is proven in subsection \ref{sub:proof_thm_reg}. We point out that all the notions and notations introduced in this section are specific to it and will not be used in the rest of the paper. Nevertheless, to make the paper more consistent, we use a formalism very close to the one we use in the other sections (and which, at first glance, may seem surprising for a kind of Birkhoff normal form procedure).

\subsection{Classes of analytic functions and pairing}
Let us introduce some further notations and basic definitions to prove Theorem~\ref{thm:reg}.

\medskip

First, we define a special set of indices which will correspond to monomial whose highest index is associated with an action.
\begin{definition}[Set of indices $\mathbf{P}_{q}$] Being given $q\geq 0$, $\mathbf{P}_{q}$ denotes the set of indices $(\boldsymbol{\ell},\boldsymbol{\sigma}) \in\mathbb{Z}^{q} \times \{-1,1\}^q$ such that there is a pairing between the two highest indices: there exist $j_1,j_2\in\{1,\dots,q\}$ such that 
$$
j_1 \neq j_2, \quad  \boldsymbol{\ell}_{j_1}=\boldsymbol{\ell}_{j_2}, \quad \boldsymbol{\sigma}_{j_1}=-\boldsymbol{\sigma}_{j_2}, \quad |\boldsymbol{\ell}_{j_1}| = |\boldsymbol{\ell}_{1}^*|  .
$$
\end{definition}

\medskip

Then, we define two spaces of homogeneous polynomials.
\begin{definition}[Spaces $\mathcal{P}_{\eta,r}$ and $\mathcal{R}_{\eta,r}$] Given $\eta \geq \eta_0$ and $r \in (0,1)$,  we denote by $\mathcal{P}_{\eta,r}$ the set of the formal series $P$ of the form
$$
P(\mu, v) =   \sum_{q\geq 0} \sum_{ \boldsymbol{\sigma} \in \{-1,1\}^q }  \sum_{ \boldsymbol{\ell} \in \mathbb{Z}^q } \sum_{n\geq 0} P_{n}^{\boldsymbol{\ell},\boldsymbol{\sigma}} \mu^{n} \prod_{1\leq j \leq q} v_{\boldsymbol{\ell}_j}^{\boldsymbol{\sigma}_j} 
$$
whose coefficients $ P_{n}^{\boldsymbol{\ell},\boldsymbol{\sigma}} \in \mathbb{C}$ satisfy the following conditions
\begin{enumerate}[(i)]
\item reality condition  : $  P_{n}^{\boldsymbol{\ell},-\boldsymbol{\sigma}}   = \overline{  P_{n}^{\boldsymbol{\ell},\boldsymbol{\sigma}}  }$ ,
\item preservation of the mass : if $ P_{n}^{\boldsymbol{\ell},\boldsymbol{\sigma}} \neq 0$ then $\boldsymbol{\sigma}_1 +\cdots+  \boldsymbol{\sigma}_q =0$,
\item zero momentum condition :  if $ P_{n}^{\boldsymbol{\ell},\boldsymbol{\sigma}} \neq 0$ then $\boldsymbol{\sigma}_1 \boldsymbol{\ell}_1+\cdots+  \boldsymbol{\sigma}_q \boldsymbol{\ell}_q=0$,
\item bound :
$$
\| P \|_{\mathcal{P}_{\eta,r}} :=  \sup_{\substack{q\geq 0, \  \boldsymbol{\ell} \in \mathbb{Z}^q  \\  \boldsymbol{\sigma} \in \{-1,1\}^q } } \sum_{n\geq 0} |P_{n}^{\boldsymbol{\ell},\boldsymbol{\sigma}}| r^{2n+q} e^{\eta q + 2\eta_0 n  } A_{\boldsymbol{\ell},\boldsymbol{\sigma}}^{-1} <\infty \quad \mathrm{where} \quad A_{\boldsymbol{\ell},\boldsymbol{\sigma}}^{-1} := 1 \vee  e^{-q}  \frac{  \langle \boldsymbol{\ell}_1^* \rangle }{  \langle \boldsymbol{\ell}_3^* \rangle^2  }  \mathbbm{1}_{(\boldsymbol{\ell},\boldsymbol{\sigma}) \in \mathbf{P}_{q}}.
$$
\end{enumerate}
We denote by $\mathcal{R}_{\eta,r}$, the subspace of $\mathcal{P}_{\eta,r}$ of the formal series $P$ satisfying the extra bound
$$
\| P \|_{\mathcal{R}_{\eta,r}} :=  \sup_{\substack{q\geq 0, \  \boldsymbol{\ell} \in \mathbb{Z}^q  \\  \boldsymbol{\sigma} \in \{-1,1\}^q } } \sum_{n\geq 0}   |P_{n}^{\boldsymbol{\ell},\boldsymbol{\sigma}}| r^{2n+q} e^{\eta q + 2\eta_0 n  }  B_{\boldsymbol{\ell}}^{-1}<\infty \quad \mathrm{where} \quad B_{\boldsymbol{\ell}}^{-1} := 1 \vee  e^{-q}  \frac{  \langle \boldsymbol{\ell}_1^* \rangle }{  \langle \boldsymbol{\ell}_3^* \rangle^2  }.
$$
\end{definition}
\begin{remark} The definition of these quite unusual norms are motivated by Lemma \ref{lem:cohom_is_smooting} and Lemma \ref{lem:bracket_reg} which are crucial in our proof.
\end{remark}

\begin{lemma}[Formal series define functions] \label{lem:ouf_cest_des_fonctions_reg} Let $r\in (0,1)$, $\eta \geq \eta_0$ and $P\in  \mathcal{P}_{\eta,r} $. Then 
$$
u\mapsto P(\| u\|_{L^2}^2,u )=:P(u)
$$
 defines a smooth function from $ B_{\ell^2_\varpi}(0,3r) $ into $\mathbb{R}$ and enjoys the estimate
$$
\sup_{u\in  B_{\ell^2_\varpi}(0,3r) }|P(u)| \lesssim \| P\|_{\mathcal{P}_{\eta,r}}.
$$
\end{lemma}
\begin{proof} It suffices to prove that the formal series defining $P(u)$ converges uniformly for $u\in B_{\ell^2_\varpi}(0,3r) $. By definition of the norm $\| P \|_{\mathcal{P}_{\eta,r}} $,  it suffices to prove that
$$
\sum_{n \in \mathbb{N}}\sum_{q\in \mathbb{N}}r^{-2n} e^{-\eta q - 2\eta_0 n  } (3r)^{2n} \sum_{\substack{ \boldsymbol{\ell} \in \mathbb{Z}^q,\  \boldsymbol{\sigma} \in  \{-1,1\}^q  }}  w_{\boldsymbol{\ell}_1} \cdots w_{\boldsymbol{\ell}_q}    \, \lesssim 1
$$
where $w_j= |u_j|r^{-1}$ for $j\in\Z$. Noticing that, since $s\geq1 $,  $w\in \ell^1$ with $\| w\|_{\ell^1} \leq 5r^{-1}\| u\|_{\ell_\varpi^2}\leq15$ for $u\in B_{\ell^2_\varpi}(0,3r) $, we deduce  
$$\sum_{ \boldsymbol{\ell} \in \Z^q,\  \boldsymbol{\sigma} \in  \{-1,1\}^q  }  w_{\boldsymbol{\ell}_1} \cdots w_{\boldsymbol{\ell}_q}    \, \lesssim 30^q.$$
On the other hand the summability with respect to $(q,n)$ is ensured by the exponential $e^{-\eta q - 2\eta_0 n  }$ with $\eta\geq \eta_0$.
\end{proof}

\begin{definition}[Projection $\Pi_{a}$ and $\Pi_{\leq a}$] Let $a\geq 0$, $r\in (0,1)$ and $\eta \geq \eta_0$. We denote by $\Pi_{a}$ the projection defined for all $P\in \mathscr{P}_{\eta,r}$ by
$$
\Pi_{a} P(\mu,v):=   \sum_{ \boldsymbol{\sigma} \in \{-1,1\}^a }  \sum_{ \boldsymbol{\ell} \in \mathbb{Z}^a } \sum_{n\geq 0} P_{n}^{\boldsymbol{\ell},\boldsymbol{\sigma}} \mu^{n} \prod_{1\leq j \leq q} v_{\boldsymbol{\ell}_j}^{\boldsymbol{\sigma}_j} .
$$
Moreover, we set 
$\Pi_{\leq a} = \Pi_0 +\cdots + \Pi_a.$
If $a< 0$ we set $\Pi_a = 0$.
\end{definition}

\begin{definition}[Degree and projections $\Pi_{\mathrm{deg}\leq d }$] Let $d\geq 0$. Some coefficients $(n,\boldsymbol{l},\boldsymbol{\sigma})\in \mathbb{N} \times \mathbb{Z}^q \times \{-1,1\}^q$, with $q\geq 0$, have degree $d$ if $2n+q = d.$
 We denote by $\Pi_{\mathrm{deg}=d}$ the associated projection defined by restriction of the coefficients to the indices of degree $d$. Moreover, we set $\Pi_{\mathrm{deg}\leq d } = \Pi_{\mathrm{deg}=0} + \cdots +\Pi_{\mathrm{deg}=d-1}$.
\end{definition}

\subsection{Laguerre polynomials and Wick renormalization}

This subsection is devoted to the proof of Proposition \ref{prop:PNLS} below in which we prove that the nonlinearity of \eqref{eq:NLS} belongs to some spaces $\mathcal{P}_{2\eta_0,r_0}$ for some $r_0>0$. Actually, the algebraic decomposition we prove is much stronger but not useful in this paper. Before stating the proposition, we recall that $F$ denotes the primitive of the nonlinearity $f$ vanishing at the origin and that  we only have assumed that, $F'(0) = f(0) =0$ and that it is a real entire function such that $z\mapsto F(z^2)$ grows at most exponentially fast\footnote{and that $F''(0) = f'(0) \neq 0$ but it is not useful in the section}.

\begin{proposition}\label{prop:PNLS} There exists $r_{0} \in (0,1)$ and $P^{\eqref{eq:NLS}}\in \mathcal{P}_{3\eta_0,r_0}$,  such that the nonlinear part of the Hamiltonian of \eqref{eq:NLS} satisfies
$$
\forall u\in B_{\ell^2_\varpi}(0,3r_0), \quad  \frac12 \int_{\mathbb{T}} F(|u(x)|^2) \frac{\mathrm{d}x}{2\pi} = P^{\eqref{eq:NLS}}(u) \quad \mathrm{with} \quad \Pi_{2} P^{\eqref{eq:NLS}} = \Pi_{\mathrm{deg}\leq 3}P^{\eqref{eq:NLS}}  = 0.
$$
\end{proposition}

The proof of this proposition (given at the end of this subsection) requires one more definition and a preparatory lemma.

\begin{definition}[Laguerre polynomials, Wick-ordered polynomials]~

\begin{enumerate}[(i)]
\item Given $p\geq 0$, the $p$-th Laguerre polynomial is defined as
$$L_p(x)=\sum_{n=0}^p \frac{(-1)^n}{n!}{\binom p n}x^n.$$
Monomials can then be expanded as
\begin{align}\label{monomL}
x^q=(q!)\sum_{m=0}^q(-1)^m{\binom q m}L_m(x).
\end{align}
\item Given $p\geq 0$, we define the Wick-ordered homogeneous polynomials of degree $2p$ by
\begin{align}\label{Wick}
\,:\!|u|^{2p}\!:\,= (-1)^p(p!)\|u\|_{L^2}^{2p}L_p\Big(\frac{|u|^2}{\|u\|_{L^2}^2}\Big)=\sum_{q=0}^p(-1)^{p-q}\frac{p!}{ q!}{\binom p q}\|u\|_{L^2}^{2(p-q)}|u|^{2q}.
\end{align}
\end{enumerate}

\end{definition}

Our interest in these homogeneous polynomials comes from (the proof of) \cite[Proposition 2.2]{DNY} that we recall in the next lemma. For convenience of the reader we also repeat the proof in our context.
\begin{lemma}\label{lem:W}
Given $p\in\N$, we have
\begin{align}\label{WickH}
W_{2p}(u):=\int_{\T}\,:\!|u|^{2p}\!:\,(x)\frac{dx}{2\pi} = \sum_{ \substack{\boldsymbol{\sigma} \in \{-1,1\}^{2p} } }  \sum_{ \substack{\boldsymbol{\ell} \in \mathbb{Z}^{2p} } } (W_{2p})^{\boldsymbol{\ell},\boldsymbol{\sigma}} \prod_{1\leq j \leq 2p} u_{\boldsymbol{\ell}_j}^{\boldsymbol{\sigma}_j} 
\end{align}
where the coefficients $(W_{2p})^{\boldsymbol{\ell},\boldsymbol{\sigma}} \in \mathbb{R}$ satisfy 
$$
(W_{2p})^{\boldsymbol{\ell},\boldsymbol{\sigma}} = 0 \quad \mathrm{if} \quad \boldsymbol{\sigma}_1 + \cdots + \boldsymbol{\sigma}_{2p} \neq0 \quad \mathrm{or} \quad \boldsymbol{\sigma}_1 \boldsymbol{\ell}_1  + \cdots +  \boldsymbol{\sigma}_{2p} \boldsymbol{\ell}_{2p}  \neq 0
$$
and
\begin{align}\label{Wpinfty}
|(W_{2p})^{\boldsymbol{\ell},\boldsymbol{\sigma}}| \leq p! 2^p.
\end{align}
Furthermore for any $(\boldsymbol{\ell},\boldsymbol\sigma)\in\Z^{2p}\times\{-1,1\}^{2p}$, if $(W_{2p})^{\boldsymbol{\ell},\boldsymbol\sigma} \neq 0$ then any pairing in $(\boldsymbol{\ell},\boldsymbol\sigma)$ is over-paired: if there exist $1\leq i<j\leq 2q$ such that $\boldsymbol{\ell}_i=\boldsymbol{\ell}_j$ and $\boldsymbol{\sigma}_i\boldsymbol{\sigma}_j=-1$ then there exits $1\leq k\leq 2p$ with $k\notin\{i,j\}$ such that $\boldsymbol{\ell}_k=\boldsymbol{\ell}_i=\boldsymbol{\ell}_j$. 
\end{lemma}

\begin{proof} 
First, we note that, for any $q\geq 0$,  by elementary combinatorial calculus, we have the expansion 
$$
\|u\|_{L^{2q}}^{2q} = \sum_{1\leq n\leq 2q} \sum_{\boldsymbol{j}_1<\cdots < \boldsymbol{j}_n} \sum_{\substack{\boldsymbol\alpha_1+\cdots+\boldsymbol\alpha_n=q \\ \boldsymbol\beta_1+\cdots+\boldsymbol\beta_n=q \\ \boldsymbol{j} \cdot (\boldsymbol \alpha -\boldsymbol\beta ) =0 }} \frac{q!}{\prod_{\ell=1}^n\boldsymbol\alpha_\ell !}\frac{q!}{\prod_{\ell=1}^n\boldsymbol\beta_\ell !} \prod_{\ell=1}^nu_{\boldsymbol{j}_\ell}^{\boldsymbol\alpha_\ell}\overline{u_{\boldsymbol{j}_\ell}}^{\boldsymbol\beta_\ell}.
$$
Then  using \eqref{Wick}, we deduce that
$$
W_{2p}(u) =  \sum_{1\leq n\leq 2q} \sum_{\boldsymbol{j}_1<\cdots < \boldsymbol{j}_n} \sum_{\substack{\boldsymbol\alpha_1+\cdots+\boldsymbol\alpha_n=p \\ \boldsymbol\beta_1+\cdots+\boldsymbol\beta_n=p \\ \boldsymbol{j} \cdot (\boldsymbol\alpha -\boldsymbol\beta ) =0 }} W_{2p}^{(n,\boldsymbol{\alpha},\boldsymbol{\beta})} \prod_{\ell=1}^nu_{\boldsymbol{j}_\ell}^{\boldsymbol\alpha_\ell}\overline{u_{\boldsymbol{j}_\ell}}^{\boldsymbol\beta_\ell}
$$
where
\begin{equation*}
\begin{split}
W_{2p}^{(n,\boldsymbol{\alpha},\boldsymbol{\beta})} &= \sum_{q=0}^{p}(-1)^{p-q}{\binom {p} q}\frac{p!}{q!}\sum_{\substack{\boldsymbol\gamma_1+\dots+\boldsymbol\gamma_n=p-q\\ 0\leq \boldsymbol\gamma_\ell \leq \boldsymbol\alpha_\ell \wedge \boldsymbol\beta_\ell}}\frac{(p-q)!}{\prod_{\ell=1}^n \boldsymbol\gamma_\ell!}\frac{(q!)^2}{\prod_{\ell=1}^n(\boldsymbol\alpha_\ell-\boldsymbol\gamma_\ell) !(\boldsymbol\beta_\ell-\boldsymbol\gamma_\ell) !}\\
&=(-1)^{p}(p!)^2\sum_{q=0}^{p}(-1)^{q}\sum_{\substack{\boldsymbol\gamma_1+\dots+\boldsymbol\gamma_n=p-q\\ 0\leq \boldsymbol\gamma_\ell \leq \boldsymbol\alpha_\ell  \wedge \boldsymbol\beta_\ell}}\frac1{\prod_{\ell=1}^n\boldsymbol\gamma_\ell!(\boldsymbol\alpha_\ell-\boldsymbol\gamma_\ell) !(\boldsymbol\beta_\ell-\boldsymbol\gamma_\ell) !}.
\end{split}
\end{equation*}
As a consequence, coming back in the original variables, we have
$$
 W_{2p}(u) = \sum_{ \substack{\boldsymbol{\sigma} \in \{-1,1\}^{2p}  \\  \boldsymbol{\sigma}_1 + \cdots + \boldsymbol{\sigma}_{2p} = 0  } }  \sum_{ \substack{\boldsymbol{\ell} \in \mathbb{Z}^{2p} \\ \boldsymbol{\sigma}_1 \boldsymbol{\ell}_1  + \cdots +  \boldsymbol{\sigma}_{2p} \boldsymbol{\ell}_{2p}  = 0 } } (W_{2p})^{\boldsymbol{\ell},\boldsymbol{\sigma}} \prod_{1\leq j \leq 2p} u_{\boldsymbol{\ell}_j}^{\boldsymbol{\sigma}_j} 
$$
where 
\begin{equation}
\label{eq:les_vrai_coeffs}
(W_{2p})^{\boldsymbol{\ell},\boldsymbol{\sigma}} := 2^{-p} (-1)^{p}\sum_{q=0}^{p}(-1)^{q}\sum_{\substack{\boldsymbol\gamma_1+\dots+\boldsymbol\gamma_n=p-q\\ 0\leq \boldsymbol\gamma_\ell \leq \boldsymbol\alpha_\ell  \wedge \boldsymbol\beta_\ell}}\prod_{\ell=1}^n \frac{\boldsymbol\alpha_\ell ! \boldsymbol\beta_\ell !}{\boldsymbol\gamma_\ell!(\boldsymbol\alpha_\ell-\boldsymbol\gamma_\ell) !(\boldsymbol\beta_\ell-\boldsymbol\gamma_\ell) !},
\end{equation}
$n\geq 1,\boldsymbol{\alpha},\boldsymbol{\beta} \in \mathbb{N}^n$ denoting indices such that the monomial  $\prod_{1\leq j \leq 2p} u_{\boldsymbol{\ell}_j}^{\boldsymbol{\sigma}_j}$ can be written under the form $\prod_{\ell=1}^nu_{\boldsymbol{j}_\ell}^{\boldsymbol\alpha_\ell}\overline{u_{\boldsymbol{j}_\ell}}^{\boldsymbol\beta_\ell} $ for some $\boldsymbol{j}_1<\cdots < \boldsymbol{j}_n$.

\medskip

With these notations, if for some $\ell$, $\boldsymbol\alpha_\ell=\boldsymbol\beta_\ell=1$, then $\boldsymbol\gamma_\ell\in\{0,1\}$ and moreover the $q$-th coefficient in the sum for $\boldsymbol\gamma_\ell=0$ exactly cancels out the $(q+1)$-th term for $\boldsymbol\gamma_\ell=1$. This shows that $(W_{2p})^{\boldsymbol{\ell},\boldsymbol{\sigma}} = 0$, i.e. that any pairing in $W_{2p}$ is always over-paired.

\medskip

It only remains to estimate $(W_{2p})^{\boldsymbol{\ell},\boldsymbol{\sigma}}$. It is not clear how to get a good estimate from \eqref{eq:les_vrai_coeffs}, so we prove another formula.
First, we note that for all $p\geq q$ there exists a set $E_{p,q}$ (whose definition is simple but long and not useful for us) such that
$$
\| u \|_{L^2}^{2(p-q)} \| u\|_{L^{2q}}^{2q} = \sum_{\boldsymbol{\ell} \in \mathbb{Z}^{2p}} \sum_{\boldsymbol{\sigma} \in \{-1,1\}^{2p}} \mathbbm{1}_{(\boldsymbol{\ell} ,\boldsymbol{\sigma} ) \in E_{p,q}} \prod_{1\leq j \leq 2p} u_{\boldsymbol{\ell}_j}^{\boldsymbol{\sigma}_j} .
$$
Thus, by symmetrizing this coefficients, it follows of \eqref{Wick} that
$$
(W_{2p})^{\boldsymbol{\ell},\boldsymbol{\sigma}} = \sum_{q=0}^p(-1)^{p-q}\frac{p!}{ q!}{\binom p q} S_{\boldsymbol{\ell},\boldsymbol{\sigma},p,q}
$$
where 
$$
S_{\boldsymbol{\ell},\boldsymbol{\sigma},p,q} = \frac1{(2p)!} \sum_{\varphi \in \mathfrak{S}_{2p}} \mathbbm{1}_{(\varphi\boldsymbol{\ell} ,\varphi\boldsymbol{\sigma} ) \in E_{p,q}} 
$$
So using that $|S_{\boldsymbol{\ell},\boldsymbol{\sigma},p,q}| \leq 1$, it comes
$$
|(W_{2p})^{\boldsymbol{\ell},\boldsymbol{\sigma}} | \leq \sum_{q=0}^p \frac{p!}{ q!}{\binom p q}\Big|\leq p! 2^p.
$$
\end{proof}

\begin{proof}[Proof of Proposition \ref{prop:PNLS}]
From \eqref{monomL}-\eqref{Wick} and denoting $\mu=\|u\|_{L^2}^{2}$ we formally expand
\begin{align}
\nonumber  \int_\T F(|u(x)|^2)dx &= \sum_{p\ge 2}\frac{F^{(p)}(0)}{p!}\sum_{q=0}^p \mu^{p-q} \frac{p!}{q!}{\binom p q}W_{2q}(u)\\ \label{eq:intF}
&=\sum_{q\ge 0}\sum_{ \boldsymbol{\sigma} \in \{-1,1\}^{2q} }  \sum_{ \boldsymbol{\ell} \in \mathbb{Z}^{2q} } \sum_{n\geq0}(W_{2q})^{\boldsymbol{\ell},\boldsymbol{\sigma}}\frac{ F^{(q+n)}(0)}{q!}{\binom {q+n} q}\ \prod_{1\leq j \leq 2q} u_{\boldsymbol{\ell}_j}^{\boldsymbol{\sigma}_j} \mu^{n}.
\end{align}
Let us denote $2P^{\eqref{eq:NLS}}(\mu,u)$ the formal series \eqref{eq:intF}.
It remains to prove that for a convenient choice of $r_0$
\begin{align*}
\sup_{q\geq 0} \sup_{ \substack{ \boldsymbol{\sigma} \in \{-1,1\}^{2q}  \\ \boldsymbol{\ell} \in \mathbb{Z}^{2q} } } \sum_{n\geq 0} |(W_{2q})^{\boldsymbol{\ell},\boldsymbol{\sigma}}|\frac{ |F^{(q+n)}(0)|}{q!}{\binom {q+n} q} r_0^{-2n-q} e^{3\eta_0 q + 2\eta_0 n  } \Big( 1 \vee  e^{-q}  \frac{  \langle \boldsymbol{\ell}_1^* \rangle }{  \langle \boldsymbol{\ell}_3^* \rangle^2  }  \mathbbm{1}_{(\boldsymbol{\ell},\boldsymbol{\sigma}) \in \mathbf{P}_{2q}} \Big) <\infty
\end{align*}
First we note that, taking into account the over-pairing property satisfied by $W_{2q}$ (see Lemma \ref{lem:W}) we deduce that if $(\boldsymbol{\ell},\boldsymbol{\sigma}) \in \mathbf{P}_{2q}$ and $(W_{2q})^{\boldsymbol{\ell},\boldsymbol{\sigma}}\neq0$ then $ \langle \boldsymbol{\ell}_1^* \rangle =  \langle \boldsymbol{\ell}_3^* \rangle$. As a consequence it suffices to prove that
\begin{equation}\label{true}
\sup_{q\geq 0} \sup_{ \boldsymbol{\sigma} \in \{-1,1\}^q }  \sup_{ \boldsymbol{\ell} \in \mathbb{Z}^q } \sum_{n\geq 0} |(W_{2q})^{\boldsymbol{\ell},\boldsymbol{\sigma}}|\frac{ |F^{(q+n)}(0)|}{q!}{\binom {q+n} q} r_0^{-2n-q} e^{3\eta_0 q + 2\eta_0 n  }  <\infty
\end{equation}

Let $p\geq2$. From the assumption on $F$ together with Cauchy's integral formula, we get that there exist $c>0$ such that for any $r\ge 1$,
\begin{align*}
|F^{(p)}(0)|&\lesssim p!\frac{e^{cr^\frac12}}{r^{p+1}},
\end{align*}
which, after optimizing in $r$, gives
\begin{align*}
|F^{(p)}(0)|\lesssim p!\big(\frac{ ce}{2p}\big)^{2p+1}\lesssim c^{2p} (p!)^{-1}.
\end{align*}
Using the bound \eqref{Wpinfty} on $|(W_{2q})^{\boldsymbol{\ell},\boldsymbol{\sigma}}|$, this gives
 \begin{align*}|(W_{2q})^{\boldsymbol{\ell},\boldsymbol{\sigma}}|\frac{ |F^{(q+n)}(0)|}{q!}{\binom {q+n} q} r_0^{-2n-q} e^{3\eta_0 q + 2\eta_0 n  } \lesssim
 2^q c^{2(q+n)} \frac1{n!} r_0^{-2n-q} e^{3\eta_0 q + 2\eta_0 n  }.
\end{align*}
Therefore, setting $r_0=4c^2e^{-3\eta_0}$, \eqref{true} holds true and thus 
 $P\in \mathcal{P}_{3\eta_0,r_0}$.\\

Finally in view of \eqref{Wick}-\eqref{WickH} we have  $W_2 \equiv 0$  and then, using \eqref{eq:intF}, we get $\Pi_{2} P^{\eqref{eq:NLS}} = 0$. Furthermore since \eqref{eq:intF} contains only monomials with even degree and $F(0) = F'(0)=0$, it follows that $\Pi_{\mathrm{deg}\leq 3}P^{\eqref{eq:NLS}}  = 0$

\end{proof}

\subsection{Vector field estimates, Poisson brackets and flows}

\begin{lemma}[Vector field estimates]\label{lem:gradient_reg} Let $r\in (0,1)$, $\eta \geq \eta_0$ and $\chi \in \mathscr{P}_{\eta,r}$. Then $\nabla \chi$ defines a smooth function from $B_{\ell^2_\varpi}(0,3r) $ into $\ell^2_\varpi$ and enjoys the estimates
\begin{equation}
\label{eq:lestimee_quon_veut}
\forall j \geq0, \quad \sup_{u\in B_{\ell^2_\varpi}(0,3r)}\| \mathrm{d}^j\nabla \chi(u)\|_{(\ell^2_\varpi)^j \to \ell^2_\varpi} \lesssim_j r^{-1-j}  \| \chi\|_{\mathscr{P}_{\eta,r}}.
\end{equation}
\end{lemma}
\begin{proof} We assume without loss of generality that the coefficients of $\chi$ are symmetric and that $\| \chi\|_{\mathscr{P}_{\eta,r}} = 1$.  First, we note that we have, for all  $u\in B_{\ell^2_\varpi}(0,3r)$
$$
\nabla \chi(u) = \sum_{d\geq 1} \nabla \Pi_{\mathrm{deg} =d } \chi(u) = \sum_{d\geq 1} \underbrace{2 u \partial_{\mu} \Pi_{\mathrm{deg} =d } \chi(u)}_{=:g^{(d)}(u)} +  \underbrace{ \nabla_v  \Pi_{\mathrm{deg} =d }  \chi(u)}_{=:h^{(d)}(u)}.
$$
On the one hand, proceed as in the proof of Lemma \ref{lem:ouf_cest_des_fonctions_reg}, we have that for all $u\in \ell^2_\varpi$
\begin{equation}
\label{eq:ca_commence_doucement}
\| g^{(d)}(u) \|_{\ell^2_\varpi} \lesssim (4r)^{-d} \| u \|_{\ell^2_\varpi}^{d-1}. 
\end{equation}
On the other hand, by symmetry of the coefficients, we have for all $u\in \ell^2_\varpi$ and $k\in \mathbb{Z}$
$$
h^{(d)}_k(u) = \sum_{q+2n = d} 2q \sum_{ \boldsymbol{\sigma} \in \{-1,1\}^{q-1} }  \sum_{ \boldsymbol{\ell} \in \mathbb{Z}^{q-1} }  \chi_{n}^{(\boldsymbol{\ell},k),(\boldsymbol{\sigma},-1)} \|u\|_{L^2}^{2n}  \prod_{1\leq j \leq q-1} u_{\boldsymbol{\ell}_j}^{\boldsymbol{\sigma}_j}.
$$
By definition of $\| \chi\|_{\mathscr{P}_{\eta,r}}$ and using the zero momentum condition, it implies that 
\begin{equation*}
\begin{split}
|h^{(d)}_k(u)| &\leq \sum_{\substack{q+2n=d\\ q\geq 2}}   2q \sum_{ \boldsymbol{\sigma} \in \{-1,1\}^{q-1} }  \sum_{ \boldsymbol{\sigma}_1 \boldsymbol{\ell}_1  + \cdots + \boldsymbol{\sigma}_{q-1} \boldsymbol{\ell}_{q-1} =k}  \| u\|_{\ell^2_\varpi}^{2n} r^{-q-2n} e^{-\eta q - 2\eta_0 n  } \prod_{1\leq j \leq q-1} |u_{\boldsymbol{\ell}_j}| \\
&\leq e^{-\frac{\eta}2 d} \sum_{q=2}^d    \| u\|_{\ell^2_\varpi}^{d-q}  r^{-d}  2q  e^{-\frac{\eta}2 q}   \sum_{ \boldsymbol{\sigma} \in \{-1,1\}^{q-1} }  \underbrace{\sum_{ \boldsymbol{\sigma}_1 \boldsymbol{\ell}_1  + \cdots + \boldsymbol{\sigma}_{q-1} \boldsymbol{\ell}_{q-1} =k}    \prod_{1\leq j \leq q-1} |u_{\boldsymbol{\ell}_j}| }_{=:f_k^{(q)}(u)}.
\end{split}
\end{equation*}
Now, we note that by Jensen, if $\boldsymbol{\sigma}_1 \boldsymbol{\ell}_1  + \cdots + \boldsymbol{\sigma}_{q-1} \boldsymbol{\ell}_{q-1} =k$ then
$$
\varpi_k \leq (q-1)^{s-1} ( \sum_{j=1}^{q-1} \langle  \boldsymbol{\ell}_j \rangle^s ) e^{\mathfrak{a} ( |\boldsymbol{\ell}_1|  + \cdots +  |\boldsymbol{\ell}_{q-1}|)   }.
$$
Thus applying the Young convolutional inequality and the Cauchy--Schwarz inequality, it comes
$$
\| f^{(q)}(u)\|_{\ell^2_\varpi} \leq  (q-1)^{s} \| u\|_{\ell^2_\varpi} \big( \sum_{k\in \mathbb{Z}} |u_k| e^{\mathfrak{a} |k|} \big)^{q-2} \leq  (q-1)^{s}  4^q  \| u\|_{\ell^2_\varpi}^{q-1}.
$$
As a consequence, applying the triangular inequality, it comes (since $\eta \geq \eta_0$)
\begin{equation}
\label{eq:cest_deja_moins_doux}
\| h^{(d)}(u)\|_{\ell^2_\varpi}  \leq 5^{-d} r^{-d}  \| u\|_{\ell^2_\varpi}^{d-1} \sum_{q=2}^d     2q e^{-\frac{\eta_0}2 q}   8^q (q-1)^{s} \lesssim  5^{-d}  r^{-d}  \| u\|_{\ell^2_\varpi}^{d-1}.
\end{equation}
Putting \eqref{eq:ca_commence_doucement} and \eqref{eq:cest_deja_moins_doux} together, we have proven that 
$$
\forall u\in \ell^2_\varpi, \quad  \| \nabla \Pi_{\mathrm{deg} =d } \chi(u)\|_{\ell^2_\varpi} \lesssim  5^{-d}  r^{-d}  \| u\|_{\ell^2_\varpi}^{d-1}.
$$
Then, noticing that $ \nabla \Pi_{\mathrm{deg} =d } \chi$ is a homogeneous polynomial of degree $q-1$ defined on a Hilbert space (here $\ell^2_\varpi$), this last estimate implies\footnote{We refer to \cite{BS71} for basic properties of polynomials on Banach spaces.} that $\nabla \Pi_{\mathrm{deg} =d } \chi$ is smooth map from $\ell^2_\varpi$ to $\ell^2_\varpi$ satisfying for all $u\in \ell^2_\varpi$ and all $ 0\leq j\leq d-1$,
\begin{equation}
\label{eq:toutcapourca}
\| \mathrm{d}^j \nabla \Pi_{\mathrm{deg} =d } \chi(u)\|_{(\ell^2_\varpi)^j \to \ell^2_\varpi}  \lesssim \frac{(d-1)!}{(d-1-j)!} 5^{-d}  r^{-d}  \| u\|_{\ell^2_\varpi}^{d-1-j} 
  \lesssim_j  4^{-d-1-j}  r^{-d}  \| u\|_{\ell^2_\varpi}^{d-1-j}.
\end{equation}
 Moreover, we also note that $ \mathrm{d}^j \nabla \Pi_{\mathrm{deg} =d } \chi = 0$ for $j\geq d$. It follows of \eqref{eq:toutcapourca} that the series $\sum_d \mathrm{d}^j\nabla \Pi_{\mathrm{deg} =d } \chi$ converge uniformly on $B_{\ell^2_\varpi}(0,3r)$ and so that $\nabla  \chi$ is a smooth map from $B_{\ell^2_\varpi}(0,3r)$ to $\ell^2_\varpi$ enjoying the estimate \eqref{eq:lestimee_quon_veut}.
\end{proof}

As a corollary of the previous lemma a result about the existence of Hamiltonian flows generated by Hamiltonian in $\mathscr{P}_{\eta,r}$. This kind of corollary is standard in Hamiltonian PDEs. Moreover, in Lemma \ref{lem:flow} of the next section, we prove a very similar result in a context slightly more complex. So we omit the proof.
\begin{lemma}[Hamiltonian flow]\label{lem:flow_of_chi}
Let $r\in (0,1)$, $\eta > \eta_0$ and $\chi \in    \mathcal{P}_{\eta,r} $ be such that
\begin{equation}
\label{eq:chi_assez_petit}
\| \chi\|_{\mathscr{P}_{\eta,r}}\lesssim  r^{2}.
\end{equation}
Then the flow of the equation $\ic \partial_t v = \nabla \chi(v)$ defines a smooth map $\Phi^t_\chi(\cdot)\equiv\Phi_\chi^t(\cdot):(t,u)\in[-1,1]\times B_{\ell^2_\varpi}(0,2r)\to \ell^2_\varpi$ enjoying the following properties:
 \begin{itemize} 
 \item[(i)] for each $t\in[-1,1]$,  $\Phi_\chi^t$ defines a symplectic change of variable from $B_{\ell^2_\varpi}(0,2r)$ onto an open set of $\ell^2_\varpi$  included in $B_{\ell^2_\varpi}(0,3r)$,
 \item[(ii)] $\Phi_\chi^t$ and $\dd\Phi_\chi^t$ are close to the identity:  for all $u\in B_{\ell^2_\varpi}(0,2r)$, all $t\in[0,1]$
$$
 \| \Phi_\chi^t(u)-u\|_{\ell^2_\varpi}+r\|\dd \Phi_\chi^t(u)-{\rm Id}\|_{\ell^2_\varpi\to\ell^2_\varpi}\lesssim   r^{-1} \| \chi\|_{\mathscr{P}_{\eta,r}},
 $$ 
\item[(iii)]  $\dd^2\Phi_\chi^t$ is bounded: for all $u\in B_{\ell^2_\varpi}(0,2r)$, all $t\in[0,1]$
$$
\|\dd^2 \Phi_\chi^t(u)\|_{\ell^2_\varpi\times\ell^2_\varpi\to\ell^2_\varpi}\lesssim r^{-3}\| \chi\|_{\mathscr{P}_{\eta,r}}.
$$
% \item[(iii)] $\dd\Phi_\chi^t$ is close to the identity: for all $u\in \mathcal{A}_\xi(2r)$, all $t\in[0,1]$ and all $v\in \ell^2_\varpi$
% $$\|d \Phi_\chi^t(u)(v)-v\|_{\ell^2_\varpi}\lesssim_\S r^{-4}\|\chi\|_ {\mathscr{H}_{\eta,r,\mathcal{O},\mathcal{S}}^{\mathrm{sup}}} \lesssim_\S r^\gamma\| v\|_{\ell^2_\varpi}.$$
 \end{itemize}
\end{lemma}
\begin{corollary} \label{cor:flow_hom} In the setting of Lemma \eqref{lem:flow_of_chi}. If moreover $\Pi_{\leq 3} \chi =0$ then for all $u\in B_{\ell^2_\varpi}(0,2r)$
$$ 
\| \Phi_\chi^t(u)-u\|_{\ell^2_\varpi} \lesssim   \| u \|_{\ell^2_\varpi}^3 r^{-4}  \| \chi\|_{\mathscr{P}_{\eta,r}}.
$$
\end{corollary}
\begin{proof}[Proof of Corollary \ref{cor:flow_hom}] Let $u\in B_{\ell^2_\varpi}(0,2r)$. If $r^{-1}\| u \|_{\ell^2_\varpi} \geq 1$ there is nothing new to prove. So we assume that $\rho := \| u \|_{\ell^2_\varpi} < r$. Since $\Pi_{\leq 3} \chi =0$, by homogeneity, we have that 
$$
\| \chi\|_{\mathscr{P}_{\eta,\rho}} \leq \big(\frac{\rho}{r} \big)^4 \| \chi\|_{\mathscr{P}_{\eta,r}}.
$$
Thus, applying Lemma \ref{lem:flow_of_chi} with $r \leftarrow \rho$, we get as expected that 
$$
 \| \Phi_\chi^t(u)-u\|_{\ell^2_\varpi} \lesssim   \rho^{-1} \| \chi\|_{\mathscr{P}_{\eta,\rho}} \lesssim \rho^3 r^{-4} \| \chi\|_{\mathscr{P}_{\eta,r}}.
$$
\end{proof}

Now we estimate the Poisson bracket between an  elements of $\mathcal{P}_{\eta,r}$ and an element of $\mathcal{R}_{\eta,r}$.
\begin{lemma}[Poisson bracket]\label{lem:bracket_reg}
Let $q'\geq 4$, $r\in (0,1)$, $\eta > \eta_0$. There exists a bilinear map $b( \cdot,\cdot ) :\mathcal{P}_{\eta,r} \times \Pi_{q'}   \mathcal{R}_{\eta,r} \to \bigcap_{ \eta_0 \leq\eta'<\eta} \mathcal{P}_{\eta',r} $ such that for all $P\in \mathcal{P}_{\eta,r}$ and all $\chi\in  \Pi_{q'} \mathcal{R}_{\eta,r}$, we have
\begin{equation}
\label{eq:poisson_triv}
\forall u\in B_{\ell^2_\varpi}(0,3r), \quad  b(P,\chi)(u) = \{P,\chi\}(u),
\end{equation}
\begin{equation}\label{estim:fish}
\forall q'' \geq 0, \forall \eta' \in [\eta_0,\eta), \quad \| \Pi_{q''}  b(P,\chi) \|_{\mathcal{P}_{\eta',r}} \lesssim r^{-2} q'' e^{- q''(\eta-\eta')}   \|  \Pi_{ q''-q'+2 } P\|_{ \mathcal{P}_{\eta,r}} (q')^5 \| \chi\|_{ \mathcal{R}_{\eta,r}} .
\end{equation}
From now, in view of \eqref{eq:poisson_triv}, by a slight abuse of notations, we set 
$$
\{P,\chi\} := b(P,\chi).
$$
\end{lemma}
\begin{proof} We divide the proof in two steps.

\medskip

\noindent $\bullet$ \underline{\emph{Step 1: Setting.}}
By definition, using that $P$ and $\chi$ commute with $\| \cdot \|_{L^2}^2$ (conservation of the mass),  we have
$$
\{P,\chi\}=
2\ic\sum_{j\in \Z} \partial_{\overline{v_j}} P\partial_{v_j} \chi-\partial_{\overline{v_j}} \chi\partial_{v_j} P.
$$
By uniform convergence of the series defining $P\in \mathcal{P}_{\eta,r}$ (and its derivatives) on $ B_{\ell^2_\varpi}(0,3r) $, it suffices to consider the case when $P= \Pi_{q} P \in \mathcal{P}_{\eta,r}$ is homogeneous (in $v$) of degree $q\geq 2$. In that case, using the zero momentum condition we get 
$$\{P,\chi\}=L$$
where for each $\boldsymbol{\ell}''=(\boldsymbol{\ell},\boldsymbol{\ell}')\in \Z^{q-1}\times\Z^{q'-1}$, $\boldsymbol{\sigma}''=(\boldsymbol{\sigma},\boldsymbol{\sigma}')\in \{-1,1\}^{q-1}\times\{-1,1\}^{q'-1}$, $n''\geq0$
$$
(2\ic )^{-1} L_{n''}^{\boldsymbol{\ell}'',\boldsymbol{\sigma}''}=\sum_{\substack{1\leq i\leq q,\ 1\leq i'\leq q'\\
 n''=n+n'}}  \left(P_{n}^{(\boldsymbol{\ell},j)^i,(\boldsymbol{\sigma},+1)^i} \, \chi_{n'}^{(\boldsymbol{\ell}',j)^{i'},(\boldsymbol{\sigma}',-1)^{i'}}-P_{n}^{(\boldsymbol{\ell},j)^i,(\boldsymbol{\sigma},-1)^i} \, \chi_{n'}^{(\boldsymbol{\ell}',j)^{i'},(\boldsymbol{\sigma}',+1)^{i'}}\right)
$$
where $(\boldsymbol{\ell},j)^i$ denotes the multi-index $(\boldsymbol{\ell}_1,\cdots,\boldsymbol{\ell}_{i-1},j,\boldsymbol{\ell}_{i+1},\cdots,\boldsymbol{\ell}_{q-1})$ (and similarly for $(\boldsymbol{\sigma},+1)^i$) and  as a consequence of the zero momentum condition,
$$ - j= \sum_{i=1}^{q-1} \boldsymbol\sigma_i \boldsymbol\ell_i.$$

By definition of the norms $\|\cdot\|_{ \mathcal{R}_{\eta,r}}$ and $\|\cdot\|_{ \mathcal{R}_{\eta,r}}$, since $A_{\boldsymbol{\ell},\boldsymbol{\sigma}}$ and $B_{\boldsymbol{\ell}}$ do not depend on the order of the multi-indices, we deduce that
$$
(2 qq' )^{-1} \sum_{n''\geq 0}r^{-2n''-q''} e^{\eta' q'' + 2\eta_0 n''  } |L_{n''}^{\boldsymbol{\ell}'',\boldsymbol{\sigma}''}|\leq2r^{-2}e^{-2\eta}e^{-q''(\eta-\eta')}  A_{(\boldsymbol{\ell},j),(\boldsymbol{\sigma},1)}\|P\|_{ \mathcal{P}_{\eta,r}}B_{\boldsymbol{\ell}',j}\|\chi\|_{ \mathcal{R}_{\eta,r}}.
$$
Thus, to prove \eqref{estim:fish} it suffices to prove that 
\begin{equation}\label{mu}
A_{(\boldsymbol{\ell},j),(\boldsymbol{\sigma},1)}B_{\boldsymbol{\ell}',j}\lesssim (q')^2A_{\boldsymbol{\ell}'',\boldsymbol{\sigma}''}.
\end{equation}

\medskip

\noindent $\bullet$ \underline{\emph{Step 2: Core of the proof.}} Now, we aim at proving \eqref{mu}. First we note that \eqref{mu} is trivial when $(\boldsymbol{\ell}'',\boldsymbol{\sigma}'')\notin \mathbf{P}_{q''}$. So from now we consider the case $(\boldsymbol{\ell}'',\boldsymbol{\sigma}'')\in \mathbf{P}_{q''}$ and we are going to prove
 \begin{equation}
 \label{mubis} 
 (q')^{-2}\frac{\langle( \boldsymbol{\ell}'')_1^* \rangle }{    \langle (\boldsymbol{\ell}'')_3^* \rangle^2  } \leq  \frac{\langle (\boldsymbol{\ell},j)_1^* \rangle }{    \langle (\boldsymbol{\ell},j)_3^* \rangle^2  }\mathbbm{1}_{((\boldsymbol{\ell},j),(\boldsymbol{\sigma},1)) \in \mathbf{P}_{q}} + \frac{\langle (\boldsymbol{\ell}',j)_1^* \rangle }{    \langle (\boldsymbol{\ell}',j)_3^* \rangle^2  }
  \end{equation}
which clearly implies \eqref{mu}. In order to prove \eqref{mubis} we distinguish 3 cases .
\begin{itemize}
\item \underline{Case $1$ : $\langle( \boldsymbol{\ell}'')_1^* \rangle=\langle \boldsymbol{\ell}_1^* \rangle>\langle( \boldsymbol{\ell}')_1^* \rangle$.} We decompose this case in 3 sub-cases.
\begin{itemize}
\item  \underline{Sub-case $1.1$ : $\langle j\rangle\leq \langle \boldsymbol{\ell}_3^* \rangle$.} In this case $((\boldsymbol{\ell},j),(\boldsymbol{\sigma},1))\in \mathbf{P}_{q}$, $ \langle (\boldsymbol{\ell},j)_1^* \rangle=\langle \boldsymbol{\ell}_1^* \rangle=  \langle (\boldsymbol{\ell}'')_1^* \rangle$ and $ \langle (\boldsymbol{\ell},j)_3^* \rangle=\langle \boldsymbol{\ell}_3^* \rangle\leq  \langle (\boldsymbol{\ell}'')_3^* \rangle$. Thus 
$$\frac{\langle( \boldsymbol{\ell}'')_1^* \rangle }{    \langle (\boldsymbol{\ell}'')_3^* \rangle^2  } \leq  \frac{\langle (\boldsymbol{\ell},j)_1^* \rangle }{    \langle (\boldsymbol{\ell},j)_3^* \rangle^2  }\mathbbm{1}_{((\boldsymbol{\ell},j),(\boldsymbol{\sigma},1)) \in \mathbf{P}_{q}}.$$
\item \underline{Sub-case $1.2$ : $ \langle \boldsymbol{\ell}_3^* \rangle\leq \langle j\rangle<  \langle \boldsymbol{\ell}_1^* \rangle$.} In this case we still have $((\boldsymbol{\ell},j),(\boldsymbol{\sigma},1))\in \mathbf{P}_{q}$ and $ \langle (\boldsymbol{\ell},j)_1^* \rangle=  \langle (\boldsymbol{\ell}'')_1^* \rangle$. Furthermore, since $(\boldsymbol{\ell}'',\boldsymbol{\sigma}'')$ and $((\boldsymbol{\ell},j),(\boldsymbol{\sigma},1))$ satisfies the zero momentum condition, $((\boldsymbol{\ell}',j),(\boldsymbol{\sigma}',-1))$ has also zero momentum and thus $ \langle j\rangle\leq q'\langle (\boldsymbol{\ell}')_1^* \rangle\leq q' \langle (\boldsymbol{\ell}'')_3^* \rangle$. On the other hand we also have $\langle (\boldsymbol{\ell},j)_3^* \rangle\leq \langle j\rangle$. All together we get 
$$\frac{\langle( \boldsymbol{\ell}'')_1^* \rangle }{    \langle (\boldsymbol{\ell}'')_3^* \rangle^2  } \leq  (q')^2\frac{\langle (\boldsymbol{\ell},j)_1^* \rangle }{    \langle (\boldsymbol{\ell},j)_3^* \rangle^2  }\mathbbm{1}_{((\boldsymbol{\ell},j),(\boldsymbol{\sigma},1)) \in \mathbf{P}_{q}}.$$
\item \underline{Sub-case $1.3$ : $ \langle j\rangle\geq  \langle \boldsymbol{\ell}_1^* \rangle$.} In this case $ \langle (\boldsymbol{\ell}',j)_1^* \rangle\geq  \langle (\boldsymbol{\ell}'')_1^* \rangle$ and $ \langle (\boldsymbol{\ell}',j)_3^* \rangle=\langle (\boldsymbol{\ell}')_2^* \rangle\leq \langle (\boldsymbol{\ell}'')_3^* \rangle$ and thus 
$$\frac{\langle( \boldsymbol{\ell}'')_1^* \rangle }{    \langle (\boldsymbol{\ell}'')_3^* \rangle^2  } \leq \frac{\langle (\boldsymbol{\ell}',j)_1^* \rangle }{    \langle (\boldsymbol{\ell}',j)_3^* \rangle^2  }.$$
\end{itemize}

\item  \underline{Case $2$ : $\langle( \boldsymbol{\ell}'')_1^* \rangle=\langle( \boldsymbol{\ell}')_1^* \rangle>\langle \boldsymbol{\ell}_1^* \rangle$. }
Since $(\boldsymbol{\ell}'',\boldsymbol{\sigma}'')\in \mathbf{P}_{q''}$ we know that $(\boldsymbol{\ell}',\boldsymbol{\sigma}')\in \mathbf{P}_{q'}$ and in particular $\langle (\boldsymbol{\ell}')_2^* \rangle=\langle (\boldsymbol{\ell}')_1^* \rangle$. 
 Let us denote $ \boldsymbol{\ell}'_\flat$ the multi-index $\boldsymbol{\ell}'$ from which we removed the two largest indices and $ \boldsymbol{\sigma}'_\flat$ the corresponding part of $\boldsymbol{\sigma}'$ . By construction, the momentum of $(( \boldsymbol{\ell}'_\flat,  \boldsymbol{\sigma}'_\flat),(j,-1))$ is still zero and thus 
$$
\langle (\boldsymbol{\ell}',j)_3^* \rangle\leq \langle ( \boldsymbol{\ell}'_\flat,j)_1^* \rangle\leq q'
\langle ( \boldsymbol{\ell}'_\flat,j)_2^* \rangle
=q' \langle ({ \boldsymbol{\ell}'},j)_4^*\rangle\leq q' \langle (\boldsymbol{\ell}')_3^* \rangle\leq q' \langle (\boldsymbol{\ell}'')_3^* \rangle.
$$
Therefore
$$\frac{\langle( \boldsymbol{\ell}'')_1^* \rangle }{    \langle (\boldsymbol{\ell}'')_3^* \rangle^2  } \leq (q')^2 \frac{\langle (\boldsymbol{\ell}',j)_1^* \rangle }{    \langle (\boldsymbol{\ell}',j)_3^* \rangle^2  }.$$
\item  \underline{Case $3$ : $\langle( \boldsymbol{\ell}'')_1^* \rangle=\langle \boldsymbol{\ell}_1^* \rangle=\langle( \boldsymbol{\ell}')_1^* \rangle$.} Then 
$ \langle (\boldsymbol{\ell}',j)_3^* \rangle\leq \langle (\boldsymbol{\ell}')_2^* \rangle\leq  \langle (\boldsymbol{\ell}'')_3^* \rangle$ and thus
$$\frac{\langle( \boldsymbol{\ell}'')_1^* \rangle }{    \langle (\boldsymbol{\ell}'')_3^* \rangle^2  } \leq \frac{\langle (\boldsymbol{\ell}',j)_1^* \rangle }{    \langle (\boldsymbol{\ell}',j)_3^* \rangle^2  }.$$

\end{itemize}

\end{proof}

Optimizing the Poisson bracket estimate with respect to $q''$, we get the following useful result.
\begin{corollary}
\label{cor:bracket_reg}
In the setting of Lemma \ref{lem:bracket_reg}, we have that
$$
\| \{P,\chi \} \|_{\mathcal{P}_{\eta',r}} \lesssim r^{-2} (\eta-\eta')^{-1}   \|   P\|_{ \mathcal{P}_{\eta,r}} (q')^5 \| \chi\|_{ \mathcal{R}_{\eta,r}} .
$$
\end{corollary}

We end this section with a result on the action of the Hamiltonian flow $\Phi_\chi^t$ by composition.
\begin{lemma}\label{lem:compo_reg} 
Let $r\in (0,1)$, $\eta > \eta_0$, $\alpha\in(0,\eta-\eta_0]$, $q'\geq 4$ and $\chi\in \Pi_{q'}\mathscr{R}_{\eta,r}$. Assume that $\chi$ satisfies
\begin{equation}\label{chi3_reg}
(q')^5 \| \chi\|_{\mathscr{R}_{\eta,r}}\lesssim \alpha r^2.
\end{equation}
Then for any $P\in \mathscr{P}_{\eta,r}$ and any $t\in[0,1]$, there exits $P_t \in \mathscr{P}_{\eta-\alpha,r}$ such that 
$$
P\circ\Phi^t_\chi=P_t \quad \mathrm{on} \quad  B_{\ell^2_\varpi}(0,2r)
$$ 
 and, identifying the formal series and the functions\footnote{i.e. identifying $P\circ\Phi^t_\chi$ and $P_t$}, we have 
\begin{equation}\label{estim:flow_reg}
 \| P\circ\Phi^t_\chi\|_{\mathscr{P}_{\eta-\alpha,r}} \leq 2\ \| P\|_{\mathscr{P}_{\eta,r}}.
 \end{equation}
 Moreover, if $\Pi_2 P =0$, we have that 
 \begin{equation}
 \label{eq:info_relou_sur_les_degre}
 \Pi_{\leq q'}  P\circ\Phi^t_\chi = \Pi_{\leq q'} P.
 \end{equation}
\end{lemma}
\begin{proof} Without loss of generality, we assume that $\Pi_0 P = 0$. Indeed, since $\chi$ commutes with the $L^2$ norm, we have $(\Pi_0 P) \circ\Phi^t_\chi = \Pi_0 P$.

\medskip

Then, we want to apply the Cauchy--Kowaleskaya lemma proven in the appendix (see Lemma  \ref{lem:CK}). So we set, for $p\geq 1$,
$E^{(p)} = \Pi_p \mathcal{P}_{\eta_0,r}$ and $\| \cdot \|_{E^{(p)}} := \|\cdot \|_{\mathcal{P}_{\eta_0,r}}$.
We note that, with the notations of Lemma \ref{lem:CK}, we have for any $\beta \geq0$
$$
E_\beta  = (\mathrm{Id} - \Pi_0)\mathscr{P}_{\eta_0+\beta,r} \quad \mathrm{ and } \quad \| \cdot \|_\beta=  \| \cdot\|_{\mathscr{P}_{\eta_0+\beta,r}} .
$$ 
By Lemma \ref{lem:bracket_reg}, we have for all $0\leq \beta\leq \eta-\eta_0$, all $p\geq 1$, all $P \in E_\beta$
$$
  \| \Pi_p\{P,\chi\}\|_{E^{(p)}} \lesssim \ p e^{-\beta p} \| P\|_{\beta} r^{-2}  (q')^5 \|  \chi\|_{\mathscr{P}_{\eta,r}}
  $$
and thus the linear map $L:=\{\cdot,\chi\}$ satisfies \eqref{CL}  with $C_L \sim r^{-2}  (q')^5 \| \chi\|_{\mathscr{P}_{\eta,r}}$. As a consequence of  \eqref{chi3_reg} we have $2C_L\leq\alpha$.

\medskip

 Now let $P \in \mathscr{P}_{\eta,r}$. Applying Lemma \ref{lem:CK}, there exists  $t\mapsto P_t\in C^1([0,1];\mathscr{P}_{\eta-\alpha,r})$  solution of the transport equation
$$
\left\{\begin{array}{ll} \partial_t P_t=\{P_t,\chi\},\quad t\in[0,1]\\
P_0=P. \end{array}\right. \quad \mathrm{such \ that} \quad \sup_{0\leq t\leq 1} \quad \|P_t \|_{\mathscr{P}_{\eta-\alpha,r}} \leq 2 \| P \|_{\mathscr{P}_{\eta,r}}.
$$
Now consider $g(t)=P_t\circ \Phi_\chi^{-t}(u)$ for some $u\in  B_{\ell^2_\varpi}(0,2r)$. We have
\begin{align*}
\frac{\dd}{\dd t} g(t)&=\partial_t P_t\circ \Phi_\chi^{-t}(u)+\dd P_t(\Phi_\chi^{-t}(u))(\partial_t\Phi_\chi^{-t}(u))\\
&=\{ P_t,\chi\}\circ \Phi_\chi^{-t}(u)- (\ic\nabla P_t( \Phi_\chi^{-t}(u)), \nabla \chi(\Phi_\chi^{-t}(u)))_{\ell^2} =0
\end{align*}
and thus $P_t\circ \Phi_\chi^{-t}(u) = g(t)=g(0)=P(u)$. Therefore $P_t=P\circ\Phi_\chi^t$ and we deduce that $P \circ\Phi_\chi^t$ belongs to $\mathscr{P}_{\eta-\alpha,r}$ for all $t\in[0,1]$ and satisfies \eqref{estim:flow_reg}.

\medskip

Finally it remains to prove \eqref{eq:info_relou_sur_les_degre}. First, we note that by Lemma \ref{lem:bracket_reg} and preservation of the mass $\Pi_{\leq q'-1} \{P_t,\chi\} =0$. Since $q' \geq 4$, it follows that $\partial_t \Pi_{\leq 2} P_t= \Pi_{\leq 2} \{P_t,\chi\} =0$. Since $\Pi_2 P=0$ (and $\Pi_1 P=0$ by preservation of the mass) it follows that $\Pi_{\leq 2} P_t =\Pi_{\leq 2} P=\Pi_0 P$. Therefore, applying again Lemma \ref{lem:bracket_reg} and using the preservation of the mass,  we get $\Pi_{\leq q'-1} \{P_t,\chi\} =0$ and it follows that $\partial_t \Pi_{\leq q'} P_t = \Pi_{\leq q'} \{P_t,\chi\} =0$.  Thus, we have that $\Pi_{\leq q'}P\circ\Phi_\chi^t \equiv \Pi_{\leq q'} P_t =  \Pi_{\leq q'} P$ .

%$[0,1]\ni t\mapsto G(t)=H^{(0)}\circ\Phi^t_\chi\in C^\infty(\mathcal A_\xi(2r),\R) $ which, in view of Lemma \ref{lem:ouf_cest_des_fonctions} and Lemma \ref{lem:flow} is a smooth function. One has $G(0)=H^{(0)}$ and 
%$$\partial_t G(t)= \{H^{(0)},\chi\}\circ \Phi^t_\chi=\{ G(t),\chi\}$$
%and thus $G(t)=H(t)$ for $t\in[0,1]$.
%Therefore $H(1)=H^{(0)}\circ\Phi_\chi$ and we deduce that $H^{(0)}\circ\Phi_\chi$ belongs to $\mathscr{H}_{\eta-\alpha,r,\mathcal{O},\mathcal{S}}$ and satisfies \eqref{estim:flow}.
\end{proof}

\subsection{Terms in normal form and cohomological equations}

In this subsection, we first define a notion of normal form and introduce a notation for the quadratic part of the Hamiltonian of \eqref{eq:NLS}. Then we justify it by two useful lemmas.

\begin{definition}[$Z_2$] We define the quadratic part of the Hamiltonian of \eqref{eq:NLS} by
$$
\forall u\in \ell^2_\varpi, \quad Z_2(u) = \frac12 \sum_{k \in \mathbb{Z}} k^2 |u_k|^2.
$$
\end{definition}
\begin{definition}[Terms in normal form] Let $r\in(0,1)$ and $\eta \geq 0$.
A Hamiltonian $P \in \mathscr{P}_{\eta,r}$ is in normal form if its non zero coefficients $P_{n}^{\boldsymbol{\ell},\boldsymbol{\sigma}}\neq 0$ satisfy 
\begin{equation}
\label{eq:normal_form}
 (\boldsymbol{\ell},\boldsymbol{\sigma}) \in \mathbf{P}_q  \quad \mathrm{or} \quad \big|  \boldsymbol{\sigma}_1 \boldsymbol{\ell}_1^2 + \cdots +  \boldsymbol{\sigma}_q \boldsymbol{\ell}_q^2  \big| \leq 4^{-10} q^9.
\end{equation}
 We denote by $\Pi_{\mathrm{nf}}$ the associated projection defined by restriction of the coefficients.
\end{definition}

Then, we prove that the terms in normal form are actually smoothing (up to a gauge transform).
\begin{lemma}[Terms in normal form are smoothing] \label{lem:nf_is_smoothing} For all $r\in (0,1)$ and $\eta \geq \eta_0$, we have that
$$
\Pi_{\mathrm{nf}} \mathscr{P}_{\eta,r} \subset  \mathscr{R}_{\eta,r}.
$$
\end{lemma}
\begin{proof}
We have to prove that if $q\geq 0$, $\boldsymbol{\ell} \in \mathbb{Z}^q$ and $\boldsymbol{\sigma} \in \{-1,1\}^q$ are such that, $\boldsymbol{\sigma}_1\boldsymbol{\ell}_1+ \cdots+ \boldsymbol{\sigma}_q\boldsymbol{\ell}_q=0$, $(\boldsymbol{\ell} , \boldsymbol{\sigma}) \notin \mathbf{P}_q$  and $\big| \boldsymbol{\sigma}_1 \boldsymbol{\ell}_1^2 + \cdots +  \boldsymbol{\sigma}_q \boldsymbol{\ell}_q^2 \big| \leq 4^{-10} q^9$ then  $\langle \boldsymbol{\ell}_1^* \rangle \lesssim e^q \langle \boldsymbol{\ell}_3^* \rangle^2 $.

\medskip

The cases $q\leq 2$ are implied by the case $q=3$.  So we assume from now that $q\geq 3$ and, moreover, without loss of generality, that $|\boldsymbol{\ell}_1|\geq \cdots \geq |\boldsymbol{\ell}_q|$.  We distinguish two cases.
\begin{itemize}
\item \underline{Case $(\boldsymbol{\ell}_1,\boldsymbol{\sigma}_1)=-(\boldsymbol{\ell}_2,\boldsymbol{\sigma}_2)$.} Using the zero momentum condition, we have that
$$
2|\boldsymbol{\ell}_1| = |\boldsymbol{\sigma}_3\boldsymbol{\ell}_3 +\cdots+\boldsymbol{\sigma}_q\boldsymbol{\ell}_q| \leq (q-2) |\boldsymbol{\ell}_3|.
$$
\item \underline{Case $(\boldsymbol{\ell}_1,\boldsymbol{\sigma}_1)\neq-(\boldsymbol{\ell}_2,\boldsymbol{\sigma}_2)$.} Since by assumption, we also have that $(\boldsymbol{\ell}_1,\boldsymbol{\sigma}_1)\neq (\boldsymbol{\ell}_2,-\boldsymbol{\sigma}_2)$, we deduce that
$$
|\boldsymbol{\ell}_1| \leq |  \boldsymbol{\sigma}_1 \boldsymbol{\ell}_1^2  +  \boldsymbol{\sigma}_2 \boldsymbol{\ell}_2^2 | \leq 4^{-10} q^9 + (q-2)  \boldsymbol{\ell}_3^2.
$$
\end{itemize}
We conclude by noticing that both estimates imply that  $\langle \boldsymbol{\ell}_1 \rangle \lesssim e^q \langle \boldsymbol{\ell}_3 \rangle^2 $.

\end{proof}

Finally, we prove that the solutions of the cohomological equations are also smoothing (up to a gauge transform). 
\begin{lemma}[Solutions to the cohomological equations are smoothing] \label{lem:cohom_is_smooting} For all $r\in (0,1)$, $\eta \geq \eta_0$, $q\geq 3$ and $P\in \mathcal{P}_{\eta,r}$ 
 there exists $\chi \in \Pi_q \mathcal{R}_{\eta,r}$ such that
\begin{equation}
\label{eq:coho_reg}
\forall u\in B_{\ell^2_\varpi}(0,3r), \quad \{Z_2,\chi \}(u) = - \Pi_q (\mathrm{Id} - \Pi_{\mathrm{nf}}) P(u) 
\end{equation}
and
\begin{equation}
\label{eq:est_coho_reg}
 \| \chi \|_{ \mathscr{R}_{\eta,r}} \lesssim q^{-9}  \| P \|_{ \mathscr{P}_{\eta,r}}.
\end{equation}
\end{lemma}
\begin{proof} Without loss of generality, we assume that $P= \Pi_q (\mathrm{Id} - \Pi_{\mathrm{nf}}) P$. First, we define a Hamiltonian $\chi \in  \Pi_q \mathcal{P}_{\eta,r}$ by setting
$\chi_{n}^{\boldsymbol{\ell},\boldsymbol{\sigma}} = 0$ if $P_{n}^{\boldsymbol{\ell},\boldsymbol{\sigma}} = 0$ and, else
$$
\chi_{n}^{\boldsymbol{\ell},\boldsymbol{\sigma}} := \frac{P_{n}^{\boldsymbol{\ell},\boldsymbol{\sigma}} }{\ic \Omega_{\boldsymbol{\ell,\boldsymbol{\sigma}}}} \quad \mathrm{where} \quad 
\Omega_{\boldsymbol{\ell},\boldsymbol{\sigma}} := \sum_{1\leq j\leq q} \boldsymbol{\sigma}_j \boldsymbol{\ell}_j^2.
$$
We note that it is clear that $\chi \in \Pi_q \mathcal{P}_{\eta,r}$ and a standard direct calculation proves that $\chi$ solves the comohological equation \eqref{eq:coho_reg}. The main point is to prove that $\chi \in \Pi_q \mathcal{R}_{\eta,r}$ and satisfies the bound \eqref{eq:est_coho_reg}. 

By definition, to prove it, it suffices to prove that if $ (\boldsymbol{\ell},\boldsymbol{\sigma}) \in \mathbb{Z}^q \times \{-1,1\}^q$ are such that $(\boldsymbol{\ell},\boldsymbol{\sigma}) \notin \mathbf{P}_q $, $\boldsymbol{\sigma}_1\boldsymbol{\ell}_1+ \cdots+ \boldsymbol{\sigma}_q\boldsymbol{\ell}_q=0$ and $|\Omega_{\boldsymbol{\ell},\boldsymbol{\sigma}} | > 4^{-10} q^9$ then 
$$
  |\chi_{n}^{\boldsymbol{\ell},\boldsymbol{\sigma}}| \lesssim q^{-9} |P_{n}^{\boldsymbol{\ell},\boldsymbol{\sigma}}| \quad \mathrm{and} \quad  e^{-q}  \frac{  \langle \boldsymbol{\ell}_1^* \rangle }{  \langle \boldsymbol{\ell}_3^* \rangle^2  }   |\chi_{n}^{\boldsymbol{\ell},\boldsymbol{\sigma}}| \lesssim q^{-9} |P_{n}^{\boldsymbol{\ell},\boldsymbol{\sigma}}|.
$$
The first estimate is a direct consequence of the lower bound on $|\Omega_{\boldsymbol{\ell},\boldsymbol{\sigma}} |$. For the second one, we assume moreover, without loss of generality that  $|\boldsymbol{\ell}_1|\geq \cdots \geq |\boldsymbol{\ell}_q|$ and we aim at proving that 
\begin{equation}
\label{eq:on_y_veut}
 \frac{  \langle \boldsymbol{\ell}_1 \rangle }{  \langle \boldsymbol{\ell}_3 \rangle^2  } |\Omega_{\boldsymbol{\ell},\boldsymbol{\sigma}} |^{-1} \lesssim e^{q/2}.
\end{equation}
We have to distinguish $2$ cases.
\begin{itemize}
\item \underline{Case $(\boldsymbol{\ell}_1,\boldsymbol{\sigma}_1)=-(\boldsymbol{\ell}_2,\boldsymbol{\sigma}_2)$.} Using the zero momentum condition, we have that
\begin{equation}
\label{eq:on_y_veut1}
2|\boldsymbol{\ell}_1| = |\boldsymbol{\sigma}_3\boldsymbol{\ell}_3 +\cdots+\boldsymbol{\sigma}_q\boldsymbol{\ell}_q| \leq (q-2) |\boldsymbol{\ell}_3|.
\end{equation}
\item \underline{Case $(\boldsymbol{\ell}_1,\boldsymbol{\sigma}_1)\neq-(\boldsymbol{\ell}_2,\boldsymbol{\sigma}_2)$.} Since by assumption, we also have that $(\boldsymbol{\ell}_1,\boldsymbol{\sigma}_1)\neq (\boldsymbol{\ell}_2,-\boldsymbol{\sigma}_2)$, we deduce that
$$
|\boldsymbol{\ell}_1| \leq |  \boldsymbol{\sigma}_1 \boldsymbol{\ell}_1^2  +  \boldsymbol{\sigma}_2 \boldsymbol{\ell}_2^2 | \leq |\Omega_{\boldsymbol{\ell},\boldsymbol{\sigma}}  |+ (q-2)  \boldsymbol{\ell}_3^2.
$$
As a consequence, we have
\begin{equation}
\label{eq:on_y_veut2}
\mathrm{either} \quad |\boldsymbol{\ell}_1|\leq (q-2)  \boldsymbol{\ell}_3^2 \quad \mathrm{or} \quad \frac12 |\boldsymbol{\ell}_1| \leq  |\Omega_{\boldsymbol{\ell},\boldsymbol{\sigma}}  |.
\end{equation}
\end{itemize}
Finally, recalling that, by assumption, we also have $ |\Omega_{\boldsymbol{\ell},\boldsymbol{\sigma}}  | \geq 1$, it suffices to conclude to note that both \eqref{eq:on_y_veut1} and \eqref{eq:on_y_veut2} imply \eqref{eq:on_y_veut}.
\end{proof}

\subsection{Iterative lemma}

In this subsection, we prove a lemma which is somehow the heredity step of our regularizing normal form procedure.
\begin{lemma}[Iterative lemma] Let $r\in (0,1)$, $q\geq 4$, $\eta' = \eta-q^{-2} \geq \eta_0$ and $P\in \mathcal{P}_{\eta,r}$  be such that
$$
\Pi_{\leq q-1} P =  \Pi_{\mathrm{nf}} \Pi_{\leq q-1}  P, \quad \Pi_2 P = 0  \quad \mathrm{and} \quad \| P  \|_{ \mathcal{P}_{\eta,r} } \lesssim r^2.
$$
Then, there exists $\chi \in \Pi_{q} \mathcal{R}_{\eta,r}$ and $Q \in \mathcal{P}_{\eta',r}$ such that
\begin{itemize}
\item $Q$ is in normal form up to order $q$, i.e.
$$
\Pi_{\leq q} Q =  \Pi_{\mathrm{nf}} \Pi_{\leq q}  P,
$$
\item $Z_2+P$ is conjugated to $Z_2+Q$, i.e.
$$
(Z_2 + P) \circ \Phi_\chi^1 = Z_2 + Q \quad \mathrm{on} \quad B_{\ell^2_\varpi}(0,2r)
$$
\item $Q$ and $\chi$ satisfy the bounds 
$$
 \| \chi \|_{\mathcal{R}_{\eta,r}} \lesssim q^{-9} \| P\|_{\mathcal{P}_{\eta,r}} \quad \mathrm{and} \quad
\| Q\|_{\mathcal{P}_{\eta',r}} \leq (1 + q^{-2}) \| P\|_{\mathcal{P}_{\eta,r}} .
$$
\end{itemize}
\end{lemma}
\begin{proof} Let $\chi$ be the Hamiltonian given by Lemma \ref{lem:cohom_is_smooting}. We note that $ \| \chi \|_{\mathcal{P}_{\eta,r}} \lesssim \| P\|_{\mathcal{P}_{\eta,r}} $ and so, provided that $r^{-2}\| P\|_{\mathcal{P}_{\eta,r}} $ is small enough, $\Phi_\chi^t$ is well defined up to time $1$ by Lemma \ref{lem:flow_of_chi}.

Then, by Taylor expansion, we have on $B_{\ell^2_\varpi}(0,2r)$ that
$$
(Z_2 + P) \circ \Phi_\chi^1 = Z_2 + K +R^{(1)} + R^{(2)} =: Z_2+Q
$$
where $K =\{Z_2,\chi \} +  P$,
$$
R^{(1)} = \int_0^1  \{P,\chi \}\circ \Phi_\chi^t \mathrm{d}t  \quad \mathrm{and} \quad R^{(2)} = \int_0^1 (1-t) \{\{Z_2,\chi \},\chi\}\circ \Phi_\chi^t \mathrm{d}t .
$$
Then, we recall that by construction of $\chi$, we have $\{Z_2,\chi \} = -\Pi_q (\mathrm{Id} - \Pi_{\mathrm{nf}}) P \in \mathcal{P}_{\eta,r}$. Therefore we deduce that $K\in \mathcal{P}_{\eta,r}$ satisfies
$$
 \| K \|_{\mathcal{P}_{\eta,r}} \leq  \| P \|_{\mathcal{P}_{\eta,r}} \quad \mathrm{and} \quad \Pi_{\leq q} K =  \Pi_{\mathrm{nf}} \Pi_{\leq q}  P.
$$

\medskip

Now, we focus on $ R^{(1)}$. Since $\| \chi \|_{ \mathscr{R}_{\eta,r}} \lesssim q^{-9}  \| P \|_{ \mathscr{P}_{\eta,r}}$, provided that $r^{-2} \| P \|_{ \mathscr{P}_{\eta,r}}$ is small enough, we can apply Lemma \ref{lem:compo_reg} with $\alpha = q^{-2}$, to deduce that $R^{(1)}\in   \mathscr{P}_{\eta',r}$ and
$$
 \| R^{(1)}\|_{\mathscr{P}_{\eta',r}} \lesssim  \| \{ \chi , P \} \|_{\mathscr{P}_{\eta',r}}.
$$
Then applying the Poisson bracket estimate of Corollary \ref{cor:bracket_reg}, we deduce that
$$
 \| R^{(1)}\|_{\mathscr{P}_{\eta',r}} \lesssim  r^{-2} (\eta -\eta')^{-1} q^5 \|  \chi \|_{ \mathscr{R}_{\eta,r} } \| P \|_{ \mathscr{P}_{\eta,r} } \lesssim q^{-2} \big( r^{-2}\| P \|_{ \mathscr{P}_{\eta,r} } \big) \| P \|_{ \mathscr{P}_{\eta,r} }
$$
and so, provided that $r^{-2}\| P \|_{ \mathscr{P}_{\eta,r}}$ is small enough that $ \| R^{(1)}\|_{\mathscr{P}_{\eta',r}} \leq \frac12 q^{-2}\| P \|_{ \mathscr{P}_{\eta,r} }$.
Using that $\{Z_2,\chi \} = -\Pi_q (\mathrm{Id} - \Pi_{\mathrm{nf}}) P \in \mathcal{P}_{\eta,r}$, we deduce similarly that $R^{(2)}\in   \mathscr{P}_{\eta',r}$ and that it also satisfies $ \| R^{(2)}\|_{\mathscr{P}_{\eta',r}} \leq \frac12 q^{-2}\| P \|_{ \mathscr{P}_{\eta,r} }$.

\medskip

Finally, to conclude the proof it suffices to prove that $\Pi_{\leq q} R^{(1)} =\Pi_{\leq q} R^{(2)} =0$. 
By Lemma \ref{lem:bracket_reg}, since $\Pi_2 P = 0$, we have that $\Pi_{\leq 2} P = \Pi_0 P$ and so by preservation of the mass that $\Pi_{\leq q} \{ P,\chi\} = \Pi_{\leq q} \{ \Pi_0  P,\chi\} = 0$. Therefore since $q\geq 4$, using the relation \eqref{eq:info_relou_sur_les_degre} of Lemma \ref{lem:compo_reg}, we have that $\Pi_{\leq q} \{ P,\chi\}  \circ \Phi_\chi^t = 0$ (for $t\in[0,1]$) and so that $\Pi_{\leq q} R^{(1)}=0$. The same arguments apply to prove that $\Pi_{\leq q} R^{(2)} =0$.

\end{proof}

\subsection{Convergence of the regularizing normal form : proof of Theorem \ref{thm:reg}}
\label{sub:proof_thm_reg}

We divide the proof in $4$ steps.

\noindent \underline{\emph{Step $1$: Setting.}}
First, we recall that by Proposition \ref{prop:PNLS}, there exist $r_0 \in (0,1)$ and $P^{\eqref{eq:NLS}} \in \mathcal{P}_{3 \eta_0,r_0}$ such that
$$
H^{\eqref{eq:NLS}} = Z_2 + P^{\eqref{eq:NLS}} \quad \mathrm{on} \quad B_{\ell^2_\varpi}(0,3r_0),
$$
$\Pi_{\mathrm{deg}\leq 3 }  P^{\eqref{eq:NLS}} = 0$ and $\Pi_{2 }  P^{\eqref{eq:NLS}} = 0$. Therefore, we have by homogeneity that for all $r\leq r_0$, 
$$
\forall r\leq r_0, \quad \| P^{\eqref{eq:NLS}} \|_{\mathcal{P}(3 \eta_0,r)} \lesssim r^4.
$$
From now, we fix a radius $r\leq r_0$ such that $\| P^{\eqref{eq:NLS}} \|_{\mathcal{P}(3 \eta_0,r)} \leq r^3$ and that we assume small enough (i.e. in other words, we will assume a finite number of extra smallness assumptions on $r$). 

We set
$$
\eta_q = 3\eta_0 - \sum_{4\leq p < q} p^{-2}, \quad \mathrm{for} \quad q\geq 4.
$$
Applying the iterative lemma, provided that $r$ is small enough, we get some families  of Hamiltonians $\chi^{(q)} \in \Pi_q \mathcal{R}_{\eta_q,r}$ and $P^{(q)} \in \mathcal{P}_{\eta_q,r}$ such that for all $q\geq 4$ 
\begin{itemize} 
\item $P^{(4)} = P^{\eqref{eq:NLS}}$,
\item $P^{(q+1)}$ is in normal form up to order $q$ and stationary in some sense, i.e. 
\begin{equation}
\label{eq:is_in_normal_form}
\Pi_{\leq q} P^{(q+1)} =  \Pi_{\mathrm{nf}} \Pi_{\leq q}  P^{(q)},
\end{equation}
\item $Z_2+P^{(q)}$ is conjugated to $Z_2+P^{(q+1)}$, i.e. 
\begin{equation}
\label{eq:ca_conjugue}
(Z_2 + P^{(q)}) \circ \Phi_{\chi^{(q)}}^1 = Z_2 + P^{(q+1)} \quad \mathrm{on} \quad B_{\ell^2_\varpi}(0,2r)
\end{equation}
\item they enjoy the bounds  
\begin{equation}
\label{eq:bound_induction}
\| \chi^{(q)} \|_{\mathcal{P}_{\eta_q,r}} \lesssim q^{-9} r^3 \quad \mathrm{and} \quad \| P^{(q)}\|_{\mathcal{P}_{\eta_q,r}} \leq C_q r^3  \quad \mathrm{where} \quad C_q :=\prod_{4\leq p <q} 1+p^{-2}.
\end{equation}
\end{itemize}

Thanks to the bound \eqref{eq:bound_induction} on $\chi^{(q)}$, we know, by Lemma \ref{lem:flow_of_chi} that, provided that $r$ is small enough, for all $q\geq 4$ and $\lambda \in (1,2)$, $\Phi_{\chi^{(q)}}^1$ maps $B_{\ell^2_\varpi}(0,\lambda r)$ into $B_{\ell^2_\varpi}(0,(1+ \varepsilon q^{-2}) \lambda r)$ where $\varepsilon>0$ is a constant so that
$$
\prod_{q\geq 4 } 1+\varepsilon q^{-2} \leq 2.
$$ 
As a consequence, for $q\geq 4$, the map
$$
\tau^{(q)} := \Phi_{\chi^{(4)}}^1 \circ \cdots \circ \Phi_{\chi^{(q-1)}}^1 : B_{\ell^2_\varpi}(0,r) \to B_{\ell^2_\varpi}(0,3r) \quad \mathrm{is \ well \ defined}.
$$
Moreover, for all $q\geq 4$, $\tau^{(q)}$ is clearly smooth, symplectic and satisfies, by \eqref{eq:ca_conjugue},
\begin{equation}
\label{eq:jet_lag}
(Z_2 + P^{\eqref{eq:NLS}}) \circ  \tau^{(q)} = Z_2 + P^{(q)} \quad \mathrm{on} \quad B_{\ell^2_\varpi}(0,r).
\end{equation}

\medskip

\noindent \underline{\emph{Step $2$: Convergence of $P^{(q)}$.}} We aim at passing to the limit as $q$ goes to $+\infty$. First, thanks to the relation \eqref{eq:is_in_normal_form}, it make sense to define $P^{(\infty)} \in \mathcal{P}_{\eta_\infty,r}$ by
$$
\forall q\geq 4, \quad \Pi_{\leq q} P^{(\infty)} :=  \Pi_{\mathrm{nf}}  \Pi_{\leq q} P^{(q)} 
$$
and by the estimate \eqref{eq:bound_induction}, it satisfies $\| P^{(\infty)}\|_{\mathcal{P}_{\eta_\infty,r}} \leq C_\infty r^3 $. Moreover, we note that by \eqref{eq:is_in_normal_form}, $P^{(\infty)}$ is in normal form, i.e.
$$
P^{(\infty)} = \Pi_{\mathrm{nr}} P^{(\infty)}.
$$

\medskip

Then, we aim at proving that $P^{(q)}$ converges pointwize to $P^{(\infty)}$ on $B_{\ell^2_\varpi}(0,r)$. So let $u\in B_{\ell^2_\varpi}(0,r)$, $q\geq 4$ and set $\rho := \| u\|_{\ell^2_\varpi} < r$. First, we note that by construction,
$$
\Pi_{\leq q-1} P^{(q)} = \Pi_{\leq q-1}  P^{(\infty)}.
$$
Therefore, we have by homogeneity that
$$
\| P^{(q)} - P^{(\infty)} \|_{\mathcal{P}_{\eta_\infty,\rho}} \leq \big(\frac{\rho}{r}\big)^{q} \| P^{(q)} - P^{(\infty)} \|_{\mathcal{P}_{\eta_\infty,r}} \lesssim  r^3 \big(\frac{\rho}{r}\big)^{q}.
$$
 Thus, applying Lemma \ref{lem:ouf_cest_des_fonctions_reg}, we get that
 $$
 |P^{(q)}(u) - P^{(\infty)}(u)| \lesssim \| P^{(q)} - P^{(\infty)} \|_{\mathcal{P}_{\eta_\infty,\rho}} \lesssim  r^3 \big(\frac{\rho}{r}\big)^{q} \mathop{\longrightarrow}_{q \to \infty} 0.
 $$

\medskip

\noindent \underline{\emph{Step $3$: Convergence of $\tau^{(q)}$.}} We aim at proving that $\tau^{(q)}$ converges in $C^1_b( B_{\ell^2_\varpi}(0,r) ;\ell^2_\varpi )$ to a map $\tau^{(\infty)}$. We are going to prove that it is a Cauchy sequence. So let $4\leq p \leq q$ and $u\in  B_{\ell^2_\varpi}(0,r)$. Applying the triangular inequality and the mean value inequality, we get that
\begin{equation*}
\begin{split}
\| \tau^{(q)}(u) - \tau^{(p)}(u) \|_{\ell^2_\varpi} &\leq \sum_{j = p+1}^{q} \| \tau^{(j)}(u) - \tau^{(j-1)}(u) \|_{\ell^2_\varpi} \\
 &\leq  \sum_{j = p+1}^{q}  \sup_{v\in B_{\ell^2_\varpi}(0, (1+\varepsilon j^{-2}) r ) } \| \mathrm{d} \tau^{(j-1)}(v) \|_{\ell^2_\varpi \to \ell^2_\varpi} \| u - \Phi_{\chi^{(j)}}^1(u)\|_{\ell^2_\varpi}
\end{split}
\end{equation*}
Then, we note that,  by Lemma \ref{lem:flow_of_chi}   we have
$$
\| \mathrm{d} \tau^{(j-1)}(v) \|_{\ell^2_\varpi \to \ell^2_\varpi} \leq \prod_{4\leq \ell \leq j-2} 1+ r^{-2} \| \chi^{(\ell)} \|_{\mathcal{P}_{\eta_q,r}} \mathop{\lesssim}^{\eqref{eq:bound_induction}} 1.
$$
Therefore, using \eqref{eq:bound_induction} and the bound given by Corollary \ref{cor:flow_hom}, we have that 
\begin{equation}
\label{eq:cauchy_reg}
\| \tau^{(q)}(u) - \tau^{(p)}(u) \|_{\ell^2_\varpi} \lesssim r^{-1} \| u\|_{\ell^2_\varpi}^3 \sum_{j = p+1}^{+\infty} j^{-9}.
\end{equation}
Similarly, still using the bound given by Lemma \ref{lem:flow_of_chi} and \eqref{eq:bound_induction}, we get that
\begin{equation}
\label{eq:cauchy_d}
\begin{split}
\| \mathrm{d} \tau^{(q)}(u) - \mathrm{d}\tau^{(p)}(u) \|_{\ell^2_\varpi\to \ell^2_\varpi} 
 \lesssim  \sum_{j = p+1}^{q} \| \mathrm{Id} - \mathrm{d} \Phi_{\chi^{(j)}}^1(u)\|_{\ell^2_\varpi \to \to \ell^2_\varpi} 
 \lesssim  r \sum_{j = p+1}^{q} j^{-9}.
\end{split}
\end{equation}
The estimates \eqref{eq:cauchy_reg} and \eqref{eq:cauchy_d} prove that  $(\tau^{(q)})_q$ is a Cauchy sequence in $C^1_b( B_{\ell^2_\varpi}(0,r) ;\ell^2_\varpi )$ which is a   Banach space. Thus it converges in this space and we denote by $\tau^{(\infty)} \in C^1_b( B_{\ell^2_\varpi}(0,r) ;\ell^2_\varpi )$ its limit.

\medskip

\noindent \underline{\emph{Step $4$: Conclusion.}} Passing to the limit in \eqref{eq:jet_lag}, we have that
$$
(Z_2 + P^{\eqref{eq:NLS}}) \circ  \tau^{(\infty)} = Z_2 + P^{(\infty)} \quad \mathrm{on} \quad B_{\ell^2_\varpi}(0,r).
$$
Since the maps $\tau^{(q)}$ preserve the $L^2$ norm, are symplectic and are invariant by gauge transform, the same is true for $\tau^{(\infty)}$. Moreover passing to the limit  as $q\to \infty$ in \eqref{eq:cauchy_reg} and \eqref{eq:cauchy_d} for $p=4$, we have that (provided that $r$ is small enough)
$$
\| \tau^{(\infty)}(u) - \mathrm{Id}\|_{\ell^2_\varpi}\lesssim r^{-1} \| u\|_{\ell^2_\varpi}^3 \quad \mathrm{and} \quad \| \mathrm{d}\tau^{(\infty)}(u) - \mathrm{Id}\|_{\ell^2_\varpi \to \ell^2_\varpi } \leq \frac12.
$$
Since $P^{(\infty)}$ is in normal form, we have by Lemma \ref{lem:nf_is_smoothing} that $P^{(\infty)} \in \mathcal{R}_{\eta_\infty,r}$. Finally, to conclude this proof it suffices to prove that $\Pi_{\mathrm{deg}\leq 5 } P^{(\infty)} = H^{(0)}_{1} - Z_2 $. 

\medskip

First, we note that by conservation of the mass, $\Pi_{\mathrm{deg}\leq 5 } P^{(\infty)} = \Pi_{\mathrm{deg}\leq 4 } P^{(\infty)}$. Then, we note that by construction
$$
 \Pi_{\mathrm{deg}\leq 4 } P^{(\infty)} = \Pi_{\mathrm{deg}\leq 4 } \Pi_{\leq 4} P^{(\infty)} =  \Pi_{\mathrm{deg}\leq 4 } \Pi_{\leq 4} P^{(4)} = \Pi_{\mathrm{deg}\leq 4 } \Pi_{\mathrm{nf}}\Pi_{\leq 4} P^{\eqref{eq:NLS}} = \Pi_{\mathrm{nf}} \Pi_{\mathrm{deg}\leq 4 } P^{\eqref{eq:NLS}} .
$$
Moreover, by uniqueness of the Taylor expansion, we have that for all $v\in \ell^2_\varpi$
$$
\Pi_{\mathrm{deg}\leq 4 } P^{\eqref{eq:NLS}}(v)  = \frac{f'(0)}{4}  \| v \|_{L^4}^4 = \frac{f'(0)}{4} \sum_{\boldsymbol{\ell}_1+\boldsymbol{\ell}_2 - \boldsymbol{\ell}_3 - \boldsymbol{\ell}_4 = 0} v_{\boldsymbol{\ell}_1} v_{\boldsymbol{\ell}_2} \overline{v_{\boldsymbol{\ell}_3}} \overline{v_{\boldsymbol{\ell}_4}}.
$$
Finally, we note that, as in \cite{KP96} that if $\boldsymbol{\ell}_1+\boldsymbol{\ell}_2 - \boldsymbol{\ell}_3 - \boldsymbol{\ell}_4 = 0$ and $|\boldsymbol{\ell}_1^2+\boldsymbol{\ell}_2^2 - \boldsymbol{\ell}_3^2 - \boldsymbol{\ell}_4^2| \leq \frac12$ then $(\boldsymbol{\ell}_1,\boldsymbol{\ell}_2) = (\boldsymbol{\ell}_3,\boldsymbol{\ell}_4)$ or $(\boldsymbol{\ell}_1,\boldsymbol{\ell}_2) = (\boldsymbol{\ell}_4,\boldsymbol{\ell}_3)$. Thus, we deduce that
$$
\Pi_{\mathrm{deg}\leq 5 } P^{(\infty)}(v) =  \frac{f'(0)}{2} \| v\|_{L^2}^4 - \frac{f'(0)}{4} \sum_{k\in \mathbb{Z}} |v_k|^2 = (H^{(0)}_{1} - Z_2)(v).  
$$

\section{Analytic Hamiltonians in partial action-angle variables} \label{sec:Ham_form}
In this section, we define the framework within which we will present and prove our results and the basic tools that will help us to do so.

\subsection{Spaces of Hamiltonians}
\begin{definition}[Formal series without parameters $\mathscr{H}_{\eta,r,\mathcal{S}}^{\mathrm{cste}}$] \label{def:class} Let $\mathcal{S} \subset \mathbb{Z}$ be a finite set, $r>0$ and $\eta \geq 0$. We denote by $ \mathscr{H}_{\eta,r,\mathcal{S}}^{\mathrm{cste}}$, the set of the formal series $P$, of the form
$$
P(\mu,y,z,v) := \sum_{n \in \mathbb{N}} \sum_{ \boldsymbol{k} \in \mathbb{Z}^{\mathcal{S}}  }\sum_{ \boldsymbol{m} \in \mathbb{N}^{\mathcal{S}}  } \sum_{q\in \mathbb{N}} \sum_{ \boldsymbol{\ell} \in (\mathcal{S}^c)^q  }  \sum_{ \boldsymbol{\sigma} \in  \{-1,1\}^q  } P_{n,\boldsymbol{k},\boldsymbol{m}}^{\boldsymbol{\ell},\boldsymbol{\sigma}} \, \mu^n \, \underbrace{ \prod_{i=1}^q v_{\boldsymbol{\ell}_i}^{\boldsymbol{\sigma}_i} }_{=:v_{\boldsymbol{\ell}}^{\boldsymbol{\sigma}} } \, \underbrace{ \prod_{j\in \mathcal{S}} y^{\boldsymbol{m}_j}_j z^{\boldsymbol{k}_j}_j  }_{=: y^{\boldsymbol{m}} \, z^{ \boldsymbol{k} }} \, 
$$
%$$
%P(\xi) := \sum_{n \in \mathbb{N}} \mu(\xi)^n P_n(\xi)
%$$
%with
%$$
%P_n(\xi) := \sum_{ \boldsymbol{k} \in \mathbb{Z}^{\mathcal{S}}  }\sum_{ \boldsymbol{m} \in \mathbb{N}^{\mathcal{S}}  } y(\xi)^{\boldsymbol{m}}e^{i \boldsymbol{k} \cdot \theta} P_{n,\boldsymbol{k},\boldsymbol{m}}(\xi)
%$$
%and
%$$
%P_{n,\boldsymbol{k},\boldsymbol{m}}(\xi;u) := \sum_{q\in \mathbb{N}} \sum_{ \boldsymbol{\ell} \in (\mathcal{S}^c)^q  }  \sum_{ \boldsymbol{\sigma} \in  \{-1,1\}^q  } P_{n,\boldsymbol{k},\boldsymbol{m}}^{\boldsymbol{\ell},\boldsymbol{\sigma}}(\xi) u_{\boldsymbol{\ell}}^{\boldsymbol{\sigma}}
%$$
with the usual convention $ v_{\ell}^{-1} :=  \overline{v_{\ell}}$ and whose coefficients $ P_{n,\boldsymbol{k},\boldsymbol{m}}^{\boldsymbol{\ell},\boldsymbol{\sigma}} \in \mathbb{C}$ satisfy the relations
\begin{itemize}
\item reality : $P_{n,-\boldsymbol{k},\boldsymbol{m}}^{\boldsymbol{\ell},-\boldsymbol{\sigma}} = \overline{ P_{n,\boldsymbol{k},\boldsymbol{m}}^{\boldsymbol{\ell},\boldsymbol{\sigma}} }$,
\item symmetry : $P_{n,\boldsymbol{k},\boldsymbol{m}}^{\varphi \boldsymbol{\ell},\varphi \boldsymbol{\sigma}} = P_{n,\boldsymbol{k},\boldsymbol{m}}^{\boldsymbol{\ell},\boldsymbol{\sigma}} $ for all $\varphi \in \mathfrak{S}_q$,
\item preservation of the mass : if $P_{n,\boldsymbol{k},\boldsymbol{m}}^{\boldsymbol{\ell},\boldsymbol{\sigma}} \neq 0$ then
$
\displaystyle \sum_{j\in \mathcal{S}} \boldsymbol{k}_j + \sum_{i=1}^q \boldsymbol{\sigma}_i  =0,
$
\item preservation of the momentum :  if $P_{n,\boldsymbol{k},\boldsymbol{m}}^{\boldsymbol{\ell},\boldsymbol{\sigma}} \neq 0$ then
$
\displaystyle \sum_{j\in \mathcal{S}} j\boldsymbol{k}_j + \sum_{i=1}^q \boldsymbol{\sigma}_i \boldsymbol{\ell}_i=0
$
\end{itemize}
and the bounds
$$
\| P\|_{\mathscr{H}^{\mathrm{cste}}_{\eta,r,\mathcal{S}}} := \sup_{n,\boldsymbol{k},\boldsymbol{m},\boldsymbol{\ell},\boldsymbol{\sigma}} | P_{n,\boldsymbol{k},\boldsymbol{m}}^{\boldsymbol{\ell},\boldsymbol{\sigma}} | e^{\eta ( \| \boldsymbol{k} \|_{\ell^1} +  2\| \boldsymbol{m} \|_{\ell^1}  + q +2n)}   r^{2  \| \boldsymbol{m} \|_{\ell^1} +  q+2n} \Theta_{\ell}^{-1} <\infty
$$
where \footnote{by conventions, $\Theta_{\ell}=1$ if $q=0$ and the second factor equals $(\varpi_{\boldsymbol{\ell}^*_1})^{-1}$ when $q=1$.}
%$$
%\Theta_{\ell}   :=  1 \wedge \frac{ \varpi_{\boldsymbol{\ell}_1} \cdots \varpi_{\boldsymbol{\ell}_q}  }{  \max_{1\leq j\leq q} \varpi_{\boldsymbol{\ell}_j}^2}  \wedge  \frac{ \langle \boldsymbol{\ell}_1 \rangle^{2\delta} \cdots \langle \boldsymbol{\ell}_q \rangle^{2\delta}  }{  \max_{1\leq j\leq q} \langle \boldsymbol{\ell}_j\rangle^{5\delta} }.
%$$

\begin{equation}\label{eq:theta}
\Theta_{\ell}   :=  1\ \wedge\  \varpi^\flat_{\boldsymbol{\ell}^*_3} \cdots \varpi^\flat_{\boldsymbol{\ell}^*_q}  
\frac{\varpi_{\boldsymbol{\ell}^*_2}}{\varpi_{\boldsymbol{\ell}^*_1}} 
\ \wedge\  \frac{ \langle \boldsymbol{\ell}_1 \rangle^{2\delta} \cdots \langle \boldsymbol{\ell}_q \rangle^{2\delta}  }{ \langle\boldsymbol{\ell}^*_1\rangle^{5\delta} }\ ,
\end{equation}
 $\varpi^\flat_j:=\langle j\rangle^{-1} \varpi_j$ and $\boldsymbol{\ell}^*_1,\cdots,\boldsymbol{\ell}^*_q$ is a rearrangement of $\boldsymbol{\ell}_1,\cdots,\boldsymbol{\ell}_q$ satisfying $|\boldsymbol{\ell}^*_1|\geq\cdots\geq|\boldsymbol{\ell}^*_q|$.
\end{definition}
We will justify this definition, and in particular the choice of the weight  $\Theta_{\ell} $, in Remark \ref{rem:space}.
\begin{definition}[Functions $y,\theta,z,\mu$] Being given $u \in L^2(\mathbb{T})$, $\xi \in \ell^1(\mathbb{Z};\mathbb{R})$, $k\in \mathbb{Z}$ we define
\begin{itemize}
\item the re-centered action $y_k(\xi;u) = |u_k|^2 - \xi_k$,
\item the angle $\theta_k(u) \in \mathbb{T}$ as an angle satisfying $\sqrt{|u_k|} e^{i \theta_k(u)}= u_k$,
\item $z_k(u) \in \mathbb{C}$, is the complex number of modulus one defined by $z_k(u) =e^{i\theta_k}$,
\item the re-centered mass $\mu(\xi;u) = \|u\|_{L^2}^2 - \| \xi \|_{\ell^1}$.
\end{itemize}
\end{definition}

\begin{definition}[Annulus] Being given $\xi \in \ell^1_{\varpi^2}$ and $\rho>0$, we set
$$
\mathcal{A}_\xi(\rho) = \Big\{ u\in \ell^2_{\varpi} \ | \ \sum_{k \in \mathbb{Z}}  \big| |u_k|^2 - \xi_k \big| \varpi_k^2 \, \leq \rho^2 \Big\}.
$$
%on pourrait l'√©crire avec des y_k mais c'est peu etre moins frappant
\end{definition}

\begin{lemma}[Formal series define functions] \label{lem:ouf_cest_des_fonctions}Let $\mathcal{S} \subset \mathbb{Z}$ be a finite set, $r\in (0,\frac14)$, $\eta \geq \eta_0$, $P\in  \mathscr{H}^{\mathrm{cste}}_{\eta,r,\mathcal{S}}$ and $\xi \in ((4r)^2,1)^\mathcal{S}$. Then 
$$
u\mapsto P(\mu(\xi;u),y(\xi;u),z(u),u )=:P(\xi;u)
$$
 defines a smooth function from $\mathcal{A}_\xi(3r) $ into $\mathbb{R}$ and enjoys the estimate
$$
\sup_{u\in \mathcal{A}_\xi(3r)}|P(\xi;u)| \lesssim_S \| P\|_{\mathscr{H}^{\mathrm{cste}}_{\eta,r,\mathcal{S}}}
$$
\end{lemma}

\begin{proof} It suffices to prove that, for $\xi \in ((4r)^2,1)^\mathcal{S}$, the formal series defining $P(\xi;u)$ converges uniformly for $u\in \mathcal{A}_\xi(3r)$. First we note that for  $u\in \mathcal{A}_\xi(3r)$ we have $\mu(\xi;u)\leq 9r^2$ and $y_j\leq 9r^2$ for $j\in\S$. So by definition of the norm $\| P\|_{\mathscr{H}_{\eta,r,\mathcal{S}}^{\mathrm{cste}}}$  it suffices to prove
$$\sum_{n \in \mathbb{N}} \sum_{ \boldsymbol{k} \in \mathbb{Z}^{\mathcal{S}}  }\sum_{ \boldsymbol{m} \in \mathbb{N}^{\mathcal{S}}  } \sum_{q\in \mathbb{N}} 3^{2\| \boldsymbol{m} \|_{\ell^1}  + 2n}\  e^{-\eta ( \| \boldsymbol{k} \|_{\ell^1} +  2\| \boldsymbol{m} \|_{\ell^1}  + q +2n)}\sum_{\substack{ \boldsymbol{\ell} \in (\mathcal{S}^c)^q,\  \boldsymbol{\sigma} \in  \{-1,1\}^q  \\ \mathcal M(\boldsymbol{k},\boldsymbol{\ell}, \boldsymbol{\sigma})=0}}  w_{\boldsymbol{\ell}_1} \cdots w_{\boldsymbol{\ell}_q}    \, \lesssim_\S 1
$$
where $w_j= |u_j|r^{-1}$ for $j\in\mathcal{S}^c$ and the momentum $\mathcal M(\boldsymbol{k},\boldsymbol{\ell}, \boldsymbol{\sigma})$ is defined by 
$$
\mathcal M(\boldsymbol{k},\boldsymbol{\ell}, \boldsymbol{\sigma}):= \sum_{j\in \mathcal{S}} j\boldsymbol{k}_j + \sum_{i=1}^q \sigma_i \ell_i.
$$ 
Noticing that, since $s\geq1 $,  $w\in \ell^1(\mathcal{S}^c)$ with $\| w\|_{\ell^1(\mathcal{S}^c)} \leq 3r^{-1}\| u\|_{\ell_\varpi^2(\mathcal{S}^c)}\leq9$ for $u\in \mathcal{A}_\xi(3r)$, we deduce  
$$\sum_{ \boldsymbol{\ell} \in (\mathcal{S}^c)^q,\  \boldsymbol{\sigma} \in  \{-1,1\}^q  }  w_{\boldsymbol{\ell}_1} \cdots w_{\boldsymbol{\ell}_q}    \, \leq (18)^q.$$
On the other hand the summability with respect to $\boldsymbol{k} ,q,n,\boldsymbol{m} $ is ensured by the exponential $e^{-\eta ( \| \boldsymbol{k} \|_{\ell^1} +  2\| \boldsymbol{m} \|_{\ell^1}  + q +2n)}$ with $\eta\geq \eta_0$.
\end{proof}
\begin{remark}\label{rem:space}
The definition of the class $ \mathscr{H}^{\mathrm{cste}}_{\eta,r,\mathcal{S}}$ has been chosen in such a way the function $u\mapsto  P(\mu(\xi;u),y(\xi;u),z(u),u )$ has good property. In particular,
in \eqref{eq:theta}, the third term encoded the $\delta$-regularizing property inherit from Theorem \ref{thm:reg} while the second term will be used to obtain a regularity property for the gradient: it maps $\mathcal{A}_\xi(3r) $ into $\ell^2_\varpi$ (see Lemma \ref{lem:gradient}).
\end{remark}

\begin{definition}[Hamiltonians with parameters $\mathscr{H}_{\eta,r,\mathcal{O},\mathcal{S}}$] \label{def:lip} Let $\mathcal{S} \subset \mathbb{Z}$ be a finite set, $r\in(0,\frac14)$, $\mathcal{O}\subset ((4r)^2;1)^\mathcal{S}$,  $\eta \geq 0$. We denote by
$$
 \mathscr{H}_{\eta,r,\mathcal{O},\mathcal{S}} =: \mathrm{Lip}(\mathcal{O}; \mathscr{H}_{\eta,r,\mathcal{S}}^{\mathrm{cste}} )
 $$
 the space of the Lipschitz functions from $\mathcal{O}$ into the space of the formal series $\mathscr{H}_{\eta,r,\mathcal{S}}^{\mathrm{cste}} $. Naturally we equip $ \mathscr{H}_{\eta,r,\mathcal{O},\mathcal{S}} $ with the norms
 $$
 \| P\|_{\mathscr{H}_{\eta,r,\mathcal{O},\mathcal{S}}^{\mathrm{sup}}} = \sup_{\xi \in \mathcal{O}} \| P(\xi) \|_{\mathscr{H}_{\eta,r,\mathcal{S}}^{\mathrm{cste}}},
\quad 
 \| P\|_{\mathscr{H}_{\eta,r,\mathcal{O},\mathcal{S}}^{\mathrm{lip}}} = \sup_{\substack{\xi,\zeta \in \mathcal{O}\\ \xi \neq \zeta}} \frac{\| P(\xi) - P(\zeta)  \|_{\mathscr{H}_{\eta,r,\mathcal{S}}^{\mathrm{cste}}}}{\| \xi -\zeta \|_{\ell^1}},
 $$
 and
 $$
\| P\|_{\mathscr{H}_{\eta,r,\mathcal{O},\mathcal{S}}^{\mathrm{tot}}} := \| P\|_{\mathscr{H}_{\eta,r,\mathcal{O},\mathcal{S}}^{\mathrm{sup}}}  + r^2 \| P\|_{\mathscr{H}_{\eta,r,\mathcal{O},\mathcal{S}}^{\mathrm{lip}}}. 
$$
Being given $\xi \in \mathcal{O}$, thanks to Lemma \ref{lem:ouf_cest_des_fonctions}, to shorten the notation, we set
$$
P(\xi;u) := P(\xi)(\xi;u).
$$
To avoid possible confusion, we denote by $P[\xi]:= P(\xi;\cdot)$ the smooth function from $ \mathcal{A}_\xi(3r)$ into $\mathbb{R}$ defined just above.
\end{definition}
Finally, we define a useful projection to identify some terms which require more attention.
\begin{definition}[Projection $\Pi_{n=0,\boldsymbol{m}=0}$] Let $\mathcal{S} \subset \mathbb{Z}$ be a finite set, $r>0$ and $\eta \geq 0$. We denote by $\Pi_{n=0,\boldsymbol{m}=0}$ the projection defined (by restriction of the coefficients) for all $P\in \mathscr{H}^{\mathrm{cste}}_{\eta,r,\mathcal{S}}$ 
$$
\Pi_{n=0,\boldsymbol{m}=0} P(\mu,y,z,v):=  \sum_{\boldsymbol{k} \in \mathbb{Z}^{\mathcal{S}}} \sum_{q\in \mathbb{N}}   \sum_{ \boldsymbol{\ell} \in (\mathcal{S}^c)^q  }  \sum_{ \boldsymbol{\sigma} \in  \{-1,1\}^q  } P_{0,\boldsymbol{k},0}^{\boldsymbol{\ell},\boldsymbol{\sigma}} \, v_{\boldsymbol{\ell}}^{\boldsymbol{\sigma}} \, z^{ \boldsymbol{k} } \, .
$$
\end{definition}

\subsection{Vector field estimates, Poisson brackets and flows}
In this subsection, we prove all the basic estimates  we will need later.
\begin{lemma}[Vector field]\label{lem:gradient} Let $\mathcal{S} \subset \mathbb{Z}$ be a finite set, $r\in (0,\frac14)$, $\mathcal{O}\subset ((4r)^2,1)^\mathcal{S}$, $\eta \geq \eta_0$, $\xi \in \mathcal{O}$ and $P\in \mathscr{H}_{\eta,r,\mathcal{O},\mathcal{S}}$. Then $\nabla P[\xi]$ defines a smooth function from $\mathcal{A}_\xi(3r) $ into $\ell^2_\varpi$ and enjoys the estimate
$$
\sup_{u\in \mathcal{A}_\xi(3r)}\|\nabla P(\xi;u)\|_{\ell^2_\varpi} \lesssim_{\mathcal{S}} r^{-2}  \| P\|_{\mathscr{H}_{\eta,r,\mathcal{O},\mathcal{S}}^{\mathrm{sup}}},
$$
$$
\sup_{u\in \mathcal{A}_\xi(3r)}\|\mathrm{d}\nabla P(\xi;u)\|_{\ell^2_\varpi\to \ell^2_\varpi} \lesssim_{\mathcal{S}} r^{-4}  \| P\|_{\mathscr{H}_{\eta,r,\mathcal{O},\mathcal{S}}^{\mathrm{sup}}},
$$
$$
\sup_{u\in \mathcal{A}_\xi(3r)}\|\dd^2\nabla P(\xi;u)\|_{\ell^2_\varpi\times\ell^2_\varpi\to \ell^2_\varpi} \lesssim_{\mathcal{S}} r^{-6}  \| P\|_{\mathscr{H}_{\eta,r,\mathcal{O},\mathcal{S}}^{\mathrm{sup}}}.
$$
Moreover, if $\chi = \Pi_{n=0,\boldsymbol{m}=0} P$ we have
$$
\sup_{u\in \mathcal{A}_\xi(3r)}\|\nabla\chi(\xi;u)\|_{\ell^2_\varpi} \lesssim_{\mathcal{S}} r^{-1}  \| P\|_{\mathscr{H}_{\eta,r,\mathcal{O},\mathcal{S}}^{\mathrm{sup}}},
$$
$$
\sup_{u\in \mathcal{A}_\xi(3r)}\|\mathrm{d}\nabla \chi(\xi;u)\|_{\ell^2_\varpi\to \ell^2_\varpi} \lesssim_{\mathcal{S}} r^{-2}  \| P\|_{\mathscr{H}_{\eta,r,\mathcal{O},\mathcal{S}}^{\mathrm{sup}}},
$$
%$$
%\sup_{u\in \mathcal{A}_\xi(3r)}\|\dd^2\nabla P(\xi;u)\|_{\ell^2_\varpi\times\ell^2_\varpi\to \ell^2_\varpi} \lesssim_{\mathcal{S}} r^{-3}  \| P\|_{\mathscr{H}_{\eta,r,\mathcal{O},\mathcal{S}}^{\mathrm{sup}}}.
%$$
\end{lemma}
\begin{proof} We first consider  $\nabla P$. The smoothness will be a consequence of the uniform convergence in $\ell^2_\varpi$ of the formal series defining $\nabla P(\xi;u)$ on $\mathcal{A}_\xi(3r)$, so we focus on the estimate of $\sum_{j\in\Z} \varpi_j^2  |\partial_{\overline{u_j}} P(\xi;u)|^2$. We note that for $j,k\in\Z$
$$\partial_{\overline{u_j}}\mu={u_j},\quad \partial_{\overline{u_j}}y_k=\delta_{j,k}{u_j},\quad \partial_{\overline{u_j}} z_k=-\delta_{j,k}\frac{z_j^2}{2|u_j|}$$
and thus for $j\in \mathcal{S}$
$$\partial_{\overline{u_j}} P(\xi;u)={u_j}\partial_\mu P(\xi;u)+{u_j}\partial_{y_j} P(\xi;u)-\frac{z_j^2}{2|u_j|}\partial_{z_j} P(\xi;u)$$
and for $j\in \mathcal{S}^c$
$$\partial_{\overline{u_j}} P(\xi;u)={u_j}\partial_\mu P(\xi;u)+\partial_{\overline{v_j}} P(\xi;u).$$
Since $\cS$ is finite, one easily obtains (note that $\partial_\mu P$ and $\partial_{y_j} P$ lose $r^2$ and that $|u_j|\geq r$)
$$\sum_{j\in\cS} \varpi_j^2  |\partial_{\overline{u_j}} P(\xi;u)|^2\lesssim_{\mathcal{S}} r^{-2}  \| P\|_{\mathscr{H}_{\eta,r,\mathcal{O},\mathcal{S}}^{\mathrm{sup}}}$$
by the same kind of estimates as in the proof of Lemma \ref{lem:ouf_cest_des_fonctions}. So we can focus on the sum over $\cSc$ which is an infinite sum and for which we have to deal with the weight $\Theta_{\ell}$. Notice that using again Lemma \ref{lem:ouf_cest_des_fonctions} we have 
$$\sum_{j\in\cSc} \varpi_j^2 |{u_j}\partial_\mu P(\xi;u)|^2\lesssim_{\mathcal{S}} r^{-2}  \| P\|_{\mathscr{H}_{\eta,r,\mathcal{O},\mathcal{S}}^{\mathrm{sup}}}.$$
Therefore it remains to estimate $\sum_{j\in\cSc} \varpi_j^2 |\partial_{\overline{v_j}} P(\xi;u)|^2$.
To begin with, one has using the symmetry condition on the coefficients $P_{n,\boldsymbol{k},\boldsymbol{m}}^{\boldsymbol{\ell}\boldsymbol{\sigma}}$,
$$
|\partial_{\overline{v_j}}P(\mu,y,z,v)| \leq \sum_{n \in \mathbb{N}} \sum_{ \boldsymbol{k} \in \mathbb{Z}^{\mathcal{S}}  }\sum_{ \boldsymbol{m} \in \mathbb{N}^{\mathcal{S}}  } \sum_{q\geq1} q\sum_{\substack{ \boldsymbol{\ell} \in (\mathcal{S}^c)^{q},\ \boldsymbol{\sigma} \in  \{-1,1\}^{q}  \\ \mathcal M(\boldsymbol{k},(j,\boldsymbol{\ell}),(-1,\boldsymbol{\sigma}))=0}} \big|P_{n,\boldsymbol{k},\boldsymbol{m}}^{(j,\boldsymbol{\ell}),\boldsymbol{(1,\sigma})} \, v_{\boldsymbol{\ell}}^{\boldsymbol{\sigma}} \, y^{\boldsymbol{m}} \, z^{ \boldsymbol{k} } \, \mu^n\big|\,.
$$
Notice that 
$$\varpi_j\Theta_{(j,\boldsymbol{\ell})}\leq  \varpi_{\boldsymbol{\ell}^*_1}\varpi^\flat_{\boldsymbol{\ell}^*_2}\cdots \varpi^\flat_{\boldsymbol{\ell}^*_{q}}$$
where we recall that $\boldsymbol{\ell}^*_1,\cdots,\boldsymbol{\ell}^*_q$ is a rearrangement of $\boldsymbol{\ell}_1,\cdots,\boldsymbol{\ell}_q$ satisfying $|\boldsymbol{\ell}^*_1|\geq\cdots\geq|\boldsymbol{\ell}^*_q|$ and $\varpi^\flat_j=\langle j\rangle^{-1}\varpi_j$.
Therefore, by definition of the norm $\| P\|_{\mathscr{H}_{\eta,r,\mathcal{O},\mathcal{S}}^{\mathrm{sup}}}$ and of the domain $\mathcal{A}_\xi(3r)$, it suffices to prove
$$\sum_{j\in\S^c}\left| \sum_{\substack{n \in \mathbb{N}, \ \boldsymbol{k} \in \mathbb{Z}^{\mathcal{S}} \\ \boldsymbol{m} \in \mathbb{N}^{\mathcal{S}} ,\ q\geq 0}} q\, 3^{2\| m\|_{\ell^1}+2n}\ e^{-\eta ( \| \boldsymbol{k} \|_{\ell^1} +  2\| \boldsymbol{m} \|_{\ell^1}  + (q+1) +2n)} \sum_{\substack{ \boldsymbol{\ell} \in (\mathcal{S}^c)^{q},\ \boldsymbol{\sigma} \in  \{-1,1\}^{q}  \\ \mathcal M(\boldsymbol{k}(j,\boldsymbol{\ell}),(-1,\boldsymbol{\sigma}))=0}}   t_{\boldsymbol{\ell}^*_1}w_{\boldsymbol{\ell}^*_2}\cdots w_{\boldsymbol{\ell}^*_{q}}\, 
   \right|^2\lesssim_\S 1
$$
where $w_j=\varpi^\flat_j|u_j|r^{-1}$ and $t_j= \varpi_j|u_j|r^{-1}$ for $j\in\mathcal{S}^c$.

First, since $\eta\geq\eta_0$, we can get rid of the sum with respect to $\boldsymbol{m}$ and $n$ using the exponential $e^{-\eta (  2\| \boldsymbol{m} \|_{\ell^1}  +2n)}$ and then it remains to prove that
$$\sum_{j\in\S^c}\left| \sum_{ \boldsymbol{k} \in \mathbb{Z}^{\mathcal{S}}  }\sum_{q\geq0} \ q\,e^{-\eta ( \| \boldsymbol{k} \|_{\ell^1}   + (q+1))}\ a^{(q,\boldsymbol{k})}_j\ 
   \right|^2\lesssim_\S 1$$
where
$$a^{(q,\boldsymbol{k})}_j:= \sum_{\substack{ \boldsymbol{\ell} \in (\mathcal{S}^c)^{q},\ \boldsymbol{\sigma} \in  \{-1,1\}^{q}  \\ \mathcal M(\boldsymbol{k},(j,\boldsymbol{\ell}),(-1,\boldsymbol{\sigma}))=0}}   t_{\boldsymbol{\ell}^*_1}w_{\boldsymbol{\ell}^*_2}\cdots w_{\boldsymbol{\ell}^*_{q}}\leq q\sum_{\substack{ \boldsymbol{\ell} \in (\mathcal{S}^c)^{q},\ \boldsymbol{\sigma} \in  \{-1,1\}^{q}  \\ \mathcal M(\boldsymbol{k},(j,\boldsymbol{\ell}),(-1,\boldsymbol{\sigma}))=0}}   t_{\boldsymbol{\ell}_1}w_{\boldsymbol{\ell}_2}\cdots w_{\boldsymbol{\ell}_{q}}.$$
We note that $t\in \ell^2(S^c)$ and $w\in \ell^1(S^c)$ with $\| w\|_{\ell^1},\ \| t\|_{\ell^2}\lesssim 1$ for $u\in \mathcal{A}_\xi(3r)$. Therefore by Young convolutional inequality we get that
$(a^{(q,\boldsymbol{k})}_j)_{j\in \S^c}$
belongs to $\ell^2$ and we have uniformly in $q$ and $\boldsymbol{k}$
$$\| a^{(q,\boldsymbol{k})}\|_{\ell^2}\leq qC^q$$
for an universal constant $C$.
Thus, by triangular inequality, we have
\begin{align*}
\big(\sum_{j\in\S^c}&\big| \sum_{ \boldsymbol{k} \in \mathbb{Z}^{\mathcal{S}}  }\sum_{q\geq0} q\, e^{-\eta ( \| \boldsymbol{k} \|_{\ell^1}   + q)} \ a^{(q,\boldsymbol{k})}_j\ 
   \big|^2\big)^{\frac12} 
   &\leq  \sum_{ \boldsymbol{k} \in \mathbb{Z}^{\mathcal{S}}  }\sum_{q\geq0}e^{-\eta ( \| \boldsymbol{k} \|_{\ell^1}   + q)}\big(\sum_{j\in\S^c}   |a^{(q,\boldsymbol{k})}_j|^2\big)^{\frac12}
      \lesssim_\S 1\, .\end{align*}
%   &\leq \sum_{j\in\S^c} \sum_{ \boldsymbol{k} \in \mathbb{Z}^{\mathcal{S}}  }\sum_{q\geq0} q\, e^{-\eta ( \| \boldsymbol{k} \|_{\ell^1}   + q)} \ \ a^{(q,\boldsymbol{k})}_j\ \sum_{ \boldsymbol{k}' \in \mathbb{Z}^{\mathcal{S}}  }\sum_{q'\geq0} q'\, e^{-\eta ( \| \boldsymbol{k}' \|_{\ell^1}   + q')} \ \ a^{(q',\boldsymbol{k}')}_j
%  \\
%  &\leq
%   \sum_{ \boldsymbol{k} \in \mathbb{Z}^{\mathcal{S}}  }\sum_{q\geq0} q\, e^{-\eta ( \| \boldsymbol{k} \|_{\ell^1}   + q)}
%  \sum_{ \boldsymbol{k}' \in \mathbb{Z}^{\mathcal{S}}  }\sum_{q'\geq0} q'\, e^{-\eta ( \| \boldsymbol{k}' \|_{\ell^1}   + q')}
%  \sum_{j\in\S^c} a^{(q,\boldsymbol{k})}_ja^{(q',\boldsymbol{k}')}_j
%  \\
%  & \leq \left(\sum_{ \boldsymbol{k} \in \mathbb{Z}^{\mathcal{S}}  }\sum_{q\geq0} C^qq^2\, e^{-\eta ( \| \boldsymbol{k} \|_{\ell^1}   + q)}\right)^2
%   \lesssim_\S 1\, .\end{align*}

The estimate of $\mathrm{d}\nabla P(\xi;u)$ and $\mathrm{d}^2\nabla P(\xi;u)$ are obtained in a similar way. 
Moreover, the estimates concerning $\chi = \Pi_{n=0,\boldsymbol{m}=0} P$ follow by noticing that only the derivatives with respect to $\mu$ or to $y$ lose $r^2$, the derivatives with respect to $u$ or to $z$ lose only $r$.
\end{proof}

We also have to control flows of Hamiltonians in $\mathscr{H}_{\eta,r,\mathcal{O},\mathcal{S}}$.

\begin{lemma}[Flows]\label{lem:flow} Let $\mathcal{S} \subset \mathbb{Z}$ be a finite set, $r\in (0,\frac14)$, $\mathcal{O}\subset ((4r)^2,1)^\mathcal{S}$,  $\eta \geq \eta_0$ and $\chi\in \mathscr{H}_{\eta,r,\mathcal{O},\mathcal{S}}$. Assume that 
\begin{equation}\label{chi}\| \chi\|_{\mathscr{H}_{\eta,r,\mathcal{O},\mathcal{S}}^{\mathrm{sup}}}\lesssim_\S  r^{4},\end{equation}
or, in the case $\chi = \Pi_{n=0,\boldsymbol{m}=0}\chi$,
\begin{equation}\label{chi2}\| \chi\|_{\mathscr{H}_{\eta,r,\mathcal{O},\mathcal{S}}^{\mathrm{sup}}}\lesssim_\S  r^{2}.\end{equation}
Assume also that 
\begin{equation}\label{r}r\lesssim_{\S} 1.\end{equation}
Then, for all $\xi\in\mathcal O$, the flow of the equation $\ic \partial_t v = \nabla \chi(\xi;v)$ defines a smooth map $\Phi^t_\chi(\xi;\cdot)\equiv\Phi_\chi^t(\cdot):(t,u)\in[-1,1]\times \mathcal{A}_\xi(2r)\to \ell^2_\varpi$ enjoying the following properties:
 \begin{itemize} 
 \item[(i)] for each $t\in[-1,1]$,  $\Phi_\chi^t$ defines a symplectic change of variable from $\mathcal{A}_\xi(2r) $ onto an open set of $\ell^2_\varpi$  included in $\mathcal{A}_\xi(3r) $,
  \item[(ii)] for all $r\leq\rho\leq2r$, $\Phi_\chi^1$ maps $\mathcal{A}_\xi(\rho) $ into $\mathcal{A}_\xi(\rho') $ with $0<\rho'-\rho\lesssim_\S r^{-2}\|\chi\|_ {\mathscr{H}_{\eta,r,\mathcal{O},\mathcal{S}}^{\mathrm{sup}}}$.
 \item[(iii)] $\Phi_\chi^t$ and $\dd\Phi_\chi^t$ are close to the identity:  for all $u\in \mathcal{A}_\xi(2r)$, all $t\in[0,1]$
 $$ \| \Phi_\chi^t(u)-u\|_{\ell^2_\varpi}+r^2\|\dd \Phi_\chi^t(u)-{\rm Id}\|_{\ell^2_\varpi\to\ell^2_\varpi}\lesssim_\S r^{-2}\|\chi\|_ {\mathscr{H}_{\eta,r,\mathcal{O},\mathcal{S}}^{\mathrm{sup}}}$$
 and, in the case $\chi = \Pi_{n=0,\boldsymbol{m}=0}\chi$,
$$ \| \Phi_\chi^t(u)-u\|_{\ell^2_\varpi}+r\|\dd \Phi_\chi^t(u)-{\rm Id}\|_{\ell^2_\varpi\to\ell^2_\varpi}\lesssim_\S   r^{-1}\|\chi\|_ {\mathscr{H}_{\eta,r,\mathcal{O},\mathcal{S}}^{\mathrm{sup}}},$$ 
\item[(iv)]  $\dd^2\Phi_\chi^t$ is bounded: for all $u\in \mathcal{A}_\xi(2r)$, all $t\in[0,1]$
$$\|\dd^2 \Phi_\chi^t(u)\|_{\ell^2_\varpi\times\ell^2_\varpi\to\ell^2_\varpi}\lesssim_\S r^{-6}\|\chi\|_ {\mathscr{H}_{\eta,r,\mathcal{O},\mathcal{S}}^{\mathrm{sup}}}.$$
% \item[(iii)] $\dd\Phi_\chi^t$ is close to the identity: for all $u\in \mathcal{A}_\xi(2r)$, all $t\in[0,1]$ and all $v\in \ell^2_\varpi$
% $$\|d \Phi_\chi^t(u)(v)-v\|_{\ell^2_\varpi}\lesssim_\S r^{-4}\|\chi\|_ {\mathscr{H}_{\eta,r,\mathcal{O},\mathcal{S}}^{\mathrm{sup}}} \lesssim_\S r^\gamma\| v\|_{\ell^2_\varpi}.$$
 \end{itemize}
 \end{lemma}
% We note that, with a slight abuse of notation, we have written in the estimates of the lemma  $\Phi^t_\chi(\cdot)$ instead of $\Phi^t_\chi(\xi;\cdot)$.

 \begin{proof} Since, by Lemma \ref{lem:gradient}, the vector field $\nabla \chi$ is locally-Lipschitz, the local existence and the smoothness of the flow $\Phi_{\chi}^t$ is ensured by the Cauchy-Lipchitz Theorem.  So, to prove assertion (i), we now have to check  that the flow exists for $|t|\leq 1$. Without loss of generality we only consider positive times. More precisely, let $T>0$ and $v\in C^1([0,T); \ell^2_\varpi)$ be a solution of the Cauchy problem
$$
-\ic\partial_t v(t) = \nabla \chi(\xi;v(t)) \quad \mathrm{and} \quad v(0) = u \in \mathcal A_\xi(2r).
$$
It suffices to prove that if  $0\leq t< T$ and $t \leq 1$ then $v(t) \in \mathcal A_\xi(\frac52r)$. We set $I = [0,T) \cap [0,1]$ and we aim at proving that $J=I$ where
$$
J = \big\{ t \in I \ | \ \forall \tau \in [0,t], \  v(\tau) \in \mathcal A_\xi(\frac52r)  \big\}.
$$
Since $v$ is continuous, $J$ is clearly non-empty and closed in $I$. Moreover, if $t\in J$ then, using Lemma \ref{lem:gradient},
\begin{align*}
\sum_{k \in \mathbb{Z}}  \big| |v_k(t)|^2 - \xi_k \big| \varpi_k^2 &\leq \sum_{k \in \mathbb{Z}}  \big| |v_k(0)|^2 - \xi_k \big| \varpi_k^2 + \int_0^t
\sum_{k \in \mathbb{Z}}  2|\partial_t v_k(\tau)||v_k(\tau)|\varpi_k^2 \mathrm{d}\tau\\
&\leq 4r^2+ 2\sup_{w\in \mathcal{A}_\xi(3r)}\big(\|\nabla \chi(\xi;w)\|_{\ell^2_\varpi} \| w\|_{\ell^2_\varpi} \big)\, .
 \end{align*}
 Thus, using  that $w\in\mathcal{A}_\xi(3r)$ implies $ \|w\|_{\ell^2_\varpi}\lesssim_\S 1$, we get
 $$ 
  \sum_{k \in \mathbb{Z}}  \big| |v_k(t)|^2 - \xi_k \big| \varpi_k^2 -4r^2\lesssim_\S   r^{-2}\| \chi\|_{\mathscr{H}_{\eta,r,\mathcal{O},\mathcal{S}}^{\mathrm{sup}}}.
  $$
Then, using \eqref{chi} and \eqref{r},   we get  $v(t) \in \mathcal A_\xi(\frac52r)$ for all $t\in J$. Therefore, since $v$ is continuous, $J$ is open and so, since $I$ is connected, we have $J = I$. Moreover, we notice that, when $\chi = \Pi_{n=0,\boldsymbol{m}=0}\chi$,  Lemma \ref{lem:gradient} gives us a better estimate and the same conclusion remains true under the hypothesis \eqref{chi2}.
Since the symplecticity is ensured by the fact that $\Phi_{\chi}$ is a Hamiltonian flow, we have proved assertion (i).

Similarly, if initially $v(0) = u \in \mathcal A_\xi(\rho)$ then we get
$$  \sum_{k \in \mathbb{Z}}  \big| |v_k(t)|^2 - \xi_k \big| \varpi_k^2 -\rho^2 \lesssim_\S   r^{-2}\| \chi\|_{\mathscr{H}_{\eta,r,\mathcal{O},\mathcal{S}}^{\mathrm{sup}}}$$
and therefore $v(1)\in  \mathcal A_\xi(\rho')$ with $|(\rho')^2-\rho^2|\lesssim_\S r^{-2}\| \chi\|_{\mathscr{H}_{\eta,r,\mathcal{O},\mathcal{S}}^{\mathrm{sup}}}.$

To prove assertion (iii) we note that, by Lemma \ref{lem:gradient}, %and hypothesis \eqref{chi} (resp. \eqref{chi2}),
\begin{align*}
\|v(t)-u\|_{\ell^2_\varpi} \leq  \int_0^t \sup_{w\in \mathcal{A}_\xi(3r)}\|\nabla \chi(\xi;w)\|_{\ell^2_\varpi} \mathrm{d} \tau \lesssim_\S r^{-2}\|\chi\|_ {\mathscr{H}_{\eta,r,\mathcal{O},\mathcal{S}}^{\mathrm{sup}}}
.\end{align*}

Now we focus on $\mathrm{d} \Phi_\chi^t(u)$. If $w\in\ell^2_\varpi$, we have
\begin{equation}
\label{eq:ladernierejespere}
-\ic\partial_t \mathrm{d} \Phi_\chi^t(u)(w) = \mathrm{d} \nabla \chi (v(t))( \mathrm{d} \Phi_\chi^t(u)(w))
\end{equation}
and thus by Lemma \ref{lem:gradient} %and hypothesis \eqref{chi} (resp. \eqref{chi2})
\begin{equation*}
\|  \mathrm{d} \Phi_\chi^t(u)(w) -w \|_{\ell^2_\varpi} \leq \! \int_{0}^t \|  \mathrm{d} \nabla \chi (v(\tau))( \mathrm{d} \Phi_\chi^{\tau}(u)(w)) \|_{\ell^2_\varpi}\mathrm{d} \tau 
\lesssim_\S r^{-4}\|\chi\|_ {\mathscr{H}_{\eta,r,\mathcal{O},\mathcal{S}}^{\mathrm{sup}}} \! \int_{0}^t  \| \mathrm{d} \Phi_\chi^{\tau}(u)(w)\|_{\ell^2_\varpi} \mathrm{d} \tau .
\end{equation*}

Therefore,  we  get by Gr\"onwall's inequality:
 $$
 \| \mathrm{d} \Phi_\chi^{t}(u)(w)\|_{\ell^2_\varpi} \leq \| w\|_{\ell^2_\varpi} \mathrm{exp} ( C_Sr^{-4}\|\chi\|_ {\mathscr{H}_{\eta,r,\mathcal{O},\mathcal{S}}^{\mathrm{sup}}} t ) .
 $$ 
 for all $t\in[0,1]$  where  $C_\S>0$ is a constant depending on $\S$. Using hypothesis \eqref{chi}, we get that  for all $t\in[0,1]$, $\| \mathrm{d} \Phi_\chi^{t}(u)(w)\|_{\ell^2_\varpi} \leq 2\| w\|_{\ell^2_\varpi} $.
  Then we deduce  
  \begin{equation*}
\|  \mathrm{d} \Phi_\chi^t(u)(w) -w \|_{\ell^2_\varpi} \lesssim_\S r^{-4}\|\chi\|_ {\mathscr{H}_{\eta,r,\mathcal{O},\mathcal{S}}^{\mathrm{sup}}} \| w\|_{\ell^2_\varpi}. 
\end{equation*}
Again we notice that, when $\chi = \Pi_{n=0,\boldsymbol{m}=0}\chi$,  Lemma \ref{lem:gradient} gives us a better estimate which leads to better estimates of $ \| \Phi_\chi^t(u)-u\|_{\ell^2_\varpi}+r\|\dd \Phi_\chi^t(u)-\rm{Id}\|_{\ell^2_\varpi\to\ell^2_\varpi}$ as stated in (iii). 

To prove (iv) we derive once again \eqref{eq:ladernierejespere} to get for $w,w'\in\ell^2_\varpi$
$$-i\partial_t \dd^2 \Phi_\chi^t(u)(w,w') = \dd^2 \nabla \chi (v(t))( \dd \Phi_\chi^t(u)(w),\dd \Phi_\chi^t(u)(w'))+
\dd \nabla \chi (v(t))( \dd^2 \Phi_\chi^t(u)(w,w')) $$
and thus by Lemma \ref{lem:gradient} 
\begin{equation*}
\|\dd^2 \Phi_\chi^t(u)(w,w') \|_{\ell^2_\varpi}
\lesssim_\S r^{-6}\|\chi\|_ {\mathscr{H}_{\eta,r,\mathcal{O},\mathcal{S}}^{\mathrm{sup}}}+ r^{-4}\|\chi\|_ {\mathscr{H}_{\eta,r,\mathcal{O},\mathcal{S}}^{\mathrm{sup}}} \int_{0}^t  \|\dd^2 \Phi_\chi^\tau(u)(w,w') \|_{\ell^2_\varpi} \mathrm{d} \tau\, .
\end{equation*}
Then we conclude by Gr\"onwall's inequality.
\end{proof}

Now we want to estimate the Poisson bracket between to elements of $\mathscr{H}_{\eta,r,\mathcal{O},\mathcal{S}}$ but first we prove a technical lemma that we will need.
\begin{lemma}\label{ouille} Let $\S\subset \Z$, $q,q'\geq0$ and $q''=q+q'$. Let $(\S^c)^{q''}\ni\boldsymbol{\ell}''=(\boldsymbol{\ell},\boldsymbol{\ell}')\in (\S^c)^{q}\times(\S^c)^{q'}$ and $j\in\S^c$ then
\begin{equation}\label{ouille1}\Theta_{\boldsymbol{\ell}''}^{-1}\leq\Theta_{\boldsymbol{\ell}}^{-1}\Theta_{\boldsymbol{\ell}'}^{-1}  \end{equation}
and
\begin{equation}\label{ouille2}\Theta_{\boldsymbol\ell''}^{-1}\leq\Theta_{(\boldsymbol\ell,j)}^{-1}\Theta_{(\boldsymbol\ell',j)}^{-1}\ .\end{equation}
\end{lemma}
\begin{proof} Without loss of generality we can assume that $\boldsymbol{\ell}$, $\boldsymbol{\ell}'$ and $\boldsymbol{\ell}''$ are ordered and non negative (i.e. $\boldsymbol{\ell}_1\geq\cdots\boldsymbol{\ell}_q\geq0$ and the same for $\boldsymbol{\ell}'$ and $\boldsymbol{\ell}''$) and that $j$ is non negative. The definition of $\Theta_{\ell} $ (see \eqref{eq:theta}) contains two special cases $q=0,1$. In this proof we will consider only the generic case, i.e. $q,q'\geq2$, all the remaining special cases are treated similarly (i.e. without extra trick).
First we focus on \eqref{ouille1}, we recall that by definition (since we assume $q\geq2$)
$$
\Theta_{\ell}   =  1\ \wedge\  \varpi^\flat_{\boldsymbol{\ell}_3} \cdots \varpi^\flat_{\boldsymbol{\ell}_q}  
\frac{\varpi_{\boldsymbol{\ell}_2}}{\varpi_{\boldsymbol{\ell}_1}} 
\ \wedge\  \frac{ \langle \boldsymbol{\ell}_1 \rangle^{2\delta} \cdots \langle \boldsymbol{\ell}_q \rangle^{2\delta}  }{   \langle \boldsymbol{\ell}_1\rangle^{5\delta} }\ ,
$$
 where $\varpi^\flat_k=\langle k\rangle^{-1} \varpi_k$.
 To begin with,  let us denote
$$c_{\boldsymbol{\ell}}=\frac{ \langle \boldsymbol{\ell}_1 \rangle^{2} \cdots \langle \boldsymbol{\ell}_q \rangle^{2}  }{   \langle \boldsymbol{\ell}_1\rangle^{5} }.$$
Assuming for instance  $ {\boldsymbol{\ell}_1''}= {\boldsymbol{\ell}}_1$ (the case $ {\boldsymbol{\ell}_1''}= \boldsymbol{\ell}_1'$ is treated in the same way), we have
\begin{align*}
 c_{\boldsymbol{\ell}''}&= \frac{ \langle \boldsymbol{\ell}_1 \rangle^{2} \cdots \langle \boldsymbol{\ell}_{q} \rangle^{2}  }{  \langle \boldsymbol{\ell}_1\rangle^{5}}\ { \langle \boldsymbol{\ell}'_1 \rangle^{2} \cdots \langle \boldsymbol{\ell}'_{q'} \rangle^{2}  } 
 \geq \frac{ \langle \boldsymbol{\ell}_1 \rangle^{2} \cdots \langle \boldsymbol{\ell}_{q} \rangle^{2}  }{  \langle \boldsymbol{\ell}_1\rangle^{5}}\ \frac{ \langle \boldsymbol{\ell}'_1 \rangle^{2} \cdots \langle \boldsymbol{\ell}'_{q'} \rangle^{2}  }{  \langle \boldsymbol{\ell}'_1\rangle^{5}}=c_{\boldsymbol{\ell}}c_{\boldsymbol{\ell}'}
 \, .\end{align*}
So it remains to prove that 
$$ \varpi^\flat_{\boldsymbol{\ell}''_3} \cdots \varpi^\flat_{\boldsymbol{\ell}''_{q''}}  
\frac{\varpi_{\boldsymbol{\ell}''_2}}{\varpi_{\boldsymbol{\ell}''_1}}\geq \varpi^\flat_{\boldsymbol{\ell}_3} \cdots \varpi^\flat_{\boldsymbol{\ell}_q}  
\frac{\varpi_{\boldsymbol{\ell}_2}}{\varpi_{\boldsymbol{\ell}_1}}\ \varpi^\flat_{\boldsymbol{\ell}'_3} \cdots \varpi^\flat_{\boldsymbol{\ell}'_{q'}}  
\frac{\varpi_{\boldsymbol{\ell}'_2}}{\varpi_{\boldsymbol{\ell}'_1}}\, . $$
We first notice that
$$\varpi^\flat_{\boldsymbol{\ell}_3} \cdots \varpi^\flat_{\boldsymbol{\ell}_q}  
\frac{\varpi_{\boldsymbol{\ell}_2}}{\varpi_{\boldsymbol{\ell}_1}}=
a_{\boldsymbol{\ell}}\ b_{\boldsymbol{\ell}}$$
where
$$a_{\boldsymbol{\ell}}= \frac{\prod_{i=1}^q \varpi^\flat_{\boldsymbol{\ell}_i}}{(\varpi^\flat_{\boldsymbol{\ell}_1})^2} \quad \text{and}\quad
b_{\boldsymbol{\ell}}=\frac{\langle {\boldsymbol{\ell}}_2\rangle}{\langle {\boldsymbol{\ell}}_1\rangle}.$$
We note for further use that $b_{\boldsymbol{\ell}}\leq1$.
In the same way than we estimate $ c_{\boldsymbol{\ell}''}$ we easily get $$a_{\boldsymbol{\ell}''}\geq a_{\boldsymbol{\ell}}a_{\boldsymbol{\ell}'}.$$
On the other hand, when $ {\boldsymbol{\ell}_1''}= {\boldsymbol{\ell}}_1$ then $b_{\boldsymbol{\ell}''}\geq b_{\boldsymbol{\ell}}$ and thus using $b_{\boldsymbol{\ell}'}\leq 1$ we obtain
$$b_{\boldsymbol{\ell}''}\geq b_{\boldsymbol{\ell}}b_{\boldsymbol{\ell}'}.$$
The case $ {\boldsymbol{\ell}''_1}= {\boldsymbol{\ell}}'_1$ is   handled in the same way. This conclude the proof of \eqref{ouille1}.

Now we focus on  \eqref{ouille2}. We verify that
$$
c_{\boldsymbol{\ell}''}=c_{(\boldsymbol{\ell},j)}c_{(\boldsymbol{\ell}',j)}
\frac{\max(\langle \boldsymbol{\ell}_1\rangle,\langle j\rangle)^{5}\ \max(\langle \boldsymbol{\ell}'_1\rangle,\langle j\rangle)^{5}}{\langle \boldsymbol{\ell}''_1\rangle^{5}\langle j\rangle^4}
$$
which leads to  $c_{\boldsymbol{\ell}''}\geq c_{(\boldsymbol{\ell},j)}c_{(\boldsymbol{\ell}',j)}$.
Now it suffices to prove that 
$$a_{\boldsymbol{\ell}''}\geq a_{(\boldsymbol{\ell},j)}a_{(\boldsymbol{\ell}',j)} \text{ and }\quad b_{\boldsymbol{\ell}''}\geq b_{(\boldsymbol{\ell},j)}b_{(\boldsymbol{\ell}',j)}.$$
We begin with we prove  the second estimate:
\begin{itemize}
\item If $\boldsymbol{\ell}''_1\leq j$ then 
$$b_{(\boldsymbol{\ell},j)}b_{(\boldsymbol{\ell}',j)}=\frac{\langle \boldsymbol{\ell}_2\rangle}{\langle j\rangle}\frac{\langle \boldsymbol{\ell}'_2\rangle}{\langle j\rangle}\leq \frac{\langle \boldsymbol{\ell}_2\rangle}{\langle j\rangle}\leq \frac{\langle \boldsymbol{\ell}''_2\rangle}{\langle \boldsymbol{\ell}''_1\rangle}
=b_{\boldsymbol{\ell}''} .$$

\item If $\boldsymbol{\ell}''_2\leq j\leq \boldsymbol{\ell}''_1$ then either $\boldsymbol{\ell}''_1=\boldsymbol{\ell}_1$ and thus $b_{(\boldsymbol{\ell},j)}=\frac{\langle j\rangle}{\langle\boldsymbol{\ell}''_1\rangle}$, $b_{(\boldsymbol{\ell}',j)}=\frac{\langle\boldsymbol{\ell}'_1\rangle}{\langle j\rangle}\leq\frac{\langle\boldsymbol{\ell}''_2\rangle}{\langle j\rangle}$, or  $\boldsymbol{\ell}''_1=\boldsymbol{\ell}'_1$ and thus $b_{(\boldsymbol{\ell}',j)}=\frac{\langle j\rangle}{\langle\boldsymbol{\ell}''_1\rangle}$, $b_{(\boldsymbol{\ell},j)}=\frac{\langle\boldsymbol{\ell}_1\rangle}{\langle j\rangle}\leq\frac{\langle\boldsymbol{\ell}''_2\rangle}{\langle j\rangle}$, which in both cases leads to 
$$b_{(\boldsymbol{\ell},j)}b_{(\boldsymbol{\ell}',j)}\leq\frac{\langle\boldsymbol{\ell}''_2\rangle}{\langle j\rangle}\frac{\langle j\rangle}{\langle\boldsymbol{\ell}''_1\rangle}=b_{\boldsymbol{\ell}''}.$$

%\item If $\boldsymbol{\ell}''_3\leq j\leq \boldsymbol{\ell}''_2$ then, since  $\langle \boldsymbol{\ell}''_1\rangle=\max(\langle \boldsymbol{\ell}_1\rangle, \langle \boldsymbol{\ell}'_1\rangle)$ we have
%$$b_{\boldsymbol{\ell}''}\geq \frac{\langle j\rangle}{\langle \boldsymbol{\ell}''_1\rangle}\geq \min(\frac{\langle j\rangle}{\langle \boldsymbol{\ell}_1\rangle},\frac{\langle j\rangle}{\langle \boldsymbol{\ell}'_1\rangle})\geq b_{(\boldsymbol{\ell},j)}b_{(\boldsymbol{\ell}',j)}$$
%since, either $j\geq \boldsymbol{\ell}_2$ and $\frac{\langle j\rangle}{\langle \boldsymbol{\ell}_1\rangle}\geq b_{(\boldsymbol{\ell},j)}$ or $j\geq \boldsymbol{\ell}'_2$ and $\frac{\langle j\rangle}{\langle \boldsymbol{\ell}'_1\rangle}\geq b_{(\boldsymbol{\ell}',j)}$ .

\item If $j\leq\boldsymbol{\ell}''_2$ then either $\boldsymbol{\ell}_1=\boldsymbol{\ell}''_1$ and then
$b_{(\boldsymbol{\ell},j)}=\frac{\max(\langle\boldsymbol{\ell}_2\rangle,\langle j\rangle)}{\langle\boldsymbol{\ell}''_1\rangle}\leq  \frac{\langle\boldsymbol{\ell}''_2\rangle}{\langle\boldsymbol{\ell}''_1\rangle}=b_{\boldsymbol{\ell}''}$ or $\boldsymbol{\ell}'_1=\boldsymbol{\ell}''_1$ and then
$b_{(\boldsymbol{\ell}',j)}=\frac{\max(\langle\boldsymbol{\ell}'_2\rangle,\langle j\rangle)}{\langle\boldsymbol{\ell}''_1\rangle}\leq  \frac{\langle\boldsymbol{\ell}''_2\rangle}{\langle\boldsymbol{\ell}''_1\rangle}=b_{\boldsymbol{\ell}''}$. In both cases we get
$b_{\boldsymbol{\ell}''}\geq b_{(\boldsymbol{\ell},j)}b_{(\boldsymbol{\ell}',j)}.$

\end{itemize}
Finally it remains to prove  $a_{\boldsymbol{\ell}''}\geq a_{(\boldsymbol{\ell},j)}a_{(\boldsymbol{\ell}',j)}$:
$$a_{(\boldsymbol{\ell},j)}a_{(\boldsymbol{\ell}',j)}=\frac{\big(\prod_{i=1}^{q''} \varpi^\flat_{\boldsymbol{\ell}''_i}\big)(\varpi^\flat_j)^2}{ \max(\varpi^\flat_{\boldsymbol{\ell}_1},\varpi^\flat_j)^2\max(\varpi^\flat_{\boldsymbol{\ell}'_1},\varpi^\flat_j)^2}\leq \frac{\big(\prod_{i=1}^{q''} \varpi^\flat_{\boldsymbol{\ell}''_i}\big)(\varpi^\flat_j)^2}{(\varpi^\flat_{\boldsymbol{\ell}''_1})^2(\varpi^\flat_j)^2}=a_{\boldsymbol{\ell}''}.$$
\end{proof}

We define in  $\mathscr{H}^{\mathrm{cste}}_{\eta,r,\mathcal{S}}$ the projection on the total degree by
$$\Pi_p P:=  \sum_{\substack{n \in \mathbb{N},\ \boldsymbol{k} \in \mathbb{Z}^{\mathcal{S}},\ \boldsymbol{m} \in \mathbb{N}^{\mathcal{S}},\ q\in \mathbb{N}\\ \| \boldsymbol{k}\|_{\ell^1}+2\| \boldsymbol{m}\|_{\ell^1}+q+2n=p}}\ \ \sum_{ \boldsymbol{\ell} \in (\mathcal{S}^c)^q  }  \sum_{ \boldsymbol{\sigma} \in  \{-1,1\}^q  } P_{n,\boldsymbol{k},\boldsymbol{m}}^{\boldsymbol{\ell},\boldsymbol{\sigma}} \, v_{\boldsymbol{\ell}}^{\boldsymbol{\sigma}} \, y^{\boldsymbol{m}} \, z^{ \boldsymbol{k} } \, \mu^n.$$
About Poisson brackets we have
\renewcommand{\a}{{\alpha}}
\begin{lemma}[Poisson brackets]\label{lem:poisson} Let $\mathcal{S} \subset \mathbb{Z}$ be a finite set, $r\in (0,\frac14)$, $\mathcal{O}\subset ((4r)^2,1)^\mathcal{S}$,  $\eta > \eta_0$, $\beta>0$ $\xi \in \mathcal{O}$, $P,Q\in \mathscr{H}_{\eta,r,\mathcal{O},\mathcal{S}}$ . There exists a unique $R\in \mathscr{H}_{\eta-\a,r,\mathcal{O},\mathcal{S}}$ for any $0<\alpha<\eta-\eta_0$  such that
$$\{P(\xi;\cdot),Q(\xi;\cdot)\}=R(\xi;\cdot)\quad \text{on } \mathcal A_\xi(3r)$$
and 
\begin{equation}\label{estim:poisson1} \| R\|_{\mathscr{H}_{\eta-\alpha,r,\mathcal{O},\mathcal{S}}^{\mathrm{tot}}} \lesssim_{\mathcal{S}}\ \ \frac1{r^2\a^4}\| P\|_{\mathscr{H}_{\eta,r,\mathcal{O},\mathcal{S}}^{\mathrm{tot}}}\| Q\|_{\mathscr{H}_{\eta,r,\mathcal{O},\mathcal{S}}^{\mathrm{tot}}}.\end{equation}
If furthermore $Q\in \mathscr{H}_{\eta+\beta,r,\mathcal{O},\mathcal{S}}$ for some $\beta>0$ then
for any $p\in\N$, we have
\begin{equation}\label{estim:poisson2} \| \Pi_pR\|_{\mathscr{H}_{\eta-\a,r,\mathcal{O},\mathcal{S}}^{\mathrm{tot}}} \lesssim_{\mathcal{S}}\ \ \frac{(p+1)e^{-\a p}}{r^2\beta^4}\| P\|_{\mathscr{H}_{\eta,r,\mathcal{O},\mathcal{S}}^{\mathrm{tot}}}\| Q\|_{\mathscr{H}_{\eta+\beta,r,\mathcal{O},\mathcal{S}}^{\mathrm{tot}}}.\end{equation}

\end{lemma}
Since $R$ is unique from now on we set $\{P,Q\}:=R$.
%Note that, as $pe^{-\a p}\leq \frac1\a$,   \eqref{estim:poisson2} implies
%\begin{equation}\label{estim:poisson2} \| R\|_{\mathscr{H}_{\eta-\alpha,r,\mathcal{O},\mathcal{S}}^{\sharp}} \lesssim_{\mathcal{S}}\ \ \frac1{r^2\a\beta^4}\| P\|_{\mathscr{H}_{\eta,r,\mathcal{O},\mathcal{S}}^{\sharp}}\| Q\|_{\mathscr{H}_{\eta+\beta,r,\mathcal{O},\mathcal{S}}^{\sharp}}.\end{equation}
\begin{proof} 
%We begin with the proof of \eqref{estim:poisson}.
 To  keep the notations simple, we identify, for  fixed $\xi$,  the formal series $P$ and the function $P(\xi;u)$.
%and we will highlight the change required to obtain  \eqref{estim:poisson}.
By definition, using that $P$ and $Q$ commute with $\mu$ (conservation of the mass),  we have
$$
\{P,Q\}=\underbrace{\sum_{j\in \S}\partial_{\theta_j} P\partial_{y_j} Q-\partial_{y_j} Q\partial_{\theta_j} P}_K\ +\ 
2\ic\underbrace{\sum_{j\in \S^c} \partial_{\overline{v_j}} P\partial_{v_j} Q-\partial_{\overline{v_j}} Q\partial_{v_j} P}_L.
$$
First we consider the finite sum. We have, taking into account the symmetry condition (here $q''$ denotes the cardinality of $\boldsymbol{\ell}''$) 
\begin{equation}\label{KK}\sharp(\mathfrak{S}_{q''})K_{n'',\boldsymbol{k}'',\boldsymbol{m}''}^{\boldsymbol{\ell}'',\boldsymbol{\sigma}''}=
\sum_{\varphi\in\mathfrak{S}_{q''}}\tilde K_{n'',\boldsymbol{k}'',\boldsymbol{m}''}^{\varphi\boldsymbol{\ell}'',\varphi\boldsymbol{\sigma}''}\end{equation}
where

\begin{align*}\label{estim:poi1}
\tilde K_{n'',\boldsymbol{k}'',\boldsymbol{m}''}^{\boldsymbol{\ell}'',\boldsymbol{\sigma}''}=\ic\sum_{j\in \S}\sum_{\substack{\boldsymbol{\ell}''=(\boldsymbol{\ell},\boldsymbol{\ell}'),\ \boldsymbol{\sigma}''=(\boldsymbol{\sigma},\boldsymbol{\sigma}'),\\\boldsymbol{k}''=\boldsymbol{k}+\boldsymbol{k}',\ \boldsymbol{m}''=\boldsymbol{m}+\boldsymbol{m}',\ n''=n+n'}}  &\boldsymbol{k}_j(\boldsymbol{m}'_j+1)P_{n,\boldsymbol{k},\boldsymbol{m}}^{\boldsymbol{\ell},\boldsymbol{\sigma}}Q_{n',\boldsymbol{k}',(\boldsymbol{m}'+e_j)}^{\boldsymbol{\ell}',\boldsymbol{\sigma}'}\\
&-\boldsymbol{k}'_j(1+\boldsymbol{m}_j)P_{n,\boldsymbol{k},(\boldsymbol{m}+e_j)}^{\boldsymbol{\ell},\boldsymbol{\sigma}}Q_{n',\boldsymbol{k}',\boldsymbol{m}'}^{\boldsymbol{\ell}',\boldsymbol{\sigma}'}
\end{align*}
 where    $(e_j)_{j\in S}$ denotes the canonical basis of $\R^S$. The two terms of the right hand side are essentially the same, it suffices to switch $P$ and $Q$.  So we focus on  the first term and we note that the sum with respect to $j$ is finite and thus not important. Further, remembering that we want to estimate $ \Pi_{p''}K$ for $p''\geq0$, we finally have to estimate
  $$I_j=\sum_{\substack{\boldsymbol{\ell}''=(\boldsymbol{\ell},\boldsymbol{\ell}'),\ \boldsymbol{\sigma}''=(\boldsymbol{\sigma},\boldsymbol{\sigma}'),\\\boldsymbol{k}''=\boldsymbol{k}+\boldsymbol{k}',\ \boldsymbol{m}''=\boldsymbol{m}+\boldsymbol{m}',\ n''=n+n'\\\|\boldsymbol k''\|_{\ell^1}+2\|\boldsymbol{m}''\|_{\ell^1}+ q''+ n''= p''}}  |\boldsymbol{k}_j|(\boldsymbol{m}'_j+1)\big|P_{n,\boldsymbol{k},\boldsymbol{m}}^{\boldsymbol{\ell},\boldsymbol{\sigma}}\big|\ \big|Q_{n',\boldsymbol{k}',(\boldsymbol{m}'+e_j)}^{\boldsymbol{\ell}',\boldsymbol{\sigma}'}\big|.$$

Let $p=\|\boldsymbol k\|_{\ell^1}+2\|\boldsymbol{m}\|_{\ell^1}+ q+ n $ , $p'=\|\boldsymbol k'\|_{\ell^1}+2\|\boldsymbol{m}'\|_{\ell^1}+ q'+ n'+2$ (here $q$, $q'$ denote as usual the cardinality of $\boldsymbol{\ell}$, $\boldsymbol{\ell}'$ respectively). Since $n''=n+n'$, $\|\boldsymbol{m}''\|_{\ell^1}=\|\boldsymbol{m}'\|_{\ell^1}+\|\boldsymbol{m}\|_{\ell^1}$, $q''=q+q'$ and $\|\boldsymbol k''\|_{\ell^1}-\|\boldsymbol k'\|_{\ell^1}\leq\|\boldsymbol k\|_{\ell^1}\leq \|\boldsymbol k''\|_{\ell^1}+\|\boldsymbol k'\|_{\ell^1}$,
 we get $p''\leq p+p'-2$ and also $p\leq p'+p''$. Thus, we have 
$$I_j\leq  \sum_{\substack{\|\boldsymbol k\|_{\ell^1}+2\|\boldsymbol{m}\|_{\ell^1}+ q+ n=p\\ \|\boldsymbol k'\|_{\ell^1}+2\|\boldsymbol{m}'\|_{\ell^1}+ q'+ n'+2=p'\\p''-p'+2\leq p\leq p''+ p'}} pp' \big|P_{n,\boldsymbol{k},\boldsymbol{m}}^{\boldsymbol{\ell},\boldsymbol{\sigma}}\big|\ \big|Q_{n',\boldsymbol{k}',\boldsymbol{m}'}^{\boldsymbol{\ell}',\boldsymbol{\sigma}'}\big|.$$
 Therefore, using  Lemma \ref{ouille}, we deduce 
$$\| \Pi_{p''}\tilde K\|_{\mathscr{H}^{\mathrm{sup}}_{\eta-\a,r,\mathcal{O},\mathcal{S}}}\lesssim_\S r^{-2}\sum_{\substack{p,p'\geq1\\p''-p'+2\leq p\leq p''+ p'}}pp' 
\| \Pi_{p}P\|_{\mathscr{H}^{\mathrm{sup}}_{\eta-\a,r,\mathcal{O},\mathcal{S}}}\| \Pi_{p'}Q\|_{\mathscr{H}^{\mathrm{sup}}_{\eta-\a,r,\mathcal{O},\mathcal{S}}}$$
and since our previous estimates are uniform with respect to permutation on $\boldsymbol{\ell}''$ and $\boldsymbol{\sigma}''$ we conclude
\begin{equation}\label{K}
\| \Pi_{p''} K\|_{\mathscr{H}^{\mathrm{sup}}_{\eta-\a,r,\mathcal{O},\mathcal{S}}}\lesssim_\S r^{-2}\sum_{\substack{p,p'\geq1\\p''-p'+2\leq p\leq p''+ p'}}pp' 
\| \Pi_{p}P\|_{\mathscr{H}^{\mathrm{sup}}_{\eta-\a,r,\mathcal{O},\mathcal{S}}}\| \Pi_{p'}Q\|_{\mathscr{H}^{\mathrm{sup}}_{\eta-\a,r,\mathcal{O},\mathcal{S}}}.
\end{equation}
Now we focus on the infinite sum $L$.  Following the same strategy as above (i.e. a decomposition of $L$ as in \eqref{KK}), we have to estimate 
\begin{equation*} 
\sum_{\substack{\boldsymbol{\ell}''=(\boldsymbol{\ell},\boldsymbol{\ell}'),\ \boldsymbol{\sigma}''=(\boldsymbol{\sigma},\boldsymbol{\sigma}'),\\
\boldsymbol{k}''=\boldsymbol{k}+\boldsymbol{k}',\, \boldsymbol{m}''=\boldsymbol{m}+\boldsymbol{m}',\\ n''=n+n'}}\hspace{-1cm} (q+1)(q'+1) \mathbbm{1}_{j^*\in \mathcal{S}^c}\left(\big|P_{n,\boldsymbol{k},\boldsymbol{m}}^{(\boldsymbol{\ell},j^*),(\boldsymbol{\sigma},+1)} \, Q_{n',\boldsymbol{k}',\boldsymbol{m}'}^{(\boldsymbol{\ell}',j^*),(\boldsymbol{\sigma}',-1)}\big|+\big|P_{n,\boldsymbol{k},\boldsymbol{m}}^{(\boldsymbol{\ell},j^*),(\boldsymbol{\sigma},-1)} \, Q_{n',\boldsymbol{k}',\boldsymbol{m}'}^{(\boldsymbol{\ell}',j^*),(\boldsymbol{\sigma}',+1)}\big|\right)
\end{equation*}
where $q$ and $q'$ denote the cardinality of $\boldsymbol\ell$ and $\boldsymbol\ell'$ respectively and, as aconsequence of the zero momentum condition,
$$ - j^*=\sum_{j\in \mathcal{S}} j\boldsymbol{k}_j + \sum_{i=1}^q \boldsymbol\sigma_i \boldsymbol\ell_i.$$
 So we obtain, using again Lemma \ref{ouille}, the same estimate for $L$ than  estimate \eqref{K} for $K$ and all together we get
\begin{equation}\label{R}
\| \Pi_{p''}R\|_{\mathscr{H}^{\mathrm{sup}}_{\eta-\a,r,\mathcal{O},\mathcal{S}}}\lesssim_\S r^{-2}\sum_{\substack{p,p'\geq1\\p''-p'+2\leq p\leq p''+ p'}}pp' 
\| \Pi_{p}P\|_{\mathscr{H}^{\mathrm{sup}}_{\eta-\a,r,\mathcal{O},\mathcal{S}}}\| \Pi_{p'}Q\|_{\mathscr{H}^{\mathrm{sup}}_{\eta-\a,r,\mathcal{O},\mathcal{S}}}.
\end{equation}
Now we prove \eqref{estim:poisson1}. We use $\| \Pi_{p}\cdot \|_{\mathscr{H}^{\mathrm{sup}}_{\eta-\a,r,\mathcal{O},\mathcal{S}}}\leq e^{-\alpha p}\| \cdot\|_{\mathscr{H}^{\mathrm{sup}}_{\eta,r,\mathcal{O},\mathcal{S}}}$ to get from \eqref{R}
$$\| \Pi_{p''}R\|_{\mathscr{H}^{\mathrm{sup}}_{\eta-\a,r,\mathcal{O},\mathcal{S}}}\lesssim_\S r^{-2} \| P\|_{\mathscr{H}^{\mathrm{sup}}_{\eta,r,\mathcal{O},\mathcal{S}}}\| Q\|_{\mathscr{H}^{\mathrm{sup}}_{\eta,r,\mathcal{O},\mathcal{S}}}\sum_{p,p'\geq 1}pp' e^{-\alpha p} e^{-\a p'}  .$$
Then we use that $\sum_{p\geq1}pe^{-\a p}\leq 2\int_0^{+\infty}te^{-\a t}\dd t =\frac2{\a^2}$,  to conclude that, uniformly in $p''\in\N$,
$$\| \Pi_{p''}R\|_{\mathscr{H}^{\mathrm{sup}}_{\eta-\a,r,\mathcal{O},\mathcal{S}}}\lesssim_\S \frac1{r^2\a^4}  \| P\|_{\mathscr{H}^{\mathrm{sup}}_{\eta,r,\mathcal{O},\mathcal{S}}}\| Q\|_{\mathscr{H}^{\mathrm{sup}}_{\eta,r,\mathcal{O},\mathcal{S}}}.$$
 We also get for all $p\geq0$
$$\| \Pi_pR\|_{\mathscr{H}^{\mathrm{lip}}_{\eta-\a,r,\mathcal{O},\mathcal{S}}} \lesssim_\S\  \frac{1}{r^2\a^4}\left(\| P\|_{\mathscr{H}^{\mathrm{lip}}_{\eta,r,\mathcal{O},\mathcal{S}}}\| Q\|_{\mathscr{H}^{\mathrm{sup}}_{\eta,r,\mathcal{O},\mathcal{S}}}
+\| P\|_{\mathscr{H}^{\mathrm{sup}}_{\eta,r,\mathcal{O},\mathcal{S}}}\| Q\|_{\mathscr{H}^{\mathrm{lip}}_{\eta,r,\mathcal{O},\mathcal{S}}}
\right)$$
and thus \eqref{estim:poisson1}  follows. 

\medskip

Now we prove \eqref{estim:poisson2}.
 We use  $\| \Pi_{p'}\cdot \|_{\mathscr{H}^{\mathrm{sup}}_{\eta-\a,r,\mathcal{O},\mathcal{S}}}\leq e^{-(\a+\beta) p'}\| \cdot\|_{\mathscr{H}^{\mathrm{sup}}_{\eta+\beta,r,\mathcal{O},\mathcal{S}}}$ to get from \eqref{R}
$$
\| \Pi_{p''}R\|_{\mathscr{H}^{\mathrm{sup}}_{\eta-\a,r,\mathcal{O},\mathcal{S}}}\lesssim_\S r^{-2} \| P\|_{\mathscr{H}^{\mathrm{sup}}_{\eta,r,\mathcal{O},\mathcal{S}}}\| Q\|_{\mathscr{H}^{\mathrm{sup}}_{\eta+\beta,r,\mathcal{O},\mathcal{S}}}\sum_{\substack{p,p'\geq1\\p''-p'+2\leq p\leq  p''+p'}}pp' e^{-\alpha p} e^{-(\beta+\a) p'}  .
$$
 We have $e^{-\alpha p}\leq e^{-\alpha p''}e^{\alpha p'}$ and thus
$$\sum_{\substack{p,p'\geq1\\p''-p'+2\leq p\leq  p''+p'}}pp' e^{-\alpha p} e^{-(\beta+\a) p'}  \leq 
\sum_{p'\geq1}2p'(p''+p')p' e^{-\alpha p''} e^{-\beta p'}\leq 4(p''+1)e^{-\alpha p''}\sum_{p'\geq1}p'^3  e^{-\beta p'}.
$$ 
 Finally we use $\sum_{p'\geq1}p'^3  e^{-\beta p'}\leq\frac{12}{\beta^4}$ to conclude
 \begin{equation*}\label{estim:K}
 \| \Pi_{p}R\|_{\mathscr{H}^{\mathrm{sup}}_{\eta-\a,r,\mathcal{O},\mathcal{S}}}\lesssim_\S 
 \frac{(p+1)e^{-\alpha p}}{r^2\beta^4}
  \| P\|_{\mathscr{H}^{\mathrm{sup}}_{\eta,r,\mathcal{O},\mathcal{S}}}\| Q\|_{\mathscr{H}^{\mathrm{sup}}_{\eta+\beta,r,\mathcal{O},\mathcal{S}}}.
 \end{equation*}
 This prove \eqref{estim:poisson2} for the $\text{sup}$ norm. We also get
$$\| \Pi_pR\|_{\mathscr{H}^{\mathrm{lip}}_{\eta-\a,r,\mathcal{O},\mathcal{S}}} \lesssim_\S\  \frac{(p+1) e^{-\alpha p}}{r^2\beta^4}\left(\| P\|_{\mathscr{H}^{\mathrm{lip}}_{\eta,r,\mathcal{O},\mathcal{S}}}\| Q\|_{\mathscr{H}^{\mathrm{sup}}_{\eta+\beta,r,\mathcal{O},\mathcal{S}}}
+\| P\|_{\mathscr{H}^{\mathrm{sup}}_{\eta,r,\mathcal{O},\mathcal{S}}}\| Q\|_{\mathscr{H}^{\mathrm{lip}}_{\eta+\beta,r,\mathcal{O},\mathcal{S}}}
\right)$$
from which \eqref{estim:poisson2}  follows.

\end{proof}

We end this section with a result on the action of the Hamiltonian flow $\Phi_\chi^t$ by composition.
\begin{lemma}\label{lem:compo} 
Let $\mathcal{S} \subset \mathbb{Z}$ be a finite set, $r\in (0,\frac14)$, $\mathcal{O}\subset ((4r)^2,1)^\mathcal{S}$, $\eta_{\mathrm{\max}}>\eta > \eta_0$, $\alpha\in(0,\eta-\eta_0)$ and $\chi\in \mathscr{H}_{\eta+\a,r,\mathcal{O},\mathcal{S}}$. Assume that $\chi$ satisfies \eqref{chi} (or \eqref{chi2} if $\chi = \Pi_{n=0,\boldsymbol{m}=0}\chi$) and
\begin{equation}\label{chi3}
\| \chi\|_{\mathscr{H}_{\eta+\alpha,r,\mathcal{O},\mathcal{S}}^{\mathrm{tot}}}\lesssim_\S \alpha^5 r^2.
\end{equation}
Assume also that $r$ satisfies \eqref{r}.
Then for any $H\in \mathscr{H}_{\eta,r,\mathcal{O},\mathcal{S}}$ and any $t\in[0,1]$, there exits $H_t \in \mathscr{H}_{\eta-\alpha,r,\mathcal{O},\mathcal{S}}$ such that, for all $\xi\in\mathcal O$, $H(\xi;\cdot)\circ\Phi^t_\chi(\xi;\cdot)=H_t(\xi;\cdot)$ on $\mathcal A_\xi(2r)$
 and, identifying the formal series and the functions\footnote{i.e. identifying $H\circ\Phi^t_\chi$ and $H_t$}, we have 
 
\begin{equation}\label{estim:flow} \| H\circ\Phi^t_\chi\|_{\mathscr{H}_{\eta-\alpha,r,\mathcal{O},\mathcal{S}}^{\mathrm{tot}}} \leq 2\ \| H\|_{\mathscr{H}_{\eta,r,\mathcal{O},\mathcal{S}}^{\mathrm{tot}}},\end{equation}
\begin{equation}\label{estim:flow1} \| H\circ\Phi^t_\chi-H\|_{\mathscr{H}_{\eta-\alpha,r,\mathcal{O},\mathcal{S}}^{\mathrm{tot}}} \lesssim_\S\ \frac1{r^2\a^5 } \| \chi\|_{\mathscr{H}_{\eta+\a,r,\mathcal{O},\mathcal{S}}^{\mathrm{tot}}}\| H\|_{\mathscr{H}_{\eta,r,\mathcal{O},\mathcal{S}}^{\mathrm{tot}}}\end{equation}
and
\begin{equation}\label{estim:flow2} \| H\circ\Phi^1_\chi-H-\{H,\chi\}\|_{\mathscr{H}_{\eta-\alpha,r,\mathcal{O},\mathcal{S}}^{\mathrm{tot}}} \lesssim_\S\ \frac1{r^4\a^{10}}  \| \chi\|^2_{\mathscr{H}_{\eta+\a,r,\mathcal{O},\mathcal{S}}^{\mathrm{tot}}}\| H\|_{\mathscr{H}_{\eta,r,\mathcal{O},\mathcal{S}}^{\mathrm{tot}}}\end{equation}
where,  with a slight abuse of notation, we have written $H\circ\Phi^t_\chi$ instead of $H(\xi;\cdot)\circ\Phi^t_\chi(\xi;\cdot)$.
\end{lemma}

\begin{proof} We want to apply the Cauchy--Kowaleskaya lemma proven in the appendix (see Lemma  \ref{lem:CK}). So we set  $E^{(p+1)}=\Pi_p \mathscr{H}_{\eta_0,r,\mathcal{O},\mathcal{S}}$ and $\| \cdot\|_{E^{(p+1)}}=\| \cdot\|_{\mathscr{H}_{\eta_0,r,\mathcal{O},\mathcal{S}}^{\mathrm{tot}}}$.
We note that, with the notation of Lemma \ref{lem:CK}, we have for any $\beta\geq0$
$$
E_\beta=\mathscr{H}_{\eta_0+\beta,r,\mathcal{O},\mathcal{S}}^{\mathrm{tot}} \quad \mathrm{ and } \quad \| \cdot \|_\beta=\| \cdot\|_{\mathscr{H}_{\eta_0+\beta,r,\mathcal{O},\mathcal{S}}^{\mathrm{tot}}} .
$$ 
By Lemma \ref{lem:poisson} we have for all $0<\beta \leq \eta-\eta_0$
$$
  \| \Pi_p\{H,\chi\}\|_{E^{(p+1)}} \lesssim_\S \ \frac{(p+1)e^{-\beta (p+1)}}{r^2\a^4}\| H\|_{\mathscr{H}_{\beta,r,\mathcal{O},\mathcal{S}}^{\mathrm{tot}}}\| \chi\|_{\mathscr{H}_{\eta+\a,r,\mathcal{O},\mathcal{S}}^{\mathrm{tot}}}
  $$
and thus the linear map $LH:=\{H,\chi\}$ satisfies \eqref{CL}  with $C_L:=c(S)\|\chi\|_{\mathscr{H}_{\eta+\alpha,r,\mathcal{O},\mathcal{S}}^{\mathrm{tot}}} (r^2\a^4)^{-1}$. Note that, as a consequence of  \eqref{chi3}, we have $2C_L\leq\a$.\\
 Now let $H\in \mathscr{H}_{\eta,r,\mathcal{O},\mathcal{S}}$.  Applying Lemma \ref{lem:CK}, there exists  $t\mapsto H_t\in C^1([0,1];\mathscr{H}_{\eta-\a,r,\mathcal{O},\mathcal{S}})$  solution of the transport equation
$$
\left\{\begin{array}{ll} \partial_t H_t=\{H_t,\chi\},\quad t\in[0,1]\\
H(0)=H. \end{array}\right..
$$
Consider $g(t)=H_t\circ \Phi_\chi^{-t}(u)$ for some $u\in   \mathcal A_\xi(2r)$. We have
\begin{align*}
\frac{\dd}{\dd t} g(t)&=\partial_t H_t\circ \Phi_\chi^{-t}(u)+\dd H_t(\Phi_\chi^{-t}(u))(\partial_t\Phi_\chi^{-t}(u))\\
&=\{ H_t,\chi\}\circ \Phi_\chi^{-t}(u)- (\ic\nabla H_t( \Phi_\chi^{-t}(u)), \nabla \chi(\Phi_\chi^{-t}(u)))_{L^2} =0
\end{align*}
and thus $H_t\circ \Phi_\chi^{-t}(u) = g(t)=g(0)=H(u)$. Therefore $H_t=H\circ\Phi_\chi^t$ and we deduce that $H\circ\Phi_\chi^t$ belongs to $\mathscr{H}_{\eta-\alpha,r,\mathcal{O},\mathcal{S}}^{\mathrm{tot}}$ for all $t\in[0,1]$ and satisfies \eqref{estim:flow}.
To prove \eqref{estim:flow1} we write 
$$
H\circ \Phi^t_\chi-H=\int_0^t \frac{\dd}{\dd \tau}(H\circ \Phi_\chi^\tau) \dd \tau
$$
and we use $\frac{\dd}{\dd \tau}(H\circ \Phi_\chi^\tau)=\{ H,\chi\}\circ \Phi_\chi^\tau$ to conclude using an adapted version of \eqref{estim:flow} with $\eta$ replaced by $\eta-\frac\a2$ and $\a$ replaced by $\frac\a2$ (and thus adapting \eqref{r})
$$ 
\| H\circ \Phi^t_\chi-H\|_{\mathscr{H}_{\eta-\a,r,\mathcal{O},\mathcal{S}}^{\mathrm{tot}}}\leq 2\ \| \{ H,\chi\}\|_{\mathscr{H}_{\eta-\frac\a2,r,\mathcal{O},\mathcal{S}}^{\mathrm{tot}}}.
$$
Then we use \eqref{estim:poisson2} and the uniform estimate $(p+1)e^{-\frac\a2 p}\leq \frac8\a$ to get
$$
\| \{ H,\chi\}\|_{\mathscr{H}_{\eta-\frac\a2,r,\mathcal{O},\mathcal{S}}^{\mathrm{tot}}}\lesssim_\S \frac1{r^2\a^5 }\| \chi\|_{\mathscr{H}_{\eta+\a,r,\mathcal{O},\mathcal{S}}^{\mathrm{tot}}}\| H\|_{\mathscr{H}_{\eta,r,\mathcal{O},\mathcal{S}}^{\mathrm{tot}}}
$$
which completes the proof of \eqref{estim:flow1}. The proof of \eqref{estim:flow2} is obtained similarly, writing
 $$
 H\circ \Phi^1_\chi-H-\{ H,\chi\}=\frac12\int_0^1 \frac{\dd^2}{\dd t^2}(H\circ \Phi_\chi^t) \dd t
 $$
and using $\frac{\dd^2}{\dd t^2}(H\circ \Phi_\chi^t)=\{\{ H,\chi\},\chi\}\circ \Phi_\chi^t$.

%$[0,1]\ni t\mapsto G(t)=H^{(0)}\circ\Phi^t_\chi\in C^\infty(\mathcal A_\xi(2r),\R) $ which, in view of Lemma \ref{lem:ouf_cest_des_fonctions} and Lemma \ref{lem:flow} is a smooth function. One has $G(0)=H^{(0)}$ and 
%$$\partial_t G(t)= \{H^{(0)},\chi\}\circ \Phi^t_\chi=\{ G(t),\chi\}$$
%and thus $G(t)=H(t)$ for $t\in[0,1]$.
%Therefore $H(1)=H^{(0)}\circ\Phi_\chi$ and we deduce that $H^{(0)}\circ\Phi_\chi$ belongs to $\mathscr{H}_{\eta-\alpha,r,\mathcal{O},\mathcal{S}}$ and satisfies \eqref{estim:flow}.
\end{proof}

\subsection{Spaces of integrable Hamiltonians}

The integrable part of Hamiltonians plays an important role in our context, we cannot eliminate them by normal forms and they nature may change when we open a new site. Thus they must be carefully followed. With this in mind, we define different spaces of integrable Hamiltonians and show embedding  properties.

\subsubsection{Spaces}

\begin{definition}[Integrable Hamiltonians] Let $\mathcal{S} \subset \mathbb{Z}$ be a finite set, $r\in (0,\frac14)$, $\mathcal{O}\subset ((4r)^2,1)^\mathcal{S}$, $\eta \geq 0$.
A Hamiltonian $P \in \mathscr{H}_{\eta,r,\mathcal{O},\mathcal{S}}$ is integrable if its non zero coefficients $P_{n,\boldsymbol{k},\boldsymbol{m}}^{\boldsymbol{\ell},\boldsymbol{\sigma}}\neq 0$ satisfy 
\begin{equation}
\label{eq:integ}
\boldsymbol{k}=0 \quad \mathrm{and} \quad \exists \varphi \in \mathfrak{S}_{\# \boldsymbol{\ell}}, \ \varphi \boldsymbol{\ell}=\boldsymbol{\ell} \ \mathrm{and} \ \varphi\boldsymbol{\sigma}=-\boldsymbol{\sigma}
\end{equation}
 We denote by $\Pi_{\mathrm{int}}$ the associated projection defined by restriction of the coefficients\footnote{i.e. for $K\in \mathscr{H}_{\eta,r,\mathcal{O},\mathcal{S}}$, $I=\Pi_{\mathrm{int}} K$ is the Hamiltonian whose coefficients are defined by $I_{n,\boldsymbol{k},\boldsymbol{m}}^{\boldsymbol{\ell},\boldsymbol{\sigma}}= \mathbbm{1}_{\eqref{eq:integ}} K_{n,\boldsymbol{k},\boldsymbol{m}}^{\boldsymbol{\ell},\boldsymbol{\sigma}}$}.
\end{definition}

\begin{definition}[Quartic integrable Hamiltonians $\mathscr{Q}^{\mathrm{cste}}_{\mathcal{S}},\mathscr{Q}_{\mathcal{O},\mathcal{S}}$] \label{def:quar} Let $\mathcal{S}\subset \mathbb{Z}$ be a finite set. We denote by $\mathscr{Q}^{\mathrm{cste}}_{\mathcal{S}}$, the set of the formal series $Q$, of the form
$$
Q(\mu,y,u) := \sum_{k,\ell\in \mathbb{Z}} Q_{k,\ell} Y_k Y_\ell + \sum_{k\in \mathbb{Z}} Q_k Y_k \mu  + Q_{\emptyset} \mu^2
$$
where $Y_k := \mathbbm{1}_{k\in \mathcal{S}} y_k + \mathbbm{1}_{k\in \mathcal{S}^c} |u_k|^2  $ and whose coefficients $Q_{k,\ell},Q_{k},Q_{\emptyset} \in \mathbb{R}$ satisfy the symmetry condition $Q_{k,\ell} = Q_{\ell,k}$ and the bound
$$
\| Q \|_{\mathscr{Q}^{\mathrm{cste}}} := \sup_{k,\ell\in \mathbb{Z}}  |Q_{k,\ell}|\Lambda_{k,\ell}^{-1} \vee |Q_k| \langle k \rangle^{\delta} \vee Q_{\emptyset} <\infty
$$
where
\begin{equation}
\label{eq:defLambda}
\Lambda_{k,\ell} :=  1 \wedge \frac{\langle k \rangle^{4\delta}}{\langle \ell\rangle^{\delta}} \wedge \frac{\langle \ell \rangle^{4\delta}}{  \langle k\rangle^{\delta}}.
\end{equation}
Moreover, if $\mathcal{O} \subset \mathbb{R}_+^\mathcal{S}$, we set
  $$
  \mathscr{Q}_{\mathcal{O},\mathcal{S}} := \mathrm{Lip}(\mathcal{O};\mathscr{Q}^{\mathrm{cste}}_{\mathcal{S}}).
  $$
\end{definition}

\begin{definition}[Linear Hamiltonians in the actions $\mathscr{L}^{\mathrm{cste}}_{\mathcal{S}},\mathscr{L}_{\mathcal{O},\mathcal{S}}$] \label{def:lin} Let $\mathcal{S}\subset \mathbb{Z}$ be a finite set. We denote by $\mathscr{L}^{\mathrm{cste}}_{\mathcal{S}}$, the set of the formal series $L$, of the form
$$
L(y) := \sum_{k\in \mathbb{Z}} L_k Y_k
$$
where $Y_k := \mathbbm{1}_{k\in \mathcal{S}} y_k + \mathbbm{1}_{k\in \mathcal{S}^c} |u_k|^2  $ and whose coefficients $L_k \in \mathbb{R}$ satisfy the bound
$$
\| L \|_{\mathscr{L}^{\mathrm{cste}}} := \sup_{k\in \mathbb{Z}} |L_k| \langle k \rangle^{\delta} <\infty.
$$
Moreover, if $\mathcal{O} \subset \mathbb{R}_+^\mathcal{S}$, we set
  $$
  \mathscr{L}_{\mathcal{O},\mathcal{S}} := \mathrm{Lip}(\mathcal{O};\mathscr{L}^{\mathrm{cste}}_{\mathcal{S}}).
  $$
\end{definition}

\begin{definition}[Almost constant Hamiltonians $\mathscr{A}^{\mathrm{cste}},\mathscr{A}_{\mathcal{O},\mathcal{S}}$] \label{def:cst} We denote by $\mathscr{A}^{\mathrm{cste}}$, the set of the affine functions $A$ of the form
$$
A(\mu) := A_0 + A_1 \mu
$$
with $A_0,A_1\in \mathbb{R}$,
that we equip with the norm
$$
\| A \|_{\mathscr{A}^{\mathrm{cste}}} := 3^{-1} |A_0| \vee |A_1|.
$$
Moreover, if $\mathcal{O} \subset \mathbb{R}_+^\mathcal{S}$, where $\mathcal{S} \subset \mathbb{Z}$ is a finite set, we set
  $$
  \mathscr{A}_{\mathcal{O},\mathcal{S}} := \mathrm{Lip}(\mathcal{O};\mathscr{A}^{\mathrm{cste}}).
  $$
\end{definition}
\begin{remark}
The coefficients $3$  is convenient but of minor importance in this paper. It is only used to have a factor close to $1$ in the estimate \eqref{eq:bound_open2}. We introduce the factor $3$ in Definition \ref{def:N_norm} below for the same reason.
\end{remark}

\begin{definition}[Hamiltonians in normal form $\mathscr{N}_{\mathcal{O},\mathcal{S}} $] Let $\mathcal{S} \subset \mathbb{Z}$ be a finite set and $\mathcal{O}\subset \mathbb{R}_+^\mathcal{S}$. 
An Hamiltonian is said in normal form if it belongs to the space $\mathscr{N}_{\mathcal{O},\mathcal{S}} $ defined by
\begin{equation}\label{eq:decomp}
\mathscr{N}_{\mathcal{O},\mathcal{S}} := \mathscr{A}_{\mathcal{O},\mathcal{S}}\oplus\mathscr{L}_{\mathcal{O},\mathcal{S}}\oplus \mathscr{Q}_{\mathcal{O},\mathcal{S}}.
\end{equation}
\end{definition}

\subsubsection{Lipschitz norms} We are going to define some well suited norms on the Lipschitz spaces defined above. 

\begin{definition}[Norms for $ \mathscr{A}_{\mathcal{O},\mathcal{S}}$] Let $\mathcal{S} \subset \mathbb{Z}$ be a finite set, $\mathcal{O}\subset \mathbb{R}_+^\mathcal{S}$ and $A\in   \mathscr{A}_{\mathcal{O},\mathcal{S}} $. We set
 $$
 \| A\|_{\mathscr{A}_{\mathcal{O},\mathcal{S}}^{\mathrm{sup}}} := \sup_{\xi \in \mathcal{O}} \| A(\xi) \|_{\mathscr{A}_{\mathcal{S}}^{\mathrm{cste}}},
\quad 
 \| A\|_{\mathscr{A}_{\mathcal{O},\mathcal{S}}^{\mathrm{lip}}} := \sup_{\substack{\xi,\zeta \in \mathcal{O}\\ \xi \neq \zeta}} \frac{\| A(\xi) - A(\zeta)  \|_{\mathscr{A}_{\mathcal{S}}^{\mathrm{cste}}}}{\| \xi -\zeta \|_{\ell^1}},
 $$
 $$
\| A\|_{\mathscr{A}_{\mathcal{O},\mathcal{S}}^{\mathrm{tot}}} := \| A\|_{\mathscr{A}_{\mathcal{O},\mathcal{S}}^{\mathrm{sup}}}  +\| A\|_{\mathscr{A}_{\mathcal{O},\mathcal{S}}^{\mathrm{lip}}}. 
$$
\end{definition}

\begin{definition}[Some norms for $ \mathscr{L}_{\mathcal{O},\mathcal{S}}$] Let $\mathcal{S} \subset \mathbb{Z}$ be a finite set, $\mathcal{O}\subset \mathbb{R}_+^\mathcal{S}$ and $L\in   \mathscr{L}_{\mathcal{O},\mathcal{S}} $. We set
 $$
 \| L\|_{\mathscr{L}_{\mathcal{O},\mathcal{S}}^{\mathrm{sup}}} := \sup_{\xi \in \mathcal{O}} \| L(\xi) \|_{\mathscr{L}_{\mathcal{S}}^{\mathrm{cste}}},
\quad 
 \| L\|_{\mathscr{L}_{\mathcal{O},\mathcal{S}}^{\mathrm{lip}}} := \sup_{\substack{\xi,\zeta \in \mathcal{O}\\ \xi \neq \zeta}} \frac{\| L(\xi) - L(\zeta)  \|_{\mathscr{L}_{\mathcal{S}}^{\mathrm{cste}}}}{\| \xi -\zeta \|_{\ell^1}},
 $$
 $$
\| L\|_{\mathscr{L}_{r,\mathcal{O},\mathcal{S}}^{\mathrm{tot}}} := \| L\|_{\mathscr{L}_{\mathcal{O},\mathcal{S}}^{\mathrm{sup}}}  +  r^2\| L\|_{\mathscr{L}_{\mathcal{O},\mathcal{S}}^{\mathrm{lip}}}. 
$$
\end{definition}

\begin{definition}[Norms for $\mathscr{Q}_{\mathcal{O},\mathcal{S}}$] Let $\mathcal{S} \subset \mathbb{Z}$ be a finite set, $\mathcal{O}\subset \mathbb{R}_+^\mathcal{S}$, $Q\in   \mathscr{Q}_{\mathcal{O},\mathcal{S}} $ and $r>0$. We set
 $$
 \| Q\|_{\mathscr{Q}_{\mathcal{O},\mathcal{S}}^{\mathrm{sup}}} := \sup_{\xi \in \mathcal{O}} \| Q(\xi) \|_{\mathscr{Q}_{\mathcal{S}}^{\mathrm{cste}}},
\quad 
 \| Q\|_{\mathscr{Q}_{\mathcal{O},\mathcal{S}}^{\mathrm{lip}}} := \sup_{\substack{\xi,\zeta \in \mathcal{O}\\ \xi \neq \zeta}} \frac{\| Q(\xi) - Q(\zeta)  \|_{\mathscr{Q}_{\mathcal{S}}^{\mathrm{cste}}}}{\| \xi -\zeta \|_{\ell^1}},
 $$
 $$
\| Q\|_{\mathscr{Q}_{r,\mathcal{O},\mathcal{S}}^{\mathrm{tot}}} := \| Q\|_{\mathscr{Q}_{\mathcal{O},\mathcal{S}}^{\mathrm{sup}}}  + r^2 \| Q\|_{\mathscr{Q}_{\mathcal{O},\mathcal{S}}^{\mathrm{lip}}}. 
$$
\end{definition}

The elements of $\mathscr{L}_{\mathcal{O},\mathcal{S}}$ are frequencies on which we need strong Lipschitz estimates to be able to prove small divisor estimates (see Lemma \ref{lem:pd_tech}). The most critical terms appear when we "recenter" an element of $\mathscr{Q}_{\mathcal{O},\mathcal{S}}$ (see Proposition \ref{prop:opennormal} and in particular equations \eqref{eq:sans_doutes_le_truc_le_plus_important_du_papier} and \eqref{eq:dir2}). Therefore, we now introduce a more subtle Lipschitz norm on $\mathscr{L}_{\mathcal{O},\mathcal{S}}$ allowing to control these critical terms.

%We introduce another more subtle Lipschitz norm on $\mathscr{L}_{\mathcal{O},\mathcal{S}}$ 
%mainly for two reasons :
%\begin{itemize}
%\item the elements of $\mathscr{L}_{\mathcal{O},\mathcal{S}}$ are frequencies on which we need strong estimates to be able to prove small divisor estimates (see Lemma \ref{lem:pd_tech}),
%\item in our construction, the elements of $\mathscr{L}_{\mathcal{O},\mathcal{S}}$ with critical Lipschitz estimates appear when we "recenter" an element of $\mathscr{Q}_{\mathcal{O},\mathcal{S}}$ (see Proposition \ref{prop:opennormal} and in particular equations \eqref{eq:sans_doutes_le_truc_le_plus_important_du_papier} and \eqref{eq:dir2}).
%\end{itemize}

\begin{definition}[Directional Lipschitz norms for $\mathscr{L}_{\mathcal{O},\mathcal{S}}$] \label{def:dirlip}  Let $\mathcal{S} \subset \mathbb{Z}$ be a finite set, $\mathcal{O}\subset \mathbb{R}_+^\mathcal{S}$, $L\in   \mathscr{L}_{\mathcal{O},\mathcal{S}} $. We set
 $$
\| L \|_{\mathscr{L}_{\mathcal{O},\mathcal{S}}^{\mathrm{dir}}} := \sup_{i,k\in \mathbb{Z}} \sup_{(\xi,\zeta) \in \Delta_i\mathcal{O}}   \frac{|L_k(\xi) -L_k(\zeta)  | }{ | \xi_i - \zeta_i |} \Gamma_{k,i}^{-1}.
$$
where (recalling that $\Lambda_{k,i}$ is defined by \eqref{eq:defLambda})
\begin{equation}
\label{eq:def_Gammaki}
\Gamma_{k,i} := \Lambda_{k,i}  \vee \frac1{  \langle k\rangle^{\delta}}
\end{equation}
and
\begin{equation}
\label{eq:def_deltaOi}
\Delta_i \mathcal{O} := \{ (\xi,\zeta) \in \mathcal{O}^2 \ | \  \xi \neq \zeta \quad \mathrm{and} \quad \forall k\neq i, \ \xi_k = \zeta_k \}.
\end{equation}
\end{definition}

\begin{definition}[A well suited norm on $\mathscr{N}_{\mathcal{O},\mathcal{S}}$] \label{def:N_norm}
Let $\mathcal{S} \subset \mathbb{Z}$ be a finite set, $\mathcal{O}\subset \mathbb{R}_+^\mathcal{S}$, $N=A+L+Q\in \mathscr{N}_{\mathcal{O},\mathcal{S}}$ and $r>0$.
We set
$$
\| N \|_{\mathscr{N}_{r,\mathcal{O},\mathcal{S}}^{\mathrm{tot}}} := \| A \|_{ \mathscr{A}_{\mathcal{O},\mathcal{S}}^{\mathrm{tot}} }  \vee \|  L   \|_{\mathscr{L}_{r,\mathcal{O},\mathcal{S}}^{\mathrm{tot}}} \vee \|  L   \|_{\mathscr{L}_{\mathcal{O},\mathcal{S}}^{\mathrm{dir}}} \vee  3  \, \|  Q  \|_{ \mathscr{Q}_{r,\mathcal{O},\mathcal{S}}^{\mathrm{tot}} }.
$$
\end{definition}
\begin{remark}\label{rem:r2}
We can notice that, in $\| N \|_{\mathscr{N}_{r,\mathcal{O},\mathcal{S}}^{\mathrm{tot}}} $, we put a factor $r^2$ in front of $ \|  L   \|_{\mathscr{L}_{r,\mathcal{O},\mathcal{S}}^{\mathrm{lip}}}$ and $\|  Q  \|_{ \mathscr{Q}_{r,\mathcal{O},\mathcal{S}}^{\mathrm{lip}} }$ but not in front of $\|  A  \|_{ \mathscr{Q}_{r,\mathcal{O},\mathcal{S}}^{\mathrm{lip}} }$ and $ \|  L   \|_{\mathscr{L}_{\mathcal{O},\mathcal{S}}^{\mathrm{dir}}}$. This extra factor is natural (and indeed useful in many estimations) since these Lipschitz norms contain $\| \xi -\zeta \|_{\ell^1}^{-1}$ that will be of order $r^2$. Nevertheless, in section \ref{sec:NR}, we need to prove some non resonance properties on the frequencies $\omega^{(\infty)}_j$ obtained as the limit when $p\to\infty$ of  $\omega^{(p)}_j$, the frequencies of the finite dimensional tori constructed as intermediates. It turns that the associated radii $r_p$ tends to zero when $p\to\infty$ in such way, with the extra factor $r^2$, we lose all control on the Lipschitz property of $\omega^{(\infty)}_j(\xi)$, something that we need in section \ref{sec:NR} (see step 1 of the proof of Lemma \ref{lem:full_meas_sharp}). 
\end{remark}

\subsubsection{Embeddings and projections}

First, by definition of the the norms, we get the following straightforward embedding estimates.
\begin{lemma}[Pointwise Embeddings] \label{lem:ptw}Let $\mathcal{S}\subset \mathbb{Z}$ be a finite set. Then for all $r\in (0,1]$ and $\eta \leq \eta_{\max}$, $\mathscr{A}^{\mathrm{cste}}$, $\mathscr{L}^{\mathrm{cste}}_{\mathcal{S}}$ and $\mathscr{Q}^{\mathrm{cste}}_{\mathcal{S}}$ are subspaces of $\mathscr{H}_{\eta,r,\mathcal{S}}^{\mathrm{cste}}$. Moreover, the associated embeddings enjoy the following continuity estimates
$$
r^2\| A \|_{\mathscr{A}^{\mathrm{cste}}} \lesssim \| A \|_{\mathscr{H}_{\eta,r,\mathcal{S}}^{\mathrm{cste}}} \lesssim \| A \|_{\mathscr{A}^{\mathrm{cste}}}, \quad \| L \|_{\mathscr{H}_{\eta,r,\mathcal{S}}^{\mathrm{cste}}} \sim_{\mathcal{S}} r^2\| L \|_{\mathscr{L}^{\mathrm{cste}}}, \quad \| Q \|_{\mathscr{H}_{\eta,r,\mathcal{S}}^{\mathrm{cste}}} \sim_{\mathcal{S}} r^4 \| Q \|_{\mathscr{Q}^{\mathrm{cste}}}
$$
for all $(A,L,Q)\in \mathscr{A}^{\mathrm{cste}} \times \mathscr{L}^{\mathrm{cste}}_{\mathcal{S}} \times \mathscr{Q}^{\mathrm{cste}}_{\mathcal{S}}$.
\end{lemma}

\begin{remark} As a consequence of Lemma \ref{lem:ouf_cest_des_fonctions}, it follows that $A,L,Q$ define smooth functions on some annulus of $\ell_\varpi^2$.
\end{remark}

\begin{corollary}\label{cor:emb}[Embeddings with parameters and definition of the projections] Let $\mathcal{S} \subset \mathbb{Z}$ be a finite set,  $r\in (0,1/4)$, $\mathcal{O}\subset ((4r)^2;1)^\mathcal{S}$, $\eta \leq \eta_{\max}$. Then $\mathscr{A}_{\mathcal{O},\mathcal{S}},\mathscr{L}_{\mathcal{O},\mathcal{S}},\mathscr{Q}_{\mathcal{O},\mathcal{S}},\mathscr{N}_{\mathcal{O},\mathcal{S}}$ are subspaces of  $\mathscr{H}_{\eta,r,\mathcal{O},\mathcal{S}}$. We denote, respectively, by $\Pi_\mathscr{A},\Pi_{\mathscr{A}_1} ,\Pi_\mathscr{L},\Pi_\mathscr{Q},\Pi_{\mathrm{nor}}$ the associated projection (defined by restriction of the coefficients).
\end{corollary}

\begin{corollary}[Equivalence of the norms] \label{cor:emb_and_proj} Let $\mathcal{S} \subset \mathbb{Z}$ be a finite set,  $r\in (0,1/4)$, $\mathcal{O}\subset ((4r)^2;1)^\mathcal{S}$, $\eta \leq \eta_{\max}$, $A \in \mathscr{A}_{\mathcal{O},\mathcal{S}}$, $L \in \mathscr{Q}_{\mathcal{O},\mathcal{S}}$ and $Q \in \mathscr{Q}_{\mathcal{O},\mathcal{S}}$. Then, we have
$$
r^4 \|  A  \|_{ \mathscr{A}_{r,\mathcal{O},\mathcal{S}}^{\mathrm{tot}} } \lesssim_{\mathcal{S}} \| A\|_{\mathscr{H}_{\eta,r,\mathcal{O},\mathcal{S}}^\mathrm{tot}} \lesssim_{\mathcal{S}}  \|  A  \|_{ \mathscr{A}_{r,\mathcal{O},\mathcal{S}}^{\mathrm{tot}} },
$$
$$
 \| L\|_{\mathscr{H}_{\eta,r,\mathcal{O},\mathcal{S}}^\mathrm{tot}} \sim_{\mathcal{S}} r^2 \|  L  \|_{ \mathscr{L}_{r,\mathcal{O},\mathcal{S}}^{\mathrm{tot}} } \quad \mathrm{and} \quad \| Q\|_{\mathscr{H}_{\eta,r,\mathcal{O},\mathcal{S}}^\mathrm{tot}} \sim_{\mathcal{S}} r^4 \|  Q  \|_{ \mathscr{Q}_{r,\mathcal{O},\mathcal{S}}^{\mathrm{tot}} }.
$$
\end{corollary}

\begin{lemma}[Continuity estimate for $\Pi_{\mathrm{nor}}$] \label{lem:continuity_estimate_pinor} Let $\mathcal{S} \subset \mathbb{Z}$ be a finite set,  $r\in (0,1/4)$, $\mathcal{O}\subset ((4r)^2;1)^\mathcal{S}$, $\eta \leq \eta_{\max}$ and $P \in \mathscr{H}_{\eta,r,\mathcal{O},\mathcal{S}}^\mathrm{tot}$ then we have
$$
\| \Pi_{\mathrm{nor}} P \|_{\mathscr{N}_{r,\mathcal{O},\mathcal{S}}^{\mathrm{tot}}} \lesssim_{\mathcal{S},\eta} r^{-4} \| P\|_{\mathscr{H}_{\eta,r,\mathcal{O},\mathcal{S}}^\mathrm{tot}}.
$$
\end{lemma}
\begin{proof} Thanks to Corollary \ref{cor:emb_and_proj}, since $r\leq 1$, by definition of the norms, it suffices to prove that
$$
\| \Pi_{\mathscr{L}} P \|_{\mathscr{L}_{\mathcal{O},\mathcal{S}}^{\mathrm{dir}}} \lesssim_{\mathcal{S}} r^{-2}\| P\|_{\mathscr{H}_{\eta,r,\mathcal{O},\mathcal{S}}^{\mathrm{lip}}}.
$$
Therefore it suffices to prove that for all $L\in \mathscr{L}_{\mathcal{O},\mathcal{S}}$, we have
\begin{equation}
\label{eq:hehe}
\| L \|_{\mathscr{L}_{\mathcal{O},\mathcal{S}}^{\mathrm{dir}}} \leq \| L\|_{\mathscr{L}_{\mathcal{O},\mathcal{S}}^{\mathrm{lip}}}.
\end{equation}
Indeed, by Corollary \ref{cor:emb_and_proj}, we will have
$$
\| \Pi_{\mathscr{L}} P \|_{\mathscr{L}_{\mathcal{O},\mathcal{S}}^{\mathrm{dir}}} \leq \| \Pi_{\mathscr{L}} P \|_{\mathscr{L}_{\mathcal{O},\mathcal{S}}^{\mathrm{lip}}} \leq r^{-2} \| \Pi_{\mathscr{L}} P \|_{\mathscr{L}_{\mathcal{O},\mathcal{S}}^{\mathrm{tot}}} \lesssim_{\mathcal{S}} r^{-4} \| P\|_{\mathscr{H}_{\eta,r,\mathcal{O},\mathcal{S}}^\mathrm{tot}}.
$$
So, from now, we focus on proving \eqref{eq:hehe}.
Let $i\in \mathcal{S}$, $k\in \mathbb{Z}$ and $\zeta,\xi \in  \Delta_i\mathcal{O}$. Noticing that, by definition of  $\Delta_i\mathcal{O}$, we have $\| \xi -\zeta \|_{\ell^1}=| \xi_i -\zeta_i |$, we have
$$
|L_k(\xi) - L_k(\zeta)| \leq \langle k \rangle^{-\delta} \| \xi -\zeta \|_{\ell^1} \| L\|_{\mathscr{L}_{\mathcal{O},\mathcal{S}}^{\mathrm{lip}}} \leq \Gamma_{i,k} | \xi_i -\zeta_i | \| L\|_{\mathscr{L}_{\mathcal{O},\mathcal{S}}^{\mathrm{lip}}}.
$$
By considering the supremum with respect to $\xi,\zeta,i$, we get, as expected, \eqref{eq:hehe}.
\end{proof}

\section{Small divisor estimates}\label{sec:PD}
\label{sec:sd}
The control of the so-called small divisors is central in the KAM theory and Birkhoff normal form theory. It quantifies the possible interaction between the various Fourier modes due to the near-resonances of linear frequencies. In our work, we need to control more small divisors than usual. Indeed we consider small divisor with four external modes whereas two external modes are sufficient in standard KAM theory. Concretely we control expressions of the form \eqref{eq:cequonveut_sd} with $d\leq4$.\\
To achieve this, we make crucial use of the fact that our frequencies, up to a rescaling, accumulate polynomially fast over integers: $\omega_j=\eps^{-2} j^2+O(j^{-\delta})$ (see for instance \eqref{eq:omega}). As a result, an estimate of small divisors with respect to the largest of the external indices involved is sufficient to obtain \eqref{eq:cequonveut_sd} (see Step 3 in the proof of Proposition \ref{prop:smalldiv}). This argument was already used in \cite{BG22}. All the small divisor estimates of this paper will rely on the following basic technical lemma. %je l'ai mis sous une forme assez souple pour s'appliquer dans la preuve du thm principal quand on veut montrer que nos tores sont vraiment non résonants. 
\begin{lemma} \label{lem:pd_tech} Let $\mathcal{S} \subset \mathbb{Z}$ be a finite set, $\mathcal{O} \subset (0,1)^{\mathcal{S}}$ be a Borel set, $\boldsymbol{k} \in \mathbb{R}^{\mathcal{S}} \setminus \{0\}$, $g: \mathcal{O} \to \mathbb{R}^{\mathcal{S}}$, $a: \mathcal{O} \to \mathbb{R}$ be some measurable functions satisfying 
\begin{equation}
\label{eq:jolie_est}
\forall i\in \mathcal{S},\forall (\xi,\zeta) \in \Delta_i\mathcal{O} ,\quad  \| \boldsymbol{k} \|_{\ell^\infty}^{-1} |a(\xi) -a(\zeta)  | + \sum_{\ell\in \mathcal{S}}  |g_\ell(\xi) -g_\ell(\zeta)  |  \leq \frac{| \xi_i - \zeta_i |}2
\end{equation}
where $ \Delta_i\mathcal{O}$ is defined by \eqref{eq:def_deltaOi}. Then for all $\gamma>0$, we have
$$
\mathrm{Leb}\Big(\Big\{ \xi \in \mathcal{O} \ | \  \big| a(\xi) + \sum_{i\in \mathcal{S}} \boldsymbol{k}_i( \xi_i +g_i(\xi)) \big| < \gamma   \Big\}\Big)\leq 4 \gamma \| \boldsymbol{k} \|_{\ell^\infty}^{-1}.
$$
\end{lemma}
\begin{proof} Let $\gamma>0$ and  $i_* \in \mathcal{S}$ be such that
\begin{equation}
\label{eq:le_plus_grand}
|\boldsymbol{k}_{i_*}| = \max_{ j\in \mathcal{S} }  |\boldsymbol{k}_{j}|.
\end{equation}
 Then, we note that by Tonelli's theorem, for all Borel set $\Xi\in (0,1)^{\mathcal{S}}$ and all $j\in \mathcal{S}$, we have
$$
\mathrm{Leb} (\Xi) = \int_{ (\xi)_{i\neq j} \in   (0,1)^{\mathcal{S}\setminus \{ j \}}}  \mathrm{Leb}( \{ \xi_j \in \mathbb{R} \ | \ \xi \in \Xi \})  \, \mathrm{d}(\xi)_{i\neq j} \leq \sup_{ (\xi)_{i\neq j} \in   (0,1)^{\mathcal{S}\setminus \{ j \}}}  \mathrm{Leb}( \underbrace{\{ \xi_j \in \mathbb{R} \ | \ \xi \in \Xi \}}_{=: \Xi^{(j)}_{(\xi)_{i\neq j} }}) .
$$
Applying this estimate with $j=i_*$ and
$$
\Xi := \big\{ \xi \in \mathcal{O} \ | \  | h(\xi) | < \gamma   \big\} \quad \mathrm{where} \quad h(\xi) := a(\xi) + \sum_{i\in \mathcal{S}} \boldsymbol{k}_i (\xi_i + g_i(\xi)),
$$
we get 
$$
\mathrm{Leb} (\Xi) \leq \sup_{ (\xi)_{i\neq i_*} \in   (0,1)^{\mathcal{S}\setminus \{ i_* \}}}  \mathrm{Leb}(\Xi^{(i_*)}_{(\xi)_{i\neq i*} }).
$$
Therefore, to conclude the proof, it suffices to prove that $\Xi^{(i_*)}_{(\xi)_{i\neq i*} }$ is included in an interval of length $4\gamma \| \boldsymbol{k} \|_{\ell^\infty}^{-1}$. In other words, it suffices to prove that if
 $\xi,\zeta \in \Delta_{i_*}(\mathcal{O})$ are such that $\xi_{i_*},\zeta_{i_*} \in \Xi^{(i_*)}_{(\xi)_{i\neq i*} }$ then $|\xi_{i_*}-\zeta_{i_*}|<4\gamma \| \boldsymbol{k} \|_{\ell^\infty}^{-1}$.
 
\medskip 
 
 Indeed, if $\xi,\zeta \in \Delta_{i_*}(\mathcal{O})$ are such that $\xi_{i_*},\zeta_{i_*} \in \Xi^{(i_*)}_{(\xi)_{i\neq i*} }$ then we have
\begin{equation*}
\begin{split}
2\gamma > \big| h(\xi) -  h(\zeta) \big| 
&\geq |\boldsymbol{k}_{i_*}| |\xi_{i_*}-\zeta_{i_*}| - \sum_{i\in \mathcal{S}}  |\boldsymbol{k}_i| |g_i (\xi) - g_i (\zeta)|  - |a(\xi) - a(\zeta)|\\
&\mathop{\geq}^{\eqref{eq:le_plus_grand}}  \| \boldsymbol{k} \|_{\ell^\infty} \Big( |\xi_{i_*}-\zeta_{i_*}| -
 \sum_{i\in \mathcal{S}}  |g_i (\xi) - g_i (\zeta)| -  \| \boldsymbol{k} \|_{\ell^\infty}^{-1} |a(\xi) -a(\zeta)  |\Big) \\
&\mathop{\geq}^{\eqref{eq:jolie_est}}  \frac{\| \boldsymbol{k} \|_{\ell^\infty}}2 |\xi_{i_*}-\zeta_{i_*}|.
\end{split}
\end{equation*}
\end{proof}

In our normal form theorems we will apply the following proposition.
\begin{proposition}\label{prop:smalldiv} Let $\mathcal{S} \subset \mathbb{Z}$ be a finite set, $\mathcal{O} \subset (0,1)^{\mathcal{S}}$ be a Borel set, $\varepsilon\leq 1$, $\mathfrak{M}\geq 2$, $L\in \mathscr{L}_{\mathcal{O},\mathcal{S} }$ and set
$$
\forall j\in \mathbb{Z}, \forall \xi \in \mathcal{O}, \quad \omega_j(\xi) := 2\varepsilon L_j(\xi) + \left\{\begin{array}{lcll}  \xi_j & \mathrm{if}  & j\in \mathcal{S} \\
 0 &\mathrm{if} & j\in \mathcal{S}^c
\end{array}\right. .
$$
Assume that $L$ satisfies the estimates
$$
\| L \|_{\mathscr{L}_{\mathcal{O},\mathcal{S} }^{\mathrm{sup}}} \vee \| L \|_{\mathscr{L}_{\mathcal{O},\mathcal{S} }^{\mathrm{dir}}} \leq 4
$$
and, recalling that the coefficients $\Gamma$ are defined by \eqref{eq:def_Gammaki}, assume that $\varepsilon$ satisfies
$$
\varepsilon \leq (24\, \mathfrak{M})^{-1} \quad \mathrm{and} \quad \varepsilon \big( \sup_{i\in \mathcal{S}} \sum_{k\in \mathcal{S}} \Gamma_{k,i} \big) \leq 10^{-3} 
$$
Then, for all $\gamma \in (0,1)$ there exists a Borel subset $\mathcal{O}' \subset \mathcal{O}$ such that
$$
\mathrm{Leb}( \mathcal{O}  \setminus \mathcal{O}' ) \leq \gamma
$$
and for all $\xi \in \mathcal{O}'$, all $\boldsymbol{k}\in \mathbb{Z}^{\mathcal{S}}\setminus \{0\}$, all $d \in \llbracket 0,4 \rrbracket$, all vector $\boldsymbol{j} \in (\mathcal{S}^c)^d$, all $\boldsymbol{b} \in (\mathbb{Z}^*)^d$, all $n\in \mathbb{Z}$, provided that
$\|\boldsymbol{b}\|_{\ell^1} \leq \mathfrak{M}$, we have
\begin{equation}
\label{eq:cequonveut_sd}
\Big| \varepsilon^{-2} n + \sum_{i\in \mathcal{S}} \boldsymbol{k}_i \omega_i(\xi) + \sum_{p=1}^d \boldsymbol{b}_p \omega_{\boldsymbol{j}_p}(\xi) \Big| \gtrsim_{\mathcal{S},\mathfrak{M}} \Big(\frac{\gamma}{ \| \boldsymbol{k} \|_{\ell^\infty}^{2\# \mathcal{S}} }\Big)^{c_d}
\end{equation}
where\footnote{these exponents are far from sharp, we could improve them proceeding as in \cite{BGR23}.} $c_0=1$, $c_1 = 3$, $c_2 = 9 $, $c_3 =34$ and $c_4=166$.
\end{proposition}
\begin{proof} We set
$$
E= \{ (\boldsymbol{k},d,\boldsymbol{j},\boldsymbol{b},n ) \ | \ \boldsymbol{k}\in \mathbb{Z}^{\mathcal{S}}\setminus \{0\},\ d \in \llbracket 0,4 \rrbracket,\ \boldsymbol{j} \in (\mathcal{S}^c)^d,\ n\in \mathbb{Z}, \boldsymbol{b} \in (\mathbb{Z}^*)^d\, \ \mathrm{and} \ \|\boldsymbol{b}\|_{\ell^1} \leq \mathfrak{M} \}
$$
$$
\mathrm{and} \quad \forall \boldsymbol{e}=(\boldsymbol{k},d,\boldsymbol{j},\boldsymbol{b},n ) \in E, \quad h_{\boldsymbol{e}}=  \varepsilon^{-2} n + \sum_{i\in \mathcal{S}} \boldsymbol{k}_i \omega_i + \sum_{k=1}^d \boldsymbol{b}_k \omega_{\boldsymbol{j}_k}.
$$
\noindent \underline{$\bullet$ \emph{Step $1$ : local estimates.}}
Let $\boldsymbol{e}=(\boldsymbol{k},d,\boldsymbol{j},\boldsymbol{b},n ) \in E$. We aim at proving that 
\begin{equation}
\label{eq:pas_dimagination}
\forall \gamma>0, \quad \mathrm{Leb}( \underbrace{\{\xi \in \mathcal{O} \ | \ |h_{\boldsymbol{e}}|<\gamma \}}_{=:\Xi_{\boldsymbol{e},\gamma}}) \leq 4 \gamma  \| \boldsymbol{k} \|_{\ell^\infty}^{-1}.
\end{equation}
To prove it we apply Lemma \ref{lem:pd_tech} with 
$$
a =  \varepsilon^{-2} n + \sum_{p=1}^d \boldsymbol{b}_p \omega_{\boldsymbol{j}_p} \quad \mathrm{and} \quad g = 2\varepsilon (L_k)_{k\in \mathcal{S}}.
$$
So we just have to check the estimate \eqref{eq:jolie_est}. So let $i\in \mathcal{S}$ and $(\xi,\zeta) \in \Delta_i \mathcal{O}$. On the one hand, by definition of $\| L \|_{\mathscr{L}_{\mathcal{O},\mathcal{S} }^{\mathrm{dir}}} $ (see \eqref{def:dirlip}), using the smallness assumptions on $\| L \|_{\mathscr{L}_{\mathcal{O},\mathcal{S} }^{\mathrm{dir}}}$ and $\varepsilon$,  we have
$$
\sum_{k \in \mathcal{S}} |g_k( \xi) - g_k(\zeta)| \leq 2\varepsilon \sum_{k \in \mathcal{S}} |L_k( \xi) - L_k(\zeta)| \leq 2 \varepsilon \sum_{k \in \mathcal{S}} \Gamma_{k,i}\| L \|_{\mathscr{L}_{\mathcal{O},\mathcal{S} }^{\mathrm{dir}}} |\xi_i - \zeta_i| \leq \frac{|\xi_i - \zeta_i|}{100}.
$$
On the other hand, using the smallness assumptions on $\| L \|_{\mathscr{L}_{\mathcal{O},\mathcal{S} }^{\mathrm{dir}}}$ and $\varepsilon$,  since the coefficient $\Gamma_{k,i}$ are smaller than or equal to $1$, we have
$$
|a( \xi) - a(\zeta)| \leq \varepsilon \sum_{p =1}^d |\boldsymbol{b}_p| |\omega_{\boldsymbol{j}_p}(\xi) - \omega_{\boldsymbol{j}_p}(\zeta)| \leq 2 \varepsilon \sum_{p =1}^d  |\boldsymbol{b}_p| \Gamma_{\boldsymbol{j}_p,i}\| L \|_{\mathscr{L}_{\mathcal{O},\mathcal{S} }^{\mathrm{dir}}} |\xi_i - \zeta_i| \leq 8\varepsilon \mathfrak{M}|\xi_i - \zeta_i|\leq \frac{|\xi_i - \zeta_i|}3.
$$
Since $\frac13 + \frac1{100} \leq \frac12$, the estimate \eqref{eq:jolie_est} is satisfied and so \eqref{eq:pas_dimagination} holds.

\medskip

\noindent \underline{$\bullet$ \emph{Step $2$ : global estimates.}} Now, let $\gamma \in (0,1)$. We first note that if $\boldsymbol{e}=(\boldsymbol{k},d,\boldsymbol{j},\boldsymbol{b},n ) \in E$ satisfies $|n|\geq 1+8 (\# \mathcal{S})\| \boldsymbol{k} \|_{\ell^\infty} + 8 \mathfrak{M} $ then $\Xi_{\boldsymbol{e},\gamma} = \emptyset$. Indeed, it suffices to observe that, thanks to assumption on $\| L \|_{\mathscr{L}_{\mathcal{O},\mathcal{S} }^{\mathrm{sup}}}$, if  there exists $\xi \in \Xi_{\boldsymbol{e},\gamma} $ then
$$
 \varepsilon^{-2} |n| \leq 1+ \sum_{i\in \mathcal{S}} |\boldsymbol{k}_i| |\omega_i(\xi) |+ \sum_{p=1}^d |\boldsymbol{b}_p| |\omega_{\boldsymbol{j}_p}(\xi)| \leq 1+  8(\# \mathcal{S})\| \boldsymbol{k} \|_{\ell^\infty} + 8 \mathfrak{M}.
$$
Then, we set
$$
\upsilon_{\boldsymbol{e}} := \gamma\, c_{\mathcal{S},\mathfrak{M}} \| \boldsymbol{k} \|_{\ell^\infty}^{-2 \# \mathcal{S}} \langle \boldsymbol{j}_1 \rangle^{-\alpha} \cdots \langle \boldsymbol{j}_d \rangle^{-\alpha}
$$
where $\alpha= 11/10$ and $c_{\mathcal{S},\mathfrak{M}}\in (0,1)$ is a constant depending only on $\mathcal{S}$ and $\mathfrak{M}$ that will be determined later. 
 Using \eqref{eq:pas_dimagination} we have
$$
\mathrm{Leb} \Big( \bigcup_{\boldsymbol{e} \in E}  \Xi_{\boldsymbol{e},\upsilon_{\boldsymbol{e}}} \Big) \leq 8\gamma c_{\mathcal{S},\mathfrak{M}} \! \! \! \! \sum_{\boldsymbol{k} \in \mathbb{Z}^{\mathcal{S}} \setminus \{ 0\} } (1+8 (\# \mathcal{S})\| \boldsymbol{k} \|_{\ell^\infty} + 8\mathfrak{M})\ \| \boldsymbol{k} \|_{\ell^\infty}^{(-2\# \mathcal{S}-1)} \sum_{d=0}^4\sum_{\boldsymbol{j} \in (\mathcal{S}^{c})^d} \langle \boldsymbol{j}_1 \rangle^{-\alpha} \cdots \langle \boldsymbol{j}_d \rangle^{-\alpha}.$$
It follows that, provided that $c_{\mathcal{S},\mathfrak{M}}$ is small enough
$$
\mathrm{Leb} \Big( \bigcup_{\boldsymbol{e} \in E}  \Xi_{\boldsymbol{e},\upsilon_{\boldsymbol{e}}} \Big) %\lesssim_{\mathcal{S},\mathfrak{M}} \gamma c_{\mathcal{S},\mathfrak{M}} \sum_{\boldsymbol{k} \in \mathbb{Z}^{\mathcal{S}} \setminus \{ 0\} } \| \boldsymbol{k} \|_{\ell^\infty}^{-2\# \mathcal{S}} \sum_{d=0}^4\sum_{\boldsymbol{j} \in (\mathcal{S}^{c})^d} \langle \boldsymbol{j}_1 \rangle^{-\alpha} \cdots \langle \boldsymbol{j}_d \rangle^{-\alpha}
\leq \gamma.
$$
Consequently, setting
$$
\mathcal{O}' = \mathcal{O} \setminus \bigcup_{\boldsymbol{e} \in E}  \Xi_{\boldsymbol{e},\upsilon_{\boldsymbol{e}}} ,
$$
we have $\mathrm{Leb}( \mathcal{O}  \setminus \mathcal{O}' ) \leq \gamma$ and for all $\xi \in \mathcal{O}'$,
\begin{equation}
\label{eq:est_glob_sd}
 \forall \boldsymbol{e}\in E, \quad |h_{\boldsymbol{e}}(\xi)|\geq \upsilon_{\boldsymbol{e}} = \gamma \, c_{\mathcal{S},\mathfrak{M}}   \| \boldsymbol{k} \|_{\ell^\infty}^{-2 \# \mathcal{S}} \langle \boldsymbol{j}_1 \rangle^{-\alpha} \cdots \langle \boldsymbol{j}_d \rangle^{-\alpha}.
 \end{equation}

\medskip

\noindent \underline{$\bullet$ \emph{Step $3$ \label{step3} : uniform estimates.}} Finally, in order to improve estimate \eqref{eq:est_glob_sd}  we have to remove the factor $\langle \boldsymbol{j}_1 \rangle^{-\a} \cdots \langle \boldsymbol{j}_d \rangle^{-\a}$. We follow the approach introduced in \cite{BG21} to get small divisor estimate strong enough to prove long time stability of small solution to nonlinear dispersive PDEs at low regularity. We decompose $E$ depending on the value of $d$ :
$$
E_{d_*} := \{ (\boldsymbol{k},d,\boldsymbol{j},\boldsymbol{b},n ) \in E \ | \  d=d_* \}.
$$
Now, let $\xi \in \mathcal{O}'$. First, we note that if $\boldsymbol{e} \in E_{0}$ then \eqref{eq:est_glob_sd} gives \eqref{eq:cequonveut_sd} with $c_0=1$. Then, we proceed by induction. We assume that \eqref{eq:cequonveut_sd} holds for all $\boldsymbol{e} \in E_{d-1}$ with $1\leq d \leq 4$. We fix $\boldsymbol{e} =  (\boldsymbol{k},d,\boldsymbol{j},\boldsymbol{b},n )\in E_{d}$ and we aim at proving \eqref{eq:cequonveut_sd}. Without loss of generality, we assume that 
$$
|\boldsymbol{j}_d| = \max_{1\leq p\leq d} |\boldsymbol{j}_p|.
$$
Thanks to the induction hypothesis and the assumption on $\| L \|_{\mathscr{L}_{\mathcal{O},\mathcal{S} }^{\mathrm{sup}}}$, we have
$$
|h_{\boldsymbol{e}}(\xi)| \geq \nu_{d-1,\mathcal{S},\mathfrak{M}}  \Big(\frac{\gamma}{ \| \boldsymbol{k} \|_{\ell^\infty}^{2\# \mathcal{S}} }\Big)^{c_{d-1}} -|\boldsymbol{b}_d| |\omega_{\boldsymbol{j}_d}(\xi)|  \geq \nu_{d-1,\mathcal{S},\mathfrak{M}}  \Big(\frac{\gamma}{ \| \boldsymbol{k} \|_{\ell^\infty}^{2\# \mathcal{S}} }\Big)^{c_{d-1}} - 8 \mathfrak{M} \langle \boldsymbol{j}_d \rangle^{-\delta}
$$
where $\nu_{d,\mathcal{S},\mathfrak{M}} < 1$ is a constant depending only on $d,\mathcal{S},\mathfrak{M}$. Either, $|\boldsymbol{j}_d|$ is large enough, i.e.
$$
\langle \boldsymbol{j}_d \rangle^{-\delta} \leq  \frac{\nu_{d-1,\mathcal{S},\mathfrak{M}}}{2^4\mathfrak{M}}  \Big(\frac{\gamma}{ \| \boldsymbol{k} \|_{\ell^\infty}^{2\# \mathcal{S}} }\Big)^{c_{d-1}}
$$
and so
$$
|h_{\boldsymbol{e}}(\xi)| \geq  \frac{\nu_{d-1,\mathcal{S},\mathfrak{M}}}2  \Big(\frac{\gamma}{ \| \boldsymbol{k} \|_{\ell^\infty}^{2\# \mathcal{S}} }\Big)^{c_{d-1}}
$$
or $ \boldsymbol{j}_d $ is not too large, i.e.
$$
\langle \boldsymbol{j}_d \rangle\leq  \Big( \frac{\nu_{d-1,\mathcal{S},\mathfrak{M}}}{2^4\mathfrak{M}} \Big)^{-1/\delta}  \Big(\frac{\gamma}{ \| \boldsymbol{k} \|_{\ell^\infty}^{2\# \mathcal{S}} }\Big)^{- \frac{c_{d-1}}{\delta}}
$$
and so, since $|\boldsymbol{j}_d| = \max_{1\leq p\leq d} |\boldsymbol{j}_p|$, thanks to \eqref{eq:est_glob_sd}, we have
$$
|h_{\boldsymbol{e}}(\xi)|\geq  c_{\mathcal{S},\mathfrak{M}}   \frac{\gamma}{ \| \boldsymbol{k} \|_{\ell^\infty}^{2\# \mathcal{S}}} \left(  \Big( \frac{\nu_{d-1,\mathcal{S},\mathfrak{M}}}{2^4\mathfrak{M}} \Big)^{1/\delta}  \Big(\frac{\gamma}{ \| \boldsymbol{k} \|_{\ell^\infty}^{2\# \mathcal{S}} }\Big)^{\frac{c_{d-1}}{\delta}} \right)^{\alpha d}.
$$
Since $\delta \geq \alpha^{-1} = \frac{10}{11}$, it follows that \eqref{eq:cequonveut_sd} holds with\footnote{Here $\lceil\cdot\rceil $ denotes the floor function. } 
$$
c_d := 1+\lceil\alpha^2 d c_{d-1}\rceil .
$$
Since $c_0=1$, we have $c_1 = 3$, $c_2 = 9 $, $c_3 =34$ and $c_4=166$.

\end{proof}

\section{A KAM theorem}\label{sec:KAM}
In this section we state and then prove a normal form theorem of KAM type adapted to our problem. Basically it is just an adaptation of the standard  Theorem (see \cite{Kuk87,Pos96}) but with one principal modification: we eliminate more terms in our Hamiltonians (that we call "adapted jet", see Definition \ref{def:jet}).

 \subsection{Statement and strategy of proof}
\begin{definition}[Non perturbative part] \label{def:non_pert_part}Being given $\varepsilon\in (0,1)$, we set
$$
H^{(0)}_{\varepsilon}(u) := \frac12 \sum_{k\in \mathbb{Z}} \big( \varepsilon^{-2} |k|^2 + \frac{|u_k|^2}2 \big) |u_k|^2 - \frac12\| u\|_{L^2}^4.
$$ 
\end{definition}

\begin{definition}[Adapted jet and remainder terms]\label{def:jet} Let $\mathcal{S} \subset \mathbb{Z}$ be a finite set, $r\in (0,1/4)$, $\mathcal{O} \subset ((4r)^2,1)^{\mathcal{S}}$ be a Borel set, $\eta \geq 0$.
A Hamiltonian $P \in \mathscr{H}_{\eta,r,\mathcal{O},\mathcal{S}}$ is an \emph{adapted jet} (a-jet) if $\Pi_{\mathrm{int}}P=0$ and its non zero coefficients $P_{n,\boldsymbol{k},\boldsymbol{m}}^{\boldsymbol{\ell},\boldsymbol{\sigma}}\neq 0$ satisfy  $\|\boldsymbol{m}\|_{\ell^1}\leq 2$, $n\leq 2$, $\#\boldsymbol{\ell} \leq 3$.\\
 We denote by $\Pi_{\mathrm{ajet}}$ the associated projection defined by restriction of the coefficients and we set $\Pi_{\mathrm{rem}} := \mathrm{Id} - \Pi_{\mathrm{nor}} -\Pi_{\mathrm{ajet}}$ (recall that $\Pi_{\mathrm{nor}}$ is defined in Corollary \ref{cor:emb} ).
\end{definition}

In the following KAM theorem we consider the Hamiltonian
$$
H= H^{(0)}_{\varepsilon}+\eps P= H^{(0)}_{\varepsilon}+\eps \Pi_{\mathrm{nor}} P+\eps \Pi_{\mathrm{ajet}} P+\eps \Pi_{\mathrm{rem}}  P
$$
and, under some hypothesis, we construct a canonical transformation that eliminate the adapted jet $\eps \Pi_{\mathrm{ajet}} P$ and slightly modify the normal form part.

\begin{theorem}[KAM] \label{thm:KAM} Let $\mathcal{S} \subset \mathbb{Z}$ be a finite set, $r\in (0,1/4)$, $\mathcal{O}^\natural \subset ((4r)^2,1)^{\mathcal{S}}$ be a Borel set, $\eta_{{\rm max}}\geq \eta^{\natural}>\eta \geq \eta_0$, $\varepsilon <1$ and $P\in \mathscr{H}_{\eta^\natural,r,\mathcal{O}^\natural,\mathcal{S}}$.

\medskip

\noindent Assume that the following assumptions hold
\begin{itemize}
\item the adapted jet is small enough, the remainder term and the normal form part are not too large, i.e.
\begin{equation}
\label{eq:small_assump_P_KAM}
\| \Pi_{\mathrm{ajet}} P\|_{\mathscr{H}_{\eta^\natural,r,\mathcal{O}^\natural,\mathcal{S}}^{\mathrm{tot}}}  \leq \ \eps^{\frac{9}{10}}r^{4000}, \quad \| \Pi_{\mathrm{nor}} P\|_{\mathscr{N}_{r,\mathcal{O}^\natural,\mathcal{S}}^{\mathrm{tot}}}\leq 7/2 \quad \mathrm{and} \quad    \| \Pi_{\mathrm{rem}} P\|_{\mathscr{H}_{\eta^\natural,r,\mathcal{O}^\natural,\mathcal{S}}^{\mathrm{tot}}} \leq r^3,
\end{equation}
\item  the frequencies can be modulated\footnote{ in the sense that the assumption of Proposition \ref{prop:smalldiv} will be satisfied}, i.e.
$$
\varepsilon \leq 10^{-2} \quad \mathrm{and} \quad   \varepsilon \big( \sup_{i\in \mathcal{S}} \sum_{k\in \mathcal{S}} \Gamma_{k,i} \big) \leq 10^{-3} ,
$$
\item $r$ is small enough, i.e.
\begin{equation}
\label{eq:r_small_KAM}
r\lesssim_{\mathcal{S},\eta^\natural-\eta} 1.
\end{equation}
\end{itemize}

\medskip

\noindent Then, there exists a Borel set $\mathcal{O}\subset \mathcal{O}^\natural$ and a Hamiltonian $K \in \mathscr{H}_{\eta,r,\mathcal{O},\mathcal{S}}$, such that 
\begin{itemize}
\item $\mathcal{O}$ is relatively large in the sense that
$$
\mathrm{Leb}(\mathcal{O}^\natural \setminus \mathcal{O}) \leq \, r^2 \varepsilon^{10^{-3}} \, 
$$
\item the adapted jet has been removed, the remainder term is not too large, and the normal form part does not have changed too much, i.e.
$$
 \Pi_{\mathrm{ajet}} K  =0,  \quad \|  \Pi_{\mathrm{rem}} K\|_{\mathscr{H}_{\eta,r,\mathcal{O},\mathcal{S}}^{\mathrm{tot}}} \leq 1  \quad \mathrm{and} \quad \|  \Pi_{\mathrm{nor}} (K-P)\|_{\mathscr{N}_{r,\mathcal{O},\mathcal{S}}^{\mathrm{tot}}} \leq r 
$$
\end{itemize}
and for all $\xi \in \mathcal{O}$, there exists a $C^1$ symplectic map $\tau_\xi:\mathcal{A}_\xi(r)\to \mathcal{A}_\xi(2r)$ commuting with the gauge transform and enjoying, for all $u\in \mathcal{A}_\xi(r)$, the bound
\begin{equation}
\label{eq:proche_id_KAM}
\| \tau_\xi(u) - u\|_{\ell^2_\varpi} + r^2 \| \mathrm{d} \tau_\xi(u) - \mathrm{Id}\|_{\ell^2_\varpi \to \ell^2_\varpi } \leq r^{3} \varepsilon^{1/2}
\end{equation}
 and the property that
\begin{equation}
\label{eq:cest_un_cght_de_var_KAM}
(H^{(0)}_{\varepsilon}+\varepsilon P)(\xi; \tau_\xi(u)) = (H^{(0)}_{\varepsilon}+\varepsilon K)(\xi;u).
\end{equation}
\end{theorem}
\begin{remark}\label{rem:kam}
We notice that $R:=(I-\Pi_{\mathrm{\mathscr{L}}}-\Pi_{\mathrm{\mathscr{A}}})K$ satisfies $\partial_v R$, $\partial^2_v R$ and   $\partial_y R$ vanish when $u=0$ and $y=0$ (in the new variables), i.e. on $\mathrm{T}_\xi$ (see \eqref{eq:def_T_droit}). Therefore, for all $\xi\in\mathcal O$, $\mathrm{T}_\xi$ is an invariant torus for the Hamiltonian system governed by $H^{(0)}_{\varepsilon}+\varepsilon K$ and $\tau_\xi^{-1}(\mathrm{T}_\xi)$ is an invariant torus for the Hamiltonian system governed by $H^{(0)}_{\varepsilon}+\varepsilon P$. So we recover the conclusions of a standard KAM theorem. In particular we note that the constructed tori are linearly stable. 
\end{remark}

Let us  explain the KAM strategy in our context. The aim of the KAM theorem  is to make $ \Pi_{\mathrm{ajet}} P$ disappear by changing the variable. In fact we will eliminate $ \Pi_{\mathrm{ajet}} P$ iteratively writting $K_0=P$ and for $k\geq0$
\begin{equation}\label{eq:goal} (H^{(0)}_{\varepsilon}+\varepsilon K_k)\circ  \Phi^1_{\chi_k}= (H^{(0)}_{\varepsilon}+\varepsilon K_{k+1})\end{equation}
where $\Phi_{\chi_k}^1$ is the time-one flow-map generated by the Hamiltonian $\chi_k$ which is chosen in such a way the size of $ \Pi_{\mathrm{ajet}} K_{k+1}$ is, roughly speaking, the square of the size of $ \Pi_{\mathrm{ajet}} K_k$. This allows a quadratic scheme which is rapidly convergent and can absorb all the unfavorable constants that will appear in our estimates. To obtain the equation that $\chi_k$ has to satisfied (the so called cohomological equation) we can do the following formal calculation based on the fact that  $ \Pi_{\mathrm{ajet}}\chi_k=\chi_k$ and  $\chi_k$ has the size of  $ \Pi_{\mathrm{ajet}} K_k$ (of course everything will be justified later on)
\begin{align*} (H^{(0)}_{\varepsilon}+\varepsilon K_k)\circ  \Phi_{\chi_k}&= H^{(0)}_{\varepsilon}+\varepsilon K_k+\{H^{(0)}_{\varepsilon}+\varepsilon K_k,\chi_k\}+\ h.o.t.\\
&=H^{(0)}_{\varepsilon}+ \eps \Pi_{\mathrm{ajet}} K_k+ \Pi_{\mathrm{ajet}} \{ H^{(0)}_{\varepsilon}+ \eps( K_k-\Pi_{\mathrm{ajet}} K_k),\chi_k\}+ R_k+\ h.o.t.
\end{align*}
where $\Pi_{\mathrm{ajet}} R_k=0$.
Therefore to achieve our goal we want $\chi_k$ to solve the {\it cohomological} equation
\begin{equation} \label{eq-homo}
 \eps \Pi_{\mathrm{ajet}} K_k+ \Pi_{\mathrm{ajet}} \{ H^{(0)}_{\varepsilon}+ \eps( K_k-\Pi_{\mathrm{ajet}} K_k),\chi_k\}=0.\end{equation}
So we will focus on the resolution of this equation in the next subsection.

\subsection{Preliminaries}
%To simplify notation, in this section  we will omit the $\eps$ index on all the new objects we will construct. Bear in mind that almost everything depends on $\eps$.

\subsubsection{Cohomological equation}
%In this section we solve  equations of the form \eqref{eq-homo}.

 First we decompose $H^{(0)}_{\varepsilon}$, which is an integrable Hamiltonian, in normal form. According to \eqref{eq:decomp}
 $$H^{(0)}_{\varepsilon}=Z_2+A^{(0)}+Q^{(0)}$$
where  $A^{(0)}= H^{(0)}_{\varepsilon}(\xi) - \mu \|\xi\|_{\ell^1}$ is the almost constant  part (see Definition \ref{def:cst}) of $H^{(0)}_{\varepsilon}$,
 $$
 Z_{2}:=\frac12\sum_{j\in\Z}(\varepsilon^{-2} j^2 +    \xi_j \mathbbm{1}_{j\in \mathcal{S}}) Y_j
 $$
is  a linear  in actions Hamiltonians (see Definition \ref{def:lin}) and $Q^{(0)}=\frac14 \sum_{k\in \mathbb{Z}} Y_k^2 - \frac12 \mu^2$
 is an integrable quartic Hamiltonian (see Definition \ref{def:quar}). We note that both $A^{(0)}$ and $Q^{(0)}$ do not depend on $\eps$. 
 
 \medskip

We first solve a linear cohomological equation.
\begin{lemma}\label{lem:homo}
Let $\mathcal{S} \subset \mathbb{Z}$ be a finite set, $r\in (0,1/4)$, $\mathcal{O} \subset ((4r)^2,1)^{\mathcal{S}}$ be a Borel set, $\eta > \eta_0$, $\varepsilon <1$, $L\in \mathscr{L}_{\mathcal{O},\mathcal{S} }$ and $R\in \mathscr{H}_{\eta,r,\mathcal{O},\mathcal{S}}$ be an adapted jet, i.e. $\Pi_{\mathrm{ajet}} R=R$. 

Assume that $L$ satisfies the estimate
$$
\| L \|_{\mathscr{L}_{r,\mathcal{O},\mathcal{S} }^{\mathrm{tot}}} \vee \| L \|_{\mathscr{L}_{\mathcal{O},\mathcal{S} }^{\mathrm{dir}}} \leq 4
$$
and assume that $\varepsilon$ satisfies
$$
\varepsilon \leq (72)^{-1} \quad \mathrm{and} \quad \varepsilon \big( \sup_{i\in \mathcal{S}} \sum_{k\in \mathcal{S}} \Gamma_{k,i} \big) \leq 10^{-3}. 
$$
Then for any $\gamma\in(0,1)$ there exists $\mathcal{O}^+\subset \mathcal{O}$ satisfying
 $${\rm Leb}(\mathcal{O}\setminus \mathcal{O}^{+}) \leq \gamma$$
and there exists $\chi\in\mathscr{H}_{\eta^+,r,\mathcal{O}^+,\mathcal{S}}$ (for all $\eta_0<\eta^+<\eta$) with $\Pi_{\mathrm{ajet}} \chi=\chi$ 
%and $\Pi_{\rm{int}}\chi=0$ and $\tilde Z_{2}\in\mathscr{H}_{\eta,r,\mathcal{O}^+,\mathcal{S}}$ with $\Pi_{\mathscr{L}}\tilde Z_{2}=\tilde Z_{2}$
such that 
\begin{equation}\label{eq:cohol}
\{ Z_{2}(\xi;\cdot)+\eps L,\chi(\xi;\cdot)\}+  R(\xi;\cdot)=0,\quad \forall \xi\in \mathcal O^+.
\end{equation}  
Furthermore $\chi$ satisfies the following estimate
\begin{align}\label{eq:estimchi}
\|  \chi \|_{\mathscr{H}_{\eta^+,r,\mathcal{O}^+,\mathcal{S}}^{\mathrm{tot}}} &\lesssim_\S \frac{1}{(\eta-\eta^+)^{138\# \mathcal{S}}\gamma^{68}}\| R\|_{\mathscr{H}_{\eta,r,\mathcal{O}^+,\mathcal{S}}^{\mathrm{tot}}}.
%\label{eq:tildeZ2}
%\|\tilde Z_{2}\|_{\mathscr{H}_{\eta,r,\mathcal{O}^+,\mathcal{S}}^{\mathrm{tot}}}&\leq \|R\|_{\mathscr{H}_{\eta,r,\mathcal{O},\mathcal{S}}^{\mathrm{tot}}}. 
\end{align}

\end{lemma}

\proof 
For $\boldsymbol{k}\in\Z^\S$,  $q\in\N$, $\boldsymbol{\ell}\in (\S^c)^q$ and  $\boldsymbol{\sigma}\in \{-1,1\}^q$ we define the small divisor
$$\Omega(\xi;\boldsymbol{k},\boldsymbol{\ell},\boldsymbol{\sigma})=\sum_{j\in\S} \boldsymbol{k}_j\omega_j(\xi)+\sum_{i=1}^q \boldsymbol{\sigma}_i\omega_{\boldsymbol{\ell}_i}(\xi)$$
where 
\begin{equation}
\label{eq:omega}
\omega_j(\xi) := \varepsilon^{-2} j^2 +    \xi_j \mathbbm{1}_{j\in \mathcal{S}}+2\eps L_j(\xi)
  ,\ j\in \Z.
\end{equation}
Then for $\boldsymbol{k}\in\Z^\S$, $\boldsymbol{m}\in\N^\S$, $q,n\in\N$, $\boldsymbol{\ell}\in (\S^c)^q$ and  $\boldsymbol{\sigma}\in \{-1,1\}^q$ 
  we set
\begin{align*}
 \chi_{n,\boldsymbol{k},\boldsymbol{m}}^{\boldsymbol{\ell},\boldsymbol{\sigma}}(\xi)&=\frac{ - R_{n,\boldsymbol{k},\boldsymbol{m}}^{\boldsymbol{\ell},\boldsymbol{\sigma}}(\xi)}{\ic\Omega(\xi;\boldsymbol{k},\boldsymbol{\ell},\boldsymbol{\sigma})}\quad \text{when }\Omega(\xi;\boldsymbol{k},\boldsymbol{\ell},\boldsymbol{\sigma})\neq0,\\
  \chi_{n,\boldsymbol{k},\boldsymbol{m}}^{\boldsymbol{\ell},\boldsymbol{\sigma}}(\xi)&=0 \quad\text{when }\Omega(\xi;\boldsymbol{k},\boldsymbol{\ell},\boldsymbol{\sigma})=0.
 \end{align*}
 Then we have
 \begin{equation}\label{eq:homobof}\{ Z_{2}(\xi;\cdot)+\eps L,\chi(\xi;\cdot)\}+ R(\xi;\cdot)=\sum_{\Omega(\xi;\boldsymbol{k},\boldsymbol{\ell},\boldsymbol{\sigma})=0} R_{n,\boldsymbol{k},\boldsymbol{m}}^{\boldsymbol{\ell},\boldsymbol{\sigma}} \, u_{\boldsymbol{\ell}}^{\boldsymbol{\sigma}} \, y^{\boldsymbol{m}} \, z^{ \boldsymbol{k} } \, \mu^n.\end{equation}
 We note that if $\boldsymbol{k}=0$ then  $\Omega(\xi;\boldsymbol{k},\boldsymbol{\ell},\boldsymbol{\sigma})=0$ implies $\sum_{i=1}^q  \boldsymbol{\sigma}_i \boldsymbol{\ell}_i^2=0$. Furthermore if  $R_{n,0,\boldsymbol{m}}^{\boldsymbol{\ell},\boldsymbol{\sigma}}\neq0$ then $q\leq3$ (since $\Pi_{\mathrm{ajet}} R=R$) and by definition of our class of Hamiltonians $\sum_{i=1}^q  \boldsymbol{\sigma}_i =0$ and $\sum_{i=1}^q  \boldsymbol{\sigma}_i \boldsymbol{\ell}_i=0$. Since $q\leq3$ these three equalities are incompatible for a non-integrable term and thus $\Omega(\xi;0,\boldsymbol{\ell},\boldsymbol{\sigma}) $ can vanish only when $R_{n,0,\boldsymbol{m}}^{\boldsymbol{\ell},\boldsymbol{\sigma}}=0$.

 On the other hand, in view of the hypothesis, we can apply Proposition \ref{prop:smalldiv} (with $\mathfrak{M}=3$ and $d=3$) to obtain that there exits $\mathcal{O}^+\subset \mathcal{O}$ satisfying
 ${\rm Leb}(\mathcal{O}\setminus \mathcal{O}^{+}) \leq \gamma$ such that for all $\xi\in\mathcal{O}^{+}$, $\boldsymbol{k}\in\Z^\S\setminus\{0\}$, $q\leq 3$, $\boldsymbol{\ell}\in (\S^c)^q$ and  $\boldsymbol{\sigma}\in \{-1,1\}^q$ the small divisor $\Omega(\xi;\boldsymbol{k},\boldsymbol{\ell},\boldsymbol{\sigma})$ does not vanish and furthermore
\begin{equation}\label{estim:Om}|\Omega(\xi;\boldsymbol{k},\boldsymbol{\ell},\boldsymbol{\sigma})|\gtrsim_{\mathcal{S}} \Big(\frac{\gamma}{ \| \boldsymbol{k} \|_{\ell^\infty}^{2\# \mathcal{S}} }\Big)^{34}.\end{equation}

Thus the right hand side of equation \eqref{eq:homobof} vanishes and thus the cohomological equation \eqref{eq:cohol} is satisfied. Furthermore using \eqref{estim:Om} and the fact that
$\sum_{j\geq 1}j^pe^{-\a j}\lesssim_p \a^{-(p+1)}$ for all $\a>0$ and all $p\geq0$, we conclude that
 $\chi\in\mathscr{H}_{\eta^+,r,\mathcal{O}^+,\mathcal{S}}$ and 
$$
\|  \chi \|_{\mathscr{H}_{\eta^+,r,\mathcal{O}^+,\mathcal{S}}^{\mathrm{sup}}} \lesssim_\S \frac{1}{(\eta-\eta^+)^{69\# \mathcal{S}}\gamma^{34}}\|R\|_{\mathscr{H}_{\eta,r,\mathcal{O}^+,\mathcal{S}}^{\mathrm{sup}}}.
$$
It remains to consider the Lipschitz norm. We have for $\xi,\zeta\in \mathcal O^{+}$
$$
|\chi_{n,\boldsymbol{k},\boldsymbol{m}}^{\boldsymbol{\ell},\boldsymbol{\sigma}}(\xi)-\chi_{n,\boldsymbol{k},\boldsymbol{m}}^{\boldsymbol{\ell},\boldsymbol{\sigma}}(\zeta)|\leq \frac{|R_{n,\boldsymbol{k},\boldsymbol{m}}^{\boldsymbol{\ell},\boldsymbol{\sigma}}(\xi)-R_{n,\boldsymbol{k},\boldsymbol{m}}^{\boldsymbol{\ell},\boldsymbol{\sigma}}(\zeta)|}{|\Omega(\xi;\boldsymbol{k},\boldsymbol{\ell},\boldsymbol{\sigma})|}+|R_{n,\boldsymbol{k},\boldsymbol{m}}^{\boldsymbol{\ell},\boldsymbol{\sigma}}(\zeta)| \frac{|\Omega(\xi;\boldsymbol{k},\boldsymbol{\ell},\boldsymbol{\sigma})-\Omega(\zeta;\boldsymbol{k},\boldsymbol{\ell},\boldsymbol{\sigma})|}{|\Omega(\xi;\boldsymbol{k},\boldsymbol{\ell},\boldsymbol{\sigma})\Omega(\zeta;\boldsymbol{k},\boldsymbol{\ell},\boldsymbol{\sigma})|}.
$$
The first term is treated as above. To estimate  the second one, we note that by definition
$$|\Omega(\xi;\boldsymbol{k},\boldsymbol{\ell},\boldsymbol{\sigma})-\Omega(\zeta;\boldsymbol{k},\boldsymbol{\ell},\boldsymbol{\sigma})|\leq \|\xi-\zeta\|_{\ell^1}(\| \boldsymbol{k}\|_{\ell^1}+q)(1+2\eps\| L\|_{\mathscr{L}_{\mathcal{O},\mathcal{S}}^{\mathrm{lip}}})
$$
and we get using \eqref{estim:Om}
\begin{align*}
\|  \chi \|_{\mathscr{H}_{\eta^+,r,\mathcal{O}^+,\mathcal{S}}^{\mathrm{lip}}} \lesssim_\S
&\gamma^{-34} \sum_{j\geq 1}j^{68\# \mathcal{S}}e^{-(\eta-\eta^+)j}\|\Pi_{\mathrm{ajet}}R\|_{\mathscr{H}_{\eta,r,\mathcal{O}^+,\mathcal{S}}^{\mathrm{lip}}}\\
&+\gamma^{-68} \sum_{j\geq 1}j^{136\# \mathcal{S}+1}e^{-(\eta-\eta^+)j}\|\Pi_{\mathrm{ajet}}R\|_{\mathscr{H}_{\eta,r,\mathcal{O}^+,\mathcal{S}}^{\mathrm{sup}}}\|L\|_{\mathscr{L}_{\mathcal{O},\mathcal{S}}^{\mathrm{lip}}}.
\end{align*}
Using again $\sum_{j\geq 1}j^pe^{-\a j}\lesssim_p \a^{-(p+1)}$,  we  obtain \eqref{eq:estimchi}.

\endproof

Now we want to solve the nonlinear cohomological equation, i.e. equation of the form \eqref{eq-homo}. First we introduce a new projection:
$$\Pi_{\rm four}=\mathrm{Id} -\Pi_{\mathscr{L}}-\Pi_{\mathscr{A}}-\Pi_{\mathrm{ajet}}=\Pi_{\mathrm{rem}}+\Pi_{\mathscr{Q}} $$
and we note that if $\Pi_{\rm four} R=R$ then the Hamiltonian $R$ starts with terms of degree\footnote{The degree of the monomials $ u_{\boldsymbol{\ell}}^{\boldsymbol{\sigma}} \, y^{\boldsymbol{m}} \, z^{ \boldsymbol{k} } \, \mu^n$ is  $2\|{\boldsymbol{m}}\|_{\ell^1}+\# {\boldsymbol{\ell}}+2n$. } four.

\begin{lemma}\label{lem:homonl}
Let $\mathcal{S} \subset \mathbb{Z}$ be a finite set, $r\in (0,1/4)$, $\mathcal{O} \subset ((4r)^2,1)^{\mathcal{S}}$ be a Borel set, $\eta_{{\rm max}}\geq \eta> \eta_0$, $\varepsilon <1$ and $R\in \mathscr{H}_{\eta,r,\mathcal{O},\mathcal{S}}$ with $\| \Pi_{\rm four} R\|_{\mathscr{H}_{\eta,r,\mathcal{O},\mathcal{S}}^{\mathrm{tot}}}\leq 9r^2$. 

Assume that $L=\Pi_{\mathrm{\mathscr{L}}}R$ satisfies the estimates
\begin{equation}\label{eq:L}
\| L \|_{\mathscr{L}_{r,\mathcal{O},\mathcal{S} }^{\mathrm{tot}}} \vee \| L \|_{\mathscr{L}_{\mathcal{O},\mathcal{S} }^{\mathrm{dir}}} \leq 4
\end{equation}
and assume that $\varepsilon$ satisfies
\begin{equation}\label{eq:eps}
\varepsilon \leq (72)^{-1} \quad \mathrm{and} \quad \varepsilon \big( \sup_{i\in \mathcal{S}} \sum_{k\in \mathcal{S}} \Gamma_{k,i} \big) \leq 10^{-3}. 
\end{equation}
Then for any $\gamma\in(0,1)$ there exists $\mathcal{O}^+\subset \mathcal{O}$ satisfying
 $${\rm Leb}(\mathcal{O}\setminus \mathcal{O}^{+}) \leq \gamma$$
and there exists $\chi\in\mathscr{H}_{\eta^+,r,\mathcal{O}^+,\mathcal{S}}$ (for all $\eta_0<\eta^+<\eta$) with $\Pi_{\mathrm{ajet}} \chi=\chi$  
such that 
\begin{equation} \label{eq:coho}
 \eps \Pi_{\mathrm{ajet}} R(\xi;\cdot)+ \Pi_{\mathrm{ajet}} \{ H^{(0)}_{\varepsilon}(\xi;\cdot)+ \eps( R-\Pi_{\mathrm{ajet}} R)(\xi;\cdot),\chi(\xi;\cdot)\}=0,\quad \xi\in \mathcal O^+.\end{equation}
Furthermore  $\chi$ satisfies the following estimate
\begin{align}\label{eq:estimchinl}
\|  \chi \|_{\mathscr{H}_{\eta^+,r,\mathcal{O}^+,\mathcal{S}}^{\mathrm{tot}}} &\lesssim_\S \frac{\eps}{(\eta-\eta^+)^{848\# \mathcal{S}}\gamma^{408}}\|\Pi_{\mathrm{ajet}}R\|_{\mathscr{H}_{\eta,r,\mathcal{O},\mathcal{S}}^{\mathrm{tot}}}.\end{align}
%\label{eq:tildeZ2nl}
%\|\tilde Z_{2}\|_{\mathscr{H}_{\eta^+,r,\mathcal{O}^+,\mathcal{S}}^{\mathrm{tot}}}&\lesssim_\S \frac{1}{(\eta-\eta^+)^{426\# \mathcal{S}}\gamma^{204}}\|\Pi_{\mathrm{jet}} R\|_{\mathscr{H}_{\eta,r,\mathcal{O},\mathcal{S}}^{\mathrm{tot}}}. \end{align}
%begin{align}\label{eq:estimchinl}
%\|  \chi \|_{\mathscr{H}_{\eta^+,r,\mathcal{O}^+,\mathcal{S}}^{\mathrm{tot}}} &\lesssim_\S \frac{1}{r^6(\eta-\eta^+)^{564\# \mathcal{S}}\gamma^{272}}\|\Pi_{\mathrm{jet}}R\|_{\mathscr{H}_{\eta,r,\mathcal{O},\mathcal{S}}^{\mathrm{tot}}}\| R\|^3_{\mathscr{H}_{\eta,r,\mathcal{O},\mathcal{S}}^{\mathrm{tot}}},\\
%\label{eq:tildeZ2nl}
%\|\tilde Z_{2}\|_{\mathscr{H}_{\eta^+,r,\mathcal{O}^+,\mathcal{S}}^{\mathrm{tot}}}&\lesssim_\S \frac{1}{r^6(\eta-\eta^+)^{426\# \mathcal{S}}\gamma^{204}}\|\Pi_{\mathrm{jet}} R\|_{\mathscr{H}_{\eta,r,\mathcal{O},\mathcal{S}}^{\mathrm{tot}}}\| R\|^3_{\mathscr{H}_{\eta,r,\mathcal{O},\mathcal{S}}^{\mathrm{tot}}}. \end{align}
\end{lemma}
\proof
First, taking into account that $\chi$ commutes with $\mu$ and that, since $\Pi_{\mathrm{ajet}} \chi=\chi$, $(\mathrm{Id}-\Pi_{\mathrm{ajet}})\{Z_2+\eps L,\chi\}=0$, we rewrite \eqref{eq:coho} as follows
\begin{equation}\label{eq:coho2}
\{Z_2+\eps L,\chi\}+\eps \Pi_{\mathrm{ajet}}R+ \Pi_{\mathrm{ajet}}\{Q^{(0)}+\eps\Pi_{\rm four} R,\chi\}=0
\end{equation}
where $L=\Pi_{\mathscr{L}}R$.
Then using Lemma \ref{lem:homo} we can solve successively the following six linear cohomological equations
$$\{ Z_2+\eps L,\chi_0 \}+ \eps\Pi_{\mathrm{ajet}} R= 0,$$
$$\{ Z_2+\eps L,\chi_1 \}+\Pi_{\mathrm{ajet}} R_1=0,\quad R_1=\{ Q^{(0)}+\eps\Pi_{\rm four} R,\chi_0 \},$$
$$\{ Z_2+\eps L,\chi_2 \}+  \Pi_{\mathrm{ajet}} R_2= 0,\quad R_2=\{ Q^{(0)}+\eps\Pi_{\rm four} R,\chi_1 \},$$
$$\{ Z_2+\eps L,\chi_3 \}+  \Pi_{\mathrm{ajet}} R_3= 0,\quad R_3=\{Q^{(0)}+\eps\Pi_{\rm four} R ,\chi_2 \},$$
$$\{ Z_2+\eps L,\chi_4 \}+  \Pi_{\mathrm{ajet}} R_4= 0,\quad R_4=\{Q^{(0)}+\eps\Pi_{\rm four} R ,\chi_3 \},$$
$$\{ Z_2+\eps L,\chi_5 \}+  \Pi_{\mathrm{ajet}} R_5= 0,\quad R_5=\{Q^{(0)}+\eps\Pi_{\rm four} R ,\chi_4 \}.$$

We note that, by construction, $Q^{(0)}$ and $\Pi_{\rm four} R$  start with terms of degree  $4$ while an adapted jet has only monomials of degree less or equal than $11$. Furthermore the Poisson bracket of two monomials  of respective degrees $p$ and $q$ has degree  $p+q-2$. So $R_1$ starts with terms of degree $2$ and the same holds for $\chi_1$. Thus $R_2$ and $\chi_2$ start with terms of degree $4$, $R_3$ and $\chi_3$ start with terms of degree $6$, $R_4$ and $\chi_4$ start with terms of degree $8$ and $R_5$ and $\chi_5$ start with terms of degree $10$.
As a consequence $\Pi_{\mathrm{ajet}}\{ Q^{(0)}+\eps\Pi_{\rm four} R,\chi_5 \}=0$ and therefore, setting 
 $\chi=\chi_0+\chi_1+\chi_2+\chi_3+\chi_4+\chi_5$, we have solved \eqref{eq:coho2} and thus \eqref{eq:coho}. Estimates \eqref{eq:estimchinl}  follows by iterating \eqref{eq:estimchi},  using the Poisson bracket estimate \eqref{estim:poisson1} and $\|\Pi_{\rm four} R\|_{\mathscr{H}_{\eta,r,\mathcal{O},\mathcal{S}}^{\mathrm{tot}}}\leq9r^2$. Namely, denoting $X=\frac{1}{(\eta-\eta^+)^{138\# \mathcal{S}}\gamma^{68}} $, we have for all $\eta_0<\eta^+<\eta\leq \eta_{\rm max}$
 $$\|\chi_0\|_{\mathscr{H}_{\eta^+,r,\mathcal{O}^+,\mathcal{S}}^{\mathrm{tot}}} \lesssim_\S \eps X\|\Pi_{\mathrm{ajet}}R\|_{\mathscr{H}_{\eta,r,\mathcal{O},\mathcal{S}}^{\mathrm{tot}}}$$
 from which we deduce, using \eqref{estim:poisson1} and  $\|Q^{(0)}+\eps\Pi_4 R\|_{\mathscr{H}_{\eta,r,\mathcal{O},\mathcal{S}}^{\mathrm{tot}}}\lesssim r^2$, that for all $\eta_0<\eta^+<\eta\leq \eta_{\rm max}$
  $$\|R_1\|_{\mathscr{H}_{\eta^+,r,\mathcal{O}^+,\mathcal{S}}^{\mathrm{tot}}} \lesssim_\S \eps \frac{X}{(\eta-\eta^+)^4}\|\Pi_{\mathrm{ajet}}R\|_{\mathscr{H}_{\eta,r,\mathcal{O},\mathcal{S}}^{\mathrm{tot}}}.$$
  Iterating this process we get for $i=0,\cdots,5$ and for all $\eta_0<\eta^+<\eta\leq \eta_{\rm max}$ 
 $$\|\chi_i\|_{\mathscr{H}_{\eta^+,r,\mathcal{O}^+,\mathcal{S}}^{\mathrm{tot}}} \lesssim_\S \eps X \left(\frac{X}{(\eta-\eta^+)^4}\right)^{i}\|\Pi_{\mathrm{ajet}}R\|_{\mathscr{H}_{\eta,r,\mathcal{O},\mathcal{S}}^{\mathrm{tot}}}.$$
 Therefore 
  $$\|\chi\|_{\mathscr{H}_{\eta^+,r,\mathcal{O}^+,\mathcal{S}}^{\mathrm{tot}}} \lesssim_\S \eps \frac{X^6}{(\eta-\eta^+)^{20}}\|\Pi_{\mathrm{ajet}}R\|_{\mathscr{H}_{\eta,r,\mathcal{O},\mathcal{S}}^{\mathrm{tot}}}$$
  which leads to \eqref{eq:estimchinl}.
  \endproof

\subsubsection{New remainder term}
We begin with the so called KAM step, i.e.   we estimate $R^{+}$, the new remainder term, given by the formula 
\begin{equation}\label{eq:goalbis}
(H^{(0)}_{\varepsilon}+\varepsilon R)\circ  \Phi^1_{\chi}= (H^{(0)}_{\varepsilon}+\varepsilon R^+)
\end{equation}
 when $\chi$ satisfies the cohomological equation \eqref{eq:coho}.    \begin{lemma}\label{lem:R+}
Under the hypothesis of Lemma \ref{lem:homonl} and assuming that $\chi$, constructed in Lemma \ref{lem:homonl}, satisfies the hypothesis of   Lemma \ref{lem:compo}, we have for all $\eta_0<\eta^+<\eta\leq \eta_{\rm max}$
\begin{equation}\label{eq:estimR+jet}\| \Pi_{\mathrm{ajet}}  R^{+} \|_{\mathscr{H}_{\eta^+,r,\mathcal{O}^+,\mathcal{S}}^{\mathrm{tot}}} \lesssim_\S  \frac{\eps}{r^4\gamma^{816}(\eta-\eta^+)^{1706\sharp\mathcal S}}\| \Pi_{\mathrm{ajet}}  R \|^2_{\mathscr{H}_{\eta,r,\mathcal{O},\mathcal{S}}^{\mathrm{tot}}}\end{equation}
and
\begin{equation}\label{eq:estimR+} \| R- R^{+} \|_{\mathscr{H}_{\eta^+,r,\mathcal{O}^+,\mathcal{S}}^{\mathrm{tot}}}\lesssim_\S \frac{1}{r^4\gamma^{816}(\eta-\eta^+)^{1706\sharp\mathcal S}}\| \Pi_{\mathrm{ajet}}  R \|_{\mathscr{H}_{\eta,r,\mathcal{O},\mathcal{S}}^{\mathrm{tot}}} .\end{equation}
\end{lemma}
\proof 
We first introduce a local notation: for  $\chi$ satisfying hypothesis of Lemma \ref{lem:compo} and $F$ a Hamiltonian defined on $\mathcal A_\xi(2r)$ we denote
 $$
 \mathcal T_\chi(F)=F\circ \Phi^1_\chi-F-\{F,\chi\}.
 $$
We  have 
$$
(H^{(0)}_{\varepsilon}+\varepsilon R)\circ  \Phi^1_{\chi}=Z_2+A^{(0)}+\{Z_2+\eps L,\chi\}+  \mathcal T_\chi(Z_2)
+Q^{(0)}+\eps R+ \{ Q^{(0)}+\eps (R-L),\chi\} +\mathcal T_\chi(Q^{(0)}+\eps R)
$$
and thus using \eqref{eq:coho2} we get
\begin{multline*}
 (H^{(0)}_{\varepsilon}+\varepsilon R)\circ  \Phi^1_{\chi}=H^{(0)}_{\varepsilon}+\eps(\mathrm{Id}-\Pi_{\mathrm{ajet}})R+(\mathrm{Id}-\Pi_{\mathrm{ajet}})\{Q^{(0)}+\eps \Pi_{\rm four}R,\chi\}+\eps\{\Pi_{\mathrm{ajet}}R,\chi\}\\ + \mathcal T_\chi(Z_2+Q^{(0)}+\eps R).
 \end{multline*}
 Therefore comparing with \eqref{eq:goalbis} we obtain an explicit expression for $R^{+}$:
 \begin{equation}\label{eq:R+}
 \eps R^{+}=\eps(\mathrm{Id}-\Pi_{\mathrm{ajet}})R+(\mathrm{Id}-\Pi_{\mathrm{ajet}})\{Q^{(0)}+\eps \Pi_{\rm four}R,\chi\}+\eps\{\Pi_{\mathrm{ajet}}R,\chi\}+ \mathcal T_\chi(Z_2+Q^{(0)}+\eps R).
 \end{equation}
Furthermore we have 
 \begin{equation}\label{eq:R+jet}\eps\Pi_{\mathrm{ajet}}R^{+}=\Pi_{\mathrm{ajet}}\big(\eps\{\Pi_{\mathrm{ajet}}R,\chi\}+ \mathcal T_\chi(Z_2+Q^{(0)}+\eps R)\big).\end{equation}
In view of \eqref{eq:R+jet} we have
\begin{align*}\eps\| \Pi_{\mathrm{ajet}}  R^{+} \|_{\mathscr{H}_{\eta^+,r,\mathcal{O}^+,\mathcal{S}}^{\mathrm{tot}}}\leq& 
\eps \|\{\Pi_{\mathrm{ajet}}R,\chi\}\|_{\mathscr{H}_{\eta^+,r,\mathcal{O}^+,\mathcal{S}}^{\mathrm{tot}}}+\| \mathcal T_\chi(Z_2+\eps L) \|_{\mathscr{H}_{\eta^+,r,\mathcal{O}^+,\mathcal{S}}^{\mathrm{tot}}}\\
&+\| \mathcal T_\chi(Q^{(0)}) \|_{\mathscr{H}_{\eta^+,r,\mathcal{O}^+,\mathcal{S}}^{\mathrm{tot}}}
+\| \mathcal T_\chi(\eps (R-L)) \|_{\mathscr{H}_{\eta^+,r,\mathcal{O}^+,\mathcal{S}}^{\mathrm{tot}}}
\end{align*}
and it remains to estimate the four terms on the right hand side.
Using the bound on $\chi$ \eqref{eq:estimchinl} and the Poisson bracket estimate \eqref{estim:poisson1} we get
$$ \|\{\Pi_{\mathrm{ajet}}R,\chi\}\|_{\mathscr{H}_{\eta^+,r,\mathcal{O}^+,\mathcal{S}}^{\mathrm{tot}}}\lesssim_\S \frac{\eps}{r^2(\eta-\eta^+)^{848\sharp\S+4}\gamma^{408}}\| \Pi_{\mathrm{ajet}}  R \|^2_{\mathscr{H}_{\eta,r,\mathcal{O},\mathcal{S}}^{\mathrm{tot}}}.$$
By Lemma \ref{lem:compo} and  \eqref{eq:estimchinl} we have 
$$\| \mathcal T_\chi(\eps (R-L)) \|_{\mathscr{H}_{\eta^+,r,\mathcal{O}^+,\mathcal{S}}^{\mathrm{tot}}}\lesssim_\S\frac{\eps}{r^4(\eta-\eta^+)^{10}}\left(\frac{\eps}{(\eta-\eta^+)^{848\sharp\S}\gamma^{408}} \right)^2
\| \Pi_{\mathrm{ajet}}  R \|^2_{\mathscr{H}_{\eta,r,\mathcal{O},\mathcal{S}}^{\mathrm{tot}}}\|  R \|_{\mathscr{H}_{\eta,r,\mathcal{O},\mathcal{S}}^{\mathrm{tot}}}.
$$
Similarly using that $\|  Q^{(0)} \|_{\mathscr{H}_{\eta,r,\mathcal{O},\mathcal{S}}^{\mathrm{tot}}}\lesssim_\S r^4$ we get
$$\| \mathcal T_\chi(Q^{(0)}) \|_{\mathscr{H}_{\eta^+,r,\mathcal{O}^+,\mathcal{S}}^{\mathrm{tot}}}\lesssim_\S
\frac{1}{(\eta-\eta^+)^{10}}\left(\frac{\eps}{(\eta-\eta^+)^{848\sharp\S}\gamma^{408}} \right)^2
\| \Pi_{\mathrm{ajet}}  R \|^2_{\mathscr{H}_{\eta,r,\mathcal{O},\mathcal{S}}^{\mathrm{tot}}}.
$$
Thus the sum of the last two terms is less than 
$$\frac{\eps^2}{r^4\gamma^{816}(\eta-\eta^+)^{1706\sharp\mathcal S}}\| \Pi_{\mathrm{ajet}}  R \|^2_{\mathscr{H}_{\eta,r,\mathcal{O},\mathcal{S}}^{\mathrm{tot}}}.$$
It remains to estimate $\mathcal T_\chi(Z_2+\eps L)$. By Taylor expansion we have
$$\mathcal T_\chi(Z_2+\eps L)= \int_0^1(1-t)\{\{(Z_2+\eps L,\chi\},\chi\}\circ \Phi_\chi^t \dd t$$
and thus using Lemma \ref{lem:compo} and \eqref{estim:poisson1} we get
$$\| \mathcal T_\chi(Z_2+\eps L) \|_{\mathscr{H}_{\eta^+,r,\mathcal{O}^+,\mathcal{S}}^{\mathrm{tot}}}
\lesssim_\S \frac1{r^2(\eta-\eta^+)^5}\| \{Z_2+\eps L,\chi\}\|_{\mathscr{H}_{\eta-\frac{\eta-\eta^+}2,r,\mathcal{O}^+,\mathcal{S}}^{\mathrm{tot}}} \|\chi\|_{\mathscr{H}_{\eta-\frac{\eta-\eta^+}2,r,\mathcal{O}^+,\mathcal{S}}^{\mathrm{tot}}}
$$
Then we replace  $\{Z_2+\eps L,\chi\}$ by the expression given in \eqref{eq:coho2} and we use \eqref{eq:estimchinl} and \eqref{estim:poisson1} to get, first that
\begin{equation}\label{OK}\| \{Q^{(0)}+\eps \Pi_{\rm four}R,\chi\}\|_{\mathscr{H}_{\eta^+,r,\mathcal{O}^+,\mathcal{S}}^{\mathrm{tot}}}\lesssim_\S \frac{\eps}{(\eta-\eta^+)^{848\# \mathcal{S}+4}\gamma^{408}}\|\Pi_{\mathrm{ajet}} R \|_{\mathscr{H}_{\eta,r,\mathcal{O}^+,\mathcal{S}}^{\mathrm{tot}}},\end{equation}
and then
$$\| \mathcal T_\chi(Z_2+\eps L) \|_{\mathscr{H}_{\eta^+,r,\mathcal{O}^+,\mathcal{S}}^{\mathrm{tot}}}\lesssim_\S \frac{\eps^2}{r^2\gamma^{816}(\eta-\eta^+)^{1706\sharp\mathcal S}}\| \Pi_{\mathrm{ajet}}  R \|^2_{\mathscr{H}_{\eta,r,\mathcal{O},\mathcal{S}}^{\mathrm{tot}}}.$$
All together we obtain \eqref{eq:estimR+jet}.

Now we estimate $  R- R^{+}$ using \eqref{eq:R+}. Since we already estimated the last two terms of \eqref{eq:R+} to obtain the estimate \eqref{eq:estimR+jet}, it remains to estimate 
$$-\eps\Pi_{\mathrm{ajet}}R+(\mathrm{Id}-\Pi_{\mathrm{ajet}})\{Q^{(0)}+\eps \Pi_{\rm four}R,\chi\}.$$
So using \eqref{OK} we get
\begin{align*} \| \eps(R- R^{+}) \|_{\mathscr{H}_{\eta^+,r,\mathcal{O}^+,\mathcal{S}}^{\mathrm{tot}}}\lesssim_\S &\eps \|\Pi_{\mathrm{ajet}} R \|_{\mathscr{H}_{\eta,r,\mathcal{O}^+,\mathcal{S}}^{\mathrm{tot}}}+ \frac{\eps}{(\eta-\eta^+)^{848\# \mathcal{S}+4}\gamma^{408}}\|\Pi_{\mathrm{ajet}} R \|_{\mathscr{H}_{\eta,r,\mathcal{O}^+,\mathcal{S}}^{\mathrm{tot}}}\\
&+\frac{\eps^2}{r^4\gamma^{816}(\eta-\eta^+)^{1706\sharp\mathcal S}}\| \Pi_{\mathrm{ajet}}  R \|^2_{\mathscr{H}_{\eta,r,\mathcal{O},\mathcal{S}}^{\mathrm{tot}}}\\
\lesssim_\S& \frac{\eps}{r^4\gamma^{816}(\eta-\eta^+)^{1706\sharp\mathcal S}}\| \Pi_{\mathrm{ajet}}  R \|_{\mathscr{H}_{\eta,r,\mathcal{O},\mathcal{S}}^{\mathrm{tot}}}
\end{align*}
\endproof

\subsection{Proof}
As we explain, our strategy of proof consists in eliminating  $\Pi_{\rm ajet}P$ iteratively. First we choose a set of parameters in order to realize the iteration equation \eqref{eq:goalbis}.
We set 
$$\eta_0=\eta^\natural \text{ and } \eta_k-\eta_{k+1}=\frac{\eta^\natural-\eta}{2^{k+1}} \text{ for }k\geq0,$$
$$\gamma=r^2\ \eps^{10^{-3}} \text{ and } \gamma_k=\frac{\gamma}{2^{k+1}} \text{ for }k\geq0,$$
$$\nu_0=r^{4000} \text{ and } \nu_{k+1}=\nu_k^{\frac43} \text{ for }k\geq0.$$
Then we set $K_0=P$ and $\mathcal O_0=\mathcal O^\natural$. 
\subsubsection{Iterative lemma}
\begin{lemma}[Iterative lemma]\label{lem:ite}
Under the hypothesis of Theorem \ref{thm:KAM}, for all $k\geq0$ there exist
\begin{itemize}
\item $\mathcal O_{k}\subset \mathcal O^\natural$ a Borelian set,
\item $K_k\in \mathscr{H}_{\eta_k,r,\mathcal{O}_k,\mathcal{S}}$, 
\item $\chi_k\in \mathscr{H}_{\eta_{k+1},r,\mathcal{O}_{k+1},\mathcal{S}}$
\end{itemize}
such that 
$$ (H^{(0)}_{\varepsilon}+\varepsilon K_k)\circ  \Phi^1_{\chi_k}= (H^{(0)}_{\varepsilon}+\varepsilon K_{k+1})$$
 Furthermore for all $k\geq0$ we have
\begin{itemize}
\item $\mathcal O_{k+1}\subset \mathcal O_k$ and ${\rm Leb}(\mathcal O_k\setminus \mathcal O_{k+1})\leq \gamma_k$
\item $\|\Pi_{\mathrm{ajet}} K_k\|_{\mathscr{H}_{\eta_k,r,\mathcal{O}_k,\mathcal{S}}^{\mathrm{tot}}}\leq \eps^{\frac9{10}} \nu_k$ 
\item $\| K_{k+1}-K_k\|_{\mathscr{H}_{\eta_{k+1},r,\mathcal{O}_{k+1},\mathcal{S}}^{\mathrm{tot}}}\leq \nu_{k}^{\frac14}$
\item $\|\chi_k\|_{\mathscr{H}_{\eta_{k+1},r,\mathcal{O}_{k+1},\mathcal{S}}^{\mathrm{tot}}}\leq \eps\nu_k^{\frac14}$
\end{itemize}
\proof
The proof is done by recurrence.
\medskip

\noindent $\bullet$ \underline{\emph{Initialization step.}} We just have to verify that $\|\Pi_{\mathrm{ajet}} K_0\|_{\mathscr{H}_{\eta_0,r,\mathcal{O}_0,\mathcal{S}}^{\mathrm{tot}}}\leq \eps^{\frac9{10}} \nu_0$  which is clear since $K_0=P$ and $P$ satisfies the hypothesis of Theorem \ref{thm:KAM}.
We also note that $\Pi_{\mathscr L}K_0$ which satisfies \eqref{eq:L} and that hypothesis \eqref{eq:eps} is part of the hypothesis of Theorem \ref{thm:KAM}. 

\medskip

\noindent $\bullet$ \underline{\emph{Iteration step.}}
 Let $k\geq0$, we assume that we have construct $\mathcal O_i$, $K_i$ and $\chi_{i-1}$ up to $i=k$ (by convention we take $\chi_{-1}=0$) and that the estimate of the Lemma are satisfied for these objects. Now we want to construct $\mathcal O_{k+1}$, $K_{k+1}$ and $\chi_{k}$. 
 %We follows the same line as in the initialization step.
First we note that
  $$\| K_k-K_0\|_{\mathscr{H}_{\eta_k,r,\mathcal{O}_k,\mathcal{S}}^{\mathrm{tot}}}\leq \sum_{i=0}^{k-1}   \| K_i-K_{i+1}\|_{\mathscr{H}_{\eta_i,r,\mathcal{O}_i,\mathcal{S}}^{\mathrm{tot}}}\leq \sum_{i=0}^{k-1} \nu_i^{\frac14}$$
  and thus $\| K_k-K_0\|_{\mathscr{H}_{\eta_k,r,\mathcal{O}_k,\mathcal{S}}^{\mathrm{tot}}}\leq \nu_0^{\frac15}$. Therefore, since $ \Pi_{\mathscr L}K_0$ satisfies an estimate better than \eqref{eq:L} (by assumption
$\| \Pi_{\mathscr L}K_0 \|_{\mathscr{L}_{r,\mathcal{O},\mathcal{S} }^{\mathrm{tot}}} \vee \| \Pi_{\mathscr L}K_0 \|_{\mathscr{L}_{\mathcal{O},\mathcal{S} }^{\mathrm{dir}}} \leq  \| \Pi_{\mathrm{nor}} P\|_{\mathscr{N}_{r,\mathcal{O}^\natural,\mathcal{S}}^{\mathrm{tot}}} \leq 7/2  $, see \eqref{eq:small_assump_P_KAM}), we have that $\Pi_{\mathscr L}K_k$  also satisfies  \eqref{eq:L}. Thus, by Lemma \ref{lem:homonl} we can construct a set $\mathcal O_{k+1}\subset \mathcal O_k$ satisfying $\text{Leb}(\mathcal O_k\setminus \mathcal O_{k+1})\leq \gamma_k$ and an adapted jet $\chi_k\in \mathscr{H}_{\eta_{k+1},r,\mathcal{O}_{k+1},\mathcal{S}} $ solving the homological equation 
$$ \eps \Pi_{\mathrm{ajet}} K_k+ \Pi_{\mathrm{ajet}} \{ H^{(0)}_{\varepsilon}+ \eps( K_k-\Pi_{\mathrm{ajet}} K_k),\chi_{k}\}=0$$
and satisfying 
$$\|\chi_k\|_{\mathscr{H}_{\eta_{k+1},r,\mathcal{O}_{k+1},\mathcal{S}}^{\mathrm{tot}}}\lesssim_\S 
 \frac{\eps}{(\eta_k-\eta_{k+1})^{848\# \mathcal{S}}\gamma_k^{408}}\eps^{\frac9{10}} \nu_k.$$
Using the definition of $\eta_k$ and $\gamma_k$ we get
$$\|\chi_k\|_{\mathscr{H}_{\eta_{k+1},r,\mathcal{O}_{k+1},\mathcal{S}}^{\mathrm{tot}}}\lesssim_\S 
 \frac{\eps 2^{(k+1)(848\# \mathcal{S}+408)}(r^{4000})^{\frac34(\frac43)^k}}{(\eta^\natural-\eta)^{848\# \mathcal{S}}(r^2\eps^{10^{-3}})^{408}}\eps^{\frac9{10}} \nu_k^{\frac14}.$$
Thus by imposing a condition of type $r\lesssim_{\S,\eta^\natural-\eta}1$ we obtain, uniformly\footnote{i.e. one condition on $r$ suffices to obtain the estimate for all $k$.} in $k$,
$$\|\chi_k\|_{\mathscr{H}_{\eta_{k+1},r,\mathcal{O}_{k+1},\mathcal{S}}^{\mathrm{tot}}}\leq\eps \nu_k^{\frac14}.$$
Then we construct $K_{k+1}$ by the formula \eqref{eq:R+} with $R=K_k$ and we get that \eqref{eq:goal} is satisfied at rank $k$. Furthermore by Lemma \ref{lem:R+} (note that hypothesis of Lemma \ref{lem:compo} are clearly satisfied for $r$ small enough uniformly in $k$)
$K_{k+1}\in \mathscr{H}_{\eta_{k+1},r,\mathcal{O}_{k+1},\mathcal{S}}$ and 
$$\| \Pi_{\mathrm{ajet}}  K_{k+1} \|_{\mathscr{H}_{\eta_{k+1},r,\mathcal{O}_{k+1},\mathcal{S}}^{\mathrm{tot}}} \lesssim_\S  \frac{\eps 2^{(k+1)(1706\sharp\mathcal S+816)}}{r^4(r^2\eps^{10^{-3}})^{816}(\eta^\natural-\eta)^{1706\sharp\mathcal S}}\eps^{\frac{18}{10}}\nu_k^2.$$
Thus by imposing a condition of type $r\lesssim_{\S,\eta^\natural-\eta}1$ we obtain, uniformly in $k$,
$$\| \Pi_{\mathrm{ajet}}  K_{k+1} \|_{\mathscr{H}_{\eta_{k+1},r,\mathcal{O}_{k+1},\mathcal{S}}^{\mathrm{tot}}}\leq \eps^{\frac{18}{10}}\nu_{k}^{\frac43}\leq\eps^{\frac9{10}}\nu_{k+1}.$$
Similarly, using \eqref{eq:estimR+}, we prove $\|  K_{k+1}-K_k \|_{\mathscr{H}_{\eta_{k+1},r,\mathcal{O}_{k+1},\mathcal{S}}^{\mathrm{tot}}}\leq \nu_k^{\frac14}$.

\endproof

\end{lemma}

\subsubsection{Transition to the limit and end of the proof.}
We set $\mathcal O=\cap_{k\geq0}\mathcal O_k$ by Lemma \ref{lem:ite} we have
$$\text{Leb}(\mathcal O^\natural\setminus\mathcal O)\leq\sum_{k\geq0} \text{Leb}(\mathcal O_k\setminus \mathcal O_{k+1})=\sum_{k\geq0}\gamma_k=\gamma=r^2\eps^{10^{-3}}.$$
Using Lemma \ref{lem:flow} and Lemma \ref{lem:ite}  we get that $\Phi_{\chi_k}^1$ is a $C^1$ symplectic change of variable mapping $\mathcal{A}_\xi(\rho_{k})$ into $\mathcal{A}_\xi(\rho_{k+1})$ where $\rho_k=r+\sum_{i=0}^{k-1}r^{-1}\nu_i^{\frac14}\leq 2r$ (for $r$ small enough).

Then for $\xi\in\mathcal O$ and all $k\geq 0$ we set 
$$\tau_\xi^{(k)}=\Phi_{\chi_0}^1\circ\cdots \circ \Phi_{\chi_{k}}^1.$$
By construction
 $\tau_\xi^{(k)}:\mathcal{A}_\xi(r)\to \mathcal{A}_\xi(2r)$ is a  $C^1$ symplectic map commuting with the gauge transform. 
By Lemma \ref{lem:appendix} (an adaptation of the mean value theorem for annulus)  we have for all $\ell\geq0$ and all $u\in \mathcal{A}_\xi(2r)$ 
 $$\|\tau_\xi^{(\ell+1)}(u)-\tau_\xi^{(\ell)}(u)\|_{\ell^2_\varpi} \lesssim_\S \sup_{w\in \mathcal{A}_\xi(2r)}\| \dd \tau_\xi^{(\ell)}(w)\|_{\ell^2_\varpi\to\ell^2_\varpi}\| \Phi_{\chi_{\ell+1}}^1(u)-u\|_{\ell^2_\varpi}.
 $$
Thus, using Lemma \ref{lem:flow}, we get
 $$\|\tau_\xi^{(\ell+1)}(u)-\tau_\xi^{(\ell)}(u)\|_{\ell^2_\varpi} \lesssim_\S r^{-2}\|\chi_{\ell+1}\|_{\mathscr{H}_{\eta_{\ell+1},r,\mathcal{O}_{\ell+1},\mathcal{S}}^{\mathrm{tot}}}.$$
 Similarly, using the control we have on $ \dd^2 \tau_\xi^{(\ell)}$  by Lemma \ref{lem:flow} , we obtain 
  $$\|\dd\tau_\xi^{(\ell+1)}(u)-\dd\tau_\xi^{(\ell)}(u)\|_{\ell^2_\varpi} \lesssim_\S r^{-6}\|\chi_{\ell+1}\|_{\mathscr{H}_{\eta_{\ell+1},r,\mathcal{O}_{\ell+1},\mathcal{S}}^{\mathrm{tot}}}.$$
 Therefore, for all $u\in \mathcal{A}_\xi(2r)$ and all $0\leq\ell\leq k$, we get
 $$
\| \tau_\xi^{(k)}(u) - \tau_\xi^{(\ell)} (u)\|_{\ell^2_\varpi} + r^2 \| \mathrm{d} \tau_\xi^{(k)}(u) - \mathrm{d} \tau_\xi^{(\ell)}(u)\|_{\ell^2_\varpi \to \ell^2_\varpi } \lesssim_\S r^{-4}\sum_{i=\ell}^k\|\chi_i\|_{\mathscr{H}_{\eta_{i+1},r,\mathcal{O}_{i+1},\mathcal{S}}^{\mathrm{tot}}}\leq
r^{-4}\sum_{i=\ell}^k\eps\nu_i^{\frac14}
$$
which leads to
\begin{equation}
\label{eq:tauk}
\| \tau_\xi^{(k)}(u) - \tau_\xi^{(\ell)} (u)\|_{\ell^2_\varpi} + r^2 \| \mathrm{d} \tau_\xi^{(k)}(u) - \mathrm{d} \tau_\xi^{(\ell)}(u)\|_{\ell^2_\varpi \to \ell^2_\varpi } \leq \eps^{\frac12}\nu_\ell^{\frac15}
\end{equation}
for $r\lesssim_\S 1$ small enough. As a consequence $(\tau_\xi^{(k)})_{k\geq0}$ is a Cauchy sequence in $C^1(\mathcal{A}_\xi(r), \ell^2_\varpi)$ and thus converges in this Banach space to an element that we denote  $\tau_\xi$. By construction $\tau_\xi:\mathcal{A}_\xi(r)\to \mathcal{A}_\xi(2r)$ is  a $C^1$ symplectic map commuting with the gauge transform and taking $\ell=0$ and the limit $k\to+\infty$ in \eqref{eq:tauk} we get
$$
\| \tau_\xi(u) - \text{Id}\|_{\ell^2_\varpi} + r^2 \| \mathrm{d} \tau_\xi(u) - \mathrm{Id}\|_{\ell^2_\varpi \to \ell^2_\varpi } \leq \eps^{\frac12}\nu_0^{\frac15}\leq \eps^{\frac12}r^3.$$
Using again Lemma \ref{lem:ite} we have for $0\leq\ell\leq k$
$$\| K_{k}-K_\ell\|_{\mathscr{H}_{\eta,r,\mathcal{O},\mathcal{S}}^{\mathrm{tot}}}
\leq\sum_{i=\ell}^k\nu_{i}^{\frac14}\leq \nu_\ell^{\frac15}$$
for $r$ small enough. Therefore $(K_k)_{k\geq0}$ is a Cauchy sequence and thus converges in $\mathscr{H}_{\eta,r,\mathcal{O},\mathcal{S}}$. We call $K$ its limit. Passing to the limit in the previous estimate leads to 
$$\| K-P\|_{\mathscr{H}_{\eta,r,\mathcal{O},\mathcal{S}}^{\mathrm{tot}}}
\leq \nu_0^{\frac15}$$
from which we deduce
$$\|  \Pi_{\mathrm{rem}} K\|_{\mathscr{H}_{\eta,r,\mathcal{O},\mathcal{S}}^{\mathrm{tot}}} \leq 1  \quad \mathrm{and} \quad \|  \Pi_{\mathrm{nor}} (K-P)\|_{\mathscr{N}_{r,\mathcal{O},\mathcal{S}}^{\mathrm{tot}}} \leq r .
$$
On the other hand, since by Lemma \ref{lem:ite} $\|\Pi_{\mathrm{ajet}} K_k\|_{\mathscr{H}_{\eta,r,\mathcal{O},\mathcal{S}}^{\mathrm{tot}}}\leq \nu_k^{\frac43}$, we get $\Pi_{\mathrm{ajet}} K=0$. 

Finally by construction (see \eqref{eq:goal} and \eqref{eq:tauk})
$$(H^{(0)}_{\varepsilon}+\varepsilon P)\circ  \tau_\xi^{(k)}= (H^{(0)}_{\varepsilon}+\varepsilon K_{k})$$
we deduce, passing to the limit, that for all $\xi\in \mathcal O$ and all $u\in\mathcal{A}_\xi(r)$
$$ (H^{(0)}_{\varepsilon}+\varepsilon P)(\xi;  \tau_\xi(u))= (H^{(0)}_{\varepsilon}+\varepsilon K)(\xi;u)$$
which ends the proof of Theorem \ref{thm:KAM}.

\section{A Birkhoff normal form theorem}\label{sec:Birk}
In this section we prove a Birkhoff normal form theorem that will be useful to prepare our Hamiltonian perturbation before each opening step. Roughly speaking, without this Birkhoff step, after an opening step the adapted jet (see Definition \ref{def:jet}) would be too big to apply our KAM theorem. So we prepare the Hamiltonian by making much smaller the terms that are linked to the new site (corresponding to the index $j_*$ in Definition \ref{def:tbr}) and that will be part of the adapted jet after opening. 
In fact we are going to make smaller a larger part of the Hamiltonian perturbation according to the following definition of "terms to be removed" (where we leave free the index $j^*$ in $\mathcal{S}^c$):

\begin{definition}[Terms to be removed]\label{def:tbr} Let $\mathcal{S} \subset \mathbb{Z}$ be a finite set, $r\in(0,1/4)$, $\mathcal{O}\subset ((4r)^2,1)^{\mathcal{S}}$, $\eta \geq 0$ and $\mathfrak{M}\geq 2$.
A Hamiltonian $P \in \mathscr{H}_{\eta,r,\mathcal{O},\mathcal{S}}$ is a term to be removed if $\Pi_{\mathrm{int}}P=0$ and its non zero coefficients $P_{n,\boldsymbol{k},\boldsymbol{m}}^{\boldsymbol{\ell},\boldsymbol{\sigma}}\neq 0$ satisfy 
\begin{equation}
\label{eq:original}
  q:=\#\boldsymbol{\ell}  \leq \mathfrak{M}, \quad \# \{ \boldsymbol{\ell}_j \ | \ 1\leq j \leq q\} \leq 4 \quad \mathrm{and}  \quad \exists j_* \in \mathcal{S}^c, \ \# \boldsymbol{\ell} ^{-1}(\mathcal{S}^c \setminus \{ j_*\} ) \leq 3.
\end{equation}
 We denote by $\Pi_{\mathrm{tbr}}^{\mathfrak{M}}$ the associated projection defined by restriction of the coefficients.
\end{definition}

In the following lemma we prove that all these terms are either non resonant (for the linear equation) or depend on the angles $(\theta_k)_{k\in \mathcal{S}}$. It will be crucial to get non resonance estimates by  Proposition \ref{prop:smalldiv}.
\begin{lemma} \label{lem:critical}
Let $\mathcal{S} \subset \mathbb{Z}$ be a finite set and $\mathfrak{M}\geq 2$. If $n\in \mathbb{N}$, $\boldsymbol{k} \in \mathbb{Z}^{\mathcal{S}}$, $\boldsymbol{m}\in \mathbb{N}^{\mathcal{S}}$, $q\geq 0$, $(\boldsymbol{\ell}\in \mathcal{S}^c)^q$, $\boldsymbol{\sigma}\in \{-1,1\}^q$ are indices associated with a non-zero coefficient of a term to be removed\footnote{i.e. they do not satisfy the integrability condition \eqref{eq:integ} but they satisfy \eqref{eq:original} and $ \sum_{j\in \mathcal{S}} j^\alpha\boldsymbol{k}_j + \sum_{i=1}^q \sigma_i \ell_i^\alpha= 0$ for $\alpha \in \{0,1\}$}, then either $\boldsymbol{k} \neq 0$ or $\boldsymbol{\sigma}_1 \boldsymbol{\ell}_1^2+ \cdots+\boldsymbol{\sigma}_q \boldsymbol{\ell}_q^2 \neq 0$.
\end{lemma}
\begin{proof} By contradiction assume that $\boldsymbol{k} = 0$ and $\boldsymbol{\sigma}_1 \boldsymbol{\ell}_1^2+ \cdots+\boldsymbol{\sigma}_q \boldsymbol{\ell}_q^2=0$. We aim at proving that $u_{\boldsymbol{\ell}}^{\boldsymbol{\sigma}}$ is integrable which is impossible by definition of $\Pi_{\mathrm{tbr}}^{\mathfrak{M}}$. Thanks to the two last assumptions of \eqref{eq:original}, we assume without loss of generality (it would suffices to re-order the indices) that there exists $a\leq 3$ such that
$$
\forall j>  a, \quad \boldsymbol{\ell}_j = \boldsymbol{\ell}_{a+1} \quad \mathrm{and} \quad \forall j\leq a, \quad \boldsymbol{\ell}_j \neq  \boldsymbol{\ell}_{a+1}.
$$
We set $p = \sigma_{a+1}+\cdots+ \sigma_q$. Without loss of generality, we assume that $p\geq 0$. We note that 
$$
u_{\boldsymbol{\ell}}^{\boldsymbol{\sigma}} = u_{\boldsymbol{\ell}_1}^{\boldsymbol{\sigma}_1}  \cdots u_{\boldsymbol{\ell}_a}^{\boldsymbol{\sigma}_a}  u_{\boldsymbol{\ell}_{a+1}}^{p}  |u_{\boldsymbol{\ell_{a+1}}}|^{q-a-p} 
$$
and that $q-a-p$ is an even number. Therefore is suffices to prove that $ u_{\boldsymbol{\ell}_1}^{\boldsymbol{\sigma}_1}  \cdots u_{\boldsymbol{\ell}_a}^{\boldsymbol{\sigma}_a}  u_{\boldsymbol{\ell}_{a+1}}^{p}$ is integrable.

 Then, using the preservation of the mass and the momentum, we have
$$
\forall \alpha \in \{0,1,2\}, \quad  \boldsymbol{\sigma}_1 \boldsymbol{\ell}_1^\alpha + \cdots + \boldsymbol{\sigma}_{a} \boldsymbol{\ell}_{a}^\alpha + p  \boldsymbol{\ell}_{a+1}^\alpha =0.
$$

By preservation of the mass (i.e. $\alpha =0$), we deduce that $p\leq 3$ and so that there are only few cases to consider.
\begin{itemize}
\item \emph{Case $p=0$.} By preservation of the mass, we deduce that $a=0$ or $a=2$ and $\boldsymbol{\sigma}_1 = -\boldsymbol{\sigma}_2$. Moreover by preservation of the momentum, if  $a=2$  then $\boldsymbol{\ell}_1 = \boldsymbol{\ell}_2$. In any case, it follows that $ u_{\boldsymbol{\ell}_1}^{\boldsymbol{\sigma}_1}  \cdots u_{\boldsymbol{\ell}_a}^{\boldsymbol{\sigma}_a}  u_{\boldsymbol{\ell}_{a+1}}^{p}$ is integrable.
\item \emph{Case $p=1$.} By preservation of the mass, either $a=1$ or $a=3$. If $a=1$, by preservation of the momentum, we get that  $ u_{\boldsymbol{\ell}_1}^{\boldsymbol{\sigma}_1}  \cdots u_{\boldsymbol{\ell}_a}^{\boldsymbol{\sigma}_a}  u_{\boldsymbol{\ell}_{a+1}}^{p}$ is integrable. Now, if $a=3$, we can assume that (without loss of generality)   $\boldsymbol{\sigma}_1=\boldsymbol{\sigma}_2= 1$ and $\boldsymbol{\sigma}_3 = -1$. As a consequence, using the cases $\alpha \in \{1,2\}$, we have
$$
0= \boldsymbol{\sigma}_1 \boldsymbol{\ell}_1^2 + \cdots + \boldsymbol{\sigma}_{3} \boldsymbol{\ell}_{3}^2 + p  \boldsymbol{\ell}_{4}^2 =  2(\boldsymbol{\ell}_{4}- \boldsymbol{\ell}_{1})(\boldsymbol{\ell}_{4}- \boldsymbol{\ell}_{2})
$$
which is impossible by definition of $a$ (we know that $\boldsymbol{\ell}_{4} \notin\{ \boldsymbol{\ell}_{1},\boldsymbol{\ell}_{2}\}$).
\item \emph{Case $p\in \{2,3\}$.} By preservation of the mass we deduce that $a=p$ and $\boldsymbol{\sigma}_1 = \cdots =\boldsymbol{\sigma}_a = -1$. Moreover using the cases $\alpha \in \{1,2\}$, we get that
$$
a \boldsymbol{\ell}_{a+1}^2 = a\big(\frac{\boldsymbol{\ell}_1+\cdots +\boldsymbol{\ell}_a}a \big)^2 = \boldsymbol{\ell}_1^2+\cdots +\boldsymbol{\ell}_a^2.
$$
By strict convexity of the square function, it implies that $\boldsymbol{\ell}_1=\cdots = \boldsymbol{\ell}_{a+1}$ which is impossible by definition of $a$.

\end{itemize}
\end{proof}

\begin{theorem}[\`A la Birkhoff] \label{thm:Birk} Let $\mathcal{S}^\flat \subset \mathbb{Z}$ be a finite set, $\mathfrak{N},\mathfrak{M} \geq 2$, $\rho\in (0,1/4)$, $\mathcal{O}^\flat \subset ( (4\rho)^2,1)^{\mathcal{S}^\flat}$ be a Borel set, $\eta_{\max}\geq \eta^\flat> \eta^\natural \geq \eta_0$, $\varepsilon<1$ and $K^\flat\in \mathscr{H}_{\eta^\flat,\rho,\mathcal{O}^\flat,\mathcal{S}^\flat}$.

\medskip

\noindent Assume that the following assumptions hold
\begin{itemize}
\item $K^\flat$ is not too large, i.e. $\| \Pi_{\mathrm{nor}} K^\flat\|_{\mathscr{N}_{\rho,\mathcal{O}^\flat,\mathcal{S}^\flat}^{\mathrm{tot}}} \leq 4,$
$$
\|  \Pi_{n=0,\boldsymbol{m}=0} (\mathrm{Id} -  \Pi_{\mathrm{nor}})  K^\flat\|_{\mathscr{H}_{\eta^\flat,\rho,\mathcal{O}^\flat,\mathcal{S}^\flat}^{\mathrm{tot}}}  \leq \rho^{7/2} , \quad \| (\mathrm{Id}- \Pi_{n=0,\boldsymbol{m}=0})(\mathrm{Id} -  \Pi_{\mathrm{nor}} ) K^\flat\|_{\mathscr{H}_{\eta^\flat,\rho,\mathcal{O}^\flat,\mathcal{S}^\flat}^{\mathrm{tot}}}  \leq \rho^{9/2}  ,
$$
\item  the frequencies can be modulated, i.e. 
$$
\varepsilon \leq (24\mathfrak{M})^{-1} \quad \mathrm{and} \quad   \varepsilon \big( \sup_{i\in \mathcal{S}^\flat} \sum_{k\in \mathcal{S}^\flat} \Gamma_{k,i} \big) \leq 10^{-3} ,
$$
\item $\rho$ is small enough, i.e. 
$$
\rho \lesssim_{\mathcal{S}^\flat,\eta^\flat-\eta^\natural,\mathfrak{N},\mathfrak{M}} 1.
$$
\end{itemize}

\medskip

\noindent Then, there exist a Borel set $\mathcal{O}'\subset \mathcal{O}^\flat$ and a Hamiltonian $R \in \mathscr{H}_{\eta^\natural,\rho,\mathcal{O}',\mathcal{S}^\flat}$, such that 
\begin{itemize}
\item $\mathcal{O}'$ is relatively large in the sense that
\begin{equation}
\label{eq:meas_Birk}
\mathrm{Leb}(\mathcal{O}^\flat \setminus \mathcal{O}') \leq \, \rho ^{10^{-3}} 
\end{equation}
\item the terms to be removed are very small, the normal form part does not have changed too much and the other terms are not too large, i.e.
\begin{equation}\label{eq:RKbemol}
 \| \Pi_{\mathrm{tbr}}^\mathfrak{M} R \|_{\mathscr{H}_{\eta^\natural,\rho,\mathcal{O}',\mathcal{S}^\flat}^{\mathrm{tot}}}  \leq \rho^{\mathfrak{N}},  \quad \! \|  ( \mathrm{Id}- \Pi_{\mathrm{tbr}}^\mathfrak{M} - \Pi_{\mathrm{nor}}) R\|_{\mathscr{H}_{\eta^\natural,\rho,\mathcal{O}',\mathcal{S}^\flat}^{\mathrm{tot}}} \leq 1  \! \! \quad \mathrm{and} \! \quad \|  \Pi_{\mathrm{nor}} (R-K^\flat)\|_{\mathscr{N}_{\rho,\mathcal{O}',\mathcal{S}^\flat}^{\mathrm{tot}}} \leq \rho^{1/4}
\end{equation}
\end{itemize}
and for all $\xi \in \mathcal{O}'$, there exists a smooth symplectic map $\varphi_\xi : \mathcal{A}_\xi(\rho)\to \mathcal{A}_\xi(2\rho)$ commuting with the gauge transform and enjoying, for all $u\in \mathcal{A}_\xi(\rho)$, the bound
\begin{equation}
\label{eq:proche_id_Birk}
\| \varphi_\xi(u) - u\|_{\ell^2_\varpi} + \rho^2 \|\mathrm{d} \varphi_\xi(u) - \mathrm{Id}\|_{\ell^2_\varpi \to \ell^2_\varpi } \leq \rho^{2+\frac1{21}}
\end{equation}
 and the property that
\begin{equation}
\label{eq:cest_un_chgt_de_var_Birk}
(H^{(0)}_{\varepsilon}+\varepsilon K^\flat)(\xi; \varphi_\xi(u)) = (H^{(0)}_{\varepsilon}+\varepsilon R)(\xi;u).
\end{equation}
\end{theorem}
\begin{proof} We divide the proof in $5$ step and we introduce the local notation
$$
\Pi_{\mathrm{kept}}^\mathfrak{M} := \mathrm{Id}- \Pi_{\mathrm{tbr}}^\mathfrak{M} - \Pi_{\mathrm{nor}}.
$$

\noindent  \underline{$\bullet$ \emph{Step 1 : Decompositions.}}
As in proof of the KAM theorem, we decompose  $H^{(0)}_{\varepsilon}$ as
 $$
 H^{(0)}_{\varepsilon}=Z_2+A^{(0)}+Q^{(0)}
 $$ 
 where  $A^{(0)}= H^{(0)}_{\varepsilon}(\xi) - \mu \|\xi\|_{\ell^1}$,  $Q^{(0)}=\frac14 \sum_{k\in \mathbb{Z}} Y_k^2 - \frac12 \mu^2$ and
 $$
 Z_{2}:=\frac12\sum_{j\in\Z}(\varepsilon^{-2} j^2 +    \xi_j \mathbbm{1}_{j\in \mathcal{S}^\flat}) Y_j.
 $$
Then, we decompose $ K^\flat$ as
$$
 K^\flat = A^\flat + L^\flat + Q^\flat +  K^{\mathrm{low}} + K^{\mathrm{high}} 
$$
where $A^\flat:= \Pi_{\mathscr{A}}  K^\flat$, $L^\flat:= \Pi_{\mathscr{L}}  K^\flat$, $Q^\flat:= \Pi_{\mathscr{Q}}  K^\flat$ and
$$
 K^{\mathrm{low}} := \Pi_{n=0,\boldsymbol{m}=0} (\mathrm{Id} -  \Pi_{\mathrm{nor}})K^\flat.
$$
Note that by assumption
\begin{equation}\label{est:KK} \|   K^{\mathrm{low}}\|_{\mathscr{H}_{\eta^\flat,\rho,\mathcal{O}^\flat,\mathcal{S}^\flat}^{\mathrm{tot}}}  \leq \rho^{7/2}\quad \text{and}\quad  \|   K^{\mathrm{high}}\|_{\mathscr{H}_{\eta^\flat,\rho,\mathcal{O}^\flat,\mathcal{S}^\flat}^{\mathrm{tot}}}  \leq \rho^{9/2}.\end{equation}

\noindent  \underline{$\bullet$ \emph{Step 2 : $\mathcal{O}'$ and small divisor estimates.}}
For $\xi\in \mathcal{O}^\flat $, $\boldsymbol{k}\in\Z^{\S^\flat}$,  $q\in\N$, $\boldsymbol{\ell}\in (\S^c)^q$ and  $\boldsymbol{\sigma}\in \{-1,1\}^q$, we define the small divisor
$$
\Omega(\xi;\boldsymbol{k},\boldsymbol{\ell},\boldsymbol{\sigma})=\sum_{j\in\S^\flat} \boldsymbol{k}_j\omega_j(\xi)+\sum_{i=1}^q \boldsymbol{\sigma}_i\omega_{\boldsymbol{\ell}_i}(\xi)
$$
where 
\begin{equation}
\label{eq:omega*}
\omega_j(\xi) := \varepsilon^{-2} j^2 +    \xi_j \mathbbm{1}_{j\in \mathcal{S}^\flat}+2\eps L_j^\flat(\xi)
  ,\ j\in \Z.
\end{equation}
Let $\mathcal{O}'\subset \mathcal{O}^\flat$ be the Borelian set given by Proposition \ref{prop:smalldiv}\footnote{Note that the requirement on $L$ and $\varepsilon$ of Proposition \ref{prop:smalldiv} are satisfied by assumption.}  such that 
$$
\mathrm{Leb}(\mathcal{O}^\flat \setminus \mathcal{O}') \leq \, \rho ^{10^{-3}} =:\gamma
$$
and, if $\boldsymbol{k}\neq 0$, $q\leq \mathfrak{M}$ and $ \# \{ \boldsymbol{\ell}_j \ | \ 1\leq j \leq q\} \leq 4 $, then for all $\xi \in \mathcal{O}'$
$$
| \Omega(\xi;\boldsymbol{k},\boldsymbol{\ell},\boldsymbol{\sigma})| \gtrsim_{\mathcal{S},\mathfrak{M}} \Big(\frac{\gamma}{ \| \boldsymbol{k} \|_{\ell^\infty}^{2\# \mathcal{S}} }\Big)^{166}.
$$

\medskip

It follows that if $n\in \mathbb{N}$, $\boldsymbol{k} \in \mathbb{Z}^{\mathcal{S}^\flat}$, $\boldsymbol{m}\in \mathbb{N}^{\mathcal{S}^\flat}$, $q\geq 0$, $\boldsymbol{\ell}\in (\mathbb{Z}\setminus \mathcal{S}^\flat)^q$, $\boldsymbol{\sigma}\in \{-1,1\}^q$ are indices associated with a non-zero coefficient of a Hamiltonian to be removed (i.e. as in Lemma \ref{lem:critical}) then for all $\xi \in \mathcal{O}'$, we have
\begin{equation}
\label{eq:est_ptdiv_birk}
| \Omega(\xi;\boldsymbol{k},\boldsymbol{\ell},\boldsymbol{\sigma})| \geq   \frac{\rho ^{2/10}}{    \langle \| \boldsymbol{k} \|_{\ell^1} \rangle^{400\# \mathcal{S}}}.
\end{equation}
Indeed, we got it by construction of $\mathcal{O}'$ when $\boldsymbol{k}\neq 0$ (using the smallness assumption of $\rho$). Conversely, if $\boldsymbol{k}= 0$, we know by Lemma \ref{lem:critical} that $\boldsymbol{\sigma}_1 \boldsymbol{\ell}_1^2+ \cdots+\boldsymbol{\sigma}_q \boldsymbol{\ell}_q^2\neq 0$ and, since it is an integer, it follows that
$$
| \Omega(\xi;\boldsymbol{k},\boldsymbol{\ell},\boldsymbol{\sigma})| \geq \varepsilon^{-2} -2\varepsilon\| L\|_{\mathscr{H}_{\rho,\mathcal{O}^\flat,\mathcal{S}^\flat}^\mathrm{sup}} \mathfrak{M} \geq 1.
$$

\noindent  \underline{$\bullet$ \emph{Step 3 : One first step of normal form.}} In this step we want to make slightly smaller $K^{\mathrm{low}}$ to obtain the same bound as for $K^{\mathrm{high}}$ (see \eqref{est:KK}). We set, for all $k\geq 1$,
$$
\eta_k := 2^{-k}\eta^\flat + (1-2^{-k}) \eta^{\natural}.
$$

Then, we define, thanks to the small divisor estimate \eqref{eq:est_ptdiv_birk}, a Hamiltonian $\chi^{(0)}\in \mathscr{H}_{\eta_1 ,\rho,\mathcal{O}',\mathcal{S}^\flat}$ by the relation\footnote{with the implicit convention that $ (\chi^{(0)})_{n,\boldsymbol{k},\boldsymbol{m}}^{\boldsymbol{\ell},\boldsymbol{\sigma}}(\xi)=0$ when $\Omega(\cdot;\boldsymbol{k},\boldsymbol{\ell},\boldsymbol{\sigma})(\xi)=0$.}
$$
 (\chi^{(0)})_{n,\boldsymbol{k},\boldsymbol{m}}^{\boldsymbol{\ell},\boldsymbol{\sigma}} :=  -\varepsilon \frac{(  \Pi_{\mathrm{tbr}}^{\mathfrak{M}} K^{\mathrm{low}}  )_{n,\boldsymbol{k},\boldsymbol{m}}^{\boldsymbol{\ell},\boldsymbol{\sigma}}}{\ic \Omega(\cdot;\boldsymbol{k},\boldsymbol{\ell},\boldsymbol{\sigma}) } .
$$
 Indeed, using the small divisor estimate \eqref{eq:est_ptdiv_birk} and proceeding as in the proof of Lemma \ref{lem:homo}, we get that (thanks to the smallness assumption on $\rho$)
\begin{equation}
\label{eq:bound_on_chi0}
\| \chi^{(0)} \|_{ \mathscr{H}_{\eta_1 ,\rho,\mathcal{O}',\mathcal{S}^\flat}^{\mathrm{tot}}} \leq \varepsilon \rho^{-\frac{9}{20}} \|K^{\mathrm{low}}\|_{ \mathscr{H}_{\eta^\flat,\rho,\mathcal{O}^\flat,\mathcal{S}^\flat}^{\mathrm{tot}}} \leq \varepsilon \rho^{3+\frac1{20}}.
\end{equation}
Note that, since by construction $\chi^{(0)}=\Pi_{n=0,\boldsymbol{m}=0}\chi^{(0)}$, this last estimate ensures the existence of the flow of $\chi^{(0)}$ for $|t|\leq 1$ by Lemma \ref{lem:flow}.
Moreover, again by construction, we have that  $\chi^{(0)}$ solves the cohomological equation
\begin{equation}
\label{eq:my_cohom}
\{ Z_2+\varepsilon L^\flat, \chi^{(0)}\} + \varepsilon \Pi_{\mathrm{tbr}}^{\mathfrak{M}} K^{\mathrm{low}}  =0.
\end{equation}
Then, using as in the proof of the KAM theorem the local notation (defined as soon as it makes sense)
$$
 \mathcal T_\chi(F):=F\circ \Phi^1_\chi-F-\{F,\chi\} \quad \mathrm{and} \quad \Delta_{\chi}(F) := F\circ \Phi^1_\chi-F
$$
we get the expansion
\begin{equation}
\label{eq:expan_Birk1}
\begin{split}
 H^{(0)}_{\varepsilon}+\varepsilon K'  :=(H^{(0)}_{\varepsilon}+\varepsilon K^\flat )\circ \Phi_{ \chi^{(0)} }^1 
 = H^{(0)}_{\varepsilon}+\varepsilon K^\flat  + \underbrace{\{ Z_2+\varepsilon L^\flat,\chi^{(0)} \}}_{\displaystyle \mathop{=}^{\eqref{eq:my_cohom}} -\varepsilon \Pi_{\mathrm{tbr}}^{\mathfrak{M}} K^{\mathrm{low}}} + \varepsilon \Lambda
\end{split}
\end{equation}
where
$$
\varepsilon\Lambda:=\underbrace{\mathcal T_{\chi^{(0)}}(Z_2+\varepsilon L^\flat) }_{=: \varepsilon\Lambda_1}+\underbrace{ \Delta_{\chi^{(0)}}(Q^{(0)} + \varepsilon (Q^\flat +  K^{\mathrm{low}} + K^{\mathrm{high}}) )}_{=: \varepsilon\Lambda_2}.
$$
To continue the discussion, we have to show that $\Lambda_1,\Lambda_2 \in  \mathscr{H}_{\eta_2 ,\rho,\mathcal{O}',\mathcal{S}^\flat}$ and to prove that they are small enough. On the one hand, by Taylor, we have
$$
\varepsilon\Lambda_1 = \int_{0}^1 (1-t) \big( \mathrm{ad}_{\chi^{(0)}}^2 (Z_2+\varepsilon L^\flat) \big)\circ\Phi_{\chi^{(0)}}^t \mathrm{d}t = -\varepsilon \int_{0}^1 (1-t) \{ \Pi_{\mathrm{tbr}}^{\mathfrak{M}} K^{\mathrm{low}}, \chi^{(0)}\} \circ\Phi_{\chi^{(0)}}^t \mathrm{d}t .
$$
It follows by Lemma \ref{lem:poisson} and Lemma \ref{lem:compo} that  $\Lambda_1 \in  \mathscr{H}_{\eta_2 ,\rho,\mathcal{O}',\mathcal{S}^\flat}$ and 
$$
\| \Lambda_1 \|_{\mathscr{H}_{\eta_2 ,\rho,\mathcal{O}',\mathcal{S}^\flat}^{\mathrm{tot}}} \lesssim_{\mathcal{S}^\flat} \varepsilon \frac{\rho^{7/2} \rho^{3+\frac1{20}}}{\rho^2(\eta^{\flat}-\eta^{\natural})^4}.
$$
On the other hand, since by assumption $\| \Pi_{\mathrm{nor}} K^\flat\|_{\mathscr{N}_{\rho,\mathcal{O}^\flat,\mathcal{S}^\flat}^{\mathrm{tot}}} \leq 4$, we have by Corollary \ref{cor:emb_and_proj} that
$$
\| Q^{(0)} + \varepsilon Q^\flat  \|_{\mathscr{H}_{\eta^{\flat} ,\rho,\mathcal{O}',\mathcal{S}^\flat}^{\mathrm{tot}}} \lesssim_{\mathcal{S}} \rho^4.
$$
It follows by Lemma \ref{lem:compo} that  $\Lambda_2 \in  \mathscr{H}_{\eta_2 ,\rho,\mathcal{O}',\mathcal{S}^\flat}$ and
$$
\| \Lambda_2 \|_{\mathscr{H}_{\eta_2 ,\rho,\mathcal{O}',\mathcal{S}^\flat}^{\mathrm{tot}}} \lesssim_{S^\flat} \frac{\rho^{7/2} \rho^{3+\frac1{20}}}{\rho^2(\eta^{\flat}-\eta^{\natural})^5}.
$$
Putting these bounds on $\Lambda_1$ and $\Lambda_2$ together, we get
$$
\| \Lambda \|_{\mathscr{H}_{\eta_2 ,\rho,\mathcal{O}',\mathcal{S}^\flat}^{\mathrm{tot}}} \leq  \rho^{9/2}.
$$

\medskip

Coming back to \eqref{eq:expan_Birk1} and using that $\chi^{(0)}$ solves the cohomological equation \eqref{eq:my_cohom}, we deduce that
$$
\Pi_{\mathrm{tbr}}^{\mathfrak{M}} K' = \Pi_{\mathrm{tbr}}^{\mathfrak{M}} K^{\mathrm{high}} + \Pi_{\mathrm{tbr}}^{\mathfrak{M}} \Lambda, \quad \Pi_{\mathrm{nor}} K' =  \Pi_{\mathrm{nor}} K^{\flat} + \Pi_{\mathrm{nor}} \Lambda,
\quad
\Pi_{\mathrm{kept}}^\mathfrak{M}  K' = \Pi_{\mathrm{kept}}^\mathfrak{M} (K^\flat + \Lambda)
$$
and thus, we get the bounds
\begin{equation}
\begin{split}
\label{eq:bounds_on_K'}
\| \Pi_{\mathrm{tbr}}^{\mathfrak{M}} K' \|_{\mathscr{H}_{\eta_2 ,\rho,\mathcal{O}',\mathcal{S}^\flat}^{\mathrm{tot}}} + \| \Pi_{\mathrm{nor}} (K' - K^{\flat}) \|_{\mathscr{H}_{\eta_2 ,\rho,\mathcal{O}',\mathcal{S}^\flat}^{\mathrm{tot}}}  \lesssim \rho^{9/2}, \ \ \
\|\Pi_{\mathrm{kept}}^\mathfrak{M}  K' \|_{\mathscr{H}_{\eta_2 ,\rho,\mathcal{O}',\mathcal{S}^\flat}^{\mathrm{tot}}} \! \lesssim  \rho^{7/2}.
\end{split}
\end{equation}

\noindent  \underline{$\bullet$ \emph{Step 4 : Iteration.}} We are going to design, by induction, two sequences of Hamiltonians $\chi^{(j)} \in  \mathscr{H}_{\eta_{2j+1} ,\rho,\mathcal{O}',\mathcal{S}^\flat}$, $K^{(j)} \in  \mathscr{H}_{\eta_{2j+2} ,\rho,\mathcal{O}',\mathcal{S}^\flat}$ for $1\leq j\leq \mathfrak{N}$,  such that $K^{(0)} := K'$ and for $1\leq j\leq \mathfrak{N}$
\begin{equation}
\label{eq:chgt_var_B}
 H^{(0)}_{\varepsilon}+\varepsilon K^{(j)}  =(H^{(0)}_{\varepsilon}+\varepsilon K^{(j-1)} )\circ \Phi_{ \chi^{(j)} }^1 ,
\end{equation}
\begin{equation}
\label{eq:devient_petit_B}
\| \Pi_{\mathrm{tbr}}^{\mathfrak{M}} K^{(j)} \|_{\mathscr{H}_{\eta_{2j+2} ,\rho,\mathcal{O}',\mathcal{S}^\flat}^{\mathrm{tot}}} + \| \Pi_{\mathrm{nor}} (K^{(j)} - K^{(j-1)}) \|_{\mathscr{H}_{\eta_{2j+2} ,\rho,\mathcal{O}',\mathcal{S}^\flat}^{\mathrm{tot}}} \lesssim_j \rho^{9/2 + j},
\end{equation}
\begin{equation}
\label{eq:reste_sympa_B_B}
\| \Pi_{\mathrm{kept}}^\mathfrak{M} K^{(j)} \|_{\mathscr{H}_{\eta_{2j+2} ,\rho,\mathcal{O}',\mathcal{S}^\flat}^{\mathrm{tot}}} \lesssim_j \rho^{7/2},
\end{equation}
\begin{equation}
\label{eq:bound_on_chij}
\| \chi^{(j)} \|_{ \mathscr{H}_{\eta_{2j+1} ,\rho,\mathcal{O}',\mathcal{S}^\flat}^{\mathrm{tot}}}  \lesssim_j \varepsilon \rho^{3+j+\frac1{20}}.
\end{equation}
Note that, since $K^{(0)} := K'$, the initialization is given by \eqref{eq:bounds_on_K'}.

\medskip

Let $1\leq j\leq \mathfrak{N}$ and assume that $K^{(i)},\chi^{(i)}$ are given for $i<j$.  We set 
$$
N^{(j-1)} := \Pi_{\mathrm{nor}} (K^{(j-1)} - K^{\flat})
$$
and we note that, by \eqref{eq:devient_petit_B} and \eqref{eq:bounds_on_K'}, we have
\begin{equation}
\label{eq:bound_on_N_B}
\| N^{(j-1)} \|_{\mathscr{H}_{\eta_{2j} ,\rho,\mathcal{O}',\mathcal{S}^\flat}^{\mathrm{tot}}} \lesssim_j \rho^{9/2}.
\end{equation}
Note that by construction, we have the following decomposition 
\begin{equation}
\label{eq:thedecjm1}
K^{(j-1)} = A^\flat + L^\flat + Q^\flat +  N^{(j-1)}  +  \Pi_{\mathrm{tbr}}^{\mathfrak{M}} K^{(j-1)}  +  \Pi_{\mathrm{kept}}^\mathfrak{M} K^{(j-1)}.
\end{equation}
Then, we define, thanks to the small divisor estimate \eqref{eq:est_ptdiv_birk}, the Hamiltonian $\chi^{(j)}\in \mathscr{H}_{\eta_{2j+1} ,\rho,\mathcal{O}',\mathcal{S}^\flat}$ by the relation
$$
 (\chi^{(j)})_{n,\boldsymbol{k},\boldsymbol{m}}^{\boldsymbol{\ell},\boldsymbol{\sigma}} :=  -\varepsilon \frac{(  \Pi_{\mathrm{tbr}}^{\mathfrak{M}} K^{(j-1)}  )_{n,\boldsymbol{k},\boldsymbol{m}}^{\boldsymbol{\ell},\boldsymbol{\sigma}}}{\ic \Omega(\cdot;\boldsymbol{k},\boldsymbol{\ell},\boldsymbol{\sigma}) } .
$$
 Indeed, using the small divisor estimate \eqref{eq:est_ptdiv_birk} and proceeding as in the proof of Lemma \ref{lem:homo}, we get that (thanks to the smallness assumption on $\rho$)
$$
\| \chi^{(j)} \|_{ \mathscr{H}_{\eta_{2j+1} ,\rho,\mathcal{O}',\mathcal{S}^\flat}^{\mathrm{tot}}} \leq \varepsilon \rho^{-\frac{9}{20}} \| \Pi_{\mathrm{tbr}}^{\mathfrak{M}} K^{(j-1)}\|_{ \mathscr{H}_{\eta^\flat,\rho,\mathcal{O}^\flat,\mathcal{S}^\flat}^{\mathrm{tot}}} \lesssim_j \varepsilon \rho^{3+j+\frac1{20}}.
$$
Note that\footnote{Here we need $\| \Pi_{\mathrm{tbr}}^{\mathfrak{M}} K^{(j-1)}\|_{ \mathscr{H}_{\eta^\flat,\rho,\mathcal{O}^\flat,\mathcal{S}^\flat}^{\mathrm{tot}}}\leq \rho^{9/2} $ and thus we need the step 3.}, since $3+j+\frac1{20}\geq4$, the last estimate  ensures the existence of the flow of $\chi^{(j)}$ for $|t|\leq 1$ by Lemma \ref{lem:flow}.
Moreover, by construction, $\chi^{(j)}$ solves the cohomological equation
\begin{equation}
\label{eq:my_cohomj}
\{ Z_2+\varepsilon L^\flat, \chi^{(j)}\} + \varepsilon \Pi_{\mathrm{tbr}}^{\mathfrak{M}} K^{(j-1)}  =0.
\end{equation}
Then, we get the expansion
\begin{equation}
\label{eq:expan_Birkj}
\begin{split}
 H^{(0)}_{\varepsilon}+\varepsilon K^{(j)}  :=(H^{(0)}_{\varepsilon}+\varepsilon K^{(j-1)} )\circ \Phi_{ \chi^{(j)} }^1 
 = H^{(0)}_{\varepsilon}+\varepsilon K^{(j-1)}  + \{ Z_2+\varepsilon L^\flat,\chi^{(j)} \} + \varepsilon \Lambda^{(j)}
\end{split}
\end{equation}
where
$$
\varepsilon\Lambda^{(j)}:=\underbrace{\mathcal T_{\chi^{(j)}}(Z_2+\varepsilon L^\flat) }_{=: \varepsilon\Lambda_1^{(j)}}+\underbrace{ \Delta_{\chi^{(j)}}(Q^{(0)} + \varepsilon (K^{(j-1)} - A^\flat - L^\flat  ) )}_{=: \varepsilon\Lambda_2^{(j)}}.
$$
On the one hand, proceeding as previously, we get (by Lemma \ref{lem:poisson} and Lemma \ref{lem:compo}) that  $\Lambda_1^{(j)} \in  \mathscr{H}_{\eta_{2j+2} ,\rho,\mathcal{O}',\mathcal{S}^\flat}$ and 
$$
\| \Lambda_1^{(j)} \|_{\mathscr{H}_{\eta_{2j+2} ,\rho,\mathcal{O}',\mathcal{S}^\flat}^{\mathrm{tot}}} \lesssim_{j,\mathcal{S}^\flat} \varepsilon \frac{\rho^{7/2 + j}  \rho^{3+j+\frac1{20}} }{\rho^2(\eta^{\flat}-\eta^{\natural})^4}.
$$
On the other hand, using the decomposition \ref{eq:thedecjm1} of $K^{(j-1)}$, the bounds \eqref{eq:devient_petit_B}, \eqref{eq:reste_sympa_B_B} and \eqref{eq:bound_on_N_B} on $ \Pi_{\mathrm{tbr}}^{\mathfrak{M}} K^{(j-1)}$, $ \Pi_{\mathrm{kept}}^{\mathfrak{M}} K^{(j-1)}$ and $N^{(j-1)}$, we have that
$$
\| Q^{(0)} + \varepsilon (K^{(j-1)} - A^\flat - L^\flat  ) \|_{\mathscr{H}_{\eta_{2j} ,\rho,\mathcal{O}',\mathcal{S}^\flat}^{\mathrm{tot}}} \lesssim_{j,\mathcal{S}^\flat}  \rho^{7/2}.
$$
It follows by Lemma \ref{lem:compo} that  $\Lambda_2^{(j)} \in  \mathscr{H}_{\eta_{2j+2} ,\rho,\mathcal{O}',\mathcal{S}^\flat}$ and
$$
\| \Lambda_2^{(j)} \|_{\mathscr{H}_{\eta_{2j+2} ,\rho,\mathcal{O}',\mathcal{S}^\flat}^{\mathrm{tot}}} \lesssim_{j,S^\flat} \frac{\rho^{7/2} \rho^{3+j+\frac1{20}}}{\rho^2(\eta^{\flat}-\eta^{\natural})^5}.
$$
Putting these bounds on $\Lambda_1^{(j)}$ and $\Lambda_2^{(j)}$ together, we get
$$
\| \Lambda^{(j)} \|_{\mathscr{H}_{\eta_{2j+2} ,\rho,\mathcal{O}',\mathcal{S}^\flat}^{\mathrm{tot}}} \leq  \rho^{9/2+j}.
$$
Now, coming back to \eqref{eq:expan_Birkj} and using that $\chi^{(j)}$ solves the cohomological equation \eqref{eq:my_cohomj}, we deduce that
$$
\Pi_{\mathrm{tbr}}^{\mathfrak{M}} K^{(j)} =\Pi_{\mathrm{tbr}}^{\mathfrak{M}} \Lambda^{(j)}, \quad \Pi_{\mathrm{nor}} K^{(j)} =  \Pi_{\mathrm{nor}} K^{(j-1)} + \Pi_{\mathrm{nor}} \Lambda^{(j)},
\quad
\Pi_{\mathrm{kept}}^\mathfrak{M}  K^{(j)} = \Pi_{\mathrm{kept}}^\mathfrak{M} (K^{(j-1)} + \Lambda^{(j)})
$$
and thus,  we get the bounds \eqref{eq:devient_petit_B} and \eqref{eq:reste_sympa_B_B}.

\medskip

\noindent  \underline{$\bullet$ \emph{Step 5 : Conclusion.}} It suffices to set
$$
\forall \xi \in \mathcal{O}', \quad \varphi_\xi := \Phi_{\chi^{(0)}}^{0} \circ \cdots \circ \Phi_{\chi^{(N)}}^{1}(\xi;\cdot) \quad \mathrm{and} \quad R = K^{(\mathfrak{N})}.
$$
Indeed, thanks to the bounds \eqref{eq:bound_on_chi0},\eqref{eq:bound_on_chij} on the $\chi^{(j)}$, we know by Lemma \ref{lem:flow} that $\varphi_\xi$ is well defined and that it is close enough to the identity (i.e. it satisfies \eqref{eq:proche_id_Birk}). The fact that $(H^{(0)}_{\varepsilon}+\varepsilon K^\flat)\circ \varphi = H^{(0)}_{\varepsilon}+\varepsilon R$ is a direct consequence of the construction (see \eqref{eq:expan_Birk1} and \eqref{eq:chgt_var_B}). The bounds $\| \Pi_{\mathrm{tbr}}^\mathfrak{M} R \|_{\mathscr{H}_{\eta^\natural,\rho,\mathcal{O}',\mathcal{S}^\flat}^{\mathrm{tot}}}  \leq \rho^{\mathfrak{N}}$ and $\| \Pi_{\mathrm{kept}}^\mathfrak{M} R\|_{\mathscr{H}_{\eta^\natural,\rho,\mathcal{O}',\mathcal{S}^\flat}^{\mathrm{tot}}} \leq 1$ are given by \eqref{eq:devient_petit_B} and \eqref{eq:reste_sympa_B_B}. Finally, by Lemma \ref{lem:continuity_estimate_pinor} and using by \eqref{eq:bounds_on_K'} and  \eqref{eq:devient_petit_B}, we get that
$\|  \Pi_{\mathrm{nor}} (R-K^\flat)\|_{\mathscr{N}_{\rho,\mathcal{O}',\mathcal{S}^\flat}^{\mathrm{tot}}} \leq \rho^{1/4}$.
\end{proof}

\section{Opening}\label{sec:open}

In this section  
we describe the process of opening new sites to increase the cardinal of $\S$. We distinguish between two cases: the general case, in which we allow an arbitrary number of sites to be opened (this case will be used for the initialization step in the proof of Theorem \ref{thm:main}, see section \ref{sec:first}), and the case of normal forms, in which we allow only one new site but require more precise estimates.

\subsection{General case} In itself this step could be very short : it expresses that if we have a Hamiltonian depending analytically on the variables $(\theta,y,u)$ then, when we open new sites, i.e. when we write $u_j=\sqrt{\xi_j+y_j}e^{\ic \theta_j}$ for some $j\in \mathcal{S}^c$, we get a new Hamiltonian that depends analytically on the new variables $(\theta,y,u)$ in an adapted domain. But since we encode the analyticity in the estimates of the coefficients of the Hamiltonian (see Definition \ref{def:class}), this step becomes heavy and actually we "waste" a lot of the size of the radius $r$ to make it simpler.

\begin{proposition} \label{prop:open} Let $\mathcal{S}^\flat \subset \mathcal{S}\subset \mathbb{Z}$ be some finite sets, $\Lambda>1$, $\eta \in [\eta_0,\eta_{\max}]$, $r_\flat,\rho,r\in (0,1/4),$ $\mathcal{O}^\flat\subset ((4r_\flat)^{2},1)^{\mathcal{S}^\flat}$ be a Borel set and $P\in \mathscr{H}_{\eta,r_\flat,\mathcal{O}^{\flat},\mathcal{S}^{\flat}}$.

\medskip

\noindent Assume that  the radii enjoy the estimates
\begin{equation}
\label{eq:smallr2}
r \lesssim_{\mathcal{S},\Lambda} \rho \lesssim_{\mathcal{S},\Lambda} r_\flat.
\end{equation}

\medskip

\noindent Then, there exists a Hamiltonian $R\in \mathscr{H}_{\eta,r,\mathcal{O},\mathcal{S}}$, where
$$
\mathcal{O} := \mathcal{O}^\flat\times (\rho^2,\Lambda\rho^2)^{\mathcal{S}\setminus \mathcal{S}^\flat},
$$ 
such that
\begin{equation}
\label{eq:bound_open}
\| R\|_{\mathscr{H}_{\eta,r,\mathcal{O},\mathcal{S}}^{\mathrm{tot}}} \lesssim \|P\|_{\mathscr{H}_{\eta,r_\flat,\mathcal{O}^\flat,\mathcal{S}^\flat}^{\mathrm{tot}}}
\end{equation}
and for all $\xi=:(\xi^\flat,\xi^{\mathrm{new}})\in \mathcal{O}$ and all $u\in \mathcal{A}_{\xi}(3r)$, 
$$
R(\xi;u) = P(\xi^\flat;u).
$$
Moreover, if $\Pi_{\mathrm{int}} P=0$ then $\Pi_{\mathrm{int}} R = 0$.
\end{proposition}
\begin{proof} The proof is divided in $3$ steps. The expansions will naturally make appear the generalized binomial coefficients which are defined by
$$
C_{\alpha}^k := (k!)^{-1} \prod_{0\leq j \leq k-1} (\alpha-j), \quad k\in \mathbb{N}, \ \alpha \in \mathbb{R}.
$$

\medskip

\noindent $\bullet$ \underline{\emph{Step 1 : Identification of $R$.}}  First, we denote
$$
 \mathcal{S} \setminus \mathcal{S}^{\flat} =:\{\boldsymbol{i}_1,\cdots,\boldsymbol{i}_N\} 
$$
where $\boldsymbol{i}_1,\cdots,\boldsymbol{i}_N$ are distinct.
Then, we note that, by symmetry of the coefficients of $P$, being given $\xi^\flat \in  \mathcal{O}^\flat$ and $u\in  \mathcal{A}_\xi(3r_\flat)$, we have
$$
P(\xi^\flat;u) = \sum_{\boldsymbol{k}^\flat\in \mathbb{Z}^{\mathcal{S}^\flat}} \sum_{\boldsymbol{m}^\flat\in \mathbb{N}^{\mathcal{S}^\flat}} \sum_{q \geq 0}  \sum_{ \boldsymbol{\sigma} \in  \{-1,1\}^q } \sum_{ \boldsymbol{\ell} \in ((\mathcal{S})^c)^q  } \widehat{P}_{\boldsymbol{k}^\flat,\boldsymbol{m}^\flat}^{\boldsymbol{\ell},\boldsymbol{\sigma}}(\xi^\flat;u) \, u_{\boldsymbol{\ell}}^{\boldsymbol{\sigma}} y(\xi^\flat;u)^{\boldsymbol{m}} \, e^{\ic \boldsymbol{k} \cdot \theta(u)}
$$
where
$$
 \widehat{P}_{\boldsymbol{k}^\flat,\boldsymbol{m}^\flat}^{\boldsymbol{\ell},\boldsymbol{\sigma}}(\xi^\flat;u) := \sum_{n \in \mathbb{N}}  \sum_{\boldsymbol{a} \in \mathbb{N}^N}   \sum_{ \boldsymbol{\varsigma} \in   \{-1,1\}^{\boldsymbol{a}}  } C_{\|\boldsymbol{a}\|_{\ell^1}+q}^{\|\boldsymbol{a}\|_{\ell^1}} P_{n,\boldsymbol{k}^\flat,\boldsymbol{m}^\flat}^{\boldsymbol{\ell},\boldsymbol{i}^{.\boldsymbol{a}},\boldsymbol{\sigma}, \boldsymbol{\varsigma}}(\xi^\flat)  \mu(\xi^\flat;u)^n  u_{\boldsymbol{i}^{.\boldsymbol{a}}}^{ \boldsymbol{\varsigma}}
$$
and $\{-1,1\}^{\boldsymbol{a}}:= \prod_{p=1}^N \{-1,1\}^{\boldsymbol{a}_{p}} \equiv \{-1,1\}^{\|\boldsymbol{a}\|_{\ell^1}}$,
$$
\boldsymbol{i}^{.\boldsymbol{a}} := ( \underbrace{\boldsymbol{i}_1,\cdots,\boldsymbol{i}_1}_{\boldsymbol{a}_1 \ times}, \cdots, \underbrace{\boldsymbol{i}_N,\cdots,\boldsymbol{i}_N}_{\boldsymbol{a}_N \ times} ).
$$

\medskip

Now, assume that $9^2 r \leq 9 \rho \leq r_\flat$ and let $\xi \in \mathcal{O}$ and $u\in \mathcal{A}_{\xi}(3r)$. Thanks these estimates on $\rho$ and $r$, we have that for all $i\in \mathcal{S} \setminus \mathcal{S}^{\flat}$ and all $\varsigma \in \{-1,1\}$
$$
u_{i}^{ \varsigma } = \sqrt{\xi_i + y_i(\xi;u)} e^{\ic \varsigma \theta_i(u)} \quad \mathrm{and} \quad  \mu(\xi;u)= \mu(\xi^\flat;u)-\|\xi_{\boldsymbol{i}} \|_{\ell^1} 
$$
where $\xi_{\boldsymbol{i}} := (\xi_{\boldsymbol{i}_1} ,\cdots, \xi_{\boldsymbol{i}_N} )$ (this convention will be used in all this proof).
It follows that
\begin{equation*}
\begin{split}
\widehat{P}_{\boldsymbol{k}^\flat,\boldsymbol{m}^\flat}^{\boldsymbol{\ell},\boldsymbol{\sigma}}(\xi^\flat;u) =&  \sum_{n \in \mathbb{N}} \sum_{d=0}^n   \sum_{\boldsymbol{a} \in \mathbb{N}^N}   \sum_{ \boldsymbol{\varsigma} \in   \{-1,1\}^{\boldsymbol{a}}  } C_{\|\boldsymbol{a}\|_{\ell^1}+q}^{\|\boldsymbol{a}\|_{\ell^1}} C_n^d   P_{n,\boldsymbol{k}^\flat,\boldsymbol{m}^\flat}^{\boldsymbol{\ell},\boldsymbol{i}^{.\boldsymbol{a}},\boldsymbol{\sigma}, \boldsymbol{\varsigma}}(\xi^\flat)    \mu(\xi;u)^d \|\xi_{\boldsymbol{i}} \|_{\ell^1}^{n-d} \\ & \times \prod_{p=1}^N  \sum_{\boldsymbol{b}_p\in \mathbb{N}}  e^{\ic (\boldsymbol{\varsigma}_{p,1}+\cdots+\boldsymbol{\varsigma}_{p,\boldsymbol{a}_p})\theta_{\boldsymbol{i}_p}(u)} \xi_{\boldsymbol{i}_p}^{\boldsymbol{a}_p/2}  C_{\boldsymbol{a}_p/2}^{\boldsymbol{b}_p} \big( \frac{y_{\boldsymbol{i}_p}(\xi,u)}{\xi_{\boldsymbol{i}_p}} \big)^{\boldsymbol{b}_p}
\end{split}
\end{equation*}
and so we get that $R(\xi;u) = P(\xi^\flat;u)$ with
\begin{equation}
\label{eq:ilfallaitbiensouffler}
R_{d,\boldsymbol{k},\boldsymbol{m}}^{\boldsymbol{\ell},\boldsymbol{\sigma}}(\xi) =  \sum_{\boldsymbol{a} \in \mathbb{N}^N}  \sum_{ \boldsymbol{\varsigma} \in \mathcal{R}_{\boldsymbol{a},\boldsymbol{k}} } \sum_{n\geq d}  \|\xi_{\boldsymbol{i}} \|_{\ell^1}^{n-d} C_{\|\boldsymbol{a}\|_{\ell^1}+q}^{\|\boldsymbol{a}\|_{\ell^1}} C_n^d    P_{n,\boldsymbol{k}^\flat,\boldsymbol{m}^\flat}^{\boldsymbol{\ell},\boldsymbol{i}^{.\boldsymbol{a}},\boldsymbol{\sigma}, \boldsymbol{\varsigma}}(\xi^\flat) \prod_{p=1}^N  \xi_{\boldsymbol{i_p}}^{\boldsymbol{a}_p/2-\boldsymbol{m}_{\boldsymbol{i}_p}} C_{\boldsymbol{a}_p/2}^{\boldsymbol{m}_{\boldsymbol{i}_p}}    
\end{equation}
where
$$
 \mathcal{R}_{\boldsymbol{a},\boldsymbol{k}} = \{ \boldsymbol{\varsigma} \in \{-1,1\}^{\boldsymbol{a}} \ | \ \forall p, \  \boldsymbol{\varsigma}_{p,1}+\cdots+\boldsymbol{\varsigma}_{p,\boldsymbol{a}_p} = \boldsymbol{k}_{\boldsymbol{i}_p}  \}.
$$

Here, we note that all the integrable terms of $R$ come from integrable terms of $P$. As a consequence, we have that if $\Pi_{\mathrm{int}} P=0$ then $\Pi_{\mathrm{int}} R = 0$.

\medskip 

Now, the point is to estimate these coefficients to prove that $R\in  \mathscr{H}_{\eta,r,\mathcal{O},\mathcal{S}}$ and satisfies the bound \eqref{eq:bound_open}. To lighten the estimates, without loss of generality, we assume that $\|P\|_{\mathscr{H}_{\eta^\flat,r_\flat,\mathcal{O}^\flat,\mathcal{S}^\flat}^{\mathrm{tot}}}=1$.

\medskip

\noindent $\bullet$ \underline{\emph{Step 2 : Uniform estimate of $R$.}} 
We aim at estimating $|R_{d,\boldsymbol{k},\boldsymbol{m}}^{\boldsymbol{\ell},\boldsymbol{\sigma}}(\xi)|$. First, we note that, since $\langle\cdot \rangle^{2\delta} \leq \varpi$, we have
$$
\Theta_{\boldsymbol{\ell},\boldsymbol{i}^{\cdot \boldsymbol{a}}}  \leq  \Theta_{\boldsymbol{\ell}} \prod_{1\leq p \leq N} \varpi_{\boldsymbol{i}_p}^{\boldsymbol{a}_p} .
$$
Moreover, since $\xi_{\boldsymbol{i}} \in (\rho^2,\Lambda\rho^2)^{N}$ and $\|P\|_{\mathscr{H}_{\eta,r_\flat,\mathcal{O}^\flat,\mathcal{S}^\flat}^{\mathrm{sup}}} \leq 1$, we deduce that 
\begin{equation*}
\begin{split}
|R_{d,\boldsymbol{k},\boldsymbol{m}}^{\boldsymbol{\ell},\boldsymbol{\sigma}}(\xi)|\Theta_{\boldsymbol{\ell}}^{-1} \leq&  \sum_{\substack{\boldsymbol{a} \in \mathbb{N}^N,\ n\geq d \\ \boldsymbol{\varsigma}\in  \mathcal{R}_{\boldsymbol{a},\boldsymbol{k}} }}    (N\Lambda\rho^2)^{n-d+\frac{\|\boldsymbol{a}\|_{\ell_1}}2}   C_n^d C_{\|\boldsymbol{a}\|_{\ell^1}+q}^{\|\boldsymbol{a}\|_{\ell^1}}
e^{-\eta ( \| \boldsymbol{k}^\flat \|_{\ell^1} +  2\| \boldsymbol{m}^\flat \|_{\ell^1}  + q +\|\boldsymbol{a}\|_{\ell^1}+2n)}  \\ &\times r_\flat^{-2  \| \boldsymbol{m}^\flat \|_{\ell^1} -  q-\|\boldsymbol{a}\|_{\ell^1}-2n} \prod_{1\leq p \leq N} \rho^{-2 \boldsymbol{m}_{\boldsymbol{i}_p}} \varpi_{\boldsymbol{i}_p}^{\boldsymbol{a}_p} |C_{\boldsymbol{a}_p/2}^{\boldsymbol{m}_{\boldsymbol{i}_p}}  |  .
\end{split}
\end{equation*}
Then, using that $C_{\|\boldsymbol{a}\|_{\ell^1}+q}^{\|\boldsymbol{a}\|_{\ell^1}} \leq 2^{\|\boldsymbol{a}\|_{\ell^1}+q}$, $C_n^d \leq 2^n$ and $\| \boldsymbol{k} \|_{\ell^1}  - \| \boldsymbol{k}^\flat \|_{\ell^1}  =\|\boldsymbol{k}_{\boldsymbol{i}}\|_{\ell^1}\leq \|\boldsymbol{a}\|_{\ell^1}$, we have that 
\begin{equation}
\label{eq:cequonveutcontroler}
e^{\eta ( \| \boldsymbol{k} \|_{\ell^1} +  2\| \boldsymbol{m}^\flat \|_{\ell^1} +q+2d)} r_\flat^{2  \| \boldsymbol{m}^\flat \|_{\ell^1} +q +2d } \rho^{2 \|\boldsymbol{m}_{\boldsymbol{i}} \|_{\ell^1}} |R_{d,\boldsymbol{k},\boldsymbol{m}}^{\boldsymbol{\ell},\boldsymbol{\sigma}}(\xi)| \Theta_{\boldsymbol{\ell}}^{-1}
\end{equation}
 is smaller than
$$
\sum_{\boldsymbol{a} \in \mathbb{N}^N}   \sum_{n\geq d}  2^{2\|\boldsymbol{a} \|_{\ell^1} + n +q}    (N\Lambda\rho^2)^{n-d+\frac{\|\boldsymbol{a}\|_{\ell_1}}2}
  r_\flat^{ -\|\boldsymbol{a}\|_{\ell_1}-2(n-d)}   \prod_{1\leq p \leq N}  \varpi_{\boldsymbol{i}_p}^{\boldsymbol{a}_p} |C_{\boldsymbol{a}_p/2}^{\boldsymbol{m}_{\boldsymbol{i}_p}}  |.
$$
Provided that $9(N\Lambda)^{1/2} \rho \leq  r_\flat$, summing in $n$, we get
$$
\eqref{eq:cequonveutcontroler} \lesssim  2^d\sum_{\boldsymbol{a} \in \mathbb{N}^N}   2^{2\|\boldsymbol{a} \|_{\ell^1}+q} (N\Lambda\rho^2)^{\frac{\|\boldsymbol{a} \|_{\ell^1}}2}     r_\flat^{ -\|\boldsymbol{a}\|_{\ell_1}}   \prod_{1\leq p \leq N}  \varpi_{\boldsymbol{i}_p}^{\boldsymbol{a}_p} |C_{\boldsymbol{a}_p/2}^{\boldsymbol{m}_{\boldsymbol{i}_p}}  |.
$$
Now, to estimate $C_{a/2}^{m} $, when $a/2$ is possibly not an integer, we apply the Cauchy estimate to the holomorphic function
$$
g(z) := \sum_{m\geq 0} C_{a/2}^{m} z^m = (1+z)^{a/2} \quad \mathrm{on} \quad \mathbb{D}=\{ z \in \mathbb{C} \ | \ |z|< 1 \}
$$
 and we get that
$$
|C_{a/2}^{m}| = \big|\frac{g^{(m)}(0)}{m!} \big| \leq 2^m \sup_{|z|=1/2} |g(z)| \leq 2^m \big(\frac{3}2 \big)^{a/2}.
$$
It follows that provided that $10^3 N^{3/2}\Lambda^{1/2} \rho \, \max_{1\leq p \leq N}( \varpi_{\boldsymbol{i}_p} ) \leq  r_\flat $, we have
$$
\eqref{eq:cequonveutcontroler} \lesssim  2^{d+q+\| \boldsymbol{m}_{\boldsymbol{i}} \|_{\ell^1}}\sum_{\boldsymbol{a} \in \mathbb{N}^N}    (6 N^{1/2}\Lambda^{1/2}\rho r_\flat^{-1} )^{\|\boldsymbol{a} \|_{\ell^1}}      \lesssim   2^{d+q+\| \boldsymbol{m}_{\boldsymbol{i}} \|_{\ell^1}}.
$$
As a consequence, provided that $\rho\leq r_\flat$, we have proven that
$$
|R_{d,\boldsymbol{k},\boldsymbol{m}}^{\boldsymbol{\ell},\boldsymbol{\sigma}}(\xi)|  \lesssim e^{-\eta ( \| \boldsymbol{k} \|_{\ell^1} +  2\| \boldsymbol{m}^\flat \|_{\ell^1}  + q +2d )} \rho^{-2\| \boldsymbol{m}\|_{\ell^1} -q -2d } 2^{d+q+\| \boldsymbol{m}_{\boldsymbol{i}} \|_{\ell^1}} \Theta_{\boldsymbol{\ell}}
$$
and provided that $9r\leq   \rho e^{-\eta}$, we get (as expected)
$$
|R_{d,\boldsymbol{k},\boldsymbol{m}}^{\boldsymbol{\ell},\boldsymbol{\sigma}}(\xi)| \Theta_{\boldsymbol{\ell}}^{-1} e^{\eta ( \| \boldsymbol{k} \|_{\ell^1} +  2\| \boldsymbol{m} \|_{\ell^1}  + q +2d)}   r^{2  \| \boldsymbol{m} \|_{\ell^1} +  q+2d} 
\lesssim   2^{d+q+\| \boldsymbol{m}_{\boldsymbol{i}} \|_{\ell^1}} e^{2\eta \| \boldsymbol{m}_{\boldsymbol{i}} \|_{\ell^1} } \big( \frac{r}{\rho} \big)^{2 \| \boldsymbol{m}_{\boldsymbol{i}} \|_{\ell^1}+q+2d}  \lesssim 1,
$$
i.e. $
\| R \|_{\mathscr{H}_{\eta,r,\mathcal{O},\mathcal{S}}^{\mathrm{sup}}} \lesssim 1.
$

\medskip

\noindent $\bullet$ \underline{\emph{Step 3 : Lipschitz estimates.}} Being given $\xi,\zeta \in \mathcal{O}$, we have
\begin{equation}
\label{eq:itsimlpe}
|R_{d,\boldsymbol{k},\boldsymbol{m}}^{\boldsymbol{\ell},\boldsymbol{\sigma}}(\xi)- R_{d,\boldsymbol{k},\boldsymbol{m}}^{\boldsymbol{\ell},\boldsymbol{\sigma}}(\zeta)| \leq \underbrace{|R_{d,\boldsymbol{k},\boldsymbol{m}}^{\boldsymbol{\ell},\boldsymbol{\sigma}}(\xi^\flat,\zeta^{\mathrm{new}} )- R_{d,\boldsymbol{k},\boldsymbol{m}}^{\boldsymbol{\ell},\boldsymbol{\sigma}}(\zeta)|}_{=I}+ \underbrace{|R_{d,\boldsymbol{k},\boldsymbol{m}}^{\boldsymbol{\ell},\boldsymbol{\sigma}}(\xi)- R_{d,\boldsymbol{k},\boldsymbol{m}}^{\boldsymbol{\ell},\boldsymbol{\sigma}}(\xi^\flat,\zeta^{\mathrm{new}})|}_{=II} .
\end{equation}
First, we note that, by definition of $\|\cdot \|_{\mathscr{H}_{\eta^\flat,r_\flat,\mathcal{O}^\flat,\mathcal{S}^\flat}^{\mathrm{tot}}}$, the estimates of the previous step prove that
\begin{equation}
\label{eq:ac1}
I \lesssim r_\flat^{-2}  \Theta_{\boldsymbol{\ell}} e^{-\eta ( \| \boldsymbol{k} \|_{\ell^1} +  2\| \boldsymbol{m} \|_{\ell^1}  + q +2d)}   r^{-2  \| \boldsymbol{m} \|_{\ell^1} -  q-2d} \| \xi^\flat - \zeta^\flat \|_{\ell^1}.
\end{equation}
So we only have to focus on II. Coming back to the equation \eqref{eq:ilfallaitbiensouffler} giving the coefficients of $R$, we get (with multi-index notations)
\begin{equation*}
\begin{split}
II \leq& \sum_{\boldsymbol{a} \in \mathbb{N}^N}  \sum_{ \boldsymbol{\varsigma} \in \mathcal{R}_{\boldsymbol{a},\boldsymbol{k}} } \sum_{n\geq d}  C_{\|\boldsymbol{a}\|_{\ell^1}+q}^{\|\boldsymbol{a}\|_{\ell^1}} C_n^d  |  P_{n,\boldsymbol{k}^\flat,\boldsymbol{m}^\flat}^{\boldsymbol{\ell},\boldsymbol{i}^{.\boldsymbol{a}},\boldsymbol{\sigma}, \boldsymbol{\varsigma}}(\xi^\flat) | \big(\prod_{p=1}^N | C_{\boldsymbol{a}_p/2}^{\boldsymbol{m}_{\boldsymbol{i}_p}} |  \big) \\ &\times \underbrace{\Big[ | \|\xi_{\boldsymbol{i}} \|_{\ell^1}^{n-d} - \|\zeta_{\boldsymbol{i}} \|_{\ell^1}^{n-d} | \xi_{\boldsymbol{i}}^{\boldsymbol{a}/2-\boldsymbol{m}_{\boldsymbol{i}}} | +  \|\zeta_{\boldsymbol{i}} \|_{\ell^1}^{n-d} |  \xi_{\boldsymbol{i}}^{\boldsymbol{a}/2-\boldsymbol{m}_{\boldsymbol{i}}} - \zeta_{\boldsymbol{i}}^{\boldsymbol{a}/2-\boldsymbol{m}_{\boldsymbol{i}}}  |\Big] }_{=\Delta}
\end{split}
\end{equation*}
Then, recalling that $\zeta_{\boldsymbol{i}},\xi_{\boldsymbol{i}} \in (\rho^{2},\Lambda\rho^{2})^N$, by applying the mean value inequality, we have
\begin{equation*}
\begin{split}
\frac{\Delta}{\| \xi_{\boldsymbol{i}} - \zeta_{\boldsymbol{i}}  \|_{\ell^1}} &\leq (n-d +  \frac{\| \boldsymbol{a} \|_{\ell^1}}2 + \| \boldsymbol{m}_{\boldsymbol{i}} \|_{\ell^1}        )  (N \Lambda)^{\frac{\| \boldsymbol{a} \|_{\ell^1}}2 + \| \boldsymbol{m}_{\boldsymbol{i}} \|_{\ell^1} +n-d-1} ( \rho^2 )^{\frac{\| \boldsymbol{a} \|_{\ell^1}}2 - \| \boldsymbol{m}_{\boldsymbol{i}} \|_{\ell^1} +n-d-1}  \\
&\lesssim \rho^{-2}         (2 N \Lambda^2)^{\frac{\| \boldsymbol{a} \|_{\ell^1}}2 + \| \boldsymbol{m}_{\boldsymbol{i}} \|_{\ell^1} +n-d}  \xi_{\boldsymbol{i}}^{\boldsymbol{a}/2-\boldsymbol{m}_{\boldsymbol{i}}}  \|\xi_{\boldsymbol{i}} \|_{\ell^1}^{n-d}
\end{split}
\end{equation*}

The estimates are almost the same as the ones at the previous step. The only difference is that we have to strengthen a little bit the smallness assumptions on $r$ and $\rho$ (see \eqref{eq:smallr2}) to absorb the losses generated by the factor $   (2 N \Lambda^2)^{\frac{\| \boldsymbol{a} \|_{\ell^1}}2 + \| \boldsymbol{m}_{\boldsymbol{i}} \|_{\ell^1} +n+d} $.  As a consequence, we have that 
\begin{equation}
\label{eq:ac2}
II \lesssim \rho^{-2}  \Theta_{\boldsymbol{\ell}} e^{-\eta ( \| \boldsymbol{k} \|_{\ell^1} +  2\| \boldsymbol{m} \|_{\ell^1}  + q +2d)}   r^{-2  \| \boldsymbol{m} \|_{\ell^1} -  q-2d}  \| \xi_{\boldsymbol{i}} - \zeta_{\boldsymbol{i}}  \|_{\ell^1}.
\end{equation}
Putting together \eqref{eq:itsimlpe}, \eqref{eq:ac1} and \eqref{eq:ac2}, we have proven that
$
\| R \|_{\mathscr{H}_{\eta,r,\mathcal{O},\mathcal{S}}^{\mathrm{lip}}} \lesssim \rho^{-2}
$
which is slightly better than what we aimed at proving.
\end{proof}

\subsection{The normal form case}
\begin{proposition} \label{prop:opennormal} Let $\mathcal{S}^\flat \subset \mathcal{S}\subset \mathbb{Z}$ be some finite sets, $r_\flat,\rho,r\in (0,1/4)$, $\mathcal{O}^\flat\subset ((4r_\flat)^{2},1)^{\mathcal{S}^\flat}$ be a Borel set  and $N^\flat\in \mathscr{N}_{\mathcal{O}^{\flat},\mathcal{S}^{\flat}}$.

\medskip

\noindent Assume that the following assumptions hold
\begin{itemize}
\item $\mathcal{S}$ contains exactly one extra site, i.e.  
$
\# \mathcal{S} = 1+\# \mathcal{S}^{\flat},
$
\item the radii enjoy the estimate
$
r \lesssim \rho \lesssim_{\mathcal{S}} r_\flat.$
\end{itemize}

\medskip

\noindent Then, there exists a Hamiltonian $N\in \mathscr{N}_{\mathcal{O},\mathcal{S}}$, where
$$
\mathcal{O} := \mathcal{O}^\flat\times (\rho^2,2\rho^2),
$$ 
such that
\begin{equation}
\label{eq:bound_open2}
\| N\|_{\mathscr{N}_{r,\mathcal{O},\mathcal{S}}^{\mathrm{tot}}} \leq \big(1+\frac{\rho}{r_{\flat}}\big)\|N^\flat\|_{\mathscr{N}_{r_\flat,\mathcal{O}^{\flat},\mathcal{S}^{\flat}}^{\mathrm{tot}}} 
\end{equation}
and for all $\xi=:(\xi^\flat,\xi^{\mathrm{new}})\in \mathcal{O}$ and all $u\in \mathcal{A}_{\xi}(3r)$, 
$$
N(\xi;u) = N^\flat(\xi^\flat;u).
$$
Moreover, setting $L = \Pi_{\mathscr{L}}N$, $L^\flat = \Pi_{\mathscr{L}}N^\flat$, $A_0+A_1 \mu = \Pi_{\mathscr{A}}N$ and $A_0^\flat+A_1^\flat \mu = \Pi_{\mathscr{A}}N^\flat$, for all $\xi=:(\xi^\flat,\xi^{\mathrm{new}})\in \mathcal{O}$ and $k \in \mathbb{Z}$, we have
\begin{equation}
\label{eq:cvg_freq}
 |A_1^\flat(\xi^\flat) - A_1(\xi)|+ \langle k \rangle^{-\delta}  |L_k^\flat(\xi^\flat) - L_k(\xi)|\leq \rho \|N^\flat\|_{\mathscr{N}_{r_\flat,\mathcal{O}^{\flat},\mathcal{S}^{\flat}}^{\mathrm{tot}}}.
\end{equation}
\end{proposition}
\begin{proof} We decompose the proof in $4$ steps. We denote by $i_* \in \mathbb{Z}\setminus \mathcal{S}^\flat$ the index such that
$$
\mathcal{S} = \{ i_* \} \cup \mathcal{S}^\flat.
$$

\medskip

\noindent $\bullet$ \underline{\emph{Step 1 : } Construction.} Naturally, we decompose $N^\flat$ as
$$
N^\flat = A^\flat + L^\flat + Q^\flat \in \mathscr{A}_{\mathcal{O}^\flat,\mathcal{S}^\flat} \oplus \mathscr{L}_{\mathcal{O}^\flat,\mathcal{S}^\flat} \oplus \mathscr{Q}_{\mathcal{O}^\flat,\mathcal{S}^\flat}.
$$
Then, we note that for all $\xi \in (\xi^\flat,\xi_{i_*})\in \mathcal{O}$, $u\in \ell^2_{\varpi}$ and $k\in \mathbb{Z}$ we have
$$
\mu(\xi;u) = \mu(\xi^\flat;u) - \xi_{i_*} \quad \mathrm{and} \quad y_k(\xi;u) = y_k(\xi^\flat;u) - \mathbbm{1}_{k=i_*} \xi_{i_*},
$$
It follows that for all $\xi=:(\xi^\flat,\xi_{i_*})\in \mathcal{O}$ and all $u\in \mathcal{A}_{\xi}(3r)$, 
$$
A^\flat(\xi^\flat;u) = A_0^\flat(\xi^\flat) + A_1^\flat(\xi^\flat) \mu(\xi;u) + A_1^\flat(\xi^\flat) \xi_{i_*}, 
$$
$$
L^\flat(\xi^\flat;u) = \big(  \sum_{k\in \mathbb{Z}} L_k^\flat(\xi^\flat) y_k(\xi;u) \big) + L_{i_*}^\flat(\xi^\flat)  \xi_{i_*} , 
$$
\begin{equation*}
\begin{split}
Q^\flat(\xi^\flat;u) =&  \sum_{k,\ell\in \mathbb{Z}} Q^\flat_{k,\ell}(\xi^\flat) y_k(\xi;u) y_\ell(\xi;u) + \sum_{k\in \mathbb{Z}} Q^\flat_{k}(\xi^\flat) y_k(\xi;u) \mu(\xi;u) + Q^\flat_{\emptyset}(\xi^\flat) \mu(\xi;u)^2  \\
&+ 2\sum_{k\in \mathbb{Z}}  Q^\flat_{k,i_*}(\xi^\flat) y_k(\xi;u) \xi_{i_*} + Q_{i_*,i_*}^\flat(\xi^\flat)  \xi_{i_*}^2 +  \sum_{k\in \mathbb{Z}} Q^\flat_{k}(\xi^\flat) y_k(\xi;u) \xi_{i_*} + Q^\flat_{i_*}(\xi^\flat) \xi_{i_*}  \mu(\xi;u)  \\
&+ \xi_{i_*} ^2  Q^\flat_{i_*}(\xi^\flat)+  \xi_{i_*} ^2  Q^\flat_{\emptyset}(\xi^\flat) + 2  Q^\flat_{\emptyset}(\xi^\flat) \xi_{i_*}    \mu(\xi;u).
\end{split}
\end{equation*}
As a consequence, setting (for $k,\ell \in \mathbb{Z}$)
$$
Q_{k,\ell}(\xi) := Q^\flat_{k,\ell}(\xi^\flat) , \quad Q_{k}(\xi) := Q^\flat_{k}(\xi^\flat), \quad Q_{\emptyset}(\xi) := Q^\flat_{\emptyset}(\xi^\flat),
$$
\begin{equation}
\label{eq:sans_doutes_le_truc_le_plus_important_du_papier}
L_k(\xi) := L_k^\flat(\xi^\flat) + 2\xi_{i_*}  Q^\flat_{k,i_*}(\xi^\flat) +  \xi_{i_*}  Q^\flat_{k}(\xi^\flat),
\end{equation}
$$
A_1( \xi) :=A_1^\flat( \xi^\flat)+Q^\flat_{i_*}(\xi^\flat) \xi_{i_*}  +2  Q^\flat_{\emptyset}(\xi^\flat) \xi_{i_*},
$$
$$
A_0( \xi) :=A_0( \xi^\flat)  + Q_{i_*,i_*}(\xi^\flat)  \xi_{i_*}^2 + \xi_{i_*} ^2  Q^\flat_{i_*}(\xi^\flat)+  \xi_{i_*} ^2  Q^\flat_{\emptyset}(\xi^\flat) + A_1(\xi^\flat) \xi_{i_*} + L_{i_*}(\xi^\flat) \xi_{i_*}
$$
provided that $N=A+L+Q \in \mathcal{N}_{\mathcal{O},\mathcal{S}}$ (that we will prove at the next steps), we have that
$$
N^\flat(\xi^\flat;u) = N(\xi;u).
$$

\medskip

\noindent $\bullet$ \underline{\emph{Step 2 : } Non critical estimates.} We estimate the different terms of $\| N\|_{\mathscr{N}_{r,\mathcal{O},\mathcal{S}}^{\mathrm{tot}}}$ separately. Only the estimate of $\| L\|_{\mathscr{L}_{\mathcal{O},\mathcal{S}}^\mathrm{dir}}$ and $\| A\|_{\mathscr{A}_{\mathcal{O},\mathcal{S}}^\mathrm{lip}}$ are critical and requires a careful analysis (done in the next steps). So we begin with the other ones.

\begin{itemize}
\item Fist, we note that since the coefficients of $Q$ and $Q^\flat$ are the same, we have $Q\in \mathscr{Q}_{\mathcal{O},\mathcal{S}}$ and (since $r\leq r_\flat$)
$$
\| Q\|_{\mathscr{Q}_{r,\mathcal{O},\mathcal{S}}^\mathrm{tot}} \leq \| Q^\flat\|_{\mathscr{Q}_{r_\flat,\mathcal{O}^\flat,\mathcal{S}^\flat}^\mathrm{tot}} \leq 3^{-1}\|N^\flat\|_{\mathscr{N}_{r_\flat,\mathcal{O}^{\flat},\mathcal{S}^{\flat}}^{\mathrm{tot}}}.
$$
\item Then we note that, provided that $\rho \lesssim_{\mathcal{S}} 1$, since $|\xi_{i_*}|\leq 2\rho^2$, we have for all $k\in \mathbb{Z}$
$$
|L_k(\xi) - L_k^\flat(\xi^\flat) | \leq \| Q^\flat\|_{\mathscr{Q}_{\mathcal{O}^\flat,\mathcal{S}^\flat}^\mathrm{sup}}  ( 2|\xi_{i_*}| \langle i_* \rangle^{4\delta} \langle k\rangle^{-\delta} + 2|\xi_{i_*}|  \langle k\rangle^{-\delta} ) \leq \frac{\rho}2 \langle k\rangle^{-\delta}  \| Q^\flat\|_{\mathscr{Q}_{\mathcal{O}^\flat,\mathcal{S}^\flat}^\mathrm{sup}} .
$$
and so 
$$
\| L\|_{\mathscr{L}_{\mathcal{O},\mathcal{S}}^\mathrm{sup}} \leq  \| L^\flat\|_{\mathscr{L}_{\mathcal{O}^\flat,\mathcal{S}^\flat}^\mathrm{sup}} + \rho \| Q^\flat\|_{\mathscr{Q}_{\mathcal{O}^\flat,\mathcal{S}^\flat}^\mathrm{sup}}   \leq (1+ \rho) \|N^\flat\|_{\mathscr{N}_{r_\flat,\mathcal{O}^{\flat},\mathcal{S}^{\flat}}^{\mathrm{tot}}}.
$$
 Similarly, we prove that
$
\| A\|_{\mathscr{A}_{\mathcal{O},\mathcal{S}}^\mathrm{sup}}   \leq (1+ \rho) \|N^\flat\|_{\mathscr{N}_{r_\flat,\mathcal{O}^{\flat},\mathcal{S}^{\flat}}^{\mathrm{tot}}}
$
and
$$
|A_1(\xi) - A_1^\flat(\xi^\flat) | \leq \frac{\rho}2 \| Q^\flat\|_{\mathscr{Q}_{\mathcal{O}^\flat,\mathcal{S}^\flat}^\mathrm{sup}} .
$$
Note that these estimates on $A$ and $L$ implies \eqref{eq:cvg_freq}.
\item Finally, we estimate  $\| L \|_{\mathscr{L}_{\mathcal{O},\mathcal{S}}^\mathrm{lip}}$. It is not critical because the weight $r^2$ in the definition on the "total" norm allows to absorb almost any loss. Indeed, being given $\xi,\zeta \in \mathcal{O}$ and $k\in \mathbb{Z}$, noticing that $\Lambda_{i_*,k} \leq \langle k \rangle^{-\delta}\langle i_* \rangle^{4\delta} $, we have
\begin{equation*}
\begin{split}
|L_k(\xi) - L_k(\zeta) | \leq& | L_k^\flat(\xi^\flat) -L_k^\flat(\zeta^\flat) |    + 2|\xi_{i_*}| | Q^\flat_{k,i_*}(\xi^\flat) -Q^\flat_{k,i_*}(\zeta^\flat) |+  |\xi_{i_*} | |Q^\flat_{k}(\xi^\flat) - Q^\flat_{k}(\zeta^\flat)| \\
&+  2|\xi_{i_*} -\zeta_{i_*}| | Q^\flat_{k,i_*}(\zeta^\flat) |+  |\xi_{i_*}-\zeta_{i_*} | | Q^\flat_{k}(\zeta^\flat)| \\
\leq& \| \xi^\flat - \zeta^\flat \|_{\ell^1} ( \|L^\flat \|_{\mathscr{L}_{\mathcal{O}^\flat,\mathcal{S}^\flat}^\mathrm{lip}} \langle k\rangle^{-\delta}  + 4 \rho^2 \|Q^\flat \|_{\mathscr{Q}_{\mathcal{O}^\flat,\mathcal{S}^\flat}^\mathrm{lip}} \Lambda_{i_*,k} +2\rho^2 \|Q^\flat \|_{\mathscr{Q}_{\mathcal{O}^\flat,\mathcal{S}^\flat}^\mathrm{lip}} \langle k\rangle^{-\delta}  )\\
&+ |\xi_{i_*} -\zeta_{i_*}| ( 2 \|Q^\flat \|_{\mathscr{Q}_{\mathcal{O}^\flat,\mathcal{S}^\flat}^\mathrm{sup}} \Lambda_{i_*,k} + \|Q^\flat \|_{\mathscr{Q}_{\mathcal{O}^\flat,\mathcal{S}^\flat}^\mathrm{sup}} \langle k \rangle^{-\delta})\\
\leq& \|N^\flat\|_{\mathscr{N}_{r_\flat,\mathcal{O}^{\flat},\mathcal{S}^{\flat}}^{\mathrm{tot}}} \langle k \rangle^{-\delta} \left( \| \xi^\flat - \zeta^\flat \|_{\ell^1} \Big[r_\flat^{-2}+ 6 \big(\frac{\rho}{r_\flat}\big)^2 \langle i_* \rangle^{4\delta}  \Big] +3 \langle i_* \rangle^4   |\xi_{i_*} -\zeta_{i_*}|  \right).
\end{split}
\end{equation*}
Provided that $r\lesssim_\mathcal{S} r_\flat$, we deduce that
$$
r^2|L_k(\xi) - L_k(\zeta) | \leq \|N^\flat\|_{\mathscr{N}_{r_\flat,\mathcal{O}^{\flat},\mathcal{S}^{\flat}}^{\mathrm{tot}}} \langle k \rangle^{-\delta}  \| \xi -\zeta\|_{\ell^1}
$$
and so $r^2 \| L \|_{\mathscr{L}_{\mathcal{O},\mathcal{S}}^\mathrm{lip}} \leq  \|N^\flat\|_{\mathscr{N}_{r_\flat,\mathcal{O}^{\flat},\mathcal{S}^{\flat}}^{\mathrm{tot}}}$. 

\end{itemize}

\medskip

\noindent $\bullet$ \underline{\emph{Step 3 : } Estimate on $\| L\|_{\mathscr{L}_{\mathcal{O},\mathcal{S}}^\mathrm{dir}}.$} Now, we estimate $\| L\|_{\mathscr{L}_{\mathcal{O},\mathcal{S}}^\mathrm{dir}}$.
Let $i\in \mathcal{S}^\flat$, $k\in \mathbb{Z}$, $(\zeta,\xi)\in \Delta_i(\mathcal{O})$. We recall that by definition of $\Delta_i(\mathcal{O})$, $\|\zeta-\xi\|_{\ell^1}=|\xi_i-\zeta_i|$ and $\xi_{i_*}  = \zeta_{i_*}$. Therefore, we have that
\begin{equation}
\label{eq:dir1}
\begin{split}
|L_k(\xi) - L_k(\zeta) | &\leq | L_k^\flat(\xi^\flat) -L_k^\flat(\zeta^\flat) |    + 2|\xi_{i_*}| | Q^\flat_{k,i_*}(\xi^\flat) -Q^\flat_{k,i_*}(\zeta^\flat) |+  |\xi_{i_*} | |Q^\flat_{k}(\xi^\flat) - Q^\flat_{k}(\zeta^\flat)| \\
&\leq\|L^\flat \|_{\mathscr{L}_{\mathcal{O}^\flat,\mathcal{S}^\flat}^\mathrm{dir}} |\xi_i - \zeta_i| \Gamma_{i,k}  + 2\rho^2 \|Q^\flat \|_{\mathscr{Q}_{\mathcal{O}^\flat,\mathcal{S}^\flat}^\mathrm{lip}} (2\Lambda_{i,k} + \langle k\rangle^{-\delta}) |\xi_i - \zeta_i|\\
&\leq \|N^\flat\|_{\mathscr{N}_{r_\flat,\mathcal{O}^{\flat},\mathcal{S}^{\flat}}^{\mathrm{tot}}}  |\xi_i - \zeta_i|\Gamma_{i,k} ( 1 + 6\big(\frac{\rho}{r_\flat}\big)^2  ).
\end{split}
\end{equation}
Now, let $k\in \mathbb{Z}$, $(\zeta,\xi)\in \Delta_{i_*}(\mathcal{O})$. We recall that by definition of $\Delta_{i_*}(\mathcal{O})$, we have $\xi^\flat = \zeta^\flat$. It follows that\footnote{recall the factor $3$ in the definition of the $\mathscr{N}^{\mathrm{tot}}$ norm.}
\begin{equation}
\label{eq:dir2}
\begin{split}
|L_k(\xi) - L_k(\zeta) | &\leq  |\xi_{i_*}-\zeta_{i_*}| (2| Q^\flat_{k,i_*}(\xi^\flat) |+ |Q^\flat_{k}(\xi^\flat)|)\\
&\leq  |\xi_{i_*}-\zeta_{i_*}| \|Q^\flat \|_{\mathscr{Q}_{\mathcal{O}^\flat,\mathcal{S}^\flat}^\mathrm{sup}} (2 \Lambda_{i_*,k}+ \langle k \rangle^{-\delta}) \\
&\leq  3 |\xi_{i_*}-\zeta_{i_*}| \|Q^\flat \|_{\mathscr{Q}_{\mathcal{O}^\flat,\mathcal{S}^\flat}^\mathrm{sup}}  \Gamma_{i_*,k} \\
&\leq  |\xi_{i_*}-\zeta_{i_*}| \|N^\flat\|_{\mathscr{N}_{r_\flat,\mathcal{O}^{\flat},\mathcal{S}^{\flat}}^{\mathrm{tot}}} 
\end{split}
\end{equation}
Putting together the estimates \eqref{eq:dir1} and \eqref{eq:dir2}, provided that $6\rho \leq r_\flat$, we get that
$$
\| L\|_{\mathscr{L}_{\mathcal{O},\mathcal{S}}^\mathrm{dir}} \leq \Big( 1 + \frac{\rho}{r_\flat}  \Big) \|N^\flat\|_{\mathscr{N}_{r_\flat,\mathcal{O}^{\flat},\mathcal{S}^{\flat}}^{\mathrm{tot}}}  .
$$

\medskip

\noindent $\bullet$ \underline{\emph{Step 4 : } Estimate on $\| A\|_{\mathscr{A}_{\mathcal{O},\mathcal{S}}^\mathrm{lip}} .$} First, we focus on $A_1$. Let $\xi,\zeta \in \mathcal{O}$. Proceeding as previously, provided that $6\rho \leq r_\flat$, we get that
\begin{equation}
\label{eq:lipA1}
\begin{split}
|A_1(\xi) - A_1(\zeta) |& \leq  \| N^\flat\|_{\mathscr{N}_{r_\flat,\mathcal{O}^\flat,\mathcal{S}^\flat}^\mathrm{tot}}  \|  \xi^\flat -\zeta^\flat\|_{\ell^1} \Big( 1 + 6\frac{\rho^2}{r_\flat^2}  \Big)+  3 \|  Q^\flat \|_{\mathscr{Q}_{\mathcal{O}^\flat,\mathcal{S}^\flat}^\mathrm{sup}} |\xi_{i_*} - \zeta_{i_*}  | \\
 &\leq \| N^\flat\|_{\mathscr{N}_{r_\flat,\mathcal{O}^\flat,\mathcal{S}^\flat}^\mathrm{tot}}  \|  \xi -\zeta\|_{\ell^1} \Big( 1 + \frac{\rho}{r_\flat}  \Big) .
\end{split}
\end{equation}
Then, we focus on $A_0$. Proceeding as previously, provided that $9\rho \leq r_\flat$, we get that
\begin{equation}
\label{eq:lipA0}
\begin{split}
|A_0(\xi) - A_0(\zeta) |& \leq  \| N^\flat\|_{\mathscr{N}_{r_\flat,\mathcal{O}^\flat,\mathcal{S}^\flat}^\mathrm{tot}} \Big(  \|  \xi^\flat -\zeta^\flat\|_{\ell^1} \Big( 3 + 12\frac{\rho^4}{r_\flat^2} + 4\frac{\rho^2}{r_\flat^2}   \Big)+  (2+12 \rho^2)|\xi_{i_*} - \zeta_{i_*}  | \Big) \\
 &\leq 3\| N^\flat\|_{\mathscr{N}_{r_\flat,\mathcal{O}^\flat,\mathcal{S}^\flat}^\mathrm{tot}}  \|  \xi -\zeta\|_{\ell^1} \Big( 1 + \frac{\rho}{r_\flat}  \Big) .
\end{split}
\end{equation}
Finally, putting together the estimates \eqref{eq:lipA1} and \eqref{eq:lipA0}, provided that $9\rho \leq r_\flat$, we get that
$$
\| A\|_{\mathscr{A}_{\mathcal{O},\mathcal{S}}^\mathrm{lip}} \leq \Big( 1 + \frac{\rho}{r_\flat}  \Big) \|N^\flat\|_{\mathscr{N}_{r_\flat,\mathcal{O}^{\flat},\mathcal{S}^{\flat}}^{\mathrm{tot}}}  .
$$
\end{proof}

\section{The loop}\label{sec:loop}

In this central  section we describe the process to put in normal form (i.e. to eliminate the adapted jet of) the perturbation after an opening step.  More explicitly, given a Hamiltonian $K^\flat\in \mathscr{H}_{\eta^\flat,r_\flat,\mathcal{O}^\flat,\mathcal{S}^\flat}$ with $\Pi_{\mathrm{ajet} }K^\flat=0$, we want to construct a symplectic change of variable $\psi_\xi$ and another Hamiltonian $K \in \mathscr{H}_{\eta,r,\mathcal{O},\mathcal{S}}$ with\footnote{Note that the projection $\Pi_{\mathrm{ajet} }$ changes after each opening step.} $\Pi_{\mathrm{ajet} }K=0$
such that 
$$
(H^{(0)}_{\varepsilon}+\varepsilon K^\flat)(\xi^\flat; \psi_\xi(u)) = (H^{(0)}_{\varepsilon}+\varepsilon K)(\xi;u).
$$
Here it is crucial that  we open only one new site: $\# \mathcal{S} = \# \mathcal{S}^\flat+1$. This loop will be achieved through the combined application of our KAM theorem (Theorem \ref{thm:KAM}) and our Birkhoff theorem (Theorem \ref{thm:Birk}). 
%This loop is the basis of the proof of our main theorem: we will iterate this loop to obtain tori of arbitrary dimension.
\begin{proposition}[The loop] \label{prop:loop} Let $\mathcal{S}^\flat \subset \mathcal{S}\subset \mathbb{Z}$ be some finite sets, $r,r_\flat\in (0,1/4)$ , $\mathcal{O} \subset ((4r_\flat)^{2},1)^{\mathcal{S}^\flat}$ be a Borel set, $\eta_{\max}\geq \eta^\flat> \eta \geq \eta_0$, $\varepsilon<1$ and $K^\flat\in \mathscr{H}_{\eta^\flat,r_\flat,\mathcal{O}^\flat,\mathcal{S}^\flat}$.

\medskip

\noindent Assume that the following assumptions hold
\begin{itemize}
\item $K^\flat$ has no adapted jet, its remainder and its normal form part are not too large, i.e.
$$
 \Pi_{\mathrm{ajet}} K^{\flat}  =0  \quad \mathrm{and} \quad \| \Pi_{\mathrm{nor}} K^{\flat} \|_{\mathscr{N}_{r_\flat,\mathcal{O}^\flat,\mathcal{S}^\flat}^{\mathrm{tot}}} + \| \Pi_{\mathrm{rem}} K^\flat\|_{\mathscr{H}_{\eta^\flat,r_\flat,\mathcal{O}^\flat,\mathcal{S}^\flat}^{\mathrm{tot}}} \leq 3,
$$
\item  the frequencies can be modulated, i.e. 
$$
\varepsilon \leq 10^{-9} \quad \mathrm{and} \quad   \varepsilon \sup_{i\in \mathcal{S}} \sum_{k\in \mathcal{S}} \Gamma_{k,i} \leq 10^{-3} ,
$$
\item $\mathcal{S}$ contains exactly one extra site, i.e.  
$$
\# \mathcal{S} = 1+\# \mathcal{S}^\flat,
$$
\item $r$ is small enough
%\begin{equation}
%\label{eq:smallr}
%r_\flat \lesssim_{\mathcal{S}^\flat,\eta-\eta^{\flat}} 1 \quad \mathrm{and} \quad r \lesssim_{\mathcal{S},\eta-\eta^{\flat},r_\flat} 1.
%\end{equation}
\begin{equation}
\label{eq:smallr}
r \lesssim_{\mathcal{S},r_\flat,\varepsilon} 1.
\end{equation}
\end{itemize}

\medskip

\noindent Then, there exists a Borel set $\mathcal{O} \subset ((4r)^{2},1)^{\mathcal{S}}$ and a Hamiltonian $K \in \mathscr{H}_{\eta,r,\mathcal{O},\mathcal{S}}$
such that 
\begin{itemize}
\item $\mathcal{O}$ is smaller than $\mathcal{O}^\flat$ in the sense that\footnote{with a slight abuse of notation.}
$$
\mathcal{O} \subset \mathcal{O}^\flat \times (r^{2\nu},2r^{2\nu}) =: \mathcal{O}^{\flat,\mathrm{e}}, \quad \mathrm{where} \quad \nu = 19/20.
$$
\item $\mathcal{O}$ is large in the sense that
$$
\mathrm{Leb}(\mathcal{O}^{\flat,\mathrm{e}} \setminus \mathcal{O}) \leq  r^{2\nu+10^{-5}}  
$$
\item there is still no adapted jet and the remainder term is not too large, i.e.
$$
 \Pi_{\mathrm{ajet}} K  =0  \quad \mathrm{and} \quad \| \Pi_{\mathrm{rem}} K\|_{\mathscr{H}_{\eta,r,\mathcal{O},\mathcal{S}}^{\mathrm{tot}}} \leq 1,
$$
\item the normal form part does not grow too much, i.e.
$$
\|  \Pi_{\mathrm{nor}} K \|_{\mathscr{N}_{r,\mathcal{O},\mathcal{S}}^{\mathrm{tot}}} \leq \|  \Pi_{\mathrm{nor}} K^\flat \|_{\mathscr{N}_{r_\flat,\mathcal{O}^\flat,\mathcal{S}^\flat}^{\mathrm{tot}}} + r^{10^{-3}},
$$
\item  the frequencies do not move too much, i.e. for all $\xi =: (\xi^\flat,\xi^{\mathrm{new}}) \in \mathcal{O}$ and all $j\in \mathbb{Z}$, we have
\begin{equation}
\label{eq:freq_move_pas_trop_trop}
\langle j \rangle^{\delta} |L_j(\xi) - L^\flat_j(\xi^\flat) | +  |A_1(\xi) - A_1^\flat(\xi^\flat) |  \leq  r^{\frac1{82}}  
\end{equation} 
where we have set $L := \Pi_{\mathscr{L}}K $, $L^\flat := \Pi_{\mathscr{L}}K^\flat $,  $ A_0^\flat + \mu A_1^\flat := \Pi_{\mathscr{A}}K^\flat$ and $ A_0 + \mu A_1 := \Pi_{\mathscr{A}}K$,

\end{itemize}
and for all $\xi =: (\xi^\flat,\xi^{\mathrm{new}}) \in \mathcal{O}$, there exists a $C^1$ symplectic map $\psi_\xi : \mathcal{A}_\xi(r) \to \mathcal{A}_{\xi^\flat}(r_\flat)$ commuting with the gauge transform and enjoying,
 for all $u\in \mathcal{A}_\xi(r)$, the bounds
\begin{equation}
\label{eq:psi_aussi_est_proche_id}
\| \psi_\xi(u) - u\|_{\ell^2_\varpi} + \| \mathrm{d} \psi_\xi(u) - \mathrm{Id}\|_{\ell^2_\varpi \to \ell^2_\varpi } \leq r^{\frac1{440}}
\end{equation}
 and the property that
$$
(H^{(0)}_{\varepsilon}+\varepsilon K^\flat)(\xi^\flat; \psi_\xi(u)) = (H^{(0)}_{\varepsilon}+\varepsilon K)(\xi;u).
$$
\end{proposition}

\begin{remark}\label{rem:homo}
We are going to use (and abuse) zooms, i.e. contractions of the annulus on which we consider the Hamiltonian. These contractions allows to decrease the norm of the Hamiltonians (depending of their homogeneity). More precisely, being given a finite set $ \mathcal{S}\subset \mathbb{Z}$, $0<\rho\leq r <1/4$, a Borel set $\mathcal{O} \subset ((4r)^{2},1)^{\mathcal{S}}$, $p\geq 0$, $ \eta \geq \eta_0$ and  $P\in \mathscr{H}_{\eta,r,\mathcal{O},\mathcal{S}},$ we have
$$
\| \Pi_{2n + 2 \|\boldsymbol{m}\|_{\ell^1} +q \geq p} P \|_{\mathscr{H}_{\eta,\rho,\mathcal{O},\mathcal{S}}^{\mathrm{tot}}} \leq \big(\frac{\rho}{r}\big)^p \| P \|_{\mathscr{H}_{\eta,r,\mathcal{O},\mathcal{S}}^{\mathrm{tot}}} .
$$
As a consequence the new radius $r$ will be much smaller than the initial one $r_\flat$, a fact reflected in the assumption \eqref{eq:smallr}.
\end{remark}

\begin{proof}[Proof of Proposition \ref{prop:loop}] The proof is divided in $3$ steps. At the first and the last step, the radius of analyticity of the Hamiltonians will decrease by the same amount given by
$$
\eta^\flat - \eta^\natural  = \eta^{\natural} - \eta := \frac{\eta^\flat - \eta}2.
$$

\medskip

\noindent $\bullet$ \underline{\emph{Step $1$ : First zoom and Birkhoff normal form.}} First, we set 
$$
\rho := r^{1/20}, 
$$
and we note that, since $\Pi_{\mathrm{ajet}} K^\flat= 0$, by homogeneity, we have (see Remark \ref{rem:homo})
$$
\|  \Pi_{n=0,\boldsymbol{m}=0} (\mathrm{Id} -  \Pi_{\mathrm{nor}})  K^\flat\|_{\mathscr{H}_{\eta^\flat,\rho,\mathcal{O}^\flat,\mathcal{S}^\flat}^{\mathrm{tot}}}  \leq 3\big(\frac{\rho}{r_\flat}\big)^4, $$
$$ \| (\mathrm{Id}- \Pi_{n=0,\boldsymbol{m}=0})(\mathrm{Id} -  \Pi_{\mathrm{nor}} ) K^\flat\|_{\mathscr{H}_{\eta^\flat,\rho,\mathcal{O}^\flat,\mathcal{S}^\flat}^{\mathrm{tot}}}  \leq  3\big(\frac{\rho}{r_\flat}\big)^5.
$$
Moreover, by monotonicity, we have 
$$
\| \Pi_{\mathrm{nor}} K^\flat\|_{\mathscr{N}_{\rho,\mathcal{O}^\flat,\mathcal{S}^\flat}^{\mathrm{tot}}} \leq \| \Pi_{\mathrm{nor}} K^\flat\|_{\mathscr{N}_{r_\flat,\mathcal{O}^\flat,\mathcal{S}^\flat}^{\mathrm{tot}}} \leq 3
$$
We set
$$
\mathfrak{N} = \mathfrak{M} =  10^5.
$$
Provided that $r$ is small enough in the sense that it satisfies an estimate of the form \eqref{eq:smallr}, we have
$$
3\big(\frac{\rho}{r_\flat}\big)^4 \leq \rho^{7/2}, \quad \mathrm{and} \quad  3\big(\frac{\rho}{r_\flat}\big)^5 \leq \rho^{9/2}
$$
and so\footnote{even if it means adding another smallness assumption on $r$ of the form \eqref{eq:smallr}.}  the assumptions of the Birkhoff normal form theorem (i.e. Theorem \ref{thm:Birk}) are satisfied. As a consequence, we get  a Borel set $\mathcal{O}'\subset \mathcal{O}^\flat$, a Hamiltonian $R \in \mathscr{H}_{\eta^\natural,\rho,\mathcal{O}',\mathcal{S}^\flat}$ and a family of symplectic maps $\varphi_\xi$ satisfying the properties described in the statement of Theorem \ref{thm:Birk}.

\bigskip

\noindent $\bullet$ \underline{\emph{Step $2$ : Second zoom and opening of a new site}.} First, let us note that, by assumption, there exists a unique new site $i_* \in \mathbb{Z}\setminus \mathcal{S}^\flat$  such that
$$
\mathcal{S} = \{ i_* \} \cup \mathcal{S}^\flat.
$$
Then, we split $R$ in several pieces that are
$$
R =: N+ I + J + E
$$
where $N:= \Pi_{\mathrm{nor}} R$ is the normal form part of $R$, $I:= \Pi_{\mathrm{int}} R -N$ contains the other integrable terms of $R$ and $E$ contains the terms of $R$ which are not integrable and at least of order $4$ with respect to $(u_i)_{i\in \mathcal{S}^c}$.

\medskip

Then we note that $(\mathrm{Id}-\Pi_{\mathrm{tbr}}^\mathfrak{M}) J$ is at least of order $\mathfrak{M}$ with respect to $u_{i_*}$. Further by construction (concretely by \eqref{eq:RKbemol} in Theorem \ref{thm:Birk}), $\|  \Pi_{\mathrm{tbr}}^\mathfrak{M} J\|_{\mathscr{H}_{\eta^\natural,\rho,\mathcal{O}',\mathcal{S}^\flat}^{\mathrm{tot}}}\leq \rho^\mathfrak{N}$, it follows that, for homogeneity reasons (see Remark \ref{rem:homo}),
we have
$$
\| J\|_{\mathscr{H}_{\eta^\natural,\rho^{18},\mathcal{O}',\mathcal{S}^\flat}^{\mathrm{tot}}} \leq \rho^\mathfrak{N} + \big(\frac{\rho^{18}}{\rho} \big)^\mathfrak{M} \leq 2 r^{10^5/20} \leq r^{4500} \rho,
$$
$$
\| E\|_{\mathscr{H}_{\eta^\natural,\rho^{18},\mathcal{O}',\mathcal{S}^\flat}^{\mathrm{tot}}} \leq  \big(\frac{\rho^{18}}{\rho} \big)^4 = r^{3} \rho^2 \quad \mathrm{and} \quad \| I \|_{\mathscr{H}_{\eta^\natural,\rho^{18},\mathcal{O}',\mathcal{S}^\flat}^{\mathrm{tot}}} \leq  \big(\frac{\rho^{18}}{\rho} \big)^6 = r^5 \rho^2.
$$
We set
$$
\mathcal{O}^\natural := \mathcal{O}' \times  (\rho^{2\cdot 19},2\,\rho^{2\cdot  19}) \subset \mathbb{R}^{\mathcal{S}}.
$$
By applying Proposition \ref{prop:open}, since we can impose that $r$ is small enough in the sense of the estimate \eqref{eq:smallr}, we get $3$ Hamiltonians, $I^\natural,J^\natural,E^\natural \in \mathscr{H}_{\eta^\natural,r,\mathcal{O}^\natural ,\mathcal{S}}^{\mathrm{tot}}$, satisfying the estimates
\begin{equation}
\label{eq:merci_Birk}
\| I^\natural\|_{\mathscr{H}_{\eta^\natural,r,\mathcal{O}^\natural ,\mathcal{S}}^{\mathrm{tot}}} \leq r^5 \rho, \quad \| J^\natural\|_{\mathscr{H}_{\eta^\natural,r,\mathcal{O}^\natural ,\mathcal{S}}^{\mathrm{tot}}} \leq \varepsilon r^{4000}  \quad \mathrm{and} \quad \| E^\natural\|_{\mathscr{H}_{\eta^\natural,r,\mathcal{O}^\natural ,\mathcal{S}}^{\mathrm{tot}}} \leq r^{3} \rho 
\end{equation}
and such that for all $\xi=:(\xi^\flat,\xi^{\mathrm{new}})\in \mathcal{O}^\natural$ and all $u\in \mathcal{A}_{\xi}(3r)$, 
$$
I^\natural(\xi;u) = I(\xi^\flat;u), \quad J^\natural(\xi;u) = J(\xi^\flat;u) \quad \mathrm{and} \quad E^\natural(\xi;u) = E(\xi^\flat;u).
$$
Similarly, by applying Proposition \ref{prop:opennormal}, we get a Hamiltonian $N^\natural\in \mathscr{N}_{r,\mathcal{O}^\natural ,\mathcal{S}}$ in normal form  such that 
\begin{equation}
\begin{split}
\label{eq:est_normal}
\| N^\natural \|_{\mathscr{N}_{r,\mathcal{O}^\natural,\mathcal{S}}^{\mathrm{tot}}} \leq (1+\rho^{18})\|N\|_{\mathscr{N}_{\rho,\mathcal{O}',\mathcal{S}^{\flat}}^{\mathrm{tot}}}  &\mathop{\leq}^{\eqref{eq:RKbemol}} (1+\rho^{18}) ( \|  \Pi_{\mathrm{nor}} K^\flat \|_{\mathscr{N}_{r_\flat,\mathcal{O}',\mathcal{S}^\flat}^{\mathrm{tot}}} +\rho^{1/4}) \\ &\leq  \|  \Pi_{\mathrm{nor}} K^\flat \|_{\mathscr{N}_{r_\flat,\mathcal{O}^\flat,\mathcal{S}^\flat}^{\mathrm{tot}}} +\rho^{1/5},
\end{split}
\end{equation}
for all $\xi=:(\xi^\flat,\xi^{\mathrm{new}})\in \mathcal{O}^\natural$ and all $u\in \mathcal{A}_{\xi}(3r)$, 
$$
N^\natural(\xi;u) = N(\xi^\flat;u)
$$
and, for all $j\in \mathbb{Z}$
\begin{equation}
\label{eq:var_freq}
\langle j \rangle^{\delta} |L_j^{\natural}(\xi) - L'_j(\xi^\flat) | +  |A_1^{\natural}(\xi) - A_1'(\xi^\flat) |  \leq  \rho
\end{equation}
where we have set $L^{\natural} := \Pi_{\mathscr{L}}N^{\natural} $, $L':= \Pi_{\mathscr{L}}N $,  $ A_0' + \mu A_1' := \Pi_{\mathscr{A}}N$ and $ A_0^\natural + \mu A_1^\natural := \Pi_{\mathscr{A}}N^{\natural}$.

\medskip

In order to apply the KAM theorem at the last step, we set
$$
P := N^\natural +  I^\natural  + J^\natural + E^\natural \in \mathscr{H}_{\eta^\natural,r,\mathcal{O}^\natural ,\mathcal{S}}^{\mathrm{tot}}.
$$
Note that, by construction, $P$ is just a way of rewriting $R$, i.e. for all $\xi=:(\xi^\flat,\xi^{\mathrm{new}})\in \mathcal{O}^\natural$ and all $u\in \mathcal{A}_{\xi}(3r)$ we have
\begin{equation}
\label{eq:cest_fait_pour}
P(\xi;u) = R(\xi^\flat;u).
\end{equation}

\medskip

Since $E$ is at least of order $4$ with respect to $(u_i)_{i\in \mathcal{S}^c}$, the same is also true for $E^\natural$ and so
$\Pi_{\mathrm{ajet}} E^\natural =0.
$
Moreover, since $I$ is integrable, $I^{\flat}$ is also integrable\footnote{because, since $I$ and $I^\flat$ are the same function, as $I$, $I^\natural$ commutes with all the actions.} and so $\Pi_{\mathrm{ajet}} I^\natural =0.$
It follows that 
$$
\Pi_{\mathrm{ajet}} P = \Pi_{\mathrm{ajet}} J^\natural \quad \mathrm{and} \quad \Pi_{\mathrm{rem}} P = \Pi_{\mathrm{rem}} I^\natural + \Pi_{\mathrm{rem}} J^\natural + \Pi_{\mathrm{rem}} E^\natural.
$$
Moreover, since by construction $\Pi_{\mathrm{int}} J =\Pi_{\mathrm{int}} E =0 $,  we know by Proposition \ref{prop:open} that $\Pi_{\mathrm{int}} J^\natural =\Pi_{\mathrm{int}} E^\natural =0 $. It follows that
$$
\Pi_{\mathrm{nor}} P = N^\natural + \Pi_{\mathrm{nor}} I^\natural. 
$$

\medskip

As a consequence, thanks to the estimates \eqref{eq:merci_Birk}, we have
$$
\| \Pi_{\mathrm{ajet}} P \|_{\mathscr{H}_{\eta^\natural,r,\mathcal{O}^\natural ,\mathcal{S}}^{\mathrm{tot}}}  \leq  \| J^\natural \|_{\mathscr{H}_{\eta^\natural,r,\mathcal{O}^\natural ,\mathcal{S}}^{\mathrm{tot}}}  \leq \varepsilon^{9/10} r^{4000}
$$
$$
\| \Pi_{\mathrm{rem}} P \|_{\mathscr{H}_{\eta^\natural,r,\mathcal{O}^\natural ,\mathcal{S}}^{\mathrm{tot}}}  \leq   \| I^\natural \|_{\mathscr{H}_{\eta^\natural,r,\mathcal{O}^\natural ,\mathcal{S}}^{\mathrm{tot}}}  + \| J^\natural \|_{\mathscr{H}_{\eta^\natural,r,\mathcal{O}^\natural ,\mathcal{S}}^{\mathrm{tot}}} +  \| E^\natural \|_{\mathscr{H}_{\eta^\natural,r,\mathcal{O}^\natural ,\mathcal{S}}^{\mathrm{tot}}}   \leq r^3.
$$
Moreover, applying the continuity estimate on $\Pi_{\mathrm{nor}}$ given by Lemma \ref{lem:continuity_estimate_pinor}, we have
$$
\| \Pi_{\mathrm{nor}} I^\natural \|_{\mathscr{N}_{r,\mathcal{O}^\natural ,\mathcal{S}}^{\mathrm{tot}}} \lesssim_{S} r^{-4} \| I^\natural \|_{\mathscr{H}_{\eta^\natural,r,\mathcal{O}^\natural ,\mathcal{S}}^{\mathrm{tot}}} \lesssim_{S} r\rho
$$
and so, thanks to the estimate \eqref{eq:est_normal} on $ N^\natural $, we have (provided that $r$ is small enough)
\begin{equation}
\label{eq:est_P_nor}
\|  \Pi_{\mathrm{nor}} P  \|_{\mathscr{N}_{r,\mathcal{O}^\natural,\mathcal{S}}^{\mathrm{tot}}} \leq  \|  \Pi_{\mathrm{nor}} K^\flat \|_{\mathscr{N}_{r_\flat,\mathcal{O}^\flat,\mathcal{S}^\flat}^{\mathrm{tot}}} +\rho^{1/6} \leq 7/2.
\end{equation}

\noindent $\bullet$ \underline{\emph{Step $3$ : KAM and conclusion}.} The previous estimates on $P$ ensure that the assumptions of the KAM theorem (i.e. Theorem \ref{thm:KAM}) are satisfied. As a consequence, we get  a Borel set $\mathcal{O}\subset \mathcal{O}^\natural$, a Hamiltonian $K \in \mathscr{H}_{\eta,r,\mathcal{O},\mathcal{S}}$ and a family of symplectic maps $\tau_\xi$ satisfying the properties described in the statement of Theorem \ref{thm:KAM}.

\medskip

Now, we just have to check that all the properties claimed in the statement of Proposition \ref{prop:loop} are satisfied. 

 First, $\mathcal{O}$ is indeed smaller that $\mathcal{O}^\flat$ because by construction
$$
\mathcal{O} \subset \mathcal{O}^{\natural} = \mathcal{O}' \times  (2\cdot \rho^{19},2\,\rho^{2\cdot 19})  \subset  \mathcal{O}^\flat \times  (\rho^{2\cdot 19},2\,\rho^{2\cdot 19})  =: \mathcal{O}^{\flat,\mathrm{e}}.
$$

 Then, provided that $r$ is small enough with respect to $r_\flat$, we have the measure estimate
\begin{equation*}
\begin{split}
\mathrm{Leb}(\mathcal{O}^{\flat,\mathrm{e}} \setminus \mathcal{O}) \leq\mathrm{Leb}\big([\mathcal{O}^\flat  \setminus \mathcal{O}']\times (\rho^{2\cdot 19},2\,\rho^{2\cdot 19}) \big) +  \mathrm{Leb}(\mathcal{O}^{\natural} \setminus \mathcal{O}) 
&\leq   \rho^{10^{-3}} \rho^{2\cdot 19} + r^2 \varepsilon^{10^{-3}}\\ &\leq    r^{2\nu+10^{-5}}  .
\end{split}
\end{equation*}

 The KAM theorem directly ensures that $ \Pi_{\mathrm{ajet}} K  =0$ and $\|  \Pi_{\mathrm{rem}} K\|_{\mathscr{H}_{\eta,r,\mathcal{O},\mathcal{S}}^{\mathrm{tot}}} \leq 1  $.

The normal form part is controlled as followed
\begin{equation*}
\begin{split}
\|  \Pi_{\mathrm{nor}} K \|_{\mathscr{N}_{r,\mathcal{O},\mathcal{S}}^{\mathrm{tot}}}  &\leq  \|  \Pi_{\mathrm{nor}} P\|_{\mathscr{N}_{r,\mathcal{O},\mathcal{S}}^{\mathrm{tot}}} + r \mathop{\leq}^{\eqref{eq:est_P_nor}} \|  \Pi_{\mathrm{nor}} K^\flat \|_{\mathscr{N}_{r_\flat,\mathcal{O}^\flat,\mathcal{S}^\flat}^{\mathrm{tot}}} +\underbrace{\rho^{1/6} +r}_{\leq  r^{10^{-3}}}.
\end{split}
\end{equation*}

 Setting $L := \Pi_{\mathscr{L}}K $, $L^\flat := \Pi_{\mathscr{L}}K $,  $ A_0^\flat + \mu A_1^\flat := \Pi_{\mathscr{A}}K^\flat$ and $ A_0 + \mu A_1 := \Pi_{\mathscr{A}}K$, for all $\xi =: (\xi^\flat,\xi^{\mathrm{new}}) \in \mathcal{O}$ and all $j\in \mathbb{Z}$, we have
\begin{equation*}
\begin{split}
\langle j \rangle^{\delta} |L_j(\xi) - L^\flat_j(\xi^\flat) | \vee  |A_1(\xi) - A_1^\flat(\xi^\flat) |  \leq&  \|  \Pi_{\mathrm{nor}} (K-P)\|_{\mathscr{N}_{r,\mathcal{O},\mathcal{S}}^{\mathrm{tot}}} +  \|  \Pi_{\mathrm{nor}} (R-K^\flat)\|_{\mathscr{N}_{\rho,\mathcal{O}',\mathcal{S}^\flat}^{\mathrm{tot}}} \\ 
& +  \langle j \rangle^{\delta} |L_j^{\natural}(\xi) - L'_j(\xi^\flat) | \vee |A_1^{\natural}(\xi) - A_1'(\xi^\flat) | \\
\mathop{\leq}^{\eqref{eq:var_freq}} & r+ \rho^{\frac14}+ \rho \leq r^{\frac1{81}}
\end{split}
\end{equation*}
where the estimates of  $\|  \Pi_{\mathrm{nor}} (K-P)\|_{\mathscr{N}_{r,\mathcal{O},\mathcal{S}}^{\mathrm{tot}}}$ and  $  \|  \Pi_{\mathrm{nor}} (R-K^\flat)\|_{\mathscr{N}_{\rho,\mathcal{O}',\mathcal{S}^\flat}^{\mathrm{tot}}}$ are given by Theorem \ref{thm:KAM} and Theorem \ref{thm:Birk} respectively.

Finally, we just have to define the change of variable $\psi_\xi$ and to check its properties. For all $\xi =: (\xi^\flat,\xi^{\mathrm{new}}) \in \mathcal{O}$, we set
$$
\psi_\xi := \varphi_{\xi^\flat} \circ \tau_\xi.
$$
By composition $\psi_\xi$ is clearly symplectic and invariant by gauge transform.
Then by construction, for all $u\in \mathcal{A}_\xi(r)$, we have
\begin{equation*}
\begin{split}
(H^{(0)}_{\varepsilon}+\varepsilon K^\flat)(\xi^\flat; \psi_\xi(u)) \mathop{=}^{\eqref{eq:cest_un_chgt_de_var_Birk}} (H^{(0)}_{\varepsilon}+\varepsilon R)(\xi^\flat; \tau_\xi(u))    
\mathop{=}^{\eqref{eq:cest_fait_pour}}  (H^{(0)}_{\varepsilon}+\varepsilon P)(\xi; \tau_\xi(u)) 
 \mathop{=}^{\eqref{eq:cest_un_cght_de_var_KAM}}  (H^{(0)}_{\varepsilon}+\varepsilon K)(\xi;u).
\end{split}
\end{equation*}

Finally, the fact that $\psi_\xi$ is close to the identity is a direct corollary of the fact that $ \varphi_{\xi^\flat} $ and $\tau_\xi$ are close to the identity (see \eqref{eq:proche_id_KAM} and \eqref{eq:proche_id_Birk} respectively).

\end{proof}

\section{Proof of the main theorem}\label{sec:proof} In this section, we prove Theorem \ref{thm:main}. Basically, after some reductions, rescaling and  an initial opening (section \ref{sec:prep}),  it consists in a first application of our KAM theorem (Theorem \ref{thm:KAM}) to obtain the finite dimensional KAM tori of Kuksin-P\"oschel (section \ref{sec:finite}), then we iterate  the loop described in section \ref{sec:loop} to obtain tori of arbitrary dimension but whose size does not depend on the dimension (section \ref{sec:arbi}) and we pass to the limit to obtain tori of infinite dimension (section \ref{sec:infinite}). Finally, we prove that they contain infinite dimensional invariant tori which are non-resonant (section \ref{sec:NR}) and we finish the proof (section \ref{sec:end}).

\subsection{Preparation}\label{sec:prep} Let $\mathcal{S}_1\subset \mathbb{Z}$ be a non empty finite set. 
\subsubsection{Some reductions} To lighten the notations, without loss of generality, we assume that 
$$
f'(0) = -1.
$$
Indeed, the value of this non zero constant plays no role and, as we will see, this normalization allows to have modulated frequencies of the form $\omega_j = \varepsilon^{-2} j^2 + \xi_j + \cdots$ which is convenient.

\medskip

Then, we will assume that the parameter $\varepsilon$ is small enough, i.e. smaller than a constant depending only on $f,\ \Xi_0$ and $\mathcal{S}_1$. Indeed, it suffices to set, in the main theorem, $\Xi_\varepsilon = \emptyset$ if $\varepsilon$ is too large.

\medskip

Finally, we assume, without loss of generality that 
$$
\Xi_0 = (\varrho_1^2,\Lambda \varrho_1^2)^{\mathcal{S}_1},
$$
 where $\Lambda>1$ and $\varrho_1$ is smaller than a constant depending only on $f,S_1$ and $\varpi$ (see \eqref{eq:def_rho1}). Indeed, by applying a change of variable of the form $(\xi,\varepsilon) \mapsto (\lambda^2 \xi,\lambda^{-1} \varepsilon)$ and choosing $\lambda$ arbitrarily small, we can assume without loss of generality that $\Xi_0$ is a compact set included in $(\varrho_1^2,\Lambda \varrho_1^2)^{\mathcal{S}_1}$. The only thing we need to bear in mind it that to prove the bound on the Hausdorff distance (estimate \eqref{eq:Hauss}) and the frequency estimate \eqref{eq:freq_thm} with constants $1$, it suffices to prove these bounds with betters exponents with respect to $\varepsilon$ (see \eqref{eq:haus} and  \eqref{eq:6+}).

\subsubsection{Regularizing normal form} Now, let us explain why, by Theorem \ref{thm:reg}, instead of considering solution to \eqref{eq:NLS}, it suffices to consider solutions to a Hamiltonian system of the form
\begin{equation}
\label{eq:NLS_reg}
\tag{$\mathrm{NLS}^{\mathrm{reg}}$}
i\partial_t u = \nabla H^{(reg)}(u)
\end{equation}
where
$$
H^{(reg)}(u) = \underbrace{ \frac12 \sum_{k\in \mathbb{Z}} \big(  |k|^2 + \frac{|u_k|^2}2 \big) |u_k|^2 - \frac12\| u\|_{L^2}^4}_{ = H^{(0)}_{1}(u)} + R(u)
$$
satisfies $R\in \mathscr{H}_{\eta_0,r_0,\emptyset}^{\mathrm{cste}}$ for some $r_0\in (0,1)$ and $\Pi_{2n+q < 6} R = 0$. Note that the function $u\mapsto R(u)$ does not depends on some parameters $\xi$. 

\medskip

Indeed, we define (up to a trivial symmetrization), with the notations of Theorem \ref{thm:reg}, the non zero coefficient of $R$ by the relation
$$
R_{n,\emptyset,\emptyset}^{\boldsymbol{\ell},\boldsymbol{\sigma}} := H_{n}^{\boldsymbol{\ell},\boldsymbol{\sigma}}.
$$
Then by Theorem \ref{thm:reg} there exists a constant $C>0$ such that
$$
|R_{n,\emptyset,\emptyset}^{\boldsymbol{\ell},\boldsymbol{\sigma}}| \lesssim C^{2n+q} \left( \frac{  \langle \boldsymbol{\ell}_3^* \rangle^2  }{  \langle \boldsymbol{\ell}_1^* \rangle }\wedge 1\right).
$$
Thus to get that $R\in \mathscr{H}_{\eta_0,r_0,\emptyset}^{\mathrm{cste}}$ for some $r_0$, it suffices to note that if $\boldsymbol{\sigma}_1 \boldsymbol{\ell}_1 + \cdots + \boldsymbol{\sigma}_q \boldsymbol{\ell}_q =0$ then
$$
e^{\mathfrak{a} |\boldsymbol{\ell}^*_1|} \leq e^{\mathfrak{a} |\boldsymbol{\ell}^*_2|} \cdots e^{\mathfrak{a} |\boldsymbol{\ell}^*_q|} \quad \mathrm{and} \quad \langle \boldsymbol{\ell}^*_1 \rangle \leq (q-1) \langle \boldsymbol{\ell}^*_2 \rangle
$$
and therefore, recalling  the definition \eqref{eq:theta} of $\Theta_{\ell}$, one verifies
$$
\left( \frac{  \langle \boldsymbol{\ell}_3^* \rangle^2  }{  \langle \boldsymbol{\ell}_1^* \rangle }\wedge 1\right) \leq q^{s+1} \Theta_{\ell}\lesssim 2^q \Theta_{\ell}
$$
and thus $R\in \mathscr{H}_{\eta_0,r_0,\emptyset}^{\mathrm{cste}}$ for some $r_0$.

\medskip

On the other hand, still by Theorem \ref{thm:reg}, on a centered ball $B_{\ell^2_\varpi}(0,\rho)$ we have
$$
H^{\eqref{eq:NLS}} \circ \tau = H^{(0)}_{1} + R =:H^{(reg)}.
$$
Now, assume that Theorem \ref{thm:main} has been proven for \eqref{eq:NLS_reg} instead of \eqref{eq:NLS} (with slightly better constants). To deduce Theorem \ref{thm:main} for \eqref{eq:NLS}, it suffices to invoke the following arguments. First, we note that thanks to the mean value inequality and the fact that $\| \mathrm{d}\tau(u) - \mathrm{Id} \|_{\ell^2_\varpi \to \ell^2_\varpi }$ is small, we have
$$
\frac12 \| u-v \|_{\ell^2_\varpi} \leq \| \tau(u) - \tau(v) \|_{\ell^2_\varpi} \leq 2 \| u-v \|_{\ell^2_\varpi} ,\quad \forall u,v \in B_{\ell^2_\varpi}(0,\rho).
$$
It follows that $\tau$ preserves (up to a factor $2$) the distance estimates between tori, that $\tau$ is an homeomorphism onto its image and so that the image of a torus by $\tau$ is also a torus. Moreover since $\tau$ is symplectic and its derivative is invertible (by a Neumann series argument), it transforms solutions to \eqref{eq:NLS_reg} into solutions of \eqref{eq:NLS} (and so it transforms an invariant torus of \eqref{eq:NLS_reg} into an  invariant torus of \eqref{eq:NLS}). Finally, since it is close to the identity, it satisfies, for $\xi \in (0,1)^{\mathcal{S}_1}$, $\mathrm{dist}(\varepsilon \mathrm{T}_\xi , \tau (\varepsilon \mathrm{T}_\xi )) \lesssim \varepsilon^3$.

Thus, from now, to lighten the notation, we only aim at proving Theorem \ref{thm:main} for \eqref{eq:NLS_reg} instead of \eqref{eq:NLS}.

\subsubsection{Rescaling} For $\varepsilon \in (0,1)$, we set
$$
H^{(reg)}_\varepsilon(u) = \varepsilon^{-4} H^{(reg)}(\varepsilon u) \quad \mathrm{and} \quad R_{\varepsilon}(u) = \varepsilon^{-5} R(\varepsilon u).
$$
As a consequence, we have 
$$
H^{(reg)}_\varepsilon =   H^{(0)}_{\varepsilon} + \varepsilon R_{\varepsilon} 
$$
 where  we recall that $$
H^{(0)}_{\varepsilon}(u) := \frac12 \sum_{k\in \mathbb{Z}} \big( \varepsilon^{-2} |k|^2 + \frac{|u_k|^2}2 \big) |u_k|^2 - \frac12\| u\|_{L^2}^4.
$$

\medskip

Note that if $v\in C^0(\mathbb{R};\ell^2_\varpi)$ is solution to the rescaled equation
\begin{equation}
\label{eq:NLS^reg_eps}
\tag{$\mathrm{NLS}^{\mathrm{reg}}_\varepsilon$}
i\partial_t v = \nabla H^{(reg)}_\varepsilon(v)
\end{equation}
then $
u: = t\mapsto \varepsilon v(\varepsilon^2 t)$ is solution to \eqref{eq:NLS_reg}.

\medskip

Thanks to this dilatation,  we have that, provided that $\varepsilon$ is small enough, $R_\varepsilon \in \mathscr{H}_{4\eta_0,\frac15,\emptyset}^{\mathrm{cste}}$ and, since $\Pi_{2n+q < 6} R = 0$, that
$$
\|  R_\varepsilon \|_{\mathscr{H}_{4\eta_0,\frac15,\emptyset}^{\mathrm{cste}}   }\leq \eps\|  R \|_{\mathscr{H}_{\eta_0,r_0,\emptyset}^{\mathrm{cste}}   } \leq  \varepsilon^{29/30}.
$$

\subsubsection{First opening}\label{sec:first} Let denote by $c_{\eta,\mathcal{S},\Lambda}$ and $ C_{\mathcal{S},\Lambda}>1$ the two constants appearing implicitly in the assumption \eqref{eq:smallr2}  of Proposition \ref{prop:open}, i.e. so that \eqref{eq:smallr2} rewrites
$
c_{\eta,\mathcal{S},\Lambda} \ r \leq  \rho \leq  C_{\mathcal{S},\Lambda}\ r_\flat.
$
It allows us to make explicit the smallness assumption on $\varrho_1$, that is
\begin{equation}
\label{eq:def_rho1}
\varrho_1 \leq \frac{C_{\mathcal{S}_1,\Lambda}^{-1}}5. 
\end{equation}
Proposition \ref{prop:open} ensure the existence of a constant
$$
\mathfrak{r}_1 \lesssim_{\mathcal{S}_1,\Lambda} \varrho_1
$$
and a Hamiltonian $P_\varepsilon\in \mathscr{H}_{4\eta_0,\mathfrak{r}_1,\Xi_0,\mathcal{S}_1}$
such that
$$
\| P_{\varepsilon}\|_{\mathscr{H}_{4\eta_0,\mathfrak{r}_1,\Xi_0,\mathcal{S}_1}^{\mathrm{tot}}} \leq \varepsilon^{28/30} 
$$
and for all $\xi\in \Xi_0$ and all $u\in \mathcal{A}_{\xi}(3\mathfrak{r}_1)$, 
$$
P_\varepsilon(\xi;u) = R_\varepsilon(u).
$$

\subsection{Finite dimensional tori based on $\mathcal{S}_1$}\label{sec:finite} We aim at applying Theorem \ref{thm:KAM} (i.e. our KAM Theorem), with $\eta^\natural = 4\eta_0$, $\eta = 3\eta_0$, $\mathcal{S}= \mathcal{S}_1$, $P=P_\varepsilon$, $\mathcal{O}^\natural = \Xi_0$ and $r=r_1$ where $r_1\leq \mathfrak{r}_1$ is a constant depending only on $\mathcal{S}_1$ and $\Xi_0$ which just has to be small enough so that the smallness assumption \eqref{eq:r_small_KAM} on $r$ is satisfied. The only assumption we have to check is the smallness assumption \eqref{eq:small_assump_P_KAM} on $P_\varepsilon$. But since $r_1\leq \mathfrak{r}_1$, by monotony, we have
$$
\| P_{\varepsilon}\|_{\mathscr{H}_{4\eta_0,r_1,\Xi_0,\mathcal{S}_1}^{\mathrm{tot}}} \leq \varepsilon^{28/30}.
$$
Therefore, provided that $\varepsilon$ is small enough with respect to $r_1$, the smallness assumption \eqref{eq:small_assump_P_KAM} is satisfied. In particular, thanks to the continuity estimate of Lemma \ref{lem:continuity_estimate_pinor}, we have
$$
\| \Pi_{\mathrm{nor}} P_\varepsilon\|_{\mathscr{N}_{r_1,\Xi_0,\mathcal{S}_1}^{\mathrm{tot}}}\leq 1.
$$
Hence, by Theorem \ref{thm:KAM}, we get Borel set $\mathscr{O}_\varepsilon \subset \Xi_0$ and a Hamiltonian $K_\varepsilon \in \mathscr{H}_{3\eta_0,r_1,\mathscr{O}_\varepsilon,\mathcal{S}_1}$, such that 
\begin{itemize}
\item $\mathscr{O}_\varepsilon$ is relatively large in the sense that
\begin{equation}
\label{eq:asympt_full_meas_Oe}
\mathrm{Leb}(\Xi_0 \setminus \mathscr{O}_\varepsilon) \leq \,  \varepsilon^{10^{-3}} \, 
\end{equation}
\item the adapted jet has been removed, the remainder term and the normal form part are not too large, i.e.
$$
 \Pi_{\mathrm{ajet}} K_\varepsilon  =0,  \quad \|  \Pi_{\mathrm{rem}} K_\varepsilon\|_{\mathscr{H}_{3\eta_0,r_1,\mathscr{O}_\varepsilon,\mathcal{S}_1}^{\mathrm{tot}}} +  \|  \Pi_{\mathrm{nor}}  K_\varepsilon\|_{\mathscr{N}_{r_1,\mathscr{O}_\varepsilon,\mathcal{S}_1}^{\mathrm{tot}}} \leq 2 
$$
\end{itemize}
and for all $\xi \in \mathscr{O}_\varepsilon$, there exists $C^1$ smooth symplectic map $\tau_\xi^{(\varepsilon)}:\mathcal{A}_\xi(r_1)\to \mathcal{A}_\xi(2r_1)$ commuting with the gauge transform and enjoying, for all $u\in \mathcal{A}_\xi(r_1)$, the bound
\begin{equation}
\label{eq:estrella}
\| \tau_\xi^{(\varepsilon)}(u) - u\|_{\ell^2_\varpi} + r_1^2 \| \mathrm{d}\tau_\xi^{(\varepsilon)}(u) - \mathrm{Id}\|_{\ell^2_\varpi \to \ell^2_\varpi } \leq \varepsilon^{1/2}
\end{equation}
 and the property that
$$
(H^{(0)}_{\varepsilon}+\varepsilon P_\varepsilon)(\xi; \tau_\xi^{(\varepsilon)}(u)) = (H^{(0)}_{\varepsilon}+\varepsilon K_\varepsilon)(\xi;u).
$$

\subsubsection{The frequencies} In this subsection, we are going to design a Borel subset $\Xi_\varepsilon \subset \mathscr{O}_{\varepsilon}$ that satisfies all properties of Theorem \ref{thm:main}. For the moment, we  just check properties $\mathrm{(i),(ii)}$ and $\mathrm{(iv)}$ of Theorem \ref{thm:main} for $\xi \in \mathscr{O}_\varepsilon$ (which is slightly stronger than needed).

\medskip

 For  $\xi \in \mathscr{O}_\varepsilon$, we define the renormalized frequencies of $ H_\eps^{(reg)}$ by
\begin{equation}
\label{eq:def_omega_eps}
\omega^{(\varepsilon)}_j(\xi) :=  j^2 +  \varepsilon^2  \big( \xi_j \mathbbm{1}_{j\in \mathcal{S}_1}-2 \sum_{i \in \mathcal{S}_1} \xi_i \big) + 2\varepsilon^{3}(L^{(\varepsilon)}_j(\xi) + A_1^{(\varepsilon)} (\xi)),\ j\in \Z,
\end{equation}
where $L^{(\varepsilon)} = \Pi_{\mathscr{L}} K_\varepsilon$ and $A_0^{(\varepsilon)} + A_1^{(\varepsilon)} \mu = \Pi_{\mathscr{A}} K_\varepsilon$. This definition is natural, in order that equation \eqref{eq:cest_moi_qui_justifie_les_freq} below holds. Then we set
$$
\lambda_j^{(\varepsilon)}(\xi) =2\varepsilon^{1/4} (L^{(\varepsilon)}_j(\xi) + A_1^{(\varepsilon)} (\xi))
$$
in such a way we have an estimate slightly better than \eqref{eq:freq_thm}:
\begin{equation}\label{eq:6+}
\omega^{(\varepsilon)}_j(\xi) =  j^2 +  \varepsilon^2  \big( \xi_j -2 \sum_{i \in \mathcal{S}_1} \xi_i \big) + \varepsilon^{11/4} \lambda_j^{(\varepsilon)}(\xi).
\end{equation}
We just have to check that 
\begin{lemma}
$\lambda^{(\varepsilon)}$ is $1-$Lipschitz and bounded by one.
\end{lemma}
\begin{proof}

 Indeed, by definition of the $\mathscr{N}_{r_1,\mathscr{O}_\varepsilon,\mathcal{S}_1}^{\mathrm{tot}}$ norm (see Definition \ref{def:N_norm}), since $\|  \Pi_{\mathrm{nor}}  K_\varepsilon\|_{\mathscr{N}_{r_1,\mathscr{O}_\varepsilon,\mathcal{S}_1}^{\mathrm{tot}}} \leq 2 $, for all $\xi,\zeta \in \mathscr{O}_\varepsilon$, we have (provided that $\varepsilon$ is small enough)
$$
|\lambda_j^{(\varepsilon)}(\xi) | \leq 2 \varepsilon^{1/4} \|  \Pi_{\mathrm{nor}}  K_\varepsilon\|_{\mathscr{N}_{r_1,\mathscr{O}_\varepsilon,\mathcal{S}_1}^{\mathrm{tot}}} \leq 1
$$
and
$$
|\lambda_j^{(\varepsilon)}(\xi) -\lambda_j^{(\varepsilon)}(\zeta) | \leq 2 r_1^{-2} \varepsilon^{1/4} \|  \Pi_{\mathrm{nor}}  K_\varepsilon\|_{\mathscr{N}_{r_1,\mathscr{O}_\varepsilon,\mathcal{S}_1}^{\mathrm{tot}}} \| \xi - \zeta \|_{\ell^1} \leq  \| \xi - \zeta \|_{\ell^1} .
$$
\end{proof}

\subsubsection{The homeomorphisms} Being given $\mathcal{S}\subset \mathbb{Z}$ (not necessarily finite), $\xi \in \ell^1_{\varpi^2}(\mathcal{S}; \mathbb{R}_+^*)$, we set
$$
\Upsilon_\xi^{\mathcal{S}} : \left\{ \begin{array}{lll} \mathbb{T}^{\mathcal{S}}& \to  & \mathrm{T}_\xi \subset \ell^2_\varpi \\ 
\theta & \mapsto & \displaystyle \sum_{j\in \mathcal{S}} \sqrt{\xi_j} e^{-\ic\theta_j} e^{\ic x j} \end{array} \right.
$$
where $\mathrm{T}_\xi$ (the image of $\Upsilon_\xi$) is defined by \eqref{eq:def_T_droit}. Note that $\Upsilon_\xi^{\mathcal{S}}$ is a homeomorphism. Indeed, it is clearly bijective, continuous by the Lebesgue's convergence theorem and  defined on   $\mathbb{T}^{\mathcal{S}}$ which is compact by Tychonoff's theorem. Thus its inverse is also continuous by a standard argument.

\medskip

For $\xi \in \mathscr{O}_\varepsilon$ and we set 
$$
 \mathcal{T}_\xi^{\varepsilon} := \varepsilon \tau_\xi^{(\varepsilon)}(\mathrm{T}_\xi) \quad \mathrm{and} \quad  \Psi_{\xi}^\varepsilon:= \varepsilon \tau_\xi^{(\varepsilon)} \circ \Upsilon_\xi^{\mathcal{S}_1}.
$$
\begin{lemma} \label{lem:homeo_uno}
 $ \Psi_{\xi}^\varepsilon : \mathbb{T}^{\mathcal{S}_1} \to \mathcal{T}_\xi^{\varepsilon}$ is a homeomorphism.
\end{lemma}
\begin{proof}
It suffices to prove that $(\tau_\xi^{(\varepsilon)})_{| \mathrm{T}_\xi} $ is an homeomorphism on its image (that is $\varepsilon^{-1}  \mathcal{T}_\xi^{\varepsilon}$). 

\medskip

Applying Lemma \ref{lem:appendix} of the appendix with $g =  \tau_\xi^{(\varepsilon)} - \mathrm{id}$, we have that for all  $u,v \in \mathrm{T}_\xi$, 
$$
\Big| \| \tau_\xi^{(\varepsilon)}(u) - \tau_\xi^{(\varepsilon)}(v)\|_{\ell^2_\varpi} - \| u - v\|_{\ell^2_\varpi} \Big| \lesssim_{\mathcal{S}_1} \varepsilon^{1/2} r_1^{-2} \| u - v\|_{\ell^2_\varpi}.
$$
As a consequence, provided that $\varepsilon$ is small enough, we have
\begin{equation}
\label{eq:bilip_1}
\frac12  \| u - v\|_{\ell^2_\varpi} \leq \| \tau_\xi^{(\varepsilon)}(u) - \tau_\xi^{(\varepsilon)}(v)\|_{\ell^2_\varpi} \leq 2  \| u - v\|_{\ell^2_\varpi}.
\end{equation}
It implies that $(\tau_\xi^{(\varepsilon)})_{| \mathrm{T}_\xi} $ is injective. Since $\varepsilon^{-1}  \mathcal{T}_\xi^{\varepsilon}$ is its image, it is a bijection between $\mathbb{T}^{\mathcal{S}_1}$ and $\varepsilon^{-1} \mathcal{T}_\xi^{\varepsilon}$. Moreover, \eqref{eq:bilip_1} implies that it is continuous and that its inverse is also continuous: it is an homeomorphism.
\end{proof}

\subsubsection{Invariance} \label{sub:invarianceS1} Now, we aim at proving the following lemma. 
\begin{lemma} \label{lem:inv_uno} Let $\xi \in \mathscr{O}_\varepsilon$, $\theta \in \mathbb{T}^{\mathcal{S}_1}$ and $\omega^{(\varepsilon)}$ defined by \eqref{eq:def_omega_eps}. Then
\begin{center}
$t\mapsto   \Psi_{\xi}^\varepsilon(  \omega^{(\varepsilon)} t + \theta)$ is global solution to \eqref{eq:NLS_reg}.
\end{center}

\end{lemma}
\begin{proof}
We set, for $t\in \mathbb{R}$,
$$
z(t) = \Upsilon_\xi^{\mathcal{S}_1}( \varepsilon^{-2} \omega^{(\varepsilon)} t + \theta) \in \mathrm{T}_\xi.
$$
Since $\Pi_{\mathrm{ajet}} K_\varepsilon = 0$, we have that 
\begin{equation}
\label{eq:cest_moi_qui_justifie_les_freq}
i\partial_t z= \nabla (H^{(0)}_{\varepsilon}+\varepsilon K_\varepsilon)( \xi;z).
\end{equation}
Then, we set for $t\in \mathbb{R}$
$$
v(t) := \varepsilon^{-1}\Psi_{\xi}^\varepsilon(  \varepsilon^{-2} \omega^{(\varepsilon)} t + \theta) = \tau_\xi^{(\varepsilon)} (z(t)).
$$
Since $ \tau_\xi^{(\varepsilon)}$ is symplectic and $\mathrm{d}\tau_\xi^{(\varepsilon)}(w)$ is surjective for all $w\in \mathcal{A}_\xi(r_1)$ (by a Neumann series argument using \eqref{eq:estrella}), by applying Lemma \ref{lem:sympl_chgtvar} of the appendix, we deduce that $v$ is solution to the equation
$$
i\partial_t v = \nabla (H^{(0)}_{\varepsilon}+\varepsilon P_\varepsilon)(\xi;v) = \nabla H^{(reg)}_\varepsilon(v)
$$
Finally, it suffices to note that, thus, $
 t\mapsto \varepsilon v(\varepsilon^2 t) = \Psi_{\xi}^\varepsilon(  \omega^{(\varepsilon)} t + \theta)$ is solution to \eqref{eq:NLS_reg}.
\end{proof}

\subsection{Tori of arbitrarily large dimension}\label{sec:arbi} Now, we aim at constructing finite dimensional tori of arbitrarily large dimension whose size does not go to zero when its dimension goes to infinity.

\subsubsection{Choice of the other sites}\label{sec:sparse} In this section we prove 
\begin{lemma} \label{lem:sites} There exists a strictly increasing family $(\mathcal{S}_p)_{p\geq 2}$ of finite subsets of $\mathbb{Z}$ of the form
$$
 \forall p\geq 1, \ \mathcal{S}_{p+1} =: \mathcal{S}_{p} \cup \{ i_{p+1} \}.
$$
and satisfying
$$
C_{\mathcal{S}_\infty}:=\sup_{i\in \mathcal{S}_\infty} \sum_{k\in \mathcal{S}_\infty} \Gamma_{k,i} < \infty
$$
where  $\Gamma_{k,i} =\Lambda_{k,i}\vee \langle k \rangle^{-\delta}$ is the weight associated with the $\|\cdot \|_{\mathscr{L}_{\mathscr{O},\mathcal{S}}^{\mathrm{dir}}}$ norm and was defined by \eqref{eq:def_Gammaki}, $\Lambda_{k,i}= 1\wedge \langle i \rangle^{-\delta} \langle k\rangle^{4\delta} \wedge  \langle k \rangle^{-\delta} \langle  i \rangle^{4\delta}  $ and
$$
\mathcal{S}_\infty := \bigcup_{p\geq 1} \mathcal{S}_p.
$$
\end{lemma}

Once Lemma \ref{lem:sites} will be proven,  we will consider a sequence of sites $(\mathcal{S}_p)_{p\geq 2}$ satisfying properties of Lemma \ref{lem:sites} as fixed. In particular we will not track the dependency of the parameters with respect to $(\mathcal{S}_p)_{p\geq 1}$. Moreover, we will assume that $\varepsilon$ is small enough so that
\begin{equation}
\label{eq:assump_eps_proof_fin}
\varepsilon \leq 10^{-9} \wedge 10^{-3} C_{\mathcal{S}_\infty}^{-1}
\end{equation}
which implies that the bound of Proposition \ref{prop:loop} to modulate the frequencies is satisfied for all the $\S_p$.
To lighten the notations, we will use the following convention.
\begin{definition}[$\xi^{(p)}$] If $\xi \in \mathbb{R}^{\mathbb{Z}}$ and $p\geq 1$, we denote by $\xi^{(p)} \in \mathbb{R}^{\mathcal{S}_p}$ its restriction to $\mathcal{S}_p$, i.e.
$$
\xi^{(p)} := \Pi_{\mathcal{S}_p} \xi.
$$
\end{definition}

\begin{proof}[Proof of Lemma \ref{lem:sites}]
A possible choice is
$$
i_{p} = 2^{5^p+ n_0}
$$
where $n_0 \geq 1$ is an index such that
$$
\max_{i\in \mathcal{S}_1} \langle i \rangle \leq 2^{n_0}.
$$
Indeed, first it is clear that 
$$
\sum_{k\in \mathcal{S}_\infty} \langle k\rangle^{-\delta} <\infty.
$$
So we just have to focus on the coefficients $\Lambda_{k,i}$. On the one hand, if $i\in \mathcal{S}_1$, we have the bound  $\Lambda_{k,i} \leq \langle k \rangle^{-\delta} \langle  i \rangle^{4\delta} \lesssim_{\mathcal{S}_1}  \langle k \rangle^{-\delta}$ so the estimate is almost the same. On the other hand, if $i\in  \mathcal{S}_\infty \setminus \mathcal{S}_1$, we have
\begin{equation*}
\begin{split}
 \sum_{k\in \mathcal{S}_\infty} \Lambda_{k,i} &= \sum_{k\in {S}_1 } \Lambda_{k,i} +  \sum_{\substack{k\in \mathcal{S}_\infty \setminus{S}_1 \\ k< i}} \Lambda_{k,i} + \Lambda_{i,i} + \sum_{\substack{k\in \mathcal{S}_\infty \\ k>i}} \Lambda_{k,i} \\
 &\leq 2^{4\delta n_0}\sum_{k\in {S}_\infty }  \langle k\rangle^{-\delta} +  \sum_{\substack{k\in \mathcal{S}_\infty \setminus{S}_1 \\ k< i}} \langle i \rangle^{-\delta} \langle k\rangle^{4\delta}  + 1 + \sum_{\substack{k\in \mathcal{S}_\infty \\ k>i}} \langle k \rangle^{-\delta} \langle  i \rangle^{4\delta} =: I+II+III+IV.
\end{split}
\end{equation*}
The terms $I$ and $III$ are clearly uniformly bounded with respect to $i$. Then, let $p \geq 2$ be such that $i=i_p$. It follows that \footnote{using the basic bound $
5^q -4\cdot 5^{p} \geq 5^q - 5^{p+1} = 5^{p+1} (5^{q-p-1} -1) \geq (5^{q-p-1} -1).
$ }
$$
IV  = \sum_{q>p} \langle i_q \rangle^{-\delta} \langle  i_p \rangle^{4\delta} \leq \sum_{q>p} 2^{-5^q \delta- n_0\delta} 2^{4\cdot 5^{p}\delta + 4n_0\delta + 4\delta } \lesssim_{\mathcal{S}_1} \sum_{q>p}  2^{-(5^q -4\cdot 5^{p}) \delta} \lesssim_{\mathcal{S}_1} \sum_{j\geq 0} 2^{-(5^j -1) \delta}
$$
which is uniformly bounded with respect to $p$. Finally, if $p=2$ then $II=0$ and if $p\geq 3$, by monotonicity, we have
$$
II = \sum_{q<p} \langle i_p \rangle^{-\delta} \langle  i_q \rangle^{4\delta} \leq  p  \langle i_p \rangle^{-\delta} \langle  i_{p-1} \rangle^{4\delta} \mathop{\longrightarrow}_{p \to \infty} 0 .
$$
\end{proof}

\subsubsection{Iteration of the loop} \label{sub:plein_plein_objets}  Let $\varrho>0$ be a small parameter. Up to a positive exponent, $\varrho$ will be the distance between $\mathcal{T}_\xi^{\varepsilon}$ and the infinite dimensional torus we are going to construct. Without loss of generality, we will assume $\varrho$ small enough  with respect to $r_1$ and $\varepsilon$.  We set,
$$
\forall p \geq 2, \quad \eta_p := 3\eta_0 - \eta_0 \sum_{q=1}^{p-2} 2^{-q},
$$
$$
K^1 := K_\varepsilon, \quad \mathcal{O}_1 := \mathscr{O}_\varepsilon \quad \mathrm{and} \quad r_{2} :=  \varrho. 
$$

Applying iteratively Proposition \ref{prop:loop}, we get for all $p\geq 2$, a radius $r_{p} \in (0,1/4)$, a Borel set $\mathcal{O}_{p} \subset ((4r_p)^2,1)^{\mathcal{S}_p}$, a Hamiltonian $K^p \in \mathscr{H}_{\eta_p,r_p,\mathcal{O}_p,\mathcal{S}_p}$ and for all $\xi \in \mathcal{O}_p$ a $C^1$ symplectic map $\psi_\xi^p: \mathcal{A}_\xi(r_p) \to \mathcal{A}_{\xi^{(p-1)}}(r_{p-1})$ (commuting with the gauge transform) satisfying the following properties.
\begin{itemize}
\item The sequence $(r_p)_{p\geq 2}$ is decaying moreover it satisfies 
\begin{equation}
\label{eq:rp_decroit_bien_trop}
\forall p\geq 3, \quad r_{p}\leq  g(p,r_{p-1},\eps,\mathcal{S}_p)
\end{equation}
where $g$ is a given function of $p$, $r_{p-1}$, $\varepsilon$ and $\mathcal{S}_p$ that could be specified (but we will not\footnote{since many of our estimate are far from sharp (to simplify the proof), it would not really make sense.}). It is just a way to say that we can still impose a finite number of upper bounds on $r_p$ with respect to $r_{p-1}$, $\varepsilon$ and $\mathcal{S}_p$\footnote{This is actually the big advantage of the P\"oschel iterative method: at each iteration we can choose $r_p$ as small as it suits us. The disadvantage is that we do not have a realistic lower bound on $r_p$. }. We have indeed this degree of freedom because Proposition \ref{prop:loop} only requires an upper bound on $r$ with respect to  $r_\flat$, $\varepsilon$ and $\mathcal{S}$. 
\item The sequence $(\mathcal{O}_{p})_{p\geq 2}$ is decaying in the sense that for all $p\geq 2$
\begin{equation}
\label{eq:ca_decroit_les_O}
\mathcal{O}_{p} \subset \mathcal{O}_{p-1} \times (r_p^{2\nu},2r_p^{2\nu}) .
\end{equation}
\item The sets $\mathcal{O}_{p}$ are asymptotically of full measure
\begin{equation}
\label{eq:measure_bof}
\forall p\geq 2, \quad  \mathrm{Leb}(\mathcal{O}_{p-1} \times (r_p^{2\nu},2r_p^{2\nu}) \setminus \mathcal{O}_p) \leq  r_p^{2\nu+10^5}
\end{equation}
\item $K^p$ is not too large
$$
\forall p\geq 1, \quad \| \Pi_{\mathrm{nor}} K^p \|_{\mathscr{N}_{r_p,\mathcal{O}_p,\mathcal{S}_p}^{\mathrm{tot}}} + \| \Pi_{\mathrm{rem}} K^p\|_{\mathscr{H}_{\eta_p,r_p,\mathcal{O}_p,\mathcal{S}_p}^{\mathrm{tot}}} \leq 2+\sum_{q=1}^{p-2} 2^{-q} \leq 3
$$
and its adapted jet is equal to zero
$$
\forall p\geq 1, \quad  \Pi_{\mathrm{ajet}} K^p=0.
$$
\item  $\psi_\xi^p$ is a change of variable such that,  for all $p\geq 2$, for all $\xi  \in \mathcal{O}_{p}$,  all $u\in \mathcal{A}_\xi(r_p)$,
$$
(H^{(0)}_{\varepsilon}+\varepsilon K^p)( \xi; u) = (H^{(0)}_{\varepsilon}+\varepsilon K^{p-1})(\xi^{(p-1)};\psi_\xi^p(u)).
$$
Moreover, it is close to the identity, i.e.  for all $p\geq 2$ and all $u\in \mathcal{A}_\xi(r_p)$,
\begin{equation}
\label{eq:proche_id_p}
\|  \psi_\xi^p(u) - u\|_{\ell^2_\varpi} + \| \mathrm{d}  \psi_\xi^p(u) - \mathrm{Id}\|_{\ell^2_\varpi \to \ell^2_\varpi } \leq r_{p}^{\frac1{440}}
\end{equation}
\item the frequencies do not move too much: we have, for all $p\geq 2$, all $\xi  \in \mathcal{O}_{p}$ and all $j\in \mathbb{Z}$, we have
\begin{equation}
\label{eq:freq_move_pas_trop}
\langle j \rangle^{\delta} |L^{(p)}_j(\xi) - L^{(p-1)}_j(\xi^{(p-1)}) | +  |A_1^{(p)}(\xi) - A^{(p-1)}(\xi^{(p-1)}) |  \leq  r_{p}^{\frac1{82}}  
\end{equation} 
where we have set, for all $p\geq 1$,
$$
L^{(p)} := \Pi_{\mathscr{L}}K^p \quad \mathrm{and} \quad  A_0^{(p)} + \mu A_1^{(p)} := \Pi_{\mathscr{A}}K^p .
$$
\end{itemize}
\begin{remark}
\label{rem:onalege_les_notations}
Of course, all these object depend on $\varrho$ and $\varepsilon$ (excepted $(r_p)_{p\geq 2}$ which only depends on $\varrho$) as well as many object we are going to construct from now. To lighten the notations (for readability), we only specify this dependency with subscripts $\varepsilon,\varrho$ when necessary (essentially in the last step of this proof, i.e. subsection \ref{sub:conclusion_of_the_proof}). For example, we have $r_{\varrho,p} = r_p$ and $K^p_{\varrho,\varepsilon}=K^p$.
\end{remark}

 \subsubsection{Homeomorphisms} We aim at proving the following lemma.
 
 \begin{lemma}\label{lem:homeop} For all $p\geq 1$ and all $\xi \in \mathcal{O}_p$, the map 
 \begin{equation}
\label{eq:def_tau_xi_p}
\tau_\xi^p := \tau_{\xi^{(1)}}^{(\varepsilon)} \circ   \psi_{\xi^{(2)}}^2 \circ \cdots    \circ  \psi_{\xi^{(p)}}^p :\mathcal{A}_{\xi^{(p)}}(r_p) \to \mathcal{A}_{\xi^{(1)}}(2r_1) .
\end{equation}
is well defined and satisfies for all $u,v\in \mathcal{A}_\xi(r_p)$
\begin{equation}
\label{eq:bilip_p}
\frac{1+2^{-(p-1)}}4  \| u - v\|_{\ell^2_\varpi} \leq \| \tau_\xi^p(u) - \tau_\xi^p(v)\|_{\ell^2_\varpi} \leq 4 (1-2^{-(p-1)})   \| u - v\|_{\ell^2_\varpi}.
\end{equation}
Moreover, $(\tau_\xi^p)_{| \mathrm{T}_\xi}$ is an homeomorphism on its image.
 \end{lemma}
 Of which we directly deduce the following corollary.
 \begin{corollary} \label{cor:homeo_p} For all $p\geq 1$ and all $\xi \in \mathcal{O}_p$, setting $\mathcal{T}_\xi^{(p)} := \varepsilon \tau_\xi^p (\mathrm{T}_\xi)$, we have that
\begin{equation}
\label{eq:Psi_last}\Psi_\xi^{(p)} :=  \varepsilon \tau_\xi^p \circ \Upsilon_{\xi}^{\mathcal{S}_p} 
\end{equation}
is a homeomorphism from $\mathbb{T}^{\mathcal{S}_p}$  onto $\mathcal{T}_\xi^{(p)}$.
 \end{corollary}
 \begin{proof}[Proof of Lemma \ref{lem:homeop}]
 To prove that this definition makes sense it suffices to proceed by iteration noticing that being $p\geq 2$ and $\xi \in \mathcal{O}_p$, we have (thanks to the decay property \eqref{eq:ca_decroit_les_O})
 $\xi^{(p-1)} \in \mathcal{O}_{p-1}$ and 
 \begin{equation}
 \label{eq:decay_annulus1}
\mathcal{A}_{ \xi}(2r_p) \subset \mathcal{A}_{\xi^{(p-1)}}(r_{p-1}).
\end{equation}
Indeed, if $u\in \mathcal{A}_{ \xi}(2r_p)$ then
$$
\sum_{i \in \mathcal{S}_{p-1}} \varpi_i^2 |\xi_i - |u_i|^2 | + \sum_{j \in \mathcal{S}_{p-1}^c} |u_j|^2 \leq 2 r_p^2 + \xi_{i_p} \leq 2 r_p^2 + 2r_p^{2\nu} \mathop{\leq}^{\eqref{eq:rp_decroit_bien_trop}} r_{p-1}^2,
$$
i.e. $u\in \mathcal{A}_{\xi^{(p-1)}}(r_{p-1})$.

\medskip

The fact that $(\tau_\xi^p)_{| \mathrm{T}_\xi}$ is an homeomorphism on its image is a direct consequence of \eqref{eq:bilip_p}. So we focus on proving \eqref{eq:bilip_p}. 
Actually, it suffices to proceed by induction, noticing that thanks to \eqref{eq:proche_id_p}, by Lemma \ref{lem:appendix} of the appendix applied with $g =  \psi_\xi^{p} - \mathrm{id}$, we have that for all $u,v\in \mathcal{A}_\xi(r_p)$, 
$$
\Big| \| \psi_\xi^{p}(u) - \psi_\xi^{p}(v)\|_{\ell^2_\varpi} - \| u - v\|_{\ell^2_\varpi} \Big| \leq (1+ (2\pi)^{\# \mathcal{S}_p}) r_{p-1} \| u - v\|_{\ell^2_\varpi} \mathop{\leq}^{\eqref{eq:rp_decroit_bien_trop}} 2^{-p-2}  \| u - v\|_{\ell^2_\varpi}  .
$$
The case $p=1$ (i.e. the initialization) has been done in the proof of Lemma \ref{lem:homeo_uno} (see \eqref{eq:bilip_1}).
 \end{proof}

 \subsubsection{Invariance and frequencies} \label{sub:inv_p} Since $\Pi_{\mathrm{ajet}}K^p = 0$, proceeding as in the proof of Lemma \ref{lem:inv_uno}, we have the following result.
 \begin{lemma} \label{lem:inv_p} For all $p\geq 2$, $\xi \in \mathcal{O}_p$ and all $\theta \in \mathbb{T}^{\mathcal{S}_p}$,
 \begin{center}
 $t\mapsto \Psi_\xi^{(p)} ( \omega^{(p)}(\xi) t + \theta)$ is a global solution to \eqref{eq:NLS_reg}
 \end{center}
 where $\omega^{(p)} \in \mathbb{R}^{\mathcal{S}_p}$ is defined by
 \begin{equation}
 \label{eq:def_omega_p}
 \forall j\in \mathcal{S}_p, \quad \omega^{(p)}_j (\xi):=  j^2 +  \varepsilon^2  \big( \xi_j -2 \sum_{i \in \mathcal{S}_p} \xi_i \big) + 2\varepsilon^{3} (A^{(p)}_1(\xi)  +  L^{(p)}_j(\xi)). 
 \end{equation}
 \end{lemma}

 \begin{remark}\label{rem:eps-indep-p}
 Note that since $\Psi_\xi^{(p)} : \mathbb{T}^{\mathcal{S}_p} \to \mathcal{T}_\xi^{(p)}$ is a homeomorphism on its image, we have construct invariant tori for \eqref{eq:NLS} of dimension $\#\mathcal{S}_p \geq p$ with $p$ arbitrarily large and size $\varepsilon$ independent of $p$. This result is interesting in itself because in all the previous works $\varepsilon$ goes to $0$ as $p$ goes to $+\infty$ (thus preventing a non-trivial result in the limit $p\to +\infty$).
 \end{remark}

\subsection{Infinite dimensional tori}\label{sec:infinite} Now, we pass to the limit when $p$ goes to $+\infty$.

\subsubsection{The parameters} We set
$$
\mathcal{O}_\infty  := \bigcap_{p\geq 1}  \widetilde{\mathcal{O}}_p
$$
where (with the same abuse of notations as usual) 
$$
\widetilde{\mathcal{O}}_p  := \mathcal{O}_p \times \prod_{q> p} (r_q^{2\nu},2r_q^{2\nu}) \subset (0,1)^{\mathcal{S}_\infty}.
$$
We note that, thanks to the decay property \eqref{eq:ca_decroit_les_O} of the sets $\mathcal{O}_p$, the sets $\widetilde{\mathcal{O}}_p$ are decaying
\begin{equation}
\label{eq:decay_tilde}
\widetilde{\mathcal{O}}_1 \supset \widetilde{\mathcal{O}}_2  \supset  \cdots \supset \cdots  \supset \widetilde{\mathcal{O}}_\infty :=\mathcal{O}_\infty .
\end{equation}
It is crucial to note that, thanks to this decay property, if $\xi \in \mathcal{O}_{\infty}$ and $p\geq 1$ then
$$
\xi^{(p)} \in \mathcal{O}_p.
$$
 Recalling that by construction
 $$
 \widetilde{\mathcal{O}}_1 = \mathscr{O}_\varepsilon \times \prod_{p\geq 2} (r_p^{2\nu},2r_p^{2\nu})
 $$
we equip $ \widetilde{\mathcal{O}}_1 $ of the the probability measure
$$
 \mathbb{P} := \mathcal{U}( \mathscr{O}_\varepsilon ) \otimes \bigotimes_{p\geq 2} \mathcal{U}( r_p^{2\nu},2r_p^{2\nu} )
$$
where $ \mathcal{U}( \mathscr{O}_\varepsilon )$ denotes the uniform law on $\mathscr{O}_\varepsilon $ (i.e. the normalized Lebesgue measure) and similarly  $\mathcal{U}( r_p^{2\nu},2r_p^{2\nu} )$ denotes the uniform law on the interval $( r_p^{2\nu},2r_p^{2\nu} )$.

\begin{lemma}  \label{lem:asympt_full_inf}$\mathcal{O}_\infty$ is asymptotically of full measure in $\mathcal{O}_1$, i.e.
$$
\mathbb{P}(\widetilde{\mathcal{O}}_1 \setminus \mathcal{O}_\infty) \lesssim  \varrho^{10^{-5}}.
$$
\end{lemma}
\begin{proof}
Since the sets $\widetilde{\mathcal{O}}_p $ are decaying, we have
\begin{equation}
\label{eq:pas_terrible_du_tout}
\begin{split}
\mathbb{P}(\widetilde{\mathcal{O}}_1 \setminus \mathcal{O}_\infty) &= \sum_{p\geq 2} \mathbb{P}(\widetilde{\mathcal{O}}_{p-1} \setminus \widetilde{\mathcal{O}}_{p} ) \\
&= \mathrm{Leb}(\mathscr{O}_{\varepsilon})^{-1} \sum_{p\geq 2} \big( \prod_{2\leq q \leq p} r_q^{-2\nu} \big) \mathrm{Leb}(\mathcal{O}_{p-1} \times (r_p^{2\nu},2r_p^{2\nu}) \setminus \mathcal{O}_p ) \\
&\mathop{\leq}^{\eqref{eq:measure_bof}} \mathrm{Leb}(\mathscr{O}_{\varepsilon})^{-1} \sum_{p\geq 2} \big( \prod_{2\leq q \leq p} r_q^{-2\nu} \big) r_p^{2\nu + 10^{-5}}.
\end{split}
\end{equation}
Then, we just have to use that $ \mathrm{Leb}(\mathscr{O}_{\varepsilon})$ goes to $ \mathrm{Leb}(\Xi_0)>0$ as $\varepsilon$ goes to $0$ (see \eqref{eq:asympt_full_meas_Oe}) and $r_2 =\varrho$  to get that, provided that $(r_q)_q$ is decaying enough (see \eqref{eq:rp_decroit_bien_trop}),
$$
r_p^{10^{-5}} \leq  \varrho^{10^{-5}} \prod_{2\leq q \leq p-1} r_q^{2\nu}
$$
and so that $\mathbb{P}(\widetilde{\mathcal{O}}_1 \setminus \mathcal{O}_\infty) \lesssim  \varrho^{10^{-5}}$.
\end{proof}

\subsubsection{The frequencies} Thanks to the estimate \eqref{eq:freq_move_pas_trop} and the fast decay of the radii $(r_p)_p$ (see \eqref{eq:rp_decroit_bien_trop}), we have the following result:
\begin{lemma} \label{lem:freq_inf} For all $\xi \in \mathcal{O}_\infty$, there exists $ A^{(\infty)}_1(\xi) \in \mathbb{R}$ and $L^{(\infty)} \in \ell^\infty(\mathbb{Z};\mathbb{R}) $ such that
$$
A_1^{(p)}(\xi^{(p)}) \mathop{\longrightarrow}_{p\to \infty} A^{(\infty)}_1(\xi)  \quad \mathrm{and} \quad L^{(p)}(\xi^{(p)}) \mathop{\longrightarrow}_{p\to \infty} L^{(\infty)}(\xi) \quad \mathrm{in} \quad \ell^\infty(\mathbb{Z}).
$$
Moreover, setting, for all $j\in \mathcal{S}_\infty$ and $\xi \in \mathcal{O}_\infty$,
$$
\omega^{(\infty)}_j (\xi):=  j^2 +  \varepsilon^2  \big( \xi_j -2 \sum_{i \in \mathcal{S}_\infty} \xi_i \big) + 2\varepsilon^{3} (A^{(\infty)}_1(\xi)  +  L^{(\infty)}_j(\xi))
$$
we have, for all $p\geq 1$,
\begin{equation}
\label{eq:cvg_freq_inf}
\sup_{\xi \in \mathcal{O}_\infty}\sup_{j\in \mathcal{S}_p}| \omega_j^{(p)}(\xi^{(p)}) - \omega_j^{(\infty)}(\xi) | \leq r_{p+1}^{1/83}.
\end{equation}
\end{lemma} 
\begin{proof}
The convergence of $A_1^{(p)}(\xi^{(p)})$ and $L^{(p)}(\xi^{(p)}) $ are direct consequences of \eqref{eq:freq_move_pas_trop}. Note that it implies the pointwize convergence of $\omega_j^{(p)}(\xi^{(p)})$ to $\omega_j^{(\infty)}(\xi)$. 
Hence we only focus on the quantitative estimate \eqref{eq:cvg_freq_inf}. So let $p\geq 1$ and $j\in \mathcal{S}_p$. Since $\omega_j^{(p)}(\xi^{(p)})$ converges to $\omega_j^{(\infty)}(\xi)$ we have
\begin{equation*}
\begin{split}
&| \omega_j^{(p)}(\xi^{(p)}) - \omega_j^{(\infty)}(\xi) |\\
 \leq& \sum_{q\geq p} |  \omega_j^{(q+1)}(\xi^{(q+1)}) -  \omega_j^{(q)}(\xi^{(q)}) | \\
=&  2\varepsilon^3 \sum_{q\geq p}  | - \varepsilon^{-1}\xi_{i_{q+1}} + L^{(q+1)}_j(\xi^{(q+1)}) - L^{(q)}_j(\xi^{(q)})  + A^{(q+1)}_1(\xi^{(q+1)}) - A^{(q)}_1(\xi^{(q)})    |  \\
\mathop{\leq}^{\eqref{eq:freq_move_pas_trop}}  & \sum_{q\geq p}  r_{q+1}^{2\nu} + r_{q+1}^{1/82} \mathop{\leq}^{\eqref{eq:rp_decroit_bien_trop}} r_{p+1}^{1/83}.
\end{split}
\end{equation*} 
\end{proof}

\subsubsection{The homeomorphisms} Let $\xi \in \mathcal{O}_\infty$. We consider the functions $\tau_{\xi^{(p)}}^p$ defined by \eqref{eq:def_tau_xi_p}. We note that by \eqref{eq:decay_annulus1}, we have
$$
\forall p \geq 2, \quad \mathcal{A}_{\xi^{(p)}} (r_p) \subset  \mathcal{A}_{\xi^{(p-1)}} (r_{p-1}). 
$$
It follows that (since $r_p$ goes to $0$ when $p$ goes to $+\infty$)
$$
\mathrm{T}_\xi = \bigcap_{p\geq 1} \mathcal{A}_{\xi^{(p)}} (r_p).
$$
In particular all the functions  $\tau_{\xi^{(p)}}^p$ are well defined on $\mathrm{T}_\xi$. 

\begin{lemma} \label{lem:petit_lemme_technique}
There exists $\mathcal{T}_\xi^{(\infty)} \subset \ell^2_\varpi$ and a homeomorphism $\tau_\xi^\infty$ such that $\tau_{\xi^{(p)}}^p$ converges uniformly to $\tau_\xi^\infty : \mathrm{T}_\xi \to \varepsilon^{-1}\mathcal{T}_\xi^{(\infty)}$ on $\mathrm{T}_\xi$. In particular, for all $p\geq 1$, we have the quantitative bound
$$
\sup_{u\in \mathrm{T}_\xi} \|\tau_{\xi^{(p)}}^p(u) -  \tau_\xi^\infty(u) \|_{\ell^2_\varpi} \leq r_{p+1}^{1/450}.
$$
\end{lemma}
\begin{proof}  Let $u\in \mathrm{T}_\xi$ and $p< q$. Using that $\tau_{\xi^{(p)}}^p$ is a $4-$Lipschitz map (see \eqref{eq:bilip_p}) and that $\psi_{\xi^{(p+1)}}^{p+1}, \cdots,   \psi_{\xi^{(q)}}^q$ are close to the identity (see \eqref{eq:proche_id_p}), we have
\begin{equation}
\label{eq:ca_cvg_unif}
\begin{split}
\| \tau_{\xi^{(p)}}^p(u) - \tau_{\xi^{(q)}}^q(u) \|_{\ell^2_\varpi}& = \| \tau_{\xi^{(p)}}^p(u) - \tau_{\xi^{(p)}}^p( \psi_{\xi^{(p+1)}}^{p+1} \circ \cdots    \circ  \psi_{\xi^{(q)}}^q(u) ) \|_{\ell^2_\varpi} \\
&\mathop{\leq}^{ \eqref{eq:bilip_p}} 4 \| u - \psi_{\xi^{(p+1)}}^{p+1} \circ \cdots    \circ  \psi_{\xi^{(q)}}^q(u)  \|_{\ell^2_\varpi} \\
&\leq 4 \sum_{j=p+1}^{q} \| \psi_{\xi^{(j)}}^{j} \circ \cdots    \circ  \psi_{\xi^{(q)}}^q(u)  - \psi_{\xi^{(j+1)}}^{j+1} \circ \cdots    \circ  \psi_{\xi^{(q)}}^q(u)  \|_{\ell^2_\varpi} \\
&\mathop{\leq}^{\eqref{eq:proche_id_p}} 4  \sum_{j=p+1}^{q} r_{j}^{1/440} \mathop{\leq}^{\eqref{eq:rp_decroit_bien_trop}} r_{p+1}^{1/450}\mathop{\longrightarrow}_{p\to +\infty} 0.
\end{split}
\end{equation}
Since $\ell^2_\varpi$ is a Banach space, it implies that $(\tau_{\xi^{(p)}}^p)_{| \mathrm{T}_\xi}$ converges uniformly. We denote by $\tau_\xi^\infty : \mathrm{T}_\xi \to \ell^2_\varpi$ its limit, i.e.
$$
\tau_\xi^\infty :=\lim_{p\to +\infty} (\tau_{\xi^{(p)}}^p)_{| \mathrm{T}_\xi} \quad \mathrm{in} \quad C^0( \mathrm{T}_\xi;\ell^2_\varpi ).
$$
Then, passing to the limit in \eqref{eq:bilip_p}, we get that for all $u,v\in \mathrm{T}_\xi$, 
\begin{equation}
\label{eq:lip_tau_inf}
\frac14\|u-v\|_{\ell^2_\varpi} \leq \| \tau_\xi^\infty(u) -  \tau_\xi^\infty(v) \|_{\ell^2_\varpi} \leq 4 \|u-v\|_{\ell^2_\varpi} .
\end{equation}
In implies that $\tau_\xi^\infty$ is a homeomorphism on its image that we denote $\varepsilon^{-1}\mathcal{T}_{\xi}^{(\infty)}:= \tau_\xi^\infty(\mathrm{T}_\xi)$.
\end{proof}

This lemma allows us to deduce the following proposition which prove that the embeddings $\Psi_\xi^{(\infty)}$ and $\Psi_{\xi^{(p)}}^{(p)}$ are close (and therefore, that the tori $\mathcal{T}_\xi^{(\infty)}$ and $\mathcal{T}_{\xi^{(p)}}^{(p)}$ are close).

\begin{proposition} \label{prop:les_plongements_sont_proches} For all $\xi \in \mathcal{O}_\infty$, the map $\Psi_\xi^{(\infty)} := \varepsilon \tau_\xi^\infty \circ \Upsilon_\xi^{\mathcal{S}_\infty} : \mathbb{T}^{\mathcal{S}_\infty} \to \mathcal{T}_\xi^{(\infty)}$ is a homeomorphism satisfying, for all $\theta \in \mathbb{T}^{\mathcal{S}_\infty}$ and all $p\geq 1$ (recall the definition of $\Psi_{\xi^{(p)}}^{(p)}$ in \eqref{eq:Psi_last}),
\begin{equation}
\label{eq:les_plongements_sont_proches}
\| \Psi_\xi^{(\infty)} (\theta) - \Psi_{\xi^{(p)}}^{(p)}(\theta^{(p)})  \|_{\ell^2_\varpi} \leq \varepsilon \, r_{p+1}^{1/460}
\end{equation}
where $\theta^{(p)} := \Pi_{\mathcal{S}_p} \theta$.
\end{proposition}
\begin{proof} Let $\xi \in \mathcal{O}_\infty$. Since by Lemma \ref{lem:petit_lemme_technique}, $ \tau_\xi^\infty: \mathrm{T}_\xi \to \varepsilon^{-1}\mathcal{T}_\xi^{(\infty)}$ is a homeomorphism, it is clear by composition that $\Psi_\xi^{(\infty)}$ is a homeomorphism. So we only focus on proving the quantitative bound \eqref{eq:les_plongements_sont_proches}. Let $\theta \in \mathbb{T}^{\mathcal{S}_\infty}$ and $p\geq 1$. We set
$$
u:=  \Upsilon_\xi^{\mathcal{S}_\infty} (\theta) \quad \mathrm{and} \quad u^{(p)} :=  \Upsilon_{\xi^{(p)}}^{\mathcal{S}_p} (\theta^{(p)}). 
$$
Note that this notation is consistent because by definition of $\Upsilon, \theta^{(p)}$ and $\xi^{(p)}$, we have that
$$
u^{(p)} = \Pi_{\mathcal{S}_p} u.
$$
By definition of $u$ and $u^{(p)}$, we have that $ \Psi_\xi^{(\infty)} (\theta) = \varepsilon \tau_\xi^\infty(u)$ and $ \Psi_\xi^{(p)} (\theta^{(p)}) = \varepsilon \tau_\xi^p(u^{(p)})$.
Applying the triangular inequality, we have
\begin{equation*}
\varepsilon^{-1}\| \Psi_\xi^{(\infty)} (\theta) - \Psi_{\xi^{(p)}}^{(p)}(\theta^{(p)})  \|_{\ell^2_\varpi}  \leq \underbrace{ \| \tau_\xi^\infty(u ) -   \tau_\xi^p(u ) \|_{\ell^2_\varpi} }_{=:I}  + \underbrace{ \| \tau_\xi^p(u ) -  \tau_\xi^p(u^{(p)} )  \|_{\ell^2_\varpi} }_{=:II}.
\end{equation*}
On the one hand, by Lemma \ref{lem:petit_lemme_technique}, we have $I\leq r_{p+1}^{1/450}$. On the other hand, using the Lipschitz estimate \eqref{eq:bilip_p}  on $\tau_\xi^p$, since $u^{(p)} = \Pi_{\mathcal{S}_p} u$, we have
\begin{equation*}
\frac{II}4 \leq  \| u  -  u^{(p)} \|_{\ell^2_\varpi} =   \| \Pi_{\mathcal{S}_p^c} u   \|_{\ell^2_\varpi} =  \big(\sum_{q>p} \xi_{i_{q}}\big)^{1/2} \mathop{\leq}^{\eqref{eq:rp_decroit_bien_trop}} r_{p+1}^{1/2}
\end{equation*}
where we have used that, by construction, for all $q\geq 2$, $\xi_{i_q} \leq 2r_q^{2\nu}$.

\end{proof}

\subsubsection{Invariance} Now we prove that $\mathcal{T}_\xi^{(\infty)} $ is a Kronecker torus for \eqref{eq:NLS_reg}.

\begin{proposition} \label{prop:inv_inf}  For all $\theta \in \mathbb{T}^{\mathcal{S}_\infty}$ and all $\xi \in \mathcal{O}_\infty$,
\begin{center}
$t \mapsto \Psi_\xi^{(\infty)}(  \omega^{(\infty)}(\xi) t + \theta)$ is a global solution to \eqref{eq:NLS_reg}.
\end{center}
\end{proposition}
\begin{proof} By Lemma \ref{lem:inv_p}, setting  $\theta^{(p)} := \Pi_{\mathcal{S}_p} \theta$, we know that  $t\mapsto \Psi_\xi^{(p)} ( \omega^{(p)}(\xi^{(p)}) t + \theta^{(p)})=: v^{(p)}(t)$ is a global solution to \eqref{eq:NLS_reg}.
 By continuity of $\nabla R$ on $\ell^2_\varpi$, mild solutions of \eqref{eq:NLS_reg} are strong solutions and so it suffices to prove that $ v^{(p)}$ converges to $v^{(\infty)} :=t\mapsto \Psi_\xi^{(\infty)}(  \omega^{(\infty)}(\xi) t + \theta)$ in $C^0(\mathbb{R};\ell^2_\varpi)$ (i.e. locally uniformly). 

\medskip

So let $\theta \in \mathbb{T}^{\mathcal{S}_\infty}$, $\xi \in \mathcal{O}_\infty$, $t\in \mathbb{R}$ and $p\geq 2$. We use proposition \ref{prop:les_plongements_sont_proches}, the Lipschitz estimate  \eqref{eq:bilip_p} on $\tau_{\xi^{(p)}}^p$ and the uniform bound \eqref{eq:cvg_freq_inf}  of Lemma \ref{lem:freq_inf} on the convergence of the frequencies, to get that we have that
\begin{equation*}
\begin{split}
\| v^{(p)}(t) - v^{(\infty)}(t) \|_{\ell^2_\varpi}&\leq  r_{p+1}^{1/460} + \| \Psi_\xi^{(p)}(  \omega^{(p)}(\xi^{(p)}) t + \theta^{(p)}) - \Psi_{\xi^{(p)}}^{(p)}(  \Pi_{\mathcal{S}_p}\omega^{(\infty)}(\xi) t + \theta^{(p)})  \|_{\ell^2_\varpi} \\
&\leq  r_{p+1}^{1/460} + 4 \| \Upsilon_{\xi^{(p)}}^{\mathcal{S}_p}(  \omega^{(p)}(\xi^{(p)}) t + \theta^{(p)} ) - \Upsilon_{\xi^{(p)}}^{\mathcal{S}_p}(  \Pi_{\mathcal{S}_p}\omega^{(\infty)}(\xi) t + \theta^{(p)})  \|_{\ell^2_\varpi}\\
&\lesssim r_{p+1}^{1/460} +  |t|\sup_{j\in \mathcal{S}_p} | \omega_j^{(p)}(\xi^{(p)}) - \omega_j^{(\infty)}(\xi) | \\
&\lesssim  r_{p+1}^{1/460} +  |t|  r_{p+1}^{1/83} \mathop{\longrightarrow}_{p\to \infty} 0.
\end{split}
\end{equation*}

\end{proof}

\subsection{Non resonant tori}\label{sec:NR} At the previous step, we have proven that for all $\xi \in \mathcal{O}_\infty$, $\mathcal{T}_{\xi}^{(\infty)} = \Psi_{\xi}^\infty(\mathbb{T}^{\mathcal{S}_\infty})$ is an infinite dimensional invariant torus on \eqref{eq:NLS_reg} whose frequencies are $\omega^{(\infty)}$. If the frequencies $\omega^{(\infty)}$ were rationally independent, $\mathcal{T}_{\xi}^{(\infty)}$ would be non resonant and we could easily conclude the proof of Theorem \ref{thm:main}. Unfortunately, we do not know how to prove that there exists $\xi \in \mathcal{O}_\infty$ such that the frequencies $\omega^{(\infty)}$ are rationally independent. To explain it more in details (and to introduce notations useful later), let us decompose $\omega^{(\infty)}$ into two parts
\begin{equation}
\label{eq:def_omeganat_et_a}
\omega^{(\infty)}_j(\xi) = \underbrace{ j^2 +  \varepsilon^2  \xi_j + 2\varepsilon^3  L^{(\infty)}_j(\xi) }_{=:\omega_j^{(\mathrm{nat})}(\xi)}  + \underbrace{2\varepsilon^3 A_1^{(\infty)}(\xi) - 2 \varepsilon^2  \sum_{i \in \mathcal{S}_\infty} \xi_i }_{=:a(\xi)}.
\end{equation}
Actually, we will prove in subsection \ref{sub:param_nr} that for $\mathbb{P}$-almost all $\xi\in \mathcal{O}_\infty$, the frequencies $\omega^{(\mathrm{nat})}(\xi)$ are rationally independent. Since $a(\xi)$ does not depend on $j$, it will imply that the space of  linear combinations vanishing the frequencies $\omega^{(\infty)}(\xi) $ is of dimension smaller than or equal to $1$ and so that trajectories of \eqref{eq:NLS_reg} living in $\mathcal{T}_{\xi}^{(\infty)}$ fill infinite dimensional tori which are potentially smaller than $\mathcal{T}_{\xi}^{(\infty)}$ (and that we are going to identify in subsection \ref{sub:9.5.2}).

\medskip

Finally let us just note that since the maps $\psi_\xi^p$ are invariant by gauge transform, we have that for all $t\in \mathbb{R}$, all $\xi \in \mathcal{O}_\infty$ and all $\theta \in \mathbb{T}^{\mathcal{S}_\infty}$
$$
\Psi_\xi^{(\infty)}( \omega^{(\infty)}(\xi) t + \theta) = e^{-\ic t\, a(\xi) } \Psi_\xi^{(\infty)}( \omega^{(nat)}(\xi) t + \theta) .
$$
Therefore, the potential default of rational independency is only due to a gauge transform. Note that it could be compensated by choosing $f(0)\neq 0$.

\subsubsection{The parameters }
\label{sub:param_nr}

In this subsection, we aim at proving that  $\mathbb{P}-$almost all  $\xi \in \mathcal{O}_\infty$, the families $\omega^{(\mathrm{nat})}(\xi)$ and  $(\omega^{(\infty)}_i(\xi))_{i\in \mathcal{S}_1}$ are rationally independent, or in other words, that :
\begin{lemma} \label{lem:full_meas_sharp} Setting 
\begin{equation}
\label{eq:def_O_sharp_inf}
\mathcal{O}_\infty^{\sharp} := \{ \xi \in \mathcal{O}_\infty \ | \quad \omega^{(\mathrm{nat})}(\xi) \quad \mathrm{and} \quad (\omega^{(\infty)}_i(\xi))_{i\in \mathcal{S}_1} \quad \mathrm{are} \quad \mathbb{Q}-\mathrm{free} \},
\end{equation}
we have
$$
\mathbb{P}(  \mathcal{O}_\infty \setminus \mathcal{O}_\infty^{\sharp} ) =0.
$$
\end{lemma}
\begin{proof} \noindent \underline{$\bullet$ \emph{Step 1 : Lipschitz estimates.}}  We are going to apply Lemma \ref{lem:pd_tech}. So first, we have to prove Lipschitz estimates for $L^{(\infty)}$ and $A_1^{(\infty)}$.  We recall that by construction $\| \Pi_{\mathrm{nor}} K^p \|_{\mathscr{N}_{r_p,\mathcal{O}_p,\mathcal{S}_p}^{\mathrm{tot}}} \leq 3$ for all $p\geq 1$ and that $L^{(p)} = \Pi_{\mathscr{L}} K^p$ and $A_0^{(p)}+A_1^{(p)}\mu = \Pi_{\mathscr{L}} K^p$. By definition of the $\mathscr{N}_{r_p,\mathcal{O}_p,\mathcal{S}_p}^{\mathrm{tot}}$ norm, it implies that\footnote{Here it is crucial that factor $r_p^2$ is not there, see Remark \ref{rem:r2}.  } $\| L^{(p)} \|_{\mathscr{L}_{\mathcal{O}_p,\mathcal{S}_p}^{\mathrm{dir}}} \leq 3$ and $\| A_0^{(p)}+A_1^{(p)}\mu \|_{\mathscr{A}_{\mathcal{O}_p,\mathcal{S}_p}^{\mathrm{lip}}} \leq 3$. In particular, for all $i\in \mathcal{S}_\infty$, all $(\xi,\zeta) \in \Delta_i \mathcal{O}_\infty $, for all $k\in \mathbb{Z}$,  and $p$ large enough so that $i\in \mathcal{S}_p$ we have
$$
|L^{(p)}_k(\xi^{(p)}) - L^{(p)}_k(\zeta^{(p)})| \leq   3  \Gamma_{k,i} |\xi_i - \zeta_i| \quad \mathrm{and} \quad |A^{(p)}_1(\xi^{(p)}) - A^{(p)}_1(\zeta^{(p)})| \leq   3  |\xi_i - \zeta_i| .
$$ 
Passing to the limit as $p$ goes to $+\infty$ in these estimates (thanks to Lemma \ref{lem:freq_inf}), we get
\begin{equation}
\label{eq:est_lip_quon_aime}
|L^{(\infty)}_k(\xi) - L^{(\infty)}_k(\zeta)| \leq   3  \Gamma_{k,i} |\xi_i - \zeta_i| \quad \mathrm{and} \quad |A^{(\infty)}_1(\xi) - A^{(\infty)}_1(\zeta)| \leq   3  |\xi_i - \zeta_i| .
\end{equation}

\medskip

\noindent \underline{$\bullet$ \emph{Step 2 : Independence for $\omega^{(\mathrm{nat})}$.}} Let $\boldsymbol{k}\in \mathbb{Z}^{\mathcal{S}_\infty}\setminus \{0\}$ be sequence of integers with finite support. We aim at proving that 
$$
B:=\mathbb{P}(\{ \xi \in \mathcal{O}_\infty \ | \ \sum_{i\in \mathcal{S}_\infty} \boldsymbol{k}_i  \omega^{(\mathrm{nat})}_i(\xi)= 0\}) = 0.
$$
Since the set of such $\boldsymbol{k}$ is countable, it will imply that for $\mathbb{P}$-almost all $\xi \in \mathcal{O}_\infty$, $ \omega^{(\mathrm{nat})}$ is $\mathbb{Q}$-free. 

\medskip

Since $\boldsymbol{k}$ has finite support, there exists $p\geq 2$ such that $\boldsymbol{k} \in \mathbb{Z}^{\mathcal{S}_p}$. Moreover, by definition of the product measure $\mathbb{P}$, there exist two probability measures $\mathbb{P}_{\mathcal{S}_p}$ and $\mathbb{P}_{ \mathcal{S}_{>p}}$ supported on $\mathbb{R}^{{\mathcal{S}_p}}$ and $\mathbb{R}^{\mathcal{S}_{>p}}$ respectively such that $\mathbb{P} = \mathbb{P}_{\mathcal{S}_p}\otimes \mathbb{P}_{\mathcal{S}_{>p}} $ (where $\mathcal{S}_{>p} = \mathcal{S}_\infty \setminus \mathcal{S}_p$). Therefore, setting
$$
\forall \xi' \in \mathbb{R}^{\mathcal{S}_{>p}}, \quad \mathcal{O}_{\xi',p} := \{ \xi \in \mathbb{R}^{p} \ | \ \xi + \xi' \in \mathcal{O}_\infty \},
$$
we have, for all $\gamma >0$
\begin{equation}
\label{eq:pas_beau}
\begin{split}
B &\leq \mathbb{P}(\{ \xi \in \mathcal{O}_\infty \ | \ \big| \sum_{i\in \mathcal{S}_\infty} \boldsymbol{k}_i  \omega^{(\mathrm{nat})}_i(\xi) \big| <  \varepsilon^2 \gamma\}) \\
&= \int_{\xi'\in \Pi_{\mathcal{S}_{>p}} \mathcal{O}_\infty}   \mathbb{P}_{\mathcal{S}_p} (\{ \xi \in \mathcal{O}_{\xi',p} \ | \ \big| \sum_{i\in \mathcal{S}_p} \boldsymbol{k}_i  \omega^{(\mathrm{nat})}_i(\xi+\xi') \big| <  \varepsilon^2 \gamma\}) \mathrm{d}\mathbb{P}_{\mathcal{S}_{>p}}(\xi') \\
&\leq \sup_{\xi'\in \Pi_{\mathcal{S}_{>p}} \mathcal{O}_\infty}  \mathrm{Leb} (\{ \xi \in \mathcal{O}_{\xi',p} \ | \ \big| \sum_{i\in \mathcal{S}_p} \boldsymbol{k}_i  \varepsilon^{-2} \omega^{(\mathrm{nat})}_i(\xi+\xi') \big| <   \gamma\}).
\end{split}
\end{equation}
In order to estimate this supremum, recalling the definition  \eqref{eq:def_omeganat_et_a} of $\omega^{(\mathrm{nat})}$, we aim at applying Lemma \ref{lem:pd_tech} with $\mathcal{S} = \mathcal{S}_p$, $\mathcal{O} = \mathcal{O}_{\xi',p}$, $a = \varepsilon^{-2} \sum_{j\in \mathcal{S}_p} \boldsymbol{k}_jj^2$ and $g = \varepsilon (L^{(\infty)}_i(\cdot+\xi'))_{i\in \mathcal{S}_p}$. We just have to check the Lipschitz assumption of Lemma \ref{lem:pd_tech}. So let $i\in \mathcal{S}_p$ and $(\xi,\zeta) \in \Delta_i  \mathcal{O}_{\xi',p} $. Observing that by construction $(\xi+\xi',\zeta+\xi')\in \Delta_i \mathcal{O}_\infty$, we have 
\begin{equation*}
\begin{split}
 \| \boldsymbol{k} \|_{\ell^\infty}^{-1} |a(\xi) -a(\zeta)  | + \sum_{k\in \mathcal{S}_p}  |g_k(\xi) -g_k(\zeta)  | &= \varepsilon \sum_{k\in \mathcal{S}_p}  |L_k(\xi+\xi') -L_k(\zeta+\xi')  | \\ 
 & \mathop{\leq}^{\eqref{eq:est_lip_quon_aime}} 3 \varepsilon \sum_{k\in \mathcal{S}_p} \Gamma_{k,i} | \xi_i - \zeta_i | \mathop{\leq}^{\eqref{eq:assump_eps_proof_fin}} \frac{| \xi_i - \zeta_i |}2.
\end{split}
\end{equation*}
Finally, by Lemma \ref{lem:pd_tech} to estimate the last term of \eqref{eq:pas_beau}, we have that
$$
\forall \gamma >0, \quad B\lesssim \gamma,
$$
i.e. $B=0$.

\noindent \underline{$\bullet$ \emph{Step 3 : Independence for $(\omega^{(\infty)}_i(\xi))_{i\in \mathcal{S}_1}$.}} Let $\boldsymbol{k}\in \mathbb{Z}^{\mathcal{S}_1}\setminus \{0\}$ be sequence of integers with finite support. We aim at proving that 
$$
C:=\mathbb{P}(\{ \xi \in \mathcal{O}_\infty \ | \ \sum_{i\in \mathcal{S}_1} \boldsymbol{k}_i  \omega^{(\infty)}_i(\xi)= 0\}) = 0.
$$
Let 
$$
D = \left\{ \begin{array}{llc} \mathbb{R}^{\mathcal{S}_1} &\to & \mathbb{R}^{\mathcal{S}_1} \\
\xi &\to & \displaystyle \big(\xi_i - 2\sum_{j\in \mathcal{S}_1} \xi_j\big)_{i\in \mathcal{S}_1} 
\end{array} \right.
$$
We note that $D$ is a symmetric and invertible\footnote{we can easily check that $D^{-1} = \mathrm{Id} - (D-\mathrm{Id})/(1-2\# \mathcal{S}_1)$.} linear operator. 

\medskip

Proceeding as in \eqref{eq:pas_beau}, for all $\gamma >0$, we have
\begin{equation}
\label{eq:plus_leger}
\begin{split}
C &\leq \sup_{\xi'\in \Pi_{\mathcal{S}_{>1}} \mathcal{O}_\infty}  \mathrm{Leb} (\{ \xi \in \mathcal{O}_{\xi',1} \ | \ \big|\underbrace{ \sum_{i\in \mathcal{S}_1} \boldsymbol{k}_i  \varepsilon^{-2} \omega^{(\infty)}_i(\xi+\xi')}_{=:h_{\xi'}(\xi)} \big| <   \gamma\}).
\end{split}
\end{equation}
From now, we fix $\xi'\in \Pi_{\mathcal{S}_{>1}} \mathcal{O}_\infty$ (recall that $\mathcal{S}_{>1} = \mathcal{S}_\infty \setminus \mathcal{S}_1$). Then, we note that, by definition of $\omega^{(\infty)}$, for all $\xi \in \mathcal{S}_1$ such that $\xi+\xi' \in \mathcal{O}_\infty$ (with $\xi \in \mathbb{R}^{\mathcal{S}_1}$), we have 
$$
h_{\xi'}(\xi)  =\underbrace{ \sum_{i\in \mathcal{S}_1} \boldsymbol{k}_i \Big(\varepsilon^{-2} i^2 -2 \sum_{j\in S_\infty \setminus S_1} \xi'_j +  2\varepsilon A_1^\infty(\xi+\xi')\Big) }_{=:a(\xi)}+ \sum_{i\in \mathcal{S}_1} \boldsymbol{k}_i ((D\xi)_i + \underbrace{2\varepsilon L^\infty_i(\xi+\xi')}_{=:v_{i}(\xi)} ).
$$
and so, using that $D$ is symmetric invertible,
$$
h_{\xi'}(\xi)  = a(\xi) +  \sum_{i\in \mathcal{S}_1} (D\boldsymbol{k})_i (\xi_i + g_i(\xi)) \quad \mathrm{where} \quad g_i(\xi) =(D^{-1}v(\xi))_i .
$$
Now we have to check the Lipschitz assumption of Lemma \ref{lem:pd_tech}.  So let $i\in \mathcal{S}_p$ and $(\xi,\zeta) \in \Delta_i  \mathcal{O}_{\xi',1} $. We recall that, as previously, by construction, we have $(\xi+\xi',\zeta+\xi')\in \Delta_i \mathcal{O}_\infty$. On the one hand, provided that $\varepsilon$ is smaller than a constant depending only on $\mathcal{S}_1$, we have
$$
 \|D \boldsymbol{k} \|_{\ell^\infty}^{-1} |a(\xi) -a(\zeta)  | \lesssim_{\mathcal{S}_1} \varepsilon | A_1^\infty(\xi+\xi') -  A_1^\infty(\zeta+\xi')| \mathop{\leq}^{\eqref{eq:est_lip_quon_aime}} \frac{|\xi_i -\zeta_i|}4.
$$
On the other hand, provided that $\varepsilon$ is smaller than a constant depending only on $\mathcal{S}_1$, we have
$$
\sum_{k\in \mathcal{S}_1} |g_k(\xi) - g_k(\zeta) | \lesssim_{\mathcal{S}_1} \sup_{k\in \mathcal{S}_1} |L^\infty_k(\xi+\xi') - L^\infty_k(\xi+\xi')|\mathop{\leq}^{\eqref{eq:est_lip_quon_aime}} \frac{|\xi_i -\zeta_i|}4.
$$
Therefore, the Lipschitz assumption of Lemma \ref{lem:pd_tech} holds and, by applying Lemma \ref{lem:pd_tech} to estimate the last term of \eqref{eq:plus_leger}, we have that
$$
\forall \gamma >0, \quad C\lesssim_{\mathcal{S}_1} \gamma,
$$
i.e. $C=0$.
\end{proof}

\subsubsection{The homeomorphisms and the non resonant frequencies}\label{sub:9.5.2} For all $\xi \in \mathcal{O}_\infty^{\sharp}$, we aim at constructing a non-resonant infinite dimensional Kronecker torus close to $\mathcal{T}_{\xi^{(1)}}^\varepsilon$
(the invariant torus based on the sites $\mathcal{S}_1$; see Lemma \ref{lem:homeo_uno}) which coincides, when $p=1$, with the more general notation $ \mathcal{T}_{\xi^{(p)}}^{(p)}$ introduced in Corollary \ref{cor:homeo_p} .

\begin{proposition} \label{prop:ultimate} For all $\xi \in \mathcal{O}_\infty^{\sharp}$, there exists a set $\mathcal{T}_\xi^{(\mathrm{nr})} \subset  \ell^2_\varpi$, a homeomorphism $\widetilde{\Psi}_\xi : \mathbb{T}^{\mathbb{N}^*} \to \mathcal{T}_\xi^{(\mathrm{nr})}$ and a sequence $\widetilde{\omega}(\xi) \in \mathbb{R}^{\mathbb{N}^*}$ of rationally independent numbers satisfying the following properties.
\begin{itemize}
\item[(a)] For all $\theta \in \mathbb{T}^{\mathbb{N}_*}$, $t\mapsto \widetilde{\Psi}_\xi(\widetilde{\omega}(\xi) t + \theta )$ is a global solution to \eqref{eq:NLS_reg}.
\item[(b)] $\mathcal{T}_\xi^{\mathrm{nr}} $ is close to $\mathcal{T}_{\xi^{(1)}}^\varepsilon$ (for the Hausdorff distance), i.e.
$$
\mathrm{dist}(\mathcal{T}_\xi^{(\mathrm{nr})} , \mathcal{T}_{\xi^{(1)}}^\varepsilon)\leq \varepsilon \, \varrho^{1/460}.
$$
\end{itemize}

\end{proposition}
\begin{proof} \underline{$\bullet$ \emph{Step 1 : Re-indexation.}} Let $N=\# \mathcal{S}_1$, $(j_1,\cdots,j_N)\in \mathbb{Z}^N$ be a sequence of distinct integers such that
$$
\mathcal{S}_1 = \{ j_1,\cdots,j_N \}
$$
 and set, for $k>N$,
$$
j_k := i_{k-N+1}.
$$
Note that it is just a way to index the sites of $\mathcal{S}_\infty$, in the sense that the numbers $(j_k)_{k\geq 1}$ are distinct and that we have
$$
\mathcal{S}_\infty = \{ j_k \ | \ k\geq 1 \}.
$$

\medskip

Then, for $\xi \in \mathcal{O}_\infty^{\sharp}$ and $1\leq p\leq \infty$, we define the re-indexed quantities
$$
\omega^{(\mathrm{ri})} := (\omega_{j_k}^{(\infty)})_{k\geq 1}, \quad \Psi_\xi^{(p,\mathrm{ri})}  = \Psi_\xi^{(p)} \circ \mathfrak{P}_p^{-1} 
$$
where $\mathfrak{P}_p : \mathbb{T}^{\mathcal{S}_p} \to \mathbb{T}^{\# \mathcal{S}_p}$ is the natural homeomorphism defined by permutation\footnote{with the convention that $\# \mathcal{S}_\infty = \infty$.}, i.e.
$$
\forall \theta \in \mathbb{T}^{\mathcal{S}_p}, \forall k\leq \# \mathcal{S}_p, \quad (\mathfrak{P}(\theta))_k := \theta_{j_k}.
$$
With these notations, by Proposition \ref{prop:les_plongements_sont_proches} and Corollary \ref{cor:homeo_p}, we know that $\Psi_\xi^{(\infty,\mathrm{ri})} : \mathbb{T}^{\mathbb{N}_*} \to \mathcal{T}_\xi^{\infty}$ and $\Psi_{\xi^{(1)}}^{(1,\mathrm{ri})} : \mathbb{T}^{N} \to \mathcal{T}_{\xi^{(1)}}^{(1)} $ are homeomorphisms, for all $\theta \in \mathbb{T}^{\mathbb{N}_*}$
\begin{equation}
\label{eq:inv_ri}
 t\mapsto \Psi_\xi^{(\infty,\mathrm{ri})}(\omega^{(\mathrm{ri})}(\xi) t + \theta ) \mathrm{\ is \ a \ global \  solution \ to \ \eqref{eq:NLS_reg}}
\end{equation}
 and (recall that $\varrho=r_2$)
\begin{equation}
\label{eq:la_super_est_dist}
 \| \Psi_\xi^{(\infty,\mathrm{ri})}(\theta) - \Psi_{\xi^{(1)}}^{(1,\mathrm{ri})}(\Pi_{\leq N} \theta) \|_{\ell^2_\varpi} \leq \varepsilon \, \varrho^{1/460}
\end{equation}
where $\Pi_{\leq N} \theta = (\theta_k)_{k\leq N}$. Finally, we recall that by construction $\mathcal{T}_{\xi^{(1)}}^{\varepsilon} =\mathcal{T}_{\xi^{(1)}}^{(1)} $ (see Corollary \ref{cor:homeo_p}).

\medskip

\noindent \underline{$\bullet$ \emph{Step 2 : Disjunction and linear algebra.}} Now, we define
$$
\mathcal{O}_\infty^{\mathrm{nr}} := \{ \xi \in \mathcal{O}_\infty^\sharp \ | \  \omega^{(\infty)}(\xi) \quad \mathrm{is} \quad \mathbb{Q}-\mathrm{free}  \}.
$$
If $\xi \in \mathcal{O}_\infty^{\mathrm{nr}}$, it suffices to set $\mathcal{T}_\xi^{(\mathrm{nr})} =\mathcal{T}_\xi^{(\infty)}$, $\widetilde{\Psi}_\varepsilon = \Psi_\xi^{(\infty,\mathrm{ri})}$ and $\widetilde{\omega}(\xi) = \omega^{(\mathrm{ri})}(\xi)$. Indeed, since $\xi \in \mathcal{O}_\infty^{\mathrm{nr}}$ the frequencies are rationally independent, the property $\mathrm{(a)}$ is given by \eqref{eq:inv_ri} and the property  $\mathrm{(b)}$ is a direct consequence of \eqref{eq:la_super_est_dist}.

\medskip

So from now on we assume that $\xi \in  \mathcal{O}_\infty^\sharp \setminus \mathcal{O}_\infty^{\mathrm{nr}}$. By definition of $\mathcal{O}_\infty^{\mathrm{nr}}$, it implies that there exist an integer $M\geq 1$ and a sequence of setwise coprime integers $\boldsymbol{d}(\xi)\in \mathbb{Z}^M \setminus \{0\}$ such that
$$
\boldsymbol{d}_1 (\xi) \omega^{(\mathrm{ri})}_1(\xi) + \cdots + \boldsymbol{d}_M(\xi) \omega^{(\mathrm{ri})}_M(\xi) =0. 
$$
Since $ (\omega^{(\infty)}_i(\xi))_{i\in \mathcal{S}_1}$ are rationally independent (by definition of $\widetilde{\mathcal{O}}_\infty$), we note that  
\begin{equation}
\label{eq:M>N}
M>N \quad \mathrm{and} \quad \exists i_* >N, \ \boldsymbol{d}_{i_*}(\xi) \neq 0.
\end{equation}
Since $\boldsymbol{d}(\xi)$ is a sequence of coprime numbers, as a consequence of B\'ezout's identity (and an induction), there exists a matrix $D(\xi) \in \mathrm{SL}_M(\mathbb{Z})$ with integer coefficients and of determinant $\mathrm{det} \, D(\xi) = 1$ whose first line is $\boldsymbol{d}(\xi)$, i.e.
$$
\forall k\leq M, \quad D_{1,k}(\xi) = \boldsymbol{d}_k(\xi).
$$
Then, we define an operator by blocks
$$
\Delta_\xi := \begin{pmatrix} D(\xi) \\ & \mathrm{Id}_{\mathbb{R}^{\rrbracket M, +\infty \llbracket }} \end{pmatrix}: \mathbb{R}^{\mathbb{N}^* } \to \mathbb{R}^{\mathbb{N}^* }.
$$
Note that this operator is invertible, of invert
\begin{equation}
\label{eq:block_structure}
\Delta_\xi^{-1} := \begin{pmatrix} (D(\xi))^{-1} \\ & \mathrm{Id}_{\mathbb{R}^{\rrbracket M, +\infty \llbracket }} \end{pmatrix}: \mathbb{R}^{\mathbb{N}^* } \to \mathbb{R}^{\mathbb{N}^* }.
\end{equation}
Since the coefficients of $D(\xi)$ are integers the map
$$
\mathfrak{D}_\xi := \left\{ \begin{array}{lll} \mathbb{T}^{\mathbb{N}^*} & \to &\mathbb{T}^{\mathbb{N}^*}  \\ \theta & \mapsto & \Delta_\xi \theta   \end{array} \right.
$$
is well define. Moreover, since $\mathrm{det} \,  D(\xi) = 1$, the coefficients of $D^{-1}(\xi)$ are integers and so $\mathfrak{D}_\xi $ is invertible, of invert $\mathfrak{D}_\xi^{-1} = \theta \mapsto \Delta_\xi^{-1} \theta$. We note that $\mathfrak{D} _\xi$ and $\mathfrak{D}_\xi^{-1}$ are clearly continuous : they are homeomorphisms.

\medskip

Now, we define the shift operator
$$
\mathfrak{S} := \left\{ \begin{array}{lll} \mathbb{T}^{\mathbb{N}^*} & \to &\mathbb{T}^{\mathbb{N}^*}  \\ \theta & \mapsto & \begin{pmatrix} 0 \\ \theta\end{pmatrix}  \end{array} \right. 
$$
and we set
$$
\widetilde{\Psi}_\xi := \Psi_\xi^{(\infty,\mathrm{ri})}\circ \mathfrak{D}_\xi^{-1} \circ \mathfrak{S}, \quad \mathcal{T}_\xi^{(\mathrm{nr})} :=\widetilde{\Psi}_\xi(\mathbb{T}^{\mathbb{N}^*} ) \quad \mathrm{and} \quad \forall k\geq 1, \quad \widetilde{\omega}_k(\xi) = (\Delta_{\xi} \omega^{(\mathrm{ri})}(\xi) )_{k+1}.
$$

\medskip

It remains to check all the claims of Proposition \ref{prop:ultimate}.
\begin{itemize}
\item First, we note that by composition $\widetilde{\Psi}_\xi$ is injective and continuous. Since, by Tychonoff's theorem, $\mathbb{T}^{\mathbb{N}^*}$ is a compact set, it follows that $\widetilde{\Psi}_\xi$ is a homeomorphism.
\item  Now, we check that $\mathcal{T}_\xi^{(\mathrm{nr})} $ is a Kronecker torus (i.e. property $\mathrm{(a)}$).  So let $\theta \in \mathbb{T}^{\mathbb{N}^*}$ and $t\in \mathbb{R}$.  We observe that by definition of $\boldsymbol{d}(\xi)$ and $D(\xi)$, we have $(\Delta_\xi \omega^{(\mathrm{ri})}(\xi))_1=0$. Thus, by construction, we get
\begin{equation}
\label{eq:utile_fin}
\begin{split}
\widetilde{\Psi}_\xi (\widetilde{\omega}(\xi)t + \theta )&=  \Psi_\xi^{(\infty,\mathrm{ri})}\circ \mathfrak{D}_\xi^{-1} ( \Delta_\xi \omega^{(\mathrm{ri})}(\xi) + \mathfrak{S} \theta ) 
= \Psi_\xi^{(\infty,\mathrm{ri})} (  \omega^{(\mathrm{ri})}(\xi) t + \mathfrak{D}_\xi^{-1} \mathfrak{S} \theta ).
\end{split}
\end{equation}
It follows by \eqref{eq:inv_ri} that, for all $\theta \in \mathbb{T}^{\mathbb{N}^*}$, $t\mapsto \widetilde{\Psi}_\xi (\widetilde{\omega}(\xi)t + \theta )$ is a global solution of \eqref{eq:NLS_reg}.
\item Then, we prove that the frequencies $\widetilde{\omega}(\xi)$ are rationally independent. Indeed, let $M'\geq 1$ and $\boldsymbol{c} \in \mathbb{Z}^{M'}$ be such that
\begin{equation}
\label{eq:liaison}
\boldsymbol{c}_1 \widetilde{\omega}_1(\xi) + \cdots + \boldsymbol{c}_{M'}  \widetilde{\omega}_{M'}(\xi)=0.
\end{equation}
We aim at proving that $\boldsymbol{c}=0$. Without loss of generality, we assume that $M'\geq M$. Then we set $\widetilde{\boldsymbol{d}} = (\boldsymbol{d}(\xi),0,\cdots,0) \in \mathbb{Z}^{M'}$,
$$
B := \begin{pmatrix} D(\xi) \\ & I_{M' - M}\end{pmatrix} \quad \mathrm{and} \quad \boldsymbol{e} := (0,\boldsymbol{c}) B.
$$
By definition of $\widetilde{\omega}(\xi)$, the linear combination \eqref{eq:liaison} rewrites
$$
\boldsymbol{e}_1 \omega^{(\mathrm{ri})}_1(\xi) + \cdots + \boldsymbol{e}_{M'}  \omega^{(\mathrm{ri})}_{M'}(\xi)=0.
$$
Moreover by definition of $\widetilde{\boldsymbol{d}}$, we also have
$$
\widetilde{\boldsymbol{d}}_1(\xi) \omega^{(\mathrm{ri})}_1(\xi) + \cdots + \widetilde{\boldsymbol{d}}_{M'}(\xi)  \omega^{(\mathrm{ri})}_{M'}(\xi)=0.
$$
Then, since $\xi \in \mathcal{O}_\infty^{\sharp}$, the frequencies $\omega^{(\mathrm{nat})}$ are rationally independent and so the frequencies $(\omega^{(\mathrm{ri})}_k(\xi) -m(\xi))_{1\leq k \leq M'}$ are rationally independent. It follows that the vectors $\boldsymbol{e}$ and $\widetilde{\boldsymbol{d}}$ are collinear, i.e. there exists $\alpha \in \mathbb{Q}$ such that
$$
\boldsymbol{e} = \alpha \widetilde{\boldsymbol{d}}.
$$
Furthermore, since $\boldsymbol{d}(\xi)$ is the first line of the matrix $D$, we have
$$
(0,\boldsymbol{c}) = \boldsymbol{e} B^{-1} = \alpha   \widetilde{\boldsymbol{d}} B^{-1} = \alpha  (\boldsymbol{d}(\xi) D^{-1}(\xi), 0,\cdots,0) = \alpha(1,0,\cdots,0),
$$
of which we deduce that $\alpha=0$ and so $\boldsymbol{c}=0$.
\item Finally, we aim at proving that $\mathcal{T}_\xi^{\mathrm{nr}} $ is close to $\mathcal{T}_{\xi^{(1)}}^\varepsilon$ (property $\mathrm{(b)}$). Since $\mathcal{T}_\xi^{\mathrm{nr}}  \subset \mathcal{T}_\xi^{(\infty)}$, by Proposition \ref{prop:les_plongements_sont_proches}, recalling that by construction $r_2 = \varrho$, we have that 
$$
\forall u\in \mathcal{T}_\xi^{\mathrm{nr}} ,\exists v\in \mathcal{T}_{\xi^{(1)}}^\varepsilon, \quad \| u-v\|_{\ell^2_\varpi} \leq  \varepsilon \, \varrho^{1/460}.
$$
Now, we have to prove that for all $ v\in \mathcal{T}_{\xi^{(1)}}^\varepsilon$ there exists $u\in \mathcal{T}_\xi^{\mathrm{nr}}$ such that $\| u-v\|_{\ell^2_\varpi}\leq  \varepsilon \, \varrho^{1/460}$. Since $\Psi_{\xi^{(1)}}^{(\mathrm{ri},1)}$ is  bijective, there exists $\varphi \in \mathbb{T}^N$ such that 
$$
v =\Psi_{\xi^{(1)}}^{(\mathrm{ri},1)}(\varphi).
$$
Then, assume for the moment that $\lambda \mapsto \Pi_{\leq N}\Delta_\xi^{-1}\begin{pmatrix} 0 \\ \lambda \end{pmatrix}  : \mathbb{R}^{\mathbb{N}^*} \to \mathbb{R}^N$ is surjective. Passing to the quotient, it implies that there exists $\theta \in \mathbb{T}^{\mathbb{N}^*}$ such that $\Pi_{\leq N}\mathfrak{D}_\xi^{-1} \mathfrak{S}\theta= \varphi$. Moreover, applying \eqref{eq:utile_fin} when $t=0$, we have
$$
\widetilde{\Psi}_\xi ( \theta ) = \Psi_\xi^{(\infty,\mathrm{ri})} (   \mathfrak{D}_\xi^{-1} \mathfrak{S} \theta ).
$$
Thus, setting $u= \widetilde{\Psi}_\xi ( \theta )$ and applying \eqref{eq:la_super_est_dist}, we have  $\| u-v\|_{\ell^2_\varpi} \leq  \varepsilon \, \varrho^{1/460}$ and we are done.

\medskip

So it remains to prove that $\lambda \mapsto \Pi_{\leq N}\Delta_\xi^{-1}\begin{pmatrix} 0 \\ \lambda \end{pmatrix}    : \mathbb{R}^{\mathbb{N}^*} \to \mathbb{R}^N$ is surjective. Let $\boldsymbol{g} \in \mathbb{R}^N$ be such that
\begin{equation}
\label{eq:coucoucoucou}
\forall \lambda \in \mathbb{R}^{\mathbb{N}^*}, \quad  \boldsymbol{g}_1 (\Delta_\xi^{-1}\begin{pmatrix} 0 \\ \lambda \end{pmatrix} )_{1} + \cdots + \boldsymbol{g}_N (\Delta_\xi^{-1}\begin{pmatrix} 0 \\ \lambda \end{pmatrix} )_{N} =0.
\end{equation}
By duality, we have to prove that $\boldsymbol{g} =0$.
We extend $\boldsymbol{g}$ as a vector of $\mathbb{R}^M$  by setting $\widetilde{\boldsymbol{g}} = (\boldsymbol{g},0,\cdots,0) \in \mathbb{R}^M$. Recalling the block structure \eqref{eq:block_structure} of $\Delta_\xi^{-1}$, it follows that \eqref{eq:coucoucoucou} rewrite
$$
\forall \lambda \in \mathbb{R}^{M-1}, \quad  \widetilde{\boldsymbol{g}} _1 (D^{-1}(\xi)\begin{pmatrix} 0 \\ \lambda \end{pmatrix} )_{1} + \cdots +  \widetilde{\boldsymbol{g}}_M (D^{-1}(\xi)\begin{pmatrix} 0 \\ \lambda \end{pmatrix} )_{M} =0.
$$
By duality, it means that there exists $\beta \in \mathbb{R}$ such that
$$
\widetilde{\boldsymbol{g}} \, D^{-1}(\xi) = \beta (1,0,\cdots,0).
$$
Recalling that $\boldsymbol{d}(\xi)$ is the first line of $D(\xi)$, it means that
$$
(\boldsymbol{g},0,\cdots,0) = \widetilde{\boldsymbol{g}} = \beta \boldsymbol{d}(\xi).
$$
But since there exists $i_*>N$ such that $\boldsymbol{d}_{i_*}(\xi) \neq 0$ (see \eqref{eq:M>N}), it implies that $\beta =0$ and, as expected, $\boldsymbol{g}=0$.
\end{itemize}

\end{proof}

\subsection{Conclusion of the proof}\label{sec:end} \label{sub:conclusion_of_the_proof} To conclude the proof of Theorem \ref{thm:main}, we still have to define the set $\Xi_\varepsilon \subset \mathscr{O}_\varepsilon$ and to prove items $\mathrm{(iii)}$ and $\mathrm{(v)}$. 

\medskip

We recall that all the objects we constructed from subsection \ref{sub:plein_plein_objets} depend in fact on the parameters $\varrho$ and $\varepsilon$. Nevertheless, for readability, as explained in Remark \ref{rem:onalege_les_notations}, we did not mention explicitly these dependencies. From now on these dependencies are crucial and we add subscripts $\varrho,\varepsilon$. For example, the Kronecker torus $ \mathcal{T}_{\xi}^{(\mathrm{nr})}$ of Proposition \ref{prop:ultimate} is denoted by $\mathcal{T}_{\varrho,\varepsilon,\xi}^{(\mathrm{nr})}$ and the set of parameters $\mathcal{O}_{\infty}^{\sharp}$ is denoted by $\mathcal{O}_{\varrho,\varepsilon,\infty}^{\sharp}$.

\medskip

Then, we set
$$
\Xi_\varepsilon := \bigcap_{k_0 \geq 1} \Xi_\varepsilon^{(k_0)} \quad \mathrm{where} \quad \Xi_\varepsilon^{(k_0)}  := \bigcup_{k \geq k_0}  \Pi_{\mathcal{S}_1}  \mathcal{O}_{2^{-k},\varepsilon,\infty}^{\sharp}.
$$

\subsubsection{Measure estimates} First, we focus on property $\mathrm{(iii)}$.
\begin{lemma} \label{lem:full_meas_last}$\Xi_\varepsilon$ is a Borel subset of $\mathscr{O}_\varepsilon$ of full measure, i.e.
$$
\mathrm{Leb}(\Xi_\varepsilon) = \mathrm{Leb}(\mathscr{O}_\varepsilon).
$$
\end{lemma}
Since  $\mathscr{O}_\varepsilon$ is asymptotically of full measure in $\Xi_0$ (see \eqref{eq:asympt_full_meas_Oe} ), as corollary, we directly get property $\mathrm{(iii)}$ of Theorem \ref{thm:main}, i.e.
$$
\lim_{\epsilon\to 0}\mathrm{Leb}(\Xi_\epsilon)=\mathrm{Leb}(\Xi_0),
$$
\begin{proof}[Proof of Lemma \ref{lem:full_meas_last}] Recalling that $\mathcal{O}_{\varrho,\varepsilon,\infty}^\sharp$ is a  subset of $\mathcal{O}_{\varrho,\varepsilon,\infty}$ (see \eqref{eq:def_O_sharp_inf}) and the decay property \eqref{eq:decay_tilde}, it is clear that $\Xi_\varepsilon$ (and actually each set $\Xi_\varepsilon^{(k_0)}$) is a Borel subset of $\mathscr{O}_\varepsilon$. Now, we aim at proving that it is of full measure. It suffices to prove that each set $\Xi_\varepsilon^{(k_0)}$ is of full measure.  So let $k_0 \geq 1$.
By definition, of the probability measure $\mathbb{P}_{\varrho,\varepsilon}$ and of the set $\widetilde{\mathcal{O}}_{\varepsilon,1}$, we have, for all $k\geq k_0$,
$$
 \frac{\mathrm{Leb}(\Xi_\varepsilon^{(k_0)}) }{\mathrm{Leb}(\mathscr{O}_\varepsilon)} \geq  \frac{\mathrm{Leb}( \Pi_{\mathcal{S}_1}  \mathcal{O}_{2^{-k},\varepsilon,\infty}^{\sharp} ) }{\mathrm{Leb}(\mathscr{O}_\varepsilon)}  = \mathbb{P}_{2^{-k},\varepsilon} (\widetilde{\mathcal{O}}_{\varepsilon,1}  \cap   \Pi_{\mathcal{S}_1}^{-1}\Pi_{\mathcal{S}_1} \mathcal{O}_{2^{-k},\varepsilon,\infty}^\sharp) \geq \mathbb{P}_{2^{-k},\varepsilon} (\mathcal{O}_{2^{-k},\varepsilon,\infty}^\sharp).
$$
Then using the measure estimates of Lemma \ref{lem:asympt_full_inf} and Lemma \ref{lem:full_meas_sharp}, we have
$$
\mathbb{P}_{2^{-k},\varepsilon} ( \mathcal{O}_{2^{-k},\varepsilon,\infty}^\sharp) = \mathbb{P}_{2^{-k},\varepsilon} ( \mathcal{O}_{2^{-k},\varepsilon,\infty}) \geq 1 - 2^{-10^{-5}k}.
$$
Thus, just let $k$ go to  $+\infty$ to have
$$
 \frac{\mathrm{Leb}(\Xi_\varepsilon^{(k_0)}) }{\mathrm{Leb}(\mathscr{O}_\varepsilon)}  \geq 1 -  2^{-10^{-5}k} \mathop{\longrightarrow}_{k \to +\infty} 1.
$$
\end{proof}
\subsubsection{Accumulation of infinite dimensional non resonant tori} 
Finally, to conclude the proof of Theorem \ref{thm:main}, we just have to check property $\mathrm{(v)}$. So let $\rho <1$  and $\xi \in  \Xi_{\varepsilon}$. By monotony, there exists $k_0 \geq 1$ such that $2^{-k_0/460}<\rho $. But by definition of $\Xi_{\varepsilon}$, there exists $k\geq k_0$ such that
$$
\xi \in \Pi_{\mathcal{S}_1} \mathcal{O}_{2^{-k},\varepsilon,\infty}^{\sharp}, 
$$
i.e. there exists $\zeta \in \mathcal{O}_{2^{-k},\varepsilon,\infty}^{\sharp}$ such that $\xi = \Pi_{\mathcal{S}_1} \zeta$. Thus, by Proposition \ref{prop:ultimate}, $\mathcal{T}_\xi^{\varepsilon,\rho}:=  \mathcal{T}_{2^{-k},\varepsilon,\zeta}^{(\mathrm{nr})}$ is a non resonant infinite dimensional Kronecker torus close enough to $ \mathcal{T}_{\xi}^{\varepsilon}$, i.e. 
\begin{equation}\label{eq:haus}
\mathrm{dist}(\mathcal{T}_{\xi}^{\varepsilon,\rho} , \mathcal{T}_{\xi}^{\varepsilon})\leq \varepsilon (2^{-k})^{\frac1{460} }\leq  \eps\rho\le \rho.
\end{equation}

\section{Appendix}

\subsection{A Cauchy Kowaleskaya lemma}
In this section, we state and prove a Cauchy Kowaleskaya lemma. Actually this result is standard and a very similar one can be found in \cite{Saf95}. We prefer to state and prove our version which is adapted to our Lemmas \ref{lem:compo_reg} and \ref{lem:compo}.

\begin{lemma}\label{lem:CK} Let $E:=\Pi_{p\geq 1}E^{(p)}$ where, for each $p\geq 1$, $E^{(p)}$ is a Banach space for the norm $\|\cdot\|_{E^{(p)}}$ and, for $\eta>0$,
 $$
 E_\eta:=\{H\equiv(H_p)_{p\geq 1}\in E\mid \| H\|_{\eta}:=\sup_{p\geq 1} e^{p\eta}\|H_p\|_{E^{(p)}} <\infty\}.
 $$
Let $H^{(0)}\in E_\eta$ with $\eta>0$. Let $L$ be a linear map which maps $ E_\beta$ into $ E_{\beta-\a}$ for all $0<\a\leq\beta\leq \eta$, and satisfies, for some constant $C_L$, the estimate
\begin{equation}\label{CL}
\|(LH)_p\|_{E^{(p)}}\leq C_L p e^{-p\beta}\| H\|_\beta,\quad H\in E_\beta,\ 0<\beta\leq\eta,\ p\geq 1.
\end{equation}
Take $\nu=2C_L$ and $T<\frac\eta{\nu}$ then there exists a  solution $H\in C^1([0,T];E_{\eta-\nu T})$ to the linear differential equation
\begin{equation}\label{eq:diff}\left\{\begin{array}{ll}
\partial_tH(t)=LH(t),\ \ t\in[0,T]\\ H(0)=H^{(0)}
\end{array}\right.
\end{equation}
satisfying
\begin{equation}\label{estim:CL}
\sup_{0\leq t\leq T}\| H(t)\|_{\eta-\nu t}\leq 2\| H^{(0)}\|_\eta.
\end{equation}
\end{lemma}
Notice that, since $pe^{-\alpha p}\leq \frac1{e\alpha}$ for all $p\geq 1$, \eqref{CL} implies that $L$ is continuous from $E_\beta$ to $E_{\beta-\a}$ for all $0<\a\leq\beta\leq \eta$ and
$$
\| LH\|_{\beta-\a}\leq \frac{C_L}{e\a}\| H\|_\beta
$$
which is actually the hypothesis in \cite{Saf95}.

\begin{proof} As usual we prefer to solve the integral equation. 
We set 
$$F_\nu(T):=\{ H\in C([0,T];E_{\eta-\nu T})\mid \|H\|_{\eta,\nu,T}:=
\sup_{0\leq t\leq T}\| H(t)\|_{\eta-\nu t}\leq 2\| H^{(0)}\|_\eta\}$$
and  we search for $H\in F_\nu(T)$ solution of
\begin{equation}\label{eq:int}
H(t)=H^{(0)}+\int_0^t LH(\tau)\dd\tau=:G(H)(t),\ t\in[0,T].
\end{equation}
We note that 
 $F_\nu(T)$ is a complete space. For $t\in[0,T]$ and $\tau\in(0,t)$ we have  $ LH(\tau)\in E_{\eta-\nu t}$ and
$$\| LH(\tau)\|_{\eta-\nu t}=\sup_{p\geq1} e^{-\nu t p}e^{p\eta}\|(LH(\tau))_p\|_{E^{(p)}}.$$
Then  using \eqref{CL} we have
$$\|(LH(\tau))_p\|_{E^{(p)}}\leq C_L pe^{-(t-\tau)\nu  p} \| H(t)\|_{\eta-\nu \tau}.$$
Thus, since $\sup_{0\leq t\leq T}\| H(t)\|_{\eta-\nu t}\leq 2\| H^{(0)}\|_\eta$, we get
$$\| \int_0^t LH(\tau)\dd \tau\|_{\eta-\nu t}\leq 2C_L\|H^{(0)}\|_\eta\sup_{p\geq 1} \int_0^t pe^{-p\nu(t-\tau)}\dd\tau.$$
We conclude, using $\int_0^t pe^{-p\nu\tau}\dd\tau\leq \frac1\nu$, that 
$$\|G(H)(t)\|_{\eta-\nu t}\leq \|H^{(0)}\|_\eta+ \| \int_0^t LH(\tau)\dd \tau\|_{\eta-\nu t}\leq (1+\frac{2C_L}{\nu})\|H^{(0)}\|_\eta=2\|H^{(0)}\|_\eta$$
and thus $G(H)$ is still in $F_\nu(T)$. Furthermore following the same lines we obtain the Lipschitz estimate
$$\|G(H)-G(H')\|_{\eta,\nu,T}\leq \frac{C_L}{\nu}\|H-H'\|_{\eta,\nu,T}\leq \frac12\|H-H'\|_{\eta,\nu,T}.$$
Therefore by the fix point Theorem we obtain that \eqref{eq:int} has a unique solution in $F_\nu(T)$. Then we easily verify that a function in $F_\nu(T)$ solution of \eqref{eq:int} belongs to $C^1([0,T];E_{\eta-\nu T})$ and satisfies \eqref{eq:diff} and \eqref{estim:CL}.
\end{proof}

\subsection{Infinite dimensional invariant tori and almost periodic functions}

In this subsection, we give a general lemma useful to prove that some functions are almost periodic but not quasi periodic.
\begin{lemma} \label{lem:ap} Let $E$ be a real normed vector space, $\mathcal{T} \subset E$, $\Psi : \mathbb{T}^\mathbb{N} \to \mathcal{T}$ be a homeomorphism and $\omega \in \mathbb{R}^{\mathbb{N}}$ be a sequence of real numbers generating an infinite dimensional $\mathbb{Q}$-vector space (i.e. $\mathrm{dim}_{\mathbb{Q}} \, \mathrm{Span}_{\mathbb{Q}} \omega = \infty$). Then, the function $u:\mathbb{R} \to E$ defined by $u(t) = \Psi(\omega t)$, $t\in \mathbb{R}$, is an almost periodic function which is not quasi periodic.
\end{lemma}

\begin{corollary} \label{cor:ap_functions} The solutions of \eqref{eq:NLS} generated by initial data in non resonant infinite dimensional Kronecker tori are almost periodic and are not quasi periodic.
\end{corollary}

\begin{proof}[Proof of Lemma \ref{lem:ap}]
First, we recall that $\mathbb{T}^\mathbb{N}$ is a metric space for the distance 
$$
\mathrm{d}(\theta,\varphi) := \sum_{n\in \mathbb{N}} 2^{-n} |e^{\ic \theta_n} - e^{\ic \varphi_n}|, \ \quad \theta,\varphi \in \mathbb{T}^{\mathbb{N}}.
$$
In other word, the topology induced by $d$ in the product topology. Being given a function $\Phi : (M_1,d_1)\to (M_2,d_2)$ between two metric spaces, we define its modulus of continuity by
$$
m_\Phi(\varepsilon) := \sup_{\mathrm{d_1}(\theta,\varphi) \leq \varepsilon} \mathrm{d}_2( \Phi(\theta) ,  \Phi(\varphi)), \quad \varepsilon>0.
$$
We recall that, if, moreover, $\Phi$ is continuous and $M_1$ is compact, we have (by Heine's theorem) 
$$
 m_\Phi(\varepsilon)  \mathop{\longrightarrow}_{\varepsilon \to 0} 0.
$$

\medskip

\noindent $\bullet$ \emph{Step $1$ : $u$ is almost periodic.} We denote by $\omega^{( N)} \in \mathbb{R}^\mathbb{N}$ the vector defined by $\omega^{( N)}_n = \mathbbm{1}_{n\leq N} \omega_n $ and we set
$$
u^{( N)}(t)  = \Psi(\omega^{( N)} t), \quad t\in \mathbb{R}.
$$
First, we note that by continuity of $\Psi$, $u^{( N)}$ is a quasi-periodic function. Moreover, still by continuity of $\Psi$ and by compactness of $\mathbb{T}^{\mathbb{N}}$ (by Tychonoff's theorem), we have
$$
\| u^{( N)}(t) - u(t) \|_{E} \leq m_{\Psi}( \mathrm{d}( \omega t, \omega^{(N)} t ) ) \leq m_{\Psi}(2^{-N+2}) \mathop{\longrightarrow}_{N \to +\infty} 0.
$$ 
Therefore, $u^{( N)}$ converges uniformly to $u$ as $N$ goes to $+\infty$. Hence, $u$ is almost periodic. 

\medskip

\noindent $\bullet$ \emph{Step $2$ : $u$ is not quasi periodic.} Let $n\geq 1$, $\lambda \in \mathbb{R}^n$ and $Q : \mathbb{T}^n \to E$ be a continuous function. We aim at proving that there exists $t \in \mathbb{R}$ such that $Q(\lambda t) \neq \Psi(\omega t)$.

\medskip

First, even if it means reindexing the indices, we note that there exists $m\leq n$ such that $(\lambda_j)_{0\leq j < m}$ is a basis of $\mathrm{Span}_{\mathbb{Q}} \, \lambda$. As a consequence, there exist
$a \in \mathbb{Z}^*$ and $b \in \mathbb{Z}^{m\times n}$, such that for all $ j <n$,
\begin{equation}
\label{eq:combin_last}
a \lambda_ j = \sum_{ i< m} b_{i,j} \lambda_i.
\end{equation}
Furthermore, since $\mathrm{dim}_{\mathbb{Q}} \, \omega = \infty$, we note that there exists $N\in \mathbb{N}$ such that
$$
\omega_N \notin \mathrm{Span}_{\mathbb{Q}} \, \lambda.
$$
Then we define $\nu \in \mathbb{R}^{m+1}$ by
$$
\forall i<m, \quad \nu_i = \lambda_i a^{-1} \quad \mathrm{and} \quad \nu_m := \omega_N.
$$
We note that by construction the  frequencies $(\nu_i)_{i\leq m}$ are rationally independent. It follows that $(t \nu)_{t\in \mathbb{R}}$ is dense in $\mathbb{T}^{m+1}$.
As a consequence, for all $\varepsilon>0$, there exists $t_\varepsilon \in \mathbb{R}$ such that
$$
\forall j<m, \quad |e^{\ic t_{\varepsilon} \nu_j} - 1| \leq \varepsilon \quad \mathrm{and} \quad |e^{\ic t_{\varepsilon} \nu_m} + 1| \leq \varepsilon.  
$$
By \eqref{eq:combin_last} and by construction of $\nu$, it follows that
$$
\forall j<n, \quad |e^{\ic t_{\varepsilon} \lambda_j} - 1| \lesssim_b \varepsilon \quad \mathrm{and} \quad |e^{\ic t_{\varepsilon} \omega_N} + 1| \leq \varepsilon.  
$$
Therefore, $t_\varepsilon \lambda$ goes to $0$ (in $\mathbb{T}^n$) as $\varepsilon$ goes to $0$ and, so by continuity of $Q$, we have
$$
Q(t_\varepsilon \lambda) \mathop{\longrightarrow}_{\varepsilon \to 0} Q(0).
$$
Moreover, provided that $\varepsilon\leq 1$, we have
$$
m_{\Psi^{-1}}(\|\Psi(0) - \Psi(t_\varepsilon \omega) \|_{E}) \geq  \mathrm{d}(0,t_\varepsilon \omega) \geq  2^{-N} >0.
$$
Noticing that, by continuity of $\Psi^{-1}$ and compactness of $\mathcal{T}$, $ m_{\Psi^{-1}}(\epsilon)$ goes to $0$ as $\epsilon$ goes to $0$, it follows that 
$$
\inf_{0<\varepsilon \leq 1 }\|\Psi(0) - \Psi(t_\varepsilon \omega) \|_{E} >0.
$$
As a consequence either $\Psi(0)\neq Q(0)$, or, provided that $\varepsilon$ is small enough, $\Psi(t_\varepsilon \omega)\neq Q(t_\varepsilon \lambda)$.

\end{proof}

\subsection{Technical lemmas}
This subsection contains two useful technical lemmas. First, we mention a quite direct application of the mean value inequality on annuli (which are not convex).
\begin{lemma}\label{lem:appendix}
Let $\mathcal{S} \subset \mathbb{Z}$ be a finite set, $r\in (0,\frac14)$, $\xi \in ((4r)^2,1)^{\mathcal{S}}$ and $g \in C^1(\mathcal{A}_\xi(r) ; \ell^2_\varpi) $. Then for all $u,v \in \mathcal{A}_\xi(r)$, we have
$$
\| g(u) - g(v) \|_{\ell^2_\varpi} \leq (1+ (2\pi)^{\# \mathcal{S}})  \| u-v\|_{\ell^2_\varpi} \sup_{w\in \mathcal{A}_\xi(r)} \| \mathrm{d}g(w) \|_{\ell^2_\varpi \to \ell^2_\varpi}.
$$
\end{lemma}

Then, we prove in our context that symplectic changes of variable preserve the Hamiltonian structure. 
 \begin{lemma} \label{lem:sympl_chgtvar} Let $A,B \subset \ell^2_\varpi$ be open,  $I\subset \mathbb{R}$ be an interval, $H : B \to \mathbb{R}$ be $C^1$, $\Psi : A \to B$ be $C^1$ and $u : I \to A$ be $C^1$. Assume that
 \begin{itemize}
 \item $u$ is solution to 
 $$
 i\partial_t u = \nabla (H\circ \Psi)(u)
 $$
 \item $\Psi$ is symplectic and for all $w\in A$, $\mathrm{d}\Psi(w)$ is surjective.
 \end{itemize}
 Then, $v = \Psi(u)$ is solution to 
 $$
 i\partial_t v = \nabla H(v).
 $$
 \end{lemma}
 \begin{proof}
Let $t\in I$ and $z \in \ell^2_\varpi$. Since  $\mathrm{d}\Psi(v(t))$ is surjective, there exists $w\in \ell^2_\varpi$ such that $z = \mathrm{d}\Psi(u(t))(w)$. It follows that, since $\Psi$ is symplectic we have
\begin{equation*}
\begin{split}
 (i\partial_t v(t),z )_{\ell^2} &= (i\partial_t \mathrm{d}\Psi(u(t)) (\partial_t u(t)), \mathrm{d}\Psi(v(t))(w) )_{\ell^2} 
 =(i\partial_t u(t), w )_{\ell^2}
 = ( \nabla (H\circ \Psi)(u(t)), w )_{\ell^2} \\
 &= \mathrm{d} (H\circ \Psi)(u(t))(w) 
 = \mathrm{d}H(v(t)) (\mathrm{d}\Psi(u(t))(w)) 
 =  \mathrm{d}H(v(t)) (z) 
 =  (\nabla H(v(t)),z )_{\ell^2}.
\end{split}
\end{equation*}
This property being true for all $z \in \ell^2_\varpi$, we have proven, by duality,  that  $i\partial_t v(t) = \nabla H(v(t))$.

\end{proof}
%
%Finally, we just mention, the two following basic estimates.
%\begin{lemma}
%\label{lem:un_peu_relou_conv} Let $q\geq 2$, $\boldsymbol{\ell} \in \mathbb{Z}^q$ and $\boldsymbol{\sigma} \in \{-1,1\}^q$ be such that $\boldsymbol{\sigma}_1 \boldsymbol{\ell}_1 + \cdots + \boldsymbol{\sigma}_q \boldsymbol{\ell}_q = 0$ and $|\boldsymbol{\ell}_1| \geq \cdots \geq |\boldsymbol{\ell}_q| $. Then, we have
%$$
%e^{\mathfrak{a} |\boldsymbol{\ell}_1|} \leq e^{\mathfrak{a} |\boldsymbol{\ell}_2|} \cdots e^{\mathfrak{a} |\boldsymbol{\ell}_q|} \quad \mathrm{and} \quad \langle \boldsymbol{\ell}_1 \rangle^s \leq (q-1)^{s} \langle \boldsymbol{\ell}_2 \rangle^s.
%$$
%\end{lemma}

\end{document}